\documentclass[11pt,english]{smfart}

\usepackage{amssymb}
\usepackage{amsmath}
\usepackage{amsfonts}
\usepackage{amsthm}
\usepackage{stmaryrd}
\usepackage[all]{xy}
\usepackage{mathrsfs}
\usepackage{graphicx}
\usepackage[colorlinks=true]{hyperref}
\usepackage{color}
\usepackage{multirow}
\usepackage{extarrows}

\usepackage[margin=1.1in]{geometry}

\numberwithin{equation}{subsection}
\newtheorem{theorem}[subsection]{Theorem}

\newtheorem{thm}{Theorem}

\newtheorem{conj}[subsection]{Conjecture}
\newtheorem{cor}[subsection]{Corollary}
\newtheorem{lemma}[subsection]{Lemma}
\newtheorem{prop}[subsection]{Proposition}

\theoremstyle{definition}

\newtheorem{example}[subsection]{Example}

\newtheorem{notation}[subsection]{Notation}

\newtheorem{remark}[subsection]{Remark}

\newtheorem{hypo}[subsection]{Hypothesis}

\newcommand{\ra}{\rightarrow}
\newcommand{\xra}{\xrightarrow}
\newcommand{\hra}{\hookrightarrow}

\def\AAA{\mathbb{A}}

\def\CC{\mathbb{C}}
\def\DD{\mathbb{D}}

\def\FF{\mathbb{F}}
\def\GG{\mathbb{G}}
\def\HH{\mathbb{H}}

\def\LL{\mathbb{L}}

\def\NN{\mathbb{N}}

\def\PP{\mathbb{P}}
\def\QQ{\mathbb{Q}}
\def\RR{\mathbb{R}}
\def\SSS{\mathbb{S}}

\def\ZZ{\mathbb{Z}}

\def\bfA{\mathbf{A}}

\def\bfX{\mathbf{X}}

\def\calA{\mathcal{A}}

\def\calC{\mathcal{C}}
\def\calD{\mathcal{D}}
\def\calE{\mathcal{E}}
\def\calF{\mathcal{F}}

\def\calH{\mathcal{H}}

\def\calM{\mathcal{M}}

\def\calO{\mathcal{O}}
\def\calP{\mathcal{P}}

\def\calW{\mathcal{W}}

\def\gothc{\mathfrak{c}}
\def\gothd{\mathfrak{d}}
\def\gothe{\mathfrak{e}}

\def\gothh{\mathfrak{h}}

\def\gothl{\mathfrak{l}}

\def\gothn{\mathfrak{n}}

\def\gothp{\mathfrak{p}}
\def\gothq{\mathfrak{q}}

\def\gothF{\mathfrak{F}}
\def\gothG{\mathfrak{G}}
\def\gothH{\mathfrak{H}}

\def\gothN{\mathfrak{N}}

\def\gothR{\mathfrak{R}}
\def\gothS{\mathfrak{S}}

\def\scrA{\mathscr{A}}
\def\scrB{\mathscr{B}}
\def\scrC{\mathscr{C}}
\def\scrD{\mathscr{D}}
\def\scrE{\mathscr{E}}
\def\scrF{\mathscr{F}}

\def\scrH{\mathscr{H}}

\def\scrL{\mathscr{L}}

\def\rmM{\mathrm{M}}

\def\ttS{\mathtt{S}}
\def\ttT{\mathtt{T}}

\DeclareMathOperator{\Art}{Art}
\DeclareMathOperator{\Aut}{Aut}
\DeclareMathOperator{\Cone}{Cone}
\DeclareMathOperator{\End}{End}
\DeclareMathOperator{\Gal}{Gal}
\DeclareMathOperator{\Ker}{Ker}

\DeclareMathOperator{\Hom}{Hom}

\DeclareMathOperator{\Ind}{Ind}

\DeclareMathOperator{\Lie}{Lie}

\DeclareMathOperator{\Res}{Res}
\DeclareMathOperator{\Spec}{Spec}
\DeclareMathOperator{\Proj}{Proj}

\newcommand{\cl}{\mathrm{cl}}

\newcommand{\Q}{\mathbb{Q}}
\newcommand{\Z}{\mathbb{Z}}
\newcommand{\R}{\mathbb{R}}
\newcommand{\C}{\mathbb{C}}
\newcommand{\G}{\mathbb{G}}
\newcommand{\cO}{\mathcal{O}}
\newcommand{\kb}{\underline{k}}

\newcommand{\Fgal}{F^{\mathrm{Gal}}} % Galois closure of F in Qbar.
\newcommand{\Qb}{\overline{\QQ}} %algebraic closure of Q
\newcommand{\Qpb}{\overline{\QQ}_p}
\newcommand{\Fpb}{\overline{\FF}_p}
\newcommand{\M}{\mathcal{M}} %  geometric Hilbert modular scheme.

\newcommand{\sm}{\mathrm{sa}} % semi-abelian
\newcommand{\dR}{\mathrm{dR}}

\newcommand{\univA}{\calA}
\newcommand{\univC}{\calC} % universal canonical subgroup
\newcommand{\fX}{\mathfrak{X}}
\newcommand{\X}{\mathfrak{X}} %  arithmetic Hilbert modular variety
\newcommand{\omegab}{\underline{\omega}} %the universal bundle of invariant differential forms
\newcommand{\tor}{\mathrm{tor}} % toroidal compactification.
\newcommand{\Kod}{\mathrm{Kod}} % Kodaira-Spencer map
\newcommand{\cris}{\mathrm{cris}}
\newcommand{\F}{\mathscr{F}} %automorphic vector bundle
\newcommand{\Sym}{\mathrm{Sym}} %Symmetric power
\newcommand{\D}{\mathsf{D}} % boundary of the toroidal compactification
\newcommand{\pr}{\mathrm{pr}}
\newcommand{\Def}{\mathcal{D}ef}
\newcommand{\Fil}{\mathrm{Fil}}
\newcommand{\DR}{\mathrm{DR}^{\bullet}}

\newcommand{\Gr}{\mathrm{Gr}}
\newcommand{\ord}{\mathrm{ord}}

\newcommand{\tF}{\mathtt{F}}
\newcommand{\BGG}{\mathrm{BGG}^{\bullet}}
\newcommand{\bgg}{\mathrm{BGG}}

\newcommand{\rb}{\underline{r}}
\newcommand{\etab}{\overline{\eta}}
\newcommand{\Spf}{\mathrm{Spf}}

\newcommand{\xb}{\overline{x}}

\newcommand{\ab}{\mathrm{ab}}

\newcommand{\bbalpha}{\boldsymbol{\alpha}}

\newcommand{\bfSh}{\mathbf{Sh}}

\newcommand{\CW}{\mathrm{CW}}

\newcommand{\et}{\mathrm{et}}

\newcommand{\Fr}{\mathrm{Fr}}
\newcommand{\Frob}{\mathrm{Frob}}

\newcommand{\gothRec}{\gothR\mathrm{ec}}
\newcommand{\GL}{\mathrm{GL}}

\newcommand{\Iw}{\mathrm{Iw}}

\newcommand{\per}{\mathrm{per}}
\newcommand{\prim}{\mathrm{prim}}

\newcommand{\Qp}{\QQ_p}
\newcommand{\res}{\mathrm{res}}
\newcommand{\rig}{\mathrm{rig}}

\newcommand{\Sh}{\mathrm{Sh}}

\newcommand{\SL}{\mathrm{SL}}

\newcommand{\Tr}{\mathrm{Tr}}

\newcommand{\ttr}{\mathtt{r}}
\newcommand{\ur}{\mathrm{ur}}
\newcommand{\Zp}{\ZZ_p}

\newcommand{\Qlb}{\overline{\Q}_l}
\newcommand{\Ob}{\mathrm{Ob}}

\newcommand{\tSigma}{\widetilde{\Sigma}}
\newcommand{\tL}{\widetilde{\scrL}}
\newcommand{\sbar}{\overline{s}}
\newcommand{\val}{\mathrm{val}}
\newcommand{\uk}{\underline{k}}
\newcommand{\St}{\mathrm{St}}
\newcommand{\Stb}{\check{\St}}

\newcommand{\fZ}{\mathfrak{Z}}

\newcommand{\spe}{\mathrm{sp}}
\newcommand{\Id}{\mathrm{Id}}
\newcommand{\Gammab}{\underline{\Gamma}}

\renewcommand{\proofname}{Proof}

\begin{document}

\title[$p$-adic cohomology and classicality]{$p$-adic cohomology and classicality of overconvergent Hilbert modular forms}

\author{Yichao Tian} 
\address{Morningside Center of Mathematics,
Chinese Academy of Sciences,
55 Zhong Guan Cun East Road,
Beijing, 100190, China}
\email{yichaot@math.ac.cn}

\author{Liang Xiao}
\address{University of Connecticut, Storrs, Department of
Mathematics, 196 Auditorium Road, Unit 3009, Storrs, CT 06250, U.S.A.}
\email{liang.xiao@uconn.edu}

\begin{abstract}
Let $F$ be a totally real field in which a prime number $p$ is unramified.
We prove that, if a cuspidal overconvergent Hilbert modular form has small slopes under the $U_p$-operators, then  it is classical.
Our method follows the original cohomological approach of R. Coleman.
The key ingredient of the proof is giving an explicit description of the Goren-Oort stratification of the special fiber of the Hilbert modular variety.
As a byproduct of the proof, we show that, at least when $p$ is inert,  the rigid cohomology of the ordinary locus is equal to the space of classical forms in the Grothendieck group of finite-dimensional modules of the Hecke algebras.
\end{abstract}


\maketitle
\tableofcontents

\section{Introduction}

The classicality results for $p$-adic overconvergent modular forms started with the pioneering work of R. Coleman \cite{coleman}, in which he proved that an overconvergent modular form of weight $k$ and slope $<k-1$ is  classical.  
Coleman proved his theorem using  $p$-adic cohomology and an ingenious  dimension counting argument. Later, P. Kassaei \cite{kassaei} reproved Coleman's theorem based on an analytic continuation result by K. Buzzard \cite{buzzard}.   
In the Hilbert case,   S. Sasaki \cite{sasaki} proved the classicality of small slope overconvergent Hilbert modular forms when the prime $p$ is  totally split in the concerning totally real field.
  With a less optimal slope condition, such a classicality result for  overconvergent Hilbert modular forms was proved by the first named author \cite{tian} in the quadratic inert case, and by V. Pilloni and B. Stroh in the general unramified case \cite{ps}. The methods of \cite{sasaki,tian,ps} followed that of  Kassaei, and used the analytic continuation of overconvergent Hilbert modular forms.

In this paper, we will follow Coleman's original cohomological approach to prove the classicality of cuspidal overconvergent Hilbert modular forms. Let us describe our main results in details.
 We fix a prime number $p$.
Let $F$ be a totally real field of degree $g=[F:\Q]\geq 2$ in which $p$ is unramified,
and   denote by  $\gothp_1, \dots, \gothp_r$ the primes of $F$ above $p$. Let $\Sigma_{\infty}$ be the set of archimedean places of $F$.
We fix an isomorphism $\iota_p: \CC \cong \overline \QQ_p$. 
 For each $\gothp_i$, we  denote by $\Sigma_{\infty/\gothp_i}$ the subset of archimedean places $\tau\in \Sigma_{\infty}$  such that $\iota_p\circ\tau$ induce the prime $\gothp_i$. 
%Let $\AAA^{\infty}_F$ denote the ring of finite adeles of $F$.
We fix an ideal $\gothN$ of $\calO_F$ coprime to $p$.
% and let $ We fix an open compact subgroup $K=K^pK_{p}\subset \GL_2(\AAA_{F}^{\infty})$ such that $K_p=\GL_2(\cO_F\otimes \Z_p)$ is hyperspecial, and 
We consider the following level structures:
\begin{eqnarray*}
&K_1(\gothN) =
\bigg\{
\begin{pmatrix}a &b\\c&d\end{pmatrix}\in \GL_2(\widehat{\cO}_F)\;\Big|\; a\equiv 1, c\equiv 0 \bmod \gothN
\bigg\}; \\
& K_1(\gothN)^p\Iw_p = \bigg\{
\begin{pmatrix}a &b\\c&d\end{pmatrix}\in K_1(\gothN)\;\Big|\; c\equiv 0 \bmod p
\bigg\}.
\end{eqnarray*}
Consider a  multi-weight $(\underline k, w)\in \NN^{\Sigma_{\infty}}\times \NN$ such that $w\geq k_{\tau}\geq 2$ and $k_{\tau}\equiv w\pmod 2$ for all $\tau$ (such a multi-weight will be called cohomological).
 The convention on weights in this paper is adapted for arithmetic applications: each archimedean component of the automorphic representation associated to a cuspidal Hilbert eigenform of multiweight $(\underline k, w)$ has central character $t\mapsto t^{w-2}$.  This agrees with \cite{saito}. Our first main theorem is the following
\begin{thm}[Theorem~\ref{T:strong classicality}]
\label{T:classicality theorem introduction}

Let $f$ be a cuspidal  overconvergent Hilbert modular form of  multiweight $(\underline{k}, w)$ and level $K_1(\gothN)$, which is an eigenform for all Hecke operators. 
Let $\lambda_{\gothp_i}$ denote the eigenvalue of $f$ for the operator $U_{\gothp_i}$ for $1\leq i\leq r$.
If the $p$-adic valuation of each $\lambda_{\gothp_i}$ satisfies  
\begin{equation}\label{E:slope-bound}
\val_{p}(\lambda_{\gothp_i})<\sum_{\tau\in \Sigma_{\infty/\gothp_i}} \frac{w-k_{\tau}}{2} + \min_{\tau\in \Sigma_{\infty/\gothp_i}} \{k_{\tau}-1\},
\end{equation}
 then $f$ is a classical (cuspidal) Hilbert  eigenform of level $K_1(\gothN)^p\Iw_p$.
\end{thm}
Here,  we normalize the $p$-adic valuation $\val_p$ so that $\val_p(p)=1$. 
The term $\sum_{\tau} \frac{w-k_{\tau}}{2}$ is a normalizing factor that appears in the definition of cuspidal overconvergent Hilbert modular forms. Any cuspidal overconvergent Hilbert eigenform has $U_{\gothp_i}$-slope greater than or equal to this quantity. 
Up to this normalizing factor, Theorem~\ref{T:classicality theorem introduction} was proved in  \cite{ps} (and also in \cite{tian} for the quadratic case) with slope bound $\val_p(\lambda_{\gothp_i})<\sum_{\tau\in \Sigma_{\infty/\gothp_i}}\frac{w-k_{\tau}}{2}+\min_{\tau\in \Sigma_{\infty/\gothp_i}}(k_\tau-[F_{\gothp_i}:\Q_p])$.
The slope bound \eqref{E:slope-bound}, believed to be optimal, was conjectured by Breuil in an unpublished note \cite{breuil}, which significantly inspires this work. 
In Theorem \ref{T:strong classicality}, we  also give some classicality results using theta operators if the slope bound \eqref{E:slope-bound} is not satisfied, as conjectured by Breuil in \emph{loc. cit.} 
 Finally,  Christian Johansson \cite{johansson} also obtained independently in his thesis similar results for overconvergent automorphic forms for rank two unitary group, but with a  less optimal slope bound.

\medskip
We now explain the proof of our theorem.
As in \cite{coleman}, the first step is to relate the cuspidal overconvergent Hilbert modular forms to a certain $p$-adic cohomology group of the Hilbert modular variety. 

We take the level structure $K=K^pK_p$ to be hyperspecial at places above $p$.
 Let $K^p\Iw_p$ denote the corresponding level structure with the same tame level $K^p$ and with Iwahori group at all places above $p$.
Let $\bfX$ be the integral model of the Hilbert modular variety of level $K$ defined over the ring of integers of a finite  extension $L$ over $\Q_p$. We choose  a toroidal compactification $\bfX^\tor$ of $\bfX$.   Let $X^\tor$ and $X$ denote respectively  the  special fibers of $\bfX^{\tor}$ and $\bfX$ over $\Fpb$, and $\D$ be the boundary $X^\tor-X$.  Let $X^{\tor,\ord}$ be the ordinary locus of $X^\tor$.  %We denote by $\gothX^{\tor}_{\rig}$ denote the associated rigid analytic space of $\bfX^\tor\otimes_{\cO_L}L$, and by $\gothX^{\tor,\ord}_{\rig}$ the ordinary locus of $c$the tube of $X^{\tor,\ord}
Let $\scrF^{(\underline{k}, w)}$ denote the corresponding overconvergent log-$F$-isocrystal sheaf of multiweight $(\uk, w)$ on $X^\tor$, and let $S^{\dagger}_{(\uk, w)}$ denote the space of cuspidal overconvergent  Hilbert modular forms. We consider  the rigid cohomology of $\F^{(\kb,w)}$ over the ordinary locus of $X^{\tor}$ with compact support at cusps, denoted by $H^\star_{\rig}(X^{\tor,\ord},\D; \F^{(\kb,w)})$ (see Subsection~\ref{subsection:rigid-coh} for its precise definition).  Using the dual BGG-complex and a cohomological computation due to Coleman \cite{coleman}, we show in Theorem \ref{Theorem:overconvergent} that, the cohomology group above is computed by a complex consisting of cuspidal overconvergent Hilbert modular forms. 

Let us explain more explicitly this result in the case when $F$ is a real quadratic field and $p$ is a prime inert in $F/\QQ$. 
%{\color{blue}We also assume that the weights $(\underline{k}, w) = (k_1, k_2, w)$ such that $k_1, k_2$ are not all equal to $2$.  This allows us to kill the residue spectrum in the calculation. }
Then Theorem~\ref{Theorem:overconvergent} says that the cohomology group $H^\star_{\rig}(X^{\tor,\ord},\D; \F^{(\kb,w)})$ (together with its Hecke action) is computed by the complex 
\[
\scrC^\bullet\colon\ 
S^\dagger_{(2-k_1, 2-k_2,w)} \xrightarrow{(\Theta_1, \Theta_2)}
S^\dagger_{(k_1, 2-k_2,w)} \oplus
S^\dagger_{(2-k_1, k_2,w)}
 \xrightarrow{-\Theta_2 \oplus \Theta_1}
S^\dagger_{(k_1, k_2,w)},
\]
where each $\Theta_{i}$ is essentially the $(k_{i}-1)$-times composition of the Hilbert analogues of the well-known $\theta$-operator for the elliptic modular forms. We refer the reader to Subsection~\ref{S:BGG} and Remark~\ref{R:BGG} for the precise expression of $\Theta_{i}$'s, and to \eqref{Equ:complex} for the definition of the complex $\scrC^\bullet$ in the general case.  
%\liang{Read this part again.  Please provide a bit more details of the comparison with [KL05], or at least the origin of the difference.}
Here, the Hecke action on its terms $S_\star^\dagger$ coincides with the one  given in \cite{KL} (see Remark~\ref{remark:U_p} for details), and the complex $\scrC^{\bullet}$ is  Hecke equivariant for this Hecke action. 

 An important fact for us is that, the slope condition \eqref{E:slope-bound} can be satisfied only for eigenforms in the  last term $S^\dagger_{(k_1,k_2,w)}$.
In other words, if an eigenform $f \in S_{(k_1, k_2, w)}^\dagger$ satisfies the slope condition, then it has nontrivial image in the cohomology group $H^g_\rig(X^{\tor, \ord}, \D; \scrF^{(\kb, w)} )$.
This result on $U_p$-action is explained in Corollary~\ref{Prop:slopes-ocv}.

Moreover, the cohomological approach allows us to prove the following strengthened version of Theorem~\ref{T:classicality theorem introduction}: if a cuspidal overconvergent Hilbert modular form $f$ of multiweight $(\underline{k}, w)$ and level $K$ does not lie in the image of all $\Theta$-maps, then $f$ is a classical (cuspidal) Hilbert modular form.

The second step of the proof of Theorem \ref{T:classicality theorem introduction} is to compute  $H^\star_{\rig}(X^{\tor,\ord},\D; \F^{(\kb,w)})$ using the Goren-Oort stratification of $X$. A key ingredient here is  the explicit description of these Goren-Oort strata of $X$ given in \cite{TX-GO}. 
In the quadratic inert case considered above, the main results of \cite{TX-GO} can be described as follows. 
 Let $X_1$ and $X_2$ be respectively the vanishing loci of the two partial Hasse invariants on  $X^{\tor}$.
 %, one for each archimedean place (identified with a $p$-adic embedding by post-composing with $\iota_p$). 
Then according to \cite{goren},  $X_1\cup X_2$ a normal crossing divisor of $X^{\tor}$, and it is complement of the ordinary locus $X^{\tor,\ord}\subseteq X^{\tor}$.  Put $X_{12}=X_1\cap X_2$. 
 It was previously known that each of $X_1$ and $X_2$ is a certain collection of $\PP^1$'s. 
 The main result of \cite{TX-GO} says that  each of $X_1$ and $X_2$ is isomorphic a  $\PP^1$-bundle over $\bfSh_{K}(B_{\infty_1, \infty_2}^\times)_{\overline \FF_p}$,
the special fiber of the  discrete Shimura variety of level $K$ associated to the quaternion algebra $B_{\infty_1, \infty_2}$ over $F$ which ramifies exactly at both archimedean places.  Their intersection $X_{12}$ may be identified with the Shimura variety $\bfSh_{K^p\Iw_p}(B_{\infty_1, \infty_2}^\times)_{\overline \FF_p}$ for the same group but with Iwahori level structure at $p$. Moreover, these isomorphisms are compatible with the tame Hecke actions. 

In the general case, for each subset $\ttT\subset \Sigma_{\infty}$, we consider the closed Goren-Oort stratum $X_{\ttT}$ defined as the vanishing locus of the partial Hasse invariants corresponding to $\ttT$. 
This is a proper and smooth closed subvariety  of $X$ of codimension $\#\ttT$ by \cite{goren}. 
The main result of \cite{TX-GO} shows that $X_{\ttT}$ is a certain  $(\PP^1)^N$-bundle over the special fiber of another quaternionic Shimura variety.
In fact, this result is more naturally stated for the Shimura variety associated to the group $\GL_2(F) \times_{F^\times} E^\times$ with $E$ a quadratic CM extension of $F$. We refer the reader to Section~\ref{Section:GO-stratification} for a more detailed discussion. 
Using this result and the Jacquet-Langlands correspondence, one can  compute the cohomology of each closed Goren-Oort stratum. General formalism of rigid cohomology then produces a spectral sequence which relates the desired cohomology group $H^\star_{\rig}(X^{\tor,\ord},\D; \F^{(\kb,w)})$ to the cohomology of the  closed Goren-Oort strata. In the general case, we prove the following
\begin{thm}[Theorems \ref{Theorem:overconvergent} and \ref{T:rigid coh of ordinary}]
\label{T:ord cohomology}
We have the following equalities  in the Grothendieck group of finite-dimensional modules of the \emph{tame} Hecke algebra $\scrH(K^p, L)$:
\[
\sum_{J \subseteq \Sigma_{\infty}}
(-1)^{\#J}\big[(S^\dagger_{(s_J\cdot \underline{k}, w)})^{\mathrm{slope} \leq T} \big]=
\big[H^\star_\rig(X^{\tor,\ord}, \D; \scrF^{(\underline{k}, w)}) \big] =
(-1)^g \big[ S_{(\underline{k},w)}(K^p\Iw_p) \big],
\]
for $T$ sufficiently large, where
\begin{itemize}
\item
$s_J\cdot \kb\in \ZZ^{\Sigma_{\infty}}$ is the multi-weight whose $\tau$-component is $ k_{\tau}$ for $\tau \in J$,  and is $2-k_\tau$ for  $\tau \notin J$; %(see \eqref{Defn:S_J-dagger} and \eqref{Equ:omega-J} for the precise definition of $S^{\dagger}_{\epsilon_J(\kb,w)}$);
\item
the superscript slope $\leq T$ means to take the \emph{finite dimensional} subspace where the slope of the product of the $U_\gothp$-operators is less than or equal to $T \in \RR$; and
\item
$S_{(\uk,w)}(K^p\Iw_p)$ is the space of classical cuspidal Hilbert modular forms of level $K^p\Iw_p$.
\end{itemize}
%Moreover, we can also compare the (power of) Frobenius action on the middle term with the product of (power of $U_p$-operators action on the other two terms.
\end{thm}

At this point, there are two ways to proceed to get Theorem \ref{T:classicality theorem introduction}.
The first approach is unconditional. We first use Theorem~\ref{T:ord cohomology} to prove the classicality result when the slope is much smaller the weight (Proposition~\ref{P:weak classicality}).
 Then we improve the slope bound by studying  global crystalline periods over the eigenvarieties (Theorem~\ref{T:strong classicality}).
 In fact, we can prove something  stronger: if an eigenform $f$ does not lie in the image of the $\Theta$-maps in the complex $\scrC^\bullet$, then $f$ is a classical Hilbert modular eigenform (Theorem~\ref{T:strong classicality}).
This approach, to some extent, relies on the strong multiplicity one of overconvergent Hilbert modular forms.
This approach is explained in Section~\ref{Section:classicality-I}.

The second approach is more involved, and we need to assume
\begin{itemize}
\item[(1)] either $p$ is inert in $F$, i.e. $p$ stays as a prime in $\calO_F$,
\item[(2)] or the action of the ``partial Frobenius" on the cohomology of quaternionic Shimura variety are as expected by general Langlands conjecture. (See Conjecture~\ref{Conj: partial frobenius})
\end{itemize}
%We are informed that Cornut and Nekov\'a\v r have made progress in the direction of the second hypothesis.
We defer the definition of partial Frobenius to the context of the paper.
Under this assumption, we can strengthen Theorem~\ref{T:ord cohomology} as
\begin{thm}
\label{T:ord cohomology strong}
Assume the assumption above, the equality in Theorem~\ref{T:ord cohomology} is an equality in the Grothendieck group of finite-dimensional $\scrH(K^p, L)[U^2_\gothp; \gothp\in \Sigma_p]$-modules.
\end{thm}

This theorem is proved in Section~\ref{Section:classicality-II}, using a  combinatorially complicated argument.

The reason that Theorem~\ref{T:ord cohomology strong} is stated for the action of $U^2_\gothp$ (instead of $U_{\gothp}$) is the following: the description of the Goren-Oort strata is obtained in \cite{TX-GO} using unitary Shimura varieties, where only the twisted partial Frobenius (instead of partial Frobenius itself) has a group theoretic  interpretation, which is, morally, the  Hecke operator  given by $\big(\begin{smallmatrix}
\varpi_\gothp &0\\ 0&\varpi_\gothp^{-1}
\end{smallmatrix}
\big)$,
where $\varpi_{\gothp}$ denotes the id\`ele of $F$ which is a uniformizer at $\gothp$ and $1$ at other places.
One might be able to fix this small defect by modifying the PEL type unitary moduli problem to a moduli problem for $\GL_{2,F} \times_{F^\times} E^\times$.
%However, we do not plan to pursue this approach.

Now it is a trivial matter to deduce Theorem \ref{T:classicality theorem introduction} from Theorem~\ref{T:ord cohomology strong} (under the assumption above).
In fact, we only need $f$ to be a generalized eigenvector for all the $U_{\gothp_i}$-operators satisfying the slope condition (i.e. $f$ does not have to be an eigenvector for the tame Hecke actions).
The upside of this approach is that one may avoid using the $q$-expansion principle for Hilbert modular forms.
This is crucial when studying other quaternionic Shimura varieties, where the $q$-expansion principle is not available. 
 Moreover, Theorem~\ref{T:ord cohomology strong} is interesting in  its own right, and it gives a concrete description of the rigid cohomology of the ordinary locus.

Another related intriguing question is whether $H^*_\rig(X^{\tor,\ord}, \D; \scrF^{(\underline{k}, w)})$ is concentrated in degree $g$. 
We hope to address this question in the forthcoming paper \cite{TX-Tate}.
It turns out that the result depends on the Satake parameter at $p$ of the corresponding automorphic representation. 

Note also that the same cohomology group (in a more general context) was also used in his recent joint work with M.~Harris, R.~Taylor and J.~Thorne \cite{HLTT}.

%Our consideration of the cohomology group $H^{\star}_{\rig}(X^{\tor,\ord}_{k_0},\D; \scrF^{(\kb,w)})$ was stimulated by a conversation with K. Lan in December 2011, and he explained to us the vanishing result Lemma~\ref{Lemma:vanishing}. 
 %After we finished the first version of the paper and talked to Lan again, we found that 
 %\liang{Why do we say this here?  Are we trying to claim some credit?}

\subsection*{Structure of the paper}
Section~\ref{Section:HMV}
reviews basic facts about Hilbert modular varieties, as well as the dual BGG complex.
We define cuspidal overconvergent Hilbert modular forms in Section \ref{Section:ocgt HMF} and show that the cohomology of the complex $\scrC^\bullet$ of  cuspidal overconvergent Hilbert modular forms computes the rigid cohomology of the ordinary locus (Theorem~\ref{Theorem:overconvergent}).  
Moreover, we show that the slopes of the $U_\gothp$-operators are always greater or equal to the normalizing factor in Theorem~\ref{T:classicality theorem introduction} (Corollary~\ref{Prop:slopes-ocv}).
After this, we set up the spectral sequence that computes the rigid cohomology of the ordinary locus in Section~\ref{Section:spectral sequence}.
The entire Section~\ref{Section:GO-stratification} is devoted to give a description of the cohomology of each Goren-Oort stratum, using the earlier work \cite{TX-GO}.
The last two sections each gives an approach to prove the classicality: one unconditional but with some help from eigenvarieties, and one more straightforward but relying on some conjecture on partial Frobenius actions.

\subsection*{Acknowledgements}

We express our sincere gratitude to Kai-Wen Lan for his invaluable help on the compactification of the Hilbert modular varieties, and for explaining to us his vanishing result Lemma~\ref{Lemma:vanishing}.
We thank Ahmed Abbes, Matthew Emerton, Jianshu Li, Yifeng Liu, and Wei Zhang,  for useful discussions.
We thank the referees for careful reading of the paper and for many useful suggestions on improving the representation.

We started working on this project when we were attending a workshop held at the Institute of Advance Study at Hongkong University of Science and Technology in December 2011.
The  hospitality of the institution and the well-organization provided us a great environment for brainstorming ideas.
We especially thank the organizers Jianshu Li and Shou-wu Zhang, as well as the staff at IAS of HKUST.
Y.T. was partially supported by National Natural Science Foundation of China (No. 11321101).
L.X. was partially supported by Simons Collaboration Grant \#278433.

\subsection*{Notation}

%Let $\SSS = \Res_{\CC/\RR}\GG_m$ be the \emph{Deligne torus}.  A homomorphism from $\SSS$ to a real reductive group is called a \emph{Hodge cocharacter}.
%For an algebraic group $G$, we use $G^\ad$ to denote the adjoint group and use $G^\der$ to denote the derived group.

For a scheme $X$ over a ring $R$ and a ring homomorphism $R \to R'$, we use $X_{R'}$ to denote the base change $X \times_{\Spec R} \Spec R'$.

For a field $F$, we use $\Gal_F$ to denote its Galois group.

%Our convention for varieties is as follows.  We use Roman letters for varieties over a field of characteristic zero.
%We use bold face letters for varieties over the ring of integers of a local field.

For a number field $F$, we use $\AAA_F$ to denote its ring of adeles, and $\AAA_F^\infty$ (resp. $\AAA_F^{\infty, p}$) to denote its finite adeles (resp. finite adeles away from places above $p$).
When $F = \QQ$, we suppress the subscript $F$ from the notation.
We put $\widehat \ZZ^{(p)} = \prod_{l \neq p} \ZZ_l$ and $\widehat{\calO}_F^{(p)} = \prod_{\gothl \nmid p} \calO_\gothl$.
For each finite place $\gothp$ of $F$, let $F_\gothp$ denote the completion of $F$ at $\gothp$, $\calO_\gothp$ the ring of integers of $F_\gothp$, and $k_\gothp$ the residue field of $\calO_\gothp$. We put $d_\gothp=[k_\gothp: \FF_p]$.
Let $\varpi_\gothp$ denote a uniformizer of $\calO_\gothp$, which we take to be the image of $p$ when $\gothp$ is unramified in $F/\QQ$.
We normalize the Artin map $\Art_F: F^\times \backslash \AAA^\times_F \to \Gal_F^\ab$ so that for each finite prime $\gothp$, the element of $\AAA^\times_F$ whose $\gothp$-component is $\varpi_\gothp$ and other components are $1$, is mapped to a geometric Frobenius at $\gothp$.

%Let $\gothp$ be a nonarchimedean place of $F$. For a $p$-adic de Rham or $l$-adic representation $\rho$ of $\Gal_F$, we use $\WD_\gothp(\rho)$ denote the associated Weil-Deligne representation of $\rho|_{\Gal_{F_\gothp}}$. We will often ignore the action of the monodromy operator. In most of the case, the representation is unramified or semistable. So this is often understood as a module of the Frobenius.

We fix a totally real field $F$ of degree $g>1$ over $\QQ$. Let $\gothd_F$ be the different of $F$.
Let $\Sigma$ denote the set of places of $F$, and $\Sigma_\infty$  the subset of all real places.
We fix a prime number $p$ which is unramified in $F$, and let $\Sigma_p$ denote the set of places of $F$ above $p$.
We fix an isomorphism $\iota_p: \CC \simeq \overline \QQ_p$; this gives rise to a natural map $i_p: \Sigma_\infty \to \Sigma_p$ sending $\tau $ to the $p$-adic place corresponding to $ \iota_p \circ \tau$.
For each $\gothp \in \Sigma_p$, we put $\Sigma_{\infty/\gothp} = i_p^{-1}(\gothp)$.

For $\ttS$ an even subset of places of $F$, we use $B_\ttS$ to denote the quaternion algebra over $F$ which is ramified at $\ttS$.

A  \emph{multiweight} is a tuple $(\underline{k}, w) =((k_\tau)_{\tau\in\Sigma_\infty}, w) \in \ZZ^{\Sigma_{\infty}}\times \ZZ$ such that $w \equiv k_\tau \pmod 2$ for each $\tau$.
 We say $(\kb,w)$ is \emph{ cohomological}, if $w\geq k_{\tau}\geq 2$ for $\tau\in \Sigma_{\infty}$.

Let $\scrA_{(\underline k, w)}$ denote the set of irreducible cuspidal automorphic representations $\pi$ of $\GL_2(\AAA_F)$ whose archimedean component $\pi_\tau$ for each $\tau \in \Sigma_\infty$  is a discrete series of weight $k_\tau-2$ with central character $x \mapsto x^{w-2}$.
For such $\pi$, let $\rho_{\pi,l}$ denote the associated $l$-adic Galois representation, normalized so that $\det(\rho_{\pi,l})$ is the $(1-w)$-power of the cyclotomic character times a finite character.

For $A$ an abelian scheme over a scheme $S$, we denote by $A^{\vee}$ the dual abelian scheme, by $\Lie(A/S)$ the Lie algebra of $A$, and by $\omega_{A/S}$ the module of \emph{invariant differential $1$-forms} of $A$ relative to $S$.
We sometimes omit $S$ from the notation when the base is clear.

%We shall frequently say Grothendieck group of modules of a $\QQ$-algebra $R$ (resp. group ring of a topological group $G$); for that, we mean the Grothendieck group of \emph{finitely generated} $R$-modules (resp. smooth admissible representations of $G$).

\section{Preliminaries on Hilbert Modular Varieties and Hilbert Modular Forms}
\label{Section:HMV}
In this section, we review the construction of the integral models of Hilbert modular varieties and their compactifications.
 We also recall the construction of the automorphic vector bundles, using the universal abelian varieties.

\subsection{Shimura varieties for $\GL_{2,F}$}\label{Subsection:uniformization}
  Let $G$ be the algebraic group $\Res_{F/\QQ}(\GL_{2,F})$ over $\QQ$. Consider the homomorphism
\[
\xymatrix@R=0pt{
h: \quad \SSS(\RR)=\Res_{\CC/\RR}\GG_m(\RR) \cong \CC^\times\ar[rr]  &&G(\RR)=\GL_2(F\otimes \RR)\\
a+\sqrt{-1}b   \ar@{|->}[rr] &&\biggl({\begin{pmatrix}a &b\\-b&a\end{pmatrix},\dots,\begin{pmatrix}a &b\\-b&a\end{pmatrix}}\biggr).
}
\]
The space of conjugacy classes of $h$ under $G(\RR)$ has a structure of complex manifold, and is  isomorphic to $(\gothh^{\pm})^{\Sigma_{\infty}}$, where $\gothh^{\pm}=\PP^1(\CC)-\PP^1(\RR)$ is the union of the upper half and lower half planes.
For any open compact subgroup $K\subset G(\AAA^{\infty})=\GL_2(\AAA_{F}^\infty)$, we have the Shimura variety $\Sh_{K}(G)$ with complex points 
\[
\Sh_K(G)(\C)= G(\QQ)\backslash  (\gothh^{\pm})^{\Sigma_{\infty}} \times G(\AAA^{\infty})/K%\simeq G(\QQ)^{+}\backslash G(\AAA^{\infty})\times \gothh^{\Sigma_{\infty}}/K,
\]
%where $G(\QQ)^+$ is the subgroup of $G(\QQ)=GL_2(F)$ with totally positive determinant.
It is well known that $\Sh_K(G)$ has a canonical structure of quasi-projective variety  defined over the reflex field $\QQ$. 
For $g\in G(\AAA^{\infty})$ and open compact subgroups $K,K'\subset G(\AAA^{\infty})$ with $g^{-1}K'g\subset K$, there is a natural surjective map 
\begin{equation}\label{E:morphism-g}
[g]: \Sh_{K'}(G)\ra \Sh_K(G)
\end{equation}
whose effect on $\C$-points is given by $(z,h)\mapsto (z,hg)$. This gives rise to a Hecke correspondence: 
\begin{equation}\label{E:Hecke-cor}
\xymatrix{
&\Sh_{K\cap gKg^{-1}}(G)\ar[ld]\ar[rd]^{[g]}\\
\Sh_{K}(G) &&\Sh_K(G),
}
\end{equation}
where the left downward arrow is induced by the natural inclusion $K\cap gKg^{-1}\hra K$, and the right downward one is given by $[g]$.
 Taking the project limit  in $K$, we get a natural right action of  $G(\AAA^\infty)$ on the projective limit $\Sh(G):=\varprojlim_K \Sh_K(G)$.

\subsection{Automorphic Bundles}\label{S:automorphic-bundle}
Let $(\uk,w)$ be a cohomological multiweight. We consider the algebraic representation of $G_{\C}$:
\[
\rho^{(\uk,w)}: =\bigotimes_{\tau\in \Sigma_{\infty}}
\biggl( \Sym^{k_{\tau}-2}(\Stb_{\tau})\otimes \mathrm{det}_{\tau}^{-\frac{w-k_{\tau}}{2}}\biggr)
\]
where $\Stb_{\tau}:G_{\C}\cong (\GL_{2,\C})^{\Sigma_{\infty}}\ra \GL_{2,\C}$ is the \emph{contragradient} of the  projection onto the $\tau$-factor, and $\det_{\tau}$ is the projection onto the $\tau$-factor composed with the determinant map. Consider the subgroup $Z_s=\Ker(\Res_{F/\Q}(\GG_m)\xra{N_{F/\Q}} \GG_{m})$ of the center $Z=\Res_{F/\Q}(\GG_m)$ of $G$; let $G^c$ denote the quotient of $G$ by $Z_s$. Then the representation $\rho^{(\uk,w)}$ factors through $G^c_{\C}$. Let $L$ be a subfield of $\CC$ that contains all the embeddings of $F$. The representation $\rho^{(\kb,w)}$ descends to a representation of $G_{L}$ on an $L$-vector space $V^{(\kb,w)}$.

 We say an open subgroup $K\subset G(\AAA^{\infty})$ is \emph{sufficiently small}, if the following two properties are satisfied:
\begin{itemize}
\item[(1)] The quotient $(g^{-1}Kg\cap \GL_2(F))/(gKg^{-1}\cap F^{\times})$ does not have non-trivial elements of finite order for all $g\in G(\AAA^{\infty})$.

\item[(2)] $N_{F/\Q}(K\cap F^{\times})^{w-2}=1$.
\end{itemize} 
If $K$ is sufficiently small,  it follows from \cite[Chap. III 3.3]{milne} that  $\rho^{(\uk,w)}$ gives rise to an algebraic vector bundle $\F^{(\kb,w)}$ on $\Sh_K(G)$ equipped with an integrable connection 
\[
\nabla: \F^{(\uk,w)}\ra \F^{(\uk,w)}\otimes \Omega^{1}_{\Sh_K(G)_L}.
\]
%
%\begin{enumerate}
%\item a local system of $L$-vector spaces $\VV^{(\kb,w)}$ on $\Sh_K(\C)$ such that the monodromy representation on each connected component of $\Sh_K(\C)$ is given by  $\rho^{(\kb,w)}$ (restricted to certain discrete subgroups of $G(\R)$);

%\item an algebraic vector bundle $\F^{(\kb,w)}$ on $\Sh_K$ equipped with an integrable connection 
%\[\nabla: \F^{(\uk,w)}\ra \F^{(\uk,w)}\otimes \Omega^{1}_{\Sh_{K,L}},\]
% such that, after base change to $\C$, the sheaf of horizontal sections $\F^{(\kb,w)}(\C)^{\nabla=0}$  on $\Sh_K(\C)$coincides with $\VV^{(\kb,w)}\otimes_{L}\C$;

%\item for each fixed isomorphism $\iota_{\ell}: \C\simeq \overline{\Q}_{\ell}$, which determines a specific $\ell$-adic place $\lambda$ of $L$, there is a   $L_\lambda$ lisse sheaf $\VV^{(\kb,w)}_{\lambda}$ on $\Sh_{K,L}$ such that $\VV^{(\kb,w)}\otimes_{L}L_{\lambda}$  over  $\Sh_K(\C)$.

%\end{enumerate}

The theory of automorphic bundles also allows us  to define an invertible sheaf on $\Sh_K(G)$ for $K$ sufficiently small as follows. Consider the compact dual $(\PP^1_{\C})^{\Sigma_{\infty}}$ of the Hermitian symmetric domain $(\gothh^{\pm})^\Sigma_{\infty}$. It has a natural action by $G_\C=(\GL_{2,\C})^{\Sigma_{\infty}}$. Let $\omega$ be the dual of the tautological quotient bundle on $\PP^{1}_{\C}$. Then the line bundle $\omega$ has a natural $\GL_{2,\C}$-equivariant action. We define 
\begin{equation}\label{E:defn-omega}
\omegab^{(\kb,w)}\colon =\bigotimes_{\tau\in \Sigma_{\infty}}\pr_{\tau}^* (\omega^{\otimes k_{\tau}}\otimes \mathrm{det}^{1-\frac{w-k_{\tau}}{2}})
\end{equation}
and a $G_{\C}$-equivariant action on $\omegab^{(\kb,w)}$ as follows. 
For each $\tau\in \Sigma_{\infty}$, the action of $G_{\C}$ on $\pr_{\tau}^{*}(\omega^{\otimes k_{\tau}}\otimes \mathrm{det}^{1-\frac{w-k_{\tau}}{2}}) $ factors through the $\tau$-copy of $\GL_{2,\C}$, which in turn acts  as the product of  $\det^{1-\frac{w-k_{\tau}}{2}}$ and the $k_{\tau}$-th power of the natural action on $\omega$. One checks easily that the action of $G_{\C}$ on $\underline \omega^{(\kb,w)}$ factors through $G^c_{\C}$, and thus $\omegab^{(\kb,w)}$ descends to an invertible sheaf on $\Sh_{K}(G)$ for $K$ sufficiently small by \cite{milne}. As usual, the invertible sheaf  $\omegab^{(\uk, w)}$ on $\Sh_{K}(G)$ has a canonical model over $L$. 

We define the space of holomorphic Hilbert modular forms of level $K$ with coefficients in $L$ to be 
\begin{equation}\label{E:defn-HMF}
M_{(\uk,w)}(K,L): =H^{0}(\Sh_{K}(G)_L, \omegab^{(\kb,w)}).
\end{equation}
Note here that the canonical bundle $\Omega^g_{\Sh_K(G)}$ accounts for a parallel weight two automorphic line bundle. So our definition is equivalent to the usual notion of holomorphic Hilbert modular forms.  The unusual twist by canonical bundle will make the relation between the de Rham cohomology of $\scrF^{(\underline k, w)}$ and the Zariski cohomology of $\underline \omega^{(\kb, w)}$ more natural, in the view of the dual BGG construction of Faltings (see Subsection~\ref{S:BGG}).

Explicitly, an element of $M_{(\uk,w)}(K,\C)$ is a function $f(z,g)$ on $(\gothh^{\pm})^{\Sigma_{\infty}} \times G(\AAA^{\infty})$ such that 
\begin{itemize}
\item[(1)] $f(z,g)$ is holomorphic in $z$ and locally constant in $g$; and

\item[(2)]  one has  $f(z,gk)=f(z,g)$ for any $k\in K$,  and 
\[
f(\gamma(z),\gamma g)=\biggl(\prod_{\tau\in \Sigma_{\infty}}\frac{(c_{\tau}z_{\tau}+d_{\tau})^{k_{\tau}}}{\det(\gamma_\tau)^{\frac{w+k_{\tau}-2}{2}}}\biggr)f(z,g)
,\]
where $\gamma\in G(\Q)$, and $\gamma_{\tau}=\begin{pmatrix}a_{\tau}&b_{\tau}\\ c_\tau&d_\tau\end{pmatrix}\in \GL_2(\R)$ is the image of $\gamma$ via $G(\QQ)\hra \GL_2(F\otimes\R)\xra{\pr_\tau}\GL_2(\R)$, and $\gamma(z)=\bigl(\frac{a_\tau z_{\tau}+b_\tau}{c_{\tau}z_{\tau}+d_\tau}\bigr)_{\tau\in \Sigma_{\infty}}$.
\end{itemize}

 We denote by  $S_{(\kb,w)}(K,L)\subset M_{(\kb,w)}(K,L)$ the subspace of cusp forms, namely those forms which tend to $0$ near all cusps. For any $g\in G(\AAA^{\infty})$ and open compact subgroups $K,K'\subset G(\AAA^{\infty})$ with $g^{-1}K'g\subset K$, by construction,  there exists a natural isomorphism of coherent sheaves on $\Sh_{K'}(G)$:
 \[
 [g]^*(\omegab^{(\kb,w)})\xra{\sim} \omegab^{(\kb,w)}.
 \]
 Together with the map \eqref{E:morphism-g}, one deduces a map $S_{(\kb,w)}(K,L)\ra S_{(\kb,w)}(K',L)$. Passing to the direct limit in $K$, one obtains a natural left action of $G(\AAA^{\infty})$ on $ S_{(\kb,w)}( L)=\varinjlim_K S_{(\kb,w)}(K,L)$ so that $S_{(\kb,w)}(K,L)$ is identified with the invariants of $S_{(\kb,w)}(L)$ under $K$.
 
  Let $\scrA_{(\kb,w)}$ be the set of cuspidal automorphic representations $\pi=\pi^{\infty}\otimes \pi_\infty$ of $\GL_2(\AAA_F)$, such that each archimedean component $\pi_{\tau}$ of $\pi$ for $\tau\in \Sigma_{\infty}$ is the discrete series of weight $k_{\tau}$ and central character $x \mapsto x^{w-2}$. Then  we have canonical decompositions
 \[
 S_{(\kb,w)}(\C)=\bigoplus_{\pi=\pi^{\infty}\otimes \pi_\infty\in \scrA_{(\kb,w)}} \pi^{\infty}\quad \text{ and }\quad S_{(\kb,w)}(K,\C)=\bigoplus_{\pi=\pi^{\infty}\otimes \pi_{\infty}\in \scrA_{(\kb,w)} }(\pi^{\infty})^{K}.
 \]
  where $\pi^{\infty}$ denotes the finite part of  $\pi$.

  \subsection{Moduli interpretation and integral models}\label{Subsection:HMV}
Recall that $p$ is a rational prime unramified in $F$.
We consider level structures of the type $K=K^pK_p$, where $K^p\subset G(\AAA^{\infty,p})$ is an open compact subgroup, and $K_p$ is hyperspecial, i.e. $K_p\cong \GL_2(\cO_{F}\otimes \Z_p)$. 
%For every finite place $v$, we denote by $K_v\subset \GL_2(F_v)$ the $v$-component of $K$. %We denote by $N_K$ the product of rational primes $\ell$ such that $K_v$ is not hyperspecial for some place $v |\ell$. Then $N_K$ is  prime to $p$.
 We will use the moduli interpretation  to define integral models of $\Sh_{K}(G)$, for sufficiently small $K^p$.

We start with a more transparent description of $\Sh_{K}(G)(\C)$. The determinant map $\det: G\ra \Res_{F/\Q}(\GG_m)$ induces a bijection between the set of geometric connected components of $\Sh_{K}(G)$ and the double coset space
\[
cl^+_{F}(K): =F^\times_{+}\backslash \AAA_F^{\infty,\times}/\det(K),
\]
where $F^\times_+$ denotes the subgroup of $F^\times$ of totally positive elements. Since $\det(K)\subseteq \prod_{v\nmid \infty}\cO_{F_v}^{\times}$, there is a natural surjective map $cl^+_F(K)\ra cl^+_F$, where $cl^+_F$ is the strict ideal class group of $F$. 
The preimage of each ideal class $[\gothc]$ is a torsor under the group $I :=\widehat{\cO}_F^{\times}/\det(K)\cO_{F,+}^{\times}$, where $\calO^\times_{F, +}$ denotes the group of totally positive units in $\calO_F$.

We fix fractional ideals $\gothc_1,\dots, \gothc_{h^+_F}$ coprime to $p$, which form a set of representatives of $cl^+_F$. For each $\gothc = \gothc_j$, we write $[\gothc]$ for the corresponding class in $cl^+_F$, and we choose a subset $[\gothc]_K=\{g_i\,|\, i\in I\}\subset G(\AAA^{\infty})$ such that the fractional ideal associated to every $\det(g_i)$ is $\gothc$ and $\{\det(g_i)\,|\, i\in I\}$ is a set of representatives of the pre-image of $[\gothc]$ in $cl^+_F(K)$. 
Let $G(\Q)^+$ denote the subgroup of $G(\Q)$ consisting of matrices with totally positive determinant, and let $\gothh$ denote the upper half plane.
By the strong approximation theorem for $\SL_{2,F}$,  we have \[
G(\AAA^{\infty})=\coprod_{[\gothc]\in cl^+_F}\coprod_{g_i\in [\gothc]_K} G(\QQ)^+g_{i}K.\] This gives rise to a decomposition
\begin{align}
\nonumber
\Sh_{K}(G)(\C)&=G(\QQ)^+\backslash \gothh^{\Sigma_{\infty}} \times  G(\AAA^{\infty})/K=\coprod_{[\gothc]\in cl^+_F}\Sh^{\gothc}_{K}(G)(\C),\\
\label{Equ:complex-uniformization}
\textrm{where} \quad \Sh^{\gothc}_{K}(G)(\C)&=\coprod_{g_i\in [\gothc]_K} \Gamma(g_i,K)\backslash \gothh^{\Sigma_{\infty}}\quad \text{with }\Gamma(g_i,K)=g_iKg_i^{-1}\cap G(\Q)^+.
\end{align}
 We note that $\Sh^{\gothc}_{K}(G)$ does not depend on the choice of the subset $[\gothc]_K=\{g_i:i\in I\}$, and  descends to an algebraic variety defined over $\Q$.
 A different choice of the fractional ideal representative $\gothc'$ will result in two canonically isomorphic moduli spaces $\Sh_K^\gothc(G)$ and $\Sh_K^{\gothc'}(G)$ (see Remark~\ref{R:polarization}).
  
 We will interpret  $\Sh_K^{\gothc}(G)$ as a  moduli space as follows.
Recall that $\gothc$ is coprime to $p$. 
 Let $\gothc^{+}$ be the  cone of totally positive elements of $\gothc$. Let $S$ be a locally noetherian $\Z_{(p)}$-scheme. % We discuss first the Hilbert-Blumenthal moduli with principal level structure.
 \begin{itemize}
 \item A \emph{Hilbert-Blumenthal abelian variety} (\emph{HBAV} for short) $(A,\iota)$ over $S$ is  an abelian variety $A/S$ of dimension $[F:\Q]$ together with a homomorphism $\iota: \cO_F\ra \End_S(A)$ such that $\Lie(A)$ is a locally free $(\cO_S\otimes_{\Z}\cO_F)$-module of rank 1. 
 
\item If $(A,\iota)$ is an HBAV over $S$, then its usual dual abelian variety $A^\vee$ has a natural action by $\cO_F$. Let $\Hom^{\Sym}_{\cO_F}(A,A^\vee)$ denote the group of symmetric homomorphisms of $A$ to $A^\vee$, and $\Hom^{\Sym}_{\cO_F}(A,A^\vee)^+$ be the cone of polarization. A \emph{$\gothc$-polarization} on $A$ is an $\cO_F$-linear isomorphism  
$$
\lambda: (\gothc,\gothc^+)\xra{\sim} \big(\Hom_{\cO_F}^{\Sym}(A,A^\vee),\Hom^{\Sym}_{\cO_F}(A,A^\vee)^{+} \big)
$$
  preserving the positive cones on both sides; in particular,  $\lambda$ induces an isomorphism of HBAVs:
 $A\otimes_{\cO_F}\gothc\simeq A^\vee$. 
  
\item  We define first the level structure for $K=K(N)$, the principal subgroup of $\GL_2(\widehat{\cO}_F)$ modulo an integer $N$ coprime to $p$. A \emph{principal level-$N$ (or a level-$K(N)^p$) structure} on a $\gothc$-polarized HBAV $(A,\iota,\lambda)$ is an $\cO_F$-linear isomorphism of finite \'etale group schemes over $S$
\[
\alpha_N: (\cO_F/N)^{\oplus 2}\xra{\sim} A[N].
\] 
 Note that there exists a natural $\cO_F$-pairing $A[N]\times A^\vee[N]\ra \mu_N\otimes_{\Z}\gothd_F^{-1}$. Its  composition with $1\otimes\lambda$ gives an $\cO_F$-linear alternating pairing $A[N]\times A[N]\ra \mu_N\otimes_{\Z}\gothc^*.$
  Hence,  $\alpha_N$ determines an isomorphism 
 \[
\nu(\alpha_N) : \cO_F/N\cO_F=\wedge^2_{\cO_F}(\cO_F/N)^{\oplus 2}\xra{\sim} \wedge^2_{\cO_F} A[N]\xra{\sim} \mu_N\otimes_{\Z}\gothc^{*}.
 \]
  For a general open compact subgroup $K\subseteq \GL_2(\widehat{\cO}_F)$ with $K_p=\GL_2(\cO_F\otimes\widehat{\Z}_p)$, we define a level-$K^p$ structure on $(A,\iota,\lambda)$ as follows. Choose an integer $N$ coprime to $p$ such that $K(N)\subseteq K$, and a geometric point $\sbar$ of $S$. The finite group $\GL_2(\cO_F/N)$ acts naturally on the set of principal level-$N$ structures $(\alpha_{N,\sbar},\nu(\alpha_{N,\sbar}))$ of $A_{\sbar}$ by putting 
  $$
  g: (\alpha_{N,\sbar},\nu(\alpha_{N,\sbar}))\mapsto (\alpha_{N,\sbar}\circ g,\det(g)\nu(\alpha_{N,\sbar})).
  $$
 Then  a \emph{level-$K^p$ structure} $\alpha_{K^p}$ on $(A,\iota,\lambda)$ is, for each connected component $S_i$ of $S$ and a geometric point $\sbar_i\in S_i$, a $\pi_1(S_i,\sbar_i)$-invariant $K/K(N)$-orbit of the pairs $(\alpha_{N,\sbar},\nu(\alpha_{N,\sbar}))$. This definition does not depend on the choice of $N$ and $\sbar$.
 \end{itemize}

   We consider the moduli problem which associates to each   locally noetherian $\ZZ_{(p)}$-schemes $S$, the set of isomorphism classes of  quadruples $(A,\iota,\lambda, \alpha_{K^p})$ as above.
If $K^p$ is sufficiently small so that any $(A,\iota,\lambda, \alpha_{K^p})$  does not admit non-trivial automorphisms, then this moduli problem  is  representable   by a smooth and quasi-projective $\Z_{(p)}$-scheme $\M^{\gothc}_K$ \cite{rap,chai}. 
After choosing a  primitive $N$-th root of unity $\zeta_{N}$ for some integer $N$ coprime to $p$ such that $K(N)\subseteq K$,  the set of geometric connected components of $\M^\gothc_K$ is in natural bijection with \cite[2.4]{chai}
 $$
 \mathrm{Isom}(\widehat \calO_F, \widehat \calO_F \otimes \gothc^*) \big/ \det(K).
 $$

Let $\cO^{\times}_{F,+}$ be the group of totally positive units of $\cO_F$. It acts on  $\M^{\gothc}_{K}$ as follows. For  $\epsilon\in \cO_{F,+}^{\times}$ and an $S$-point $(A,\iota,\lambda,\alpha_{K^p})$, we put $\epsilon\cdot (A,\iota,\lambda,\alpha_{K^p})=(A,\iota,\iota(\epsilon)\circ\lambda,\alpha_{K^p})$. 
We point out that this action will take $\nu(\alpha_{N, \bar s})$ to $\epsilon\nu(\alpha_{N, \bar s})$.
We will denote by $(A,\iota,\bar\lambda,\bar \alpha_{K^p})$ the associated $\cO_{F,+}^{\times}$-orbit of $(A,\iota,\lambda,\alpha_{K^p})$.
 The subgroup $(K\cap\cO_{F}^{\times})^2$ acts trivially on $\M^{\gothc}_K$, where $\cO_F^{\times}$ is considered as a subgroup of the center of $\GL_2(\AAA^{\infty})$. Indeed, if $\epsilon=u^2$ with $u\in K\cap \cO_F^{\times}$, the endomorphism $\iota(u):A\ra A$ induces an isomorphism of quadruples $(A,\iota,\lambda, \alpha_{K^p})\cong (A,\iota,\iota(\epsilon)\circ\lambda,\alpha_{K^p})$. Hence, the action of $\cO_{F,+}^{\times}$ on $\M^\gothc_K$ factors through the finite quotient $\cO_{F,+}^{\times}/(K\cap \cO_F^{\times})^2$.
 The  equivalent classes of the set of geometric connected components of $\M^{\gothc}_{K}$ under the  induced action of $\cO_{F,+}^{\times}/(K\cap \cO_F^{\times})^2$  is in bijection with $\widehat{\cO}_F^{\times}/\det(K)\cO_{F,+}^{\times}$, and the stabilizer of each geometric connected component is $\big(\det(K)\cap \cO_{F,+}^{\times}\big)/(K\cap \cO_F^{\times})^2$.

 \begin{prop}\label{P:complex-uniform}
There exists an isomorphism between the quotient of  $\M^{\gothc}_{K}(\C)$ by $\cO_{F,+}^{\times}/(K\cap \cO_F^{\times})^{2}$ and   $\Sh^{\gothc}_K(G)(\C)$. In other words, $\Sh^{\gothc}_{K}(G)(\C)$ is identified with the coarse moduli space over $\C$ of the quadruples $(A,\iota,\bar\lambda, \bar\alpha_{K^p})$.  Moreover, if $\det(K)\cap \cO_{F,+}^{\times}=(K\cap \cO_F^{\times})^2$, then the  quotient map $\M^{\gothc}_{K}(\C)\ra \Sh^{\gothc}_K(G)(\C)$ induces an isomorphism between any geometric connected component of $\M^{\gothc}_{K}(\C)$ with its image. 
  \end{prop} 
\begin{proof}
We fix an idele $a\in \AAA^{\infty,\times}_{F}$ whose associated fractional ideal is $\gothc$. Let $\Delta\subseteq \widehat{\cO}_F^{\times}$ be a complete subset of representatives of $\widehat{\cO}_F^{\times}/\det(K)$, and let $I\subset \Delta$ be a subset of representatives of $\widehat{\cO}_F^{\times}/\det(K)\cO_{F,+}^{\times}$. We put $g_{\delta}=\begin{pmatrix} \delta a &0\\ 0&1\end{pmatrix}$ for $\delta\in \Delta$, and $\Gamma(g,K)=gKg^{-1}\cap G(\Q)^+$ and $\Gamma^1(g,K)=\Gamma(g,K)\cap \SL_2(F)$. Then it is well known that 
\[
\M^{\gothc}_K(\C)=\coprod_{\delta\in \Delta} \Gamma^1(g_{\delta},K)\backslash\gothh^{\Sigma_{\infty}}.
\] 
The case for $K=K(N)$ is proved in \cite{rap} or \cite[4.1.3]{hida}, and the general case is similar. For $\epsilon\in \cO_{F,+}^\times$, it sends a point $\Gamma^1(g_{\delta},K)z$ in $\Gamma^1(g_{\delta},K)\backslash \gothh^{\Sigma_{\infty}}$ to $\Gamma^1(g_{\delta\epsilon},K)\epsilon z$. Hence the quotient of $\M^{\gothc}_{K}(\C)$ is isomorphic to 
$$
\coprod_{\delta\in I}\bigg(\big ( \Gamma^1(g_{\delta},K)\backslash \gothh^{\Sigma_{\infty}}\big)/(\det(K)\cap \cO_{F,+}^{\times})/(K\cap\cO_{F}^\times)^2\bigg).
$$
Now if we take the set $[\gothc]_K = \{g_{\delta}, \delta\in I\}$, then  \eqref{Equ:complex-uniformization} says
$
\Sh_K^{\gothc}(G)(\C)=\coprod_{\delta\in I} \Gamma(g_{\delta},K)\backslash \gothh^{\Sigma_{\infty}}.
$
Note that for each $\delta\in I$, $\Gamma(g_{\delta},K)\backslash \gothh^{\Sigma_{\infty}}$ is identified with the natural quotient of $\Gamma^1(g_{\delta},K)\backslash \gothh^{\Sigma_{\infty}} $ by the group
\[
F^{\times}\Gamma(g_{\delta},K)/\Gamma^1(g_{\delta},K)F^{\times}\cong \Gamma(g_{\delta},K)/(F^{\times}\cap \Gamma(g_{\delta},K)) \Gamma^1(g_{\delta},K).
\]
By the strong approximation for $\SL_{2,F}$, one sees that $\det:\Gamma(g_{\delta},K)\ra \det(K)\cap \cO_{F,+}^{\times}$ is surjective.
 Hence, the group above is isomorphic to $(\det(K)\cap\cO_{F,+}^{\times})/(K\cap \cO_{F,+}^{\times})^2$. The Proposition  follows immediately.
\end{proof}

We define $\bfSh^{\gothc}_K(G)$ to be the quotient of $\M^{\gothc}_K$ by the action of the finite group $\det(K)\cap \cO_{F,+}^{\times}/(K\cap\cO_{F}^{\times})^2$, and we put $\bfSh_K(G)=\coprod_{\gothc\in cl^{+}(F)}\bfSh^{\gothc}_{K}(G)$. In general, this is just a coarse moduli space that parametrizes the quadruples $(A,\iota,\bar\lambda,\bar\alpha_{K^p})$. However, we have the following:

\begin{lemma}\label{L:open-subgroup}
For any open compact subgroup $K^p\subset G(\AAA^{\infty,p})$, there exists an open compact normal subgroup $K'^p\subseteq K^p$ of finite index, such that $\det(K'^pK_p)\cap \cO_{F,+}^{\times}=(K'^pK_p\cap\cO_F^{\times})^2$.
\end{lemma}
\begin{proof}
By a theorem of Chevalley (see for instance \cite[Lemma 2.1]{Ta}), every finite index subgroup of $\cO_{F}^{\times}$ contains a subgroup of the form $U\cap \widehat\cO_{F}^{\times}$, where $U\subseteq \widehat{\cO}^\times_{F}$ is an open compact subgroup with $U_v=\cO_{F_v}^{\times}$ for all $v|p$.  Therefore,  one can  choose such an open compact $U\subseteq \det(K)$  such that  $U\cap (\det(K)\cap \cO_{F}^{\times, +})=U\cap (K\cap \cO_{F}^{\times})^2$. Let  $K'^p\subseteq K^p$ denotes the inverse image of $U^p$ via determinant map. Then it is easy to check that this choice of $K'^p$ answers the question.
\end{proof}

\begin{remark}\label{R:units}
In general, $\det(K)\cap\cO_{F,+}^{\times}/(K\cap \cO_F^{\times})^2$ is non-trivial even for $K^p$ sufficiently small. For instance, if $K=K(N)$ for some integer $N$ coprime to $p$, then $K\cap \cO_{F}^{\times}$ is the subgroup of units congruent to $1$ modulo $N$, and $\det(K)\cap \cO_{F,+}^{\times}$ is subgroup of $K\cap \cO_{F}^{\times}$ of positive elements. By the theorem of Chevalley cited in the proof of the Lemma, we have $\det(K)\cap \cO_{F,+}^{\times}=K\cap \cO_F^{\times}$ for $N$ sufficiently large, and hence $\det(K)\cap\cO_{F,+}^{\times}/(K\cap \cO_F^{\times})^2\simeq (\Z/2\Z)^{[F:\Q]-1}$.
\end{remark}

%Then the isomorphism \eqref{Equ:complex-uniformization} follows from the isomorphism $\Sh_{\gothc_i}(\CC)\simeq \Gamma_1(\gothc_i,\gothN)\backslash \gothh^{\Sigma_{\infty}}$ \cite[4.2]{DT}.

 From now on, we  always make the following
\begin{hypo}\label{H:fine-moduli}
$K^p$ is sufficiently small and  $\det(K)\cap \cO_{F,+}^\times=(K\cap \cO_F^{\times})^2$.
\end{hypo} 
  By Lemma~\ref{L:open-subgroup}, this hypothesis is always valid up to replacing 
 $K^p$ by an open compact subgroup.  Under this assumption, Proposition~\ref{P:complex-uniform} shows that each geometric connected component is identified with a certain geometric connected component of $\M^{\gothc}_K$. Therefore, $\bfSh_K(G)$ is quasi-projective and smooth over $\Z_{(p)}$. It is the \emph{integral model of the Hilbert modular variety with level $K$}. We can also talk about the universal family of HBAV over $\bfSh_K(G)$.

\begin{remark}
\label{R:polarization}
In the construction of $\bfSh_K(G)$, we fixed the set of representatives  $\{\gothc_1,\cdots, \gothc_{h_F^+}\}$ of $cl^+_F$ and we assumed them to be  coprime to $p$. This assumption was used to establish the smoothness of each $\M^{\gothc_i}_K$, hence that of $\bfSh^{\gothc_i}_K(G)$, using deformation theory. 
However, dropping this assumption or changing to another set of representatives will not cause any problems in practice. Suppose we are given a quadruple $(A,\iota, \lambda, \alpha_{K^p})$ over a connected locally noetherian $\Z_{(p)}$-scheme $S$, where $\lambda:\gothq\xra{\simeq} \Hom_{\cO_F}^{\Sym}(A,A^\vee)$ is an  isomorphism  preserving positivity for a not necessarily prime-to-$p$ fractional ideal $\gothq$.
 Then there exists a unique representative $\gothc_i$ and an element $\xi\in F^{\times}_{+}$ such that multiplication by $\xi$ defines an isomorphism $\xi:\gothc_i\xra{\simeq}\gothq$.  
  We put $\lambda'=\xi\circ\lambda$. Let $(\alpha_{N,\sbar},\nu(\alpha_{N,\sbar}))$ be a representative of isomorphisms in the level-$K^p$ structure $\alpha_{K^p}$ for some integer $N$ coprime to $p$ with $K(N)\subseteq K$.
   We define $\alpha'_{K^p}$ to be the $K/K(N)$-orbit of the pairs $(\alpha_{N,\sbar},\xi\cdot \nu(\alpha_{N,\sbar}))$, where $\xi\cdot \nu(\alpha_{N,\sbar})$ is the composite of isomorphisms
$$
 \cO_F/N\cO_F\xra{\nu(\alpha_{N,\sbar})} \mu_{N,\sbar}\otimes_{\Z} \gothq^*\xra{\xi}\mu_N\otimes_{\Z}\gothc_i^*.
$$
We then get a new quadruple $(A,\iota,\lambda',\alpha'_{K^p})$. Since $\xi$ is well determined up to $\cO_{F,+}^{\times}$, the $\cO_{F,+}^{\times}$-orbit $(A,\iota,\bar\lambda',\bar\alpha'_{K^p})$ is a well-defined $S$-point on $\bfSh_K^{\gothc_i}(G)$. By abuse of notation, we also use  $(A,\iota,\bar\lambda, \bar\alpha_{K^p})$ to denote this point. 
\end{remark}

%If we let $K^p$ vary, the Hecke action of $G(\AAA^{\infty,p})$ on $\Sh_K(G)$ \eqref{E:morphism-g} extends to $\bfSh_K(G)$. 
\subsection{Tame Hecke actions on $\bfSh_K(G)$} Suppose we are given  $g\in G(\AAA^{\infty, p})$, and open compact subgroups $K^p, K'^p\subset G(\AAA^{\infty,p})$   with $g^{-1}K'^pg\subseteq K^p$. 
We let $K=K^pK_p$, $K'=K'^pK'_p$ with $K_p=K_p'=\GL_2(\cO_F\otimes_{\Z} \Z_p)$.  We now define a finite \'etale  map 
\begin{equation}\label{E:morphism-g-Sh}
[g]:\bfSh_{K'}(G)\ra \bfSh_K(G)
\end{equation}
 that extends the Hecke action \eqref{E:morphism-g}.
  If $K^p$ and $g$ are both contained in $\GL_2(\widehat{\cO}^{(p)}_F)$, the morphism $[g]$ is given by $(A,\iota, \bar\lambda,\bar\alpha_{K'^p})\mapsto (A,\iota, \bar\lambda, [\bar\alpha_{K'^p}\circ g]_{K^p})$, where $[\bar\alpha_{K'^g}\circ g]_{K^p}$ denotes the $K^p$-level structure associated to $\bar\alpha_{K'^p}\circ g$. 
  To define $[g]$ in the general case, it is more natural to use the rational version of the moduli interpretation of $\M_K$ as  in \cite[6.4.3]{lan}, i.e. we consider  $\M_K$ as the classifying space of certain isogenies classes of HBAVs instead of the classifying space of isomorphism classes of HBAV.
  Then the exact same formula as above gives rise to the desired Hecke action.
   For more details on these two types of moduli interpretation for $\M_{K}$ and their equivalence, we refer the reader to  \cite[Section 1.4]{lan} and \cite[Section 4.2.1]{hida}. 

%Therefore, we get a Hecke correspondence  on $\bfSh_{K}$ similar to \eqref{E:Hecke-cor} by replacing $\Sh_K$ by $\bfSh_{K}$.

%In this paper, we will mainly concern the Hilbert modular forms of the following two types of levels. Let $N$ be an integer, and 
%\begin{equation}\label{Equ:level}
%K(N)=\biggl\{g=\begin{pmatrix}a &b \\c&d\end{pmatrix}\in G(\hat{\cO}_F)| a,d\equiv 1\mod N, b\in N\hat{\gothd}^{-1}_F, c\in N\hat{\gothd}_F
%\biggr\}.
%\end{equation}
%The subgroup $K(N)$ is sufficiently small if $N\geq 3$. We suppose from now on that $N\geq 3$. The elements of $M_{(\kb,w)}(K(N), L)$ (resp. $S_{(\kb,w)}(K(N), L)$) are called Hilbert modular forms (resp. cusp forms) of principle level $N$. 
 %Let $p$ be a prime coprime to $N$ and unramified in $F$. Let 
%\begin{equation}\label{E:level-Iw}
 %K(N;p)=\prod_{v\nmid p}K(N)_v\prod_{\gothp\mid p}I_{\gothp}, \quad \text{with}\;
 %I_{\gothp}=\biggl\{g=\begin{pmatrix}a&b\\c&d\end{pmatrix}\in \GL_2(\cO_{F_{\gothp}})| c\equiv 0\mod \gothp\biggr\}.
%\end{equation}
%Elements of $M_{(\kb,w)}(K(N;p), L)$ (resp. $S_{(\kb,w)}(K(N;p), L)$) are called Hilbert modular forms (resp. cusp forms) of principle level $N$ and Iwahori at $p$.

\subsection{Compactifications}\label{Subsection:compactification}
Let $K=K^pK_p\subset G(\AAA^\infty)$ be an open compact subgroup with $K_p$ hyperspecial and  satisfying Hypothesis~\ref{H:fine-moduli}.   We recall some results on the arithmetic toroidal compactification of  $\bfSh_K(G)$.  For more details, the reader may refer to \cite{rap,chai} and more recently \cite[Chap. VI]{lan}.

By choosing suitable admissible  rational  polyhedral cone decomposition data for $\bfSh_{K}(G)$, one can construct  arithmetic toroidal compactifications $\bfSh_K^\tor(G)$ satisfying the following conditions.  
 \begin{itemize}
 \item[(1)] The schemes $\bfSh_K^{\tor}(G)$ are projective and smooth over $\Z_{(p)}$. 
 
 \item[(2)] There exists a natural open immersion  $\bfSh_K(G)\hra \bfSh_K^{\tor}(G)$ such that the boundary $\bfSh^\tor_K(G)-\bfSh_K(G)$ is a relative simple normal crossing Cartier divisor of $\bfSh^{\tor}_K(G)$ relative  to the base.
  
 \item[(3)] There exists a polarized semi-abelian scheme $\calA^{\sm}$  over $\bfSh^{\tor}_K(G)$  equipped with an action of $\cO_F$ and a $K^p$-level structure, which extends the universal abelian scheme $\calA$ on $\bfSh_K(G)$ and degenerates to a torus at the cusps. 
 
 \item[(4)] Suppose we are given an element $g\in G(\AAA^{\infty, p})$, and open compact subgroups $K^p, K'^p\subset G(\AAA^{\infty,p})$   with $g^{-1}K'^pg\subseteq K^p$. 
  We put $K=K^pK_p$, $K'=K'^pK'_p$ with $K_p=K_p'=\GL_2(\cO_F\otimes_{\Z} \Z_p)$.
   Then by choosing compatible rational polyhedral cone decomposition data for $\bfSh_K(G)$ and for $\bfSh_{K'}(G)$, we have a proper surjective morphism \cite[6.4.3.4]{lan}:
 \begin{equation}\label{E:morphism-g-tor}
 [g]^{\tor}: \bfSh^{\tor}_{K'}(G)\ra \bfSh^{\tor}_{K}(G),
 \end{equation}
 whose restriction to $\bfSh_{K'}(G)$ is \eqref{E:morphism-g-Sh} defined by the Hecke action of $g$.  
 %Moreover, $[g]^{\tor}$ is log-\'etale if we equip $\bfSh_{K'}^{\tor}(G)$ and $\bfSh^\tor_{K}(G)$ with the canonical log-structures given by their toroidal boundaries. 
Each double coset $K^pgK^p$ with $g\in \GL_{2}(\AAA^{\infty,p})$  defines an extended Hecke correspondence 
 \begin{equation}\label{E:Hecke-tor}
 \xymatrix{
 &\bfSh^\tor_{K\cap gKg^{-1}}(G)\ar[ld]_{[1]^{\tor}}\ar[rd]^{[g]^\tor}\\
 \bfSh^{\tor}_{K}(G) &&\bfSh^{\tor}_K(G),
 }
 \end{equation}
 which extends \eqref{E:Hecke-cor}.
 
 \end{itemize}
 
 We  put $\omegab=e^*(\Omega_{\calA^{\sm}/\bfSh^{\tor}_K(G)}^1)$, where $e:\bfSh_K^{\tor}(G)\ra \calA^{\sm}$ denotes the unit section. It is an $(\cO_{\bfSh^{\tor}_K}\otimes_{\ZZ}\cO_F)$-module locally free of rank $1$, and it extends  the sheaf of invariant differential $1$-forms of  $\calA$ over $\bfSh_K(G)$.
 We define the \emph{Hodge line bundle} to be $\det(\omegab)=\bigwedge_{\cO_{\bfSh^{\tor}_K(G)}}^{g}\omegab$.
  Following \cite{chai} and \cite[Section 7.2]{lan},  we put
\[
\bfSh^*_K(G)=\Proj\Big(\bigoplus_{n\geq 0}\Gamma\big(\bfSh^{\tor}_K(G),\det(\omegab)^{\otimes n}\big)\Big).
\]
This is a normal and projective scheme over $\ZZ_{(p)}$, and $\det(\omegab)$ descends to an ample line bundle on $\bfSh^*_K(G)$. Moreover, the inclusion $\bfSh_K(G)\hra \bfSh^{\tor}_K(G)$ induces an inclusion $\bfSh_K(G)\hra \bfSh^{*}_K(G)$. Although  $\bfSh^{\tor}_K(G)$ depends on the choice of certain cone  decompositions, $\bfSh^*_K(G)$ is canonically determined by $\bfSh_K(G)$. We call $\bfSh^*_K(G)$ the \emph{minimal compactification } of  $\bfSh_K(G)$.
The boundary $\bfSh^{*}_K(G)- \bfSh_K(G)$ is a finite flat scheme over  $\ZZ_{(p)}$, and its connected components are  indexed by the cusps of  $\bfSh_{K}(G)$.

\subsection{De Rham cohomology}\label{Subsection:dR-cohomology}
  Let $\Fgal$ be the Galois closure of $F/\Q$ contained in $\CC$.
   Let $R$ be an $\cO_{\Fgal, (p)}$-algebra. 
   In practice, we will need the cases where $R$ equals to  $\CC$,  a finite field $k$ of characteristic $p$ sufficiently large, or a finite extension $L/\QQ_p$ that contains all the embeddings  of $F$ into $\Qb_p$, or the ring of integers of such  $L$. 
   The set of the $g$ distinct algebra homomorphisms from $\cO_F$ to $R$ is naturally identified with $\Sigma_{\infty}$. % Hence, we have $\Sigma_{\infty}=\Sigma_{\CC}$ in this notation.
 %Let  $\bfSh^{\tor}_K(G)$ and  $\bfSh_K(G)$ be as in previous section.
  To simplify the notation, we put $\bfSh_{K,R}:=\bfSh_{K}(G)_R$ and $\bfSh_{K,R}^{\tor}:=\bfSh_{K}^\tor(G)_R$, and we write $\bfSh_K(\C)$ and $\bfSh_K^\tor(\C)$ for the associated complex manifolds respectively.  For a coherent $(\cO_{\bfSh_{K,R}^{\tor}}\otimes_{\ZZ}\cO_F)$-module  $M$, we denote by $M=\bigoplus_{\tau\in \Sigma_\infty}M_\tau$ the canonical decomposition, where $M_{\tau}$ is the direct summand on which $\cO_F$ acts via $\tau:\cO_F\ra R\ra \cO_{\bfSh^\tor_{K,R}}$. (This uses the fact that $p$ is unramified in $F$.)

 Let $\D$ denote the boundary $\bfSh_{K,R}^{\tor}-\bfSh_{K,R}$, and $\Omega^1_{\bfSh^{\tor}_{K,R}/R}(\log \D)$ the sheaf of differential $1$-forms on $\bfSh_{K,R}^{\tor}$ over $\Spec(R)$ with logarithmic poles along the relative normal crossing Cartier divisor $\D$. 
Using a toroidal compactification of the semi-abelian scheme $\calA^{\sm}$ on $\bfSh_{K,R}^{\tor}$, there exists a unique $(\cO_{\bfSh_{K,R}^{\tor}}\otimes\cO_F)$-module $\calH^1$ locally free of rank $2$ satisfying the following properties \cite[2.15, 6.9]{lan2}:
\begin{itemize}
\item[(1)] The restriction of $\calH^1$ to $\bfSh_{K,R}$ is  the relative de Rham cohomology $\calH^1_{\dR}(\calA/\bfSh_{K,R})$ of the universal abelian scheme $\calA$. 
In \cite[6.9]{lan2}, $\calH^1$ is called the \emph{canonical extension} of $\calH^1_{\dR}(\calA/\bfSh_{K,R})$.  

\item[(2)] There exists a canonical $\calO_F$-equivariant Hodge filtration
\[0\ra \omegab\ra \calH^1\ra \Lie((\calA^{\sm})^\vee)\ra 0.
\]
Taking the $\tau$-component gives 
\begin{equation}\label{Equ:tau-sequence-sh}
0\ra \omegab_{\tau}\ra \calH^1_{\tau}\ra \wedge^2(\calH^1_{\tau})\otimes\omegab_{\tau}^{-1}\ra 0.
\end{equation}
The line bundle $\wedge^2(\calH^1_{\tau})$ can be trivialized over  $\bfSh^{\tor}_{K,R}$ using the prime-to-$p$ polarization, but we choose
to keep it in order to make the Kodaira-Spencer isomorphism Hecke-equivariant. 

%Over  $\bfSh^{\tor}_{K,R}$, the prime-to-$p$ polarization on $\calA^{\sm}$ induces an isomorphism of $(\cO_{\bfSh^{\tor}_{K,R}}\otimes\cO_F)$-modules
%\[
%\Lie((\calA^{\sm})^\vee)\xra{\sim}  \omegab^{*}\simeq \bigoplus_{\tau\in\Sigma_{R}}\omegab^{\otimes-1}_{\tau},
%\]
%where $\omegab^*=\Hom_{\cO_{\bfSh^{\tor}_{K,R}}}(\omegab,\cO_{\bfSh^{\tor}_{K,R}})$ and the last isomorphism uses the fact that $\gothd_F$ is coprime to $p$. For each $\tau\in \Sigma_R$, we get an exact sequence of locally free $\cO_{\bfSh^{\tor}_{K,R}}$-modules $0\ra \omegab_{\tau}\ra \calH^1_{\tau}\ra \omegab_{\tau}^{-1}\ra 0$. Hence the line bundle $\wedge^2(\calH^1_{\tau})$ can be trivialized over  $\bfSh^{\tor}_{K,R}$. To encode the action of $\cO_F$ on $\calH^1_{\tau}$, it is convenient to write the exact sequence as 

\item[(3)] There exists an $\calO_F$-equivariant  integral  connection with logarithmic poles
\[
\nabla: \calH^1\ra \calH^1\otimes_{\cO_{\bfSh^{\tor}_{K,R}}}\Omega^1_{\bfSh^{\tor}_{K,R}/R}(\log \D),
\] 
which extends the Gauss-Manin connection on $\calH^1_{\dR}(\calA/\bfSh_{K,R})$.

\item[(4)] Let  $\mathrm{KS}$ be the map
\[
\mathrm{KS}:\omegab\hra \calH^1\xra{\nabla}\calH^1\otimes_{\cO_{\bfSh^{\tor}_{K,R}}}\Omega^1_{\bfSh^{\tor}_{K,R}/R}(\log \D)\ra \Lie((\calA^\sm)^\vee)\otimes_{\cO_{\bfSh^{\tor}_{K,R}}}\Omega^1_{\bfSh^{\tor}_{K,R}/R}(\log \D).
\]
It induces an \emph{extended Kodaira-Spencer isomorphism} \cite[6.4.1.1]{lan}
\begin{equation}\label{Equ:Kod-isom}
\Kod: \Omega^1_{\bfSh^{\tor}_{K,R}}(\log \D)\xra{\sim}\omegab\otimes_{(\cO_{\bfSh^{\tor}_{K,R}}\otimes \cO_F)}\Lie(\calA^\vee)^*\cong \bigoplus_{\tau\in\Sigma_{\infty}}\underline \Omega_\tau,
% \quad \text{with }\underline \Omega_\tau  = \omegab_{\tau}^{\otimes 2} \otimes \wedge^2(\calH_\tau^1)^{-1}
\end{equation}
where $\underline \Omega_\tau  = \omegab_{\tau}^{\otimes 2} \otimes \wedge^2(\calH_\tau^1)^{-1}$.  

Let $\bigwedge^*_{\Z}(\Z[\Sigma_\infty])$ be the exterior algebra of the $\Z$-module $\Z[\Sigma_{R}]$, and $(e_{\tau})_{\tau\in \Sigma_\infty}$ denote the natural basis. We fix an order on $\Sigma_{R}=\{\tau_1,\dots,\tau_g\}$.
We put $e_{\emptyset }=1$ and $e_J=e_{i_1}\wedge \dots \wedge e_{i_j}$, for any subset $J=\{\tau_{i_1},\dots, \tau_{i_j}\}$ with $i_1<\cdots<i_j$.
We call these $e_J$ \emph{\v Cech symbols}, where the relations $e_j \wedge e_{j'} = -e_{j'} \wedge e_j$ are built in the definition to get the correct signs.
Using them, we can write more canonically $\Kod: \Omega^1_{\bfSh^{\tor}_{K,R}}(\log \D)\cong \bigoplus_{\tau\in \Sigma_\infty}
\underline \Omega_\tau e_{\tau}$. It induces  an isomorphism of graded algebras
\begin{equation}\label{Equ:Kod-j-isom}
\Omega^\bullet_{\bfSh^{\tor}_{K,R}/R}(\log\D)=\bigoplus_{0\leq j\leq g}\Omega^j_{\bfSh_{K,R}^{\tor}/R}(\log \D)\cong \bigoplus_{J\subseteq \Sigma_\infty}\underline \Omega^J e_{J},
\end{equation}
where 
$\underline \Omega^J = \bigotimes_{\tau \in J} \underline \Omega_\tau$.

%The K

%\item The integrable connection $\nabla$ on $\calH^1$ induces a connection on each component $\calH^1_{\tau}$ for each  $\tau\in \Sigma_R$. Let $\pr_{\tau}:\calH^1\otimes \Omega^1_{\bfSh^{\tor}_{K,R}/R}(\log \D)\ra\calH^1_{\tau}\otimes \omegab_\tau^{\otimes 2}e_{\tau}$ be the projection induced by the Kodaira-Spencer isomorphism \eqref{Equ:Kod-isom}. We define $\nabla_{\tau}$ as the composite map
%\begin{equation}\label{Equ:partial-derivative}
%\calH^1_{\tau}\hra\calH^1\xra{\nabla}\calH^1\otimes_{\cO_{\bfSh^{\tor}_{K,R}}}\Omega^1_{\bfSh_{K,R}^{\tor}/R}(\log \D)\xra{ \pr_\tau}\calH^1_\tau\otimes_{\cO_{\bfSh^{\tor}_{K,R}}}\omegab^{\otimes 2}_\tau e_{\tau}.
%\end{equation}
%This can be viewed as a partial derivative on $\calH^1_{\tau}$.

\end{itemize}

\subsection{Integral models of automorphic bundles}\label{S:descent}

For a multi-weight $(\kb,w)\in \Z^{\Sigma_{\infty}}\times \Z$ (not necessarily cohomological),  we put 
 \begin{equation}
\label{E:modular line bundle}
\omegab^{(\kb,w)}\colon =\bigotimes_{\tau\in\Sigma_\infty} \biggl((\wedge^2\calH^1_\tau)^{\frac{w-k_{\tau}}{2}-1}\otimes \omegab_\tau^{k_\tau}\biggr),
\end{equation}
which is a line bundle on $\bfSh^{\tor}_{K,R}$. 
Note that, using the Kodaira-Spencer isomorphism \eqref{Equ:Kod-isom}, we have 
\[
\underline \omega^{(\kb, w)}\cong    \bigotimes_{\tau\in\Sigma_\infty} \biggl((\wedge^2\calH^1_\tau)^{\frac{w-k_{\tau}}{2}}\otimes \omegab_\tau^{k_\tau-2}\biggr)\otimes\Omega^{g}_{\bfSh_{K,R}}(
\log \D).
\]
%We will see later that the right hand side is naturally a  sub-bundle of $\F^{(\kb,w)}\otimes \Omega^{g}_{\bfSh_{K,R}}(
%\log \D)$ given by the Hodge filtration.

We define \emph{the space of Hilbert modular forms of weight $(\kb,w)$ and level $K$ with coefficients in $R$} to be
 $$
 M_{(\kb,w)}(K,R)\colon =H^0(\bfSh^{\tor}_{K,R},\omegab^{(\kb,w)}),
 $$
and the subspace of cusp forms to be 
\[
S_{(\kb,w)}(K,R)\colon = H^0(\bfSh^{\tor}_{K,R}, \omegab^{(\kb,w)}(-\D)).
\]
By Koecher's principle, one has  $M_{(\kb,w)}(K, R)=H^0(\bfSh_{K,R},\omegab^{(\kb,w)})$, which coincides with the definition \eqref{E:defn-HMF} when $R$ is a subfield of $\C$.  
In particular, both spaces $M_{(\kb, w)}(K, R)$ and $S_{(\kb, w)}(K,R)$ do not depend on the choice of the toroidal compactification of $\bfSh_{K, R}$.

If   $(\uk,w)$ is cohomological, i.e. $w\geq k_{\tau}\geq 2$ for $\tau\in \Sigma_{\infty}$,  we put 
\[
\F^{(\uk,w)}_{\tau}\colon = (\wedge^2\calH^1_{\tau})^{\frac{w-k_{\tau}}{2}} \otimes \Sym^{k_{\tau}-2}\calH^1_{\tau},\quad \textrm{and} \quad
\F^{(\kb,w)}\colon  = \bigotimes_{\tau\in\Sigma_{\infty}} \F^{(\kb,w)}_{\tau}.
\]
The extended Gauss-Manin connection on $\calH^1$ induces by functoriality an integrable connection 
$$
\nabla: \F^{(\kb,w)}\ra \F^{(\kb,w)}\otimes\Omega^1_{\bfSh^{\tor}_{K,R}}(\log \D).
$$ 
%Here, by abuse of notation, we still denote by $\D$ the toroidal boundary of $\bfSh^{\tor}_{K,R}$.
By considering the associated  local system $(\F^{(\kb,w)})^{\nabla=0}$ on $\bfSh_{K}(\C)$, it is easy to see that $(\F^{(\kb,w)},\nabla)$ on $\bfSh^{\tor}_{K,\cO_{F',(p)}}$ gives an integral model of the corresponding automorphic bundle on $\bfSh_{K}(\C)$ considered in Subsection~\ref{S:automorphic-bundle}.

Suppose we are given an element $g\in G(\AAA^{\infty,p})$, and open subgroups $K'^p, K^p\subset G(\AAA^{\infty,p})$ with $g^{-1}K'^pg\subseteq K^p$. Let $K'=K'^pK'_p$ and $K=K^pK_p$ with $K'_p=K_p=\GL_2(\cO_{F}\otimes\Z_p)$. Let $[g]^{\tor}: \bfSh^{\tor}_{K',R}\ra \bfSh_{K,R}^{\tor}$ denote the morphism \eqref{E:morphism-g-tor}. Then according to \cite[Theorem 2.15(4)]{lan}, we have canonical isomorphisms of vector bundles on $\bfSh^{\tor}_{K',R}$ 
\begin{equation}\label{E:isom-g-F}
[g]^{\tor,*}(\F^{(\kb,w)})\xra{\cong} \F^{(\kb,w)},
\end{equation}
compatible with the connection $\nabla$ on $\F^{(\kb,w)}$ and  the Hodge filtration to be defined in Subsection~\ref{S:DR-Hodge}. Similarly, we have an isomorphism on $\bfSh^{\tor}_{K',R}$: $[g]^{\tor,*}(\omegab^{(\kb,w)}) \xra{\cong} \omegab^{(\kb,w)}$ for a general multi-weight $(\kb,w)$. 

%As pointed out in Subsection~\ref{S:automorphic-bundle}, the Kodaira-Spencer isomorphism \eqref{Equ:tau-sequence-sh} induces an isomorphism 

%\yichao{I want to modify a little the notation for modular line bundles. I wanted to define $\underline \omega^{(\kb, w)}$ as the right hand side of the formula above. In this convention, the other terms in the dual BGG-complex will be obtained simply by replacing $\omega^{(\kb,w)}$ with $\omega^{(s\cdot \kb, w)}$. Here, $s\cdot \kb$  is the usual dot action of an element of the Weyl group on the weights.}

\begin{remark}\label{R:automorphic-bundle}
Our definition of Hilbert modular forms (and cusp forms) differs slightly from those in \cite{DT,KL}, where they work over the fine moduli spaces  $\calM^\gothc_K$ and drop the factors $\wedge^2\calH^1_{\tau}$'s. 
As pointed above, the line bundles $\wedge^2\calH^1_{\tau}$  are trivialized  on each $\calM^\gothc_K$ using the polarization there. 
However, in order to descend from $\calM^\gothc_K$ to the quotient $\bfSh^{\gothc}_{K,R}$,  the authors in \emph{loc. cit.} have to modify carefully  the Hecke actions by a factor.
In this paper, we think it is  more canonical to keep the factors $\wedge^2\calH^1_{\tau}$ so that the Hecke action  descends naturally from $\calM_K^\gothc$ to $\bfSh_K^\gothc$. See also Remark~\ref{remark:U_p} below for an explanation for this issue.

%Moreover, introducing the factor $\Omega^g_{\bfSh_{K,R}}$ via Kodaira-Spencer isomorphism allows us to naturally relate the line bundle \eqref{E:modular line bundle} with the de Rham complex of $\scrF^{(\kb, w)}$ (see Subsection~\ref{S:BGG}).

Intuitively, the bundle $\calH^1_{\tau}$ on $\bfSh_{K,R}^{\tor}$ ``should be''  the automorphic vector bundle corresponding to the representation $\check{\St}_{\tau}$ of $G_{\C}=(\GL_{2,\C})^{\Sigma_{\infty}}$ in the sense of \cite[Chap. III]{milne}. 
However,   the representation $\check{\St}_{\tau}$ does not give rise to an automorphic vector bundle, because it does not factor through the quotient group $G^c_{\C}$ as explained in \emph{loc. cit.} (and hence the action of the global units $\calO_F^\times$ on the sections of $\calH_\tau^1$ is not trivial). Similarly, a line bundle of the form 
$$
\bigotimes_{\tau\in \Sigma_{\infty}}\bigl((\wedge^2\calH^1_\tau)^{m_{\tau}}\otimes \omegab_{\tau}^{k_{\tau}}\bigr)
$$
 with $m_{\tau},k_{\tau}\in \Z$, is an automorphic  vector bundle in the sense of \emph{loc. cit.} if and only if $2m_{\tau}+k_{\tau}=w$ is an integer independent of $\tau$.
\end{remark}

\subsection{De Rham complex and Hodge filtrations}\label{S:DR-Hodge}
Let $(\kb,w)$ be a cohomological multi-weight. 
We denote by   $\DR(\F^{(\kb,w)})$ the de Rham complex
\[
\F^{(\kb,w)}\xra{\nabla}\F^{(\kb,w)}\otimes\Omega^1_{\bfSh^{\tor}_{K,R}/R}(\log \D)\xra{\nabla}\cdots\xra{\nabla} \F^{(\kb,w)}\otimes\Omega^g_{\bfSh^{\tor}_{K,R}/R}(\log \D).
\]
 For a coherent sheaf $\scrL$ on $\bfSh^{\tor}_{K,R}$, we denote by $\scrL(-\D)$ the tensor product of $\scrL$ with the ideal sheaf of $\D$. For $0\leq i\leq g-1$, $\nabla$ induces a map
\[
\nabla: \F^{(\kb,w)}(-\D)\otimes \Omega_{\bfSh_{K,R}^{\tor}/R}^{i}(\log \D)\ra \F^{(\kb,w)}(-\D)\otimes \Omega^{i+1}_{\bfSh_{K,R}^{\tor}/R}(\log \D).
\]
We denote by $\DR_c(\F^{(\kb,w)})$ the resulting complex by tensoring the complex $\DR(\F^{(\kb,w)})$ with  $\cO_{\bfSh_{K,R}^{\tor}}(-\D)$. 

The complex $\DR(\F^{(\kb,w)})$ (and similarly $\DR_c(\F^{(\kb,w)})$) is equipped with a natural Hodge filtration. 
 Let $(\omega_{\tau},\eta_{\tau})$ be a local basis of $\calH^1_{\tau}$ adapted to the Hodge filtration \eqref{Equ:tau-sequence-sh}.
 We define
 $\tF^n\F_{\tau}^{(\kb,w)}$ to be  the submodule generated by the vectors 
 $$
 \{(\omega_{\tau}\wedge\eta_{\tau})^{\frac{w-k_{\tau}}{2}}\otimes \omega_{\tau}^{i}\otimes \eta_{\tau}^{k_{\tau}-2-i}:  n-\frac{w-k_{\tau}}{2}\leq i\leq k_{\tau}-2\}
 $$
 if $\frac{w-k_{\tau}}{2}\leq n\leq \frac{w-k_{\tau}}{2}+k_{\tau}-2 $,  and
\begin{equation}\label{Equ:Filt-tau}
\tF^n\F_{\tau}^{(\kb,w)}=\begin{cases}\F_{\tau}^{(\kb,w)}&\text{if }n\leq \frac{w-k_{\tau}}{2}\\
0& \text{if } n\geq \frac{w-k_{\tau}}{2}+k_{\tau}-1.\end{cases}
\end{equation}
 The filtration does not depend on the choice of $(\omega_{\tau},\eta_{\tau})$, and the graded pieces of the filtration are
\begin{equation*}\label{Equ:graded-filt-tau}
\Gr^{n}_{\tF}\F_{\tau}^{(\kb,w)}\cong
\begin{cases}
(\wedge^2\calH^1_{\tau})^{w-n-2}\otimes \omegab_{\tau}^{2n+2-w}& \text{if }n\in [\frac{w-k_{\tau}}{2},\frac{w+k_{\tau}}{2}-2]\\
0&\text{otherwise.}
\end{cases}
\end{equation*}

Now consider the sheaf $\F^{(\kb,w)}$. We endow it with the tensor product filtration induced by $(\tF^n\F^{(\kb,w)}_{\tau}, n\in
\Z)$ for $\tau\in \Sigma_{R}$. The $\tF$-filtration on $\F^{(\kb,w)}$ satisfies Griffiths transversality for $\nabla$, i.e. we have
\[
\nabla: \tF^n\F^{(\kb,w)}\ra \tF^{n-1}\F^{(\kb,w)}\otimes \Omega^{1}_{\bfSh^{\tor}_{K,R}/R}(\log \D).
\]
We define $\tF^n\DR(\F^{(\kb,w)})$ as the subcomplex  $\tF^{n-\bullet}\F^{(\kb,w)}\otimes \Omega^{\bullet}_{\bfSh^{\tor}_R/R}(\log \D)$ of $\DR(\F^{(\kb,w)})$, and call it the \emph{$\tF$-filtration}  (or \emph{Hodge filtration}) on $\DR(\F^{(\kb,w)})$.
The $\tF$-filtration on $\DR(\F^{(\kb,w)})$ induces naturally an  $\tF$-filtration on $\DR_{c}(\F^{(\kb,w)})$.

\subsection{The dual BGG-complex}\label{S:BGG}
Assume that $(k_{\tau}-2)!$ is invertible in $R$ for every $\tau\in \Sigma_{\infty}$.
It is well known that $\DR(\F^{(\kb,w)})$ (resp.  $\DR_{c}(\F^{(\kb,w)})$) is quasi-isomorphic to a  much simpler complex  $\BGG(\F^{(\kb,w)})$ (resp. $\BGG_{c}(\F^{(\kb,w)})$), called the dual BGG-complex of $\F^{(\kb,w)}$. Here, we tailor the discussion for later application and refer the reader to \cite[\S 3 and \S 7]{Fa}  and \cite{lan-polo} for details. 

The Weyl group of $G_R=(\Res_{\cO_F/\Z}\GL_2)_R$  is canonically isomorphic to $W_G=\{\pm 1\}^{\Sigma_\infty}$. 
For a subset $J\subseteq \Sigma_{\infty}$,  let $s_{J}\in W_G = \{\pm 1\}^{\Sigma_\infty}$ be the element whose $\tau$-component is $-1$ for $\tau\notin J$ and is $1$ for $\tau\in J$.
 In particular,  $s_{\Sigma_\infty}=1$ is the identity element of $W_G$, and $s_{\emptyset}$ is the longest element.
We have the usual dot action of $W_G$ on $\ZZ^{\Sigma_{\infty}}$: for $J\subseteq \Sigma_\infty$ and $\kb\in\ZZ^{\Sigma_\infty}$, the $\tau$-component of  $s_{J}\cdot \kb$ is equal to $2-k_{\tau}$ for $\tau\notin J$, and to $k_{\tau}$ for $\tau\in J$.

%We put also $t_J=\sum_{\tau \in J}\tau \in \Z[\Sigma_{R}]$, $t=t_{\Sigma_R}$, and $\omegab^{t_J}=\otimes_{\tau\in J}\omegab_{\tau}$. 
%We define
%\begin{align}\label{Equ:omega-J}
%\omegab^{\epsilon_J(\kb,w) }:
%&=\biggl(\bigl(\bigotimes_{\tau\notin J}(\wedge^2\calH^1_{\tau})^{\frac{w+k_{\tau}}{2}-2}\otimes\omegab_{\tau}^{2-k_{\tau}}\bigr) \otimes \bigl(\bigotimes_{\tau\in J}(\wedge^2\calH^1_{\tau})^{ \frac{w-k_{\tau}}{2}}\otimes \omegab_{\tau}^{k_\tau-2}\bigr)\biggr).
%\end{align}
%This is an integral model of an automorphic vector bundle (See Remark~\ref{R:automorphic-bundle}).
For any $0\leq j\leq g$, we put
\begin{align}\label{Equ:BGG}
\bgg^j(\F^{(\kb,w)})&= \bigoplus_{\substack {J\subseteq \Sigma_{R}\\  \#J=j}}\omegab^{(s_{J}\cdot \kb,w)}  e_{J},
\end{align}
where $e_J$ is the \v{C}ech symbol as in \eqref{Equ:Kod-j-isom}, and 
%Note that, via the Kodaira-Spencer isomorphism \eqref{Equ:Kod-isom}, we have 
\begin{equation}\label{E:omega-s-j}
\omegab^{(s_J\cdot \kb,w)}=\bigl(\bigotimes_{\tau\notin J}(\wedge^2\calH^1_{\tau})^{\frac{w+k_{\tau}}{2}-2}\otimes\omegab_{\tau}^{2-k_{\tau}}\bigr) \otimes \bigl(\bigotimes_{\tau\in J}(\wedge^2\calH^1_{\tau})^{ \frac{w-k_{\tau}}{2}-1}\otimes \omegab_{\tau}^{k_\tau}\bigr)
%&\cong \bigl(\bigotimes_{\tau\notin J} \Gr_{\tF}^{\frac{w-k_\tau}{2}}\F_{\tau}^{(\kb,w)}\bigr)\otimes 
%\bigr(\bigotimes_{\tau\in J}
%\Gr_{\tF}^{\frac{w+k_{\tau}}{2}-2}\F^{(\kb,w)}_{\tau}
%\bigl)\otimes \Omega_{J}e_{J}
\end{equation}
%Here,  $\omegab^{2t_{J}}e_J$ should be regarded as a direct summand of the differential module $\Omega^j_{\bfSh_{K,R}^{\tor}}(\log\D)$ (see \eqref{Equ:Kod-j-isom}).
There exists a differential operator $d^j:\bgg^j(\F^{(\kb,w)})\ra\bgg^{j+1}(\F^{(\kb,w)})$ given by 
\begin{equation}\label{Equ:diff-BGG}
d^j\colon fe_{J}\mapsto \sum_{\tau_{0}\notin J} \Theta_{\tau_0,k_{\tau_0}-1}(f) e_{\tau_0}\wedge e_{J},
\end{equation}
where  $f$ is a local section $f$ of $\omegab^{(s_J\cdot\kb,w)}$ with $\#J=j$,  and $\Theta_{\tau_0,k_{\tau_0}-1}$ is a certain differential operator of order $k_{\tau_0}-1$ (See Remark~\ref{R:BGG}(1)), and it is an analog of the classical theta operator.
We define a decreasing $\tF$-filtration on $\BGG(\F^{(\kb,w)})$ by setting:
\[
\tF^n\BGG(\F^{(\kb,w)})=\bigoplus_{\substack {J\subseteq \Sigma_{R}\\n_J\geq n}}\omegab^{(s_J\cdot \kb,w)}\;e_{J}[-\#J],
\]
where $n_J: =\sum_{\tau\in J}(k_\tau-1)+\sum_{\tau\in\Sigma_\infty}\frac{w-k_\tau}{2}$. It is easy to see that $\tF^{n}\BGG(\F^{(\kb,w)})$ is stable under the differentials $d^j$, and the graded pieces
\begin{equation}\label{Equ:graded-dR-total}
\Gr^n_{\tF}\BGG(\F^{(\kb,w)})=\bigoplus_{\substack {J\subseteq \Sigma_{R}\\n_J= n}}\omegab^{(s_J\cdot\kb,w)}\;e_{J}[-\#J]
\end{equation}
have trivial induced differentials. 
Note that, via the Kodaira-Spencer isomorphism \eqref{Equ:Kod-isom}, one has a Hecke-equivariant isomorphism  
\[
\Gr^n_{\tF} \BGG(\F^{(\kb,w)})\cong \Gr^n_{\tF}\DR(\F^{(\kb,w)}).
\]

Finally, the differential $d^j$ preserves cuspidality, i.e. it induces a map 
$$
d^j:\bgg^j(\F^{(\kb,w)})(-\D)\ra \bgg^{j+1}(\F^{(\kb,w)})(-\D).
$$
we will denote by $\BGG_c(\F^{(\kb,w)})$ the resulting complex. The $\tF$-filtration on $\BGG(\F^{(\kb,w)})$  induces an $\tF$-filtration on $\BGG_c(\F^{(\kb,w)})$, and the graded pieces $\Gr^n_{\tF}(\BGG_c(\F^{(\kb,w)}))$ are given by \eqref{Equ:graded-dR-total} twisted by $\cO_{\bfSh^{\tor}_{K,R}}(-\D)$.

\begin{theorem}[Faltings; cf. {\cite[\S 3 and \S 7]{Fa} , \cite[Chap. \S 5]{FC}, \cite[\S 5]{lan-polo}} ]\label{Theorem:BGG}
Assume that  $(k_{\tau}-2)!$ is invertible in $R$ for each $\tau\in \Sigma_\infty$. Then  there is a canonical quasi-isomorphic embedding  of $\tF$-filtered complexes of abelian sheaves on $\bfSh_{K,R}^{\tor}$
\[
\BGG(\F^{(\kb,w)})\hra \DR(\F^{(\kb,w)}).
\]
Similarly, we have a canonical quasi-isomorphism of $\tF$-filtered  complexes
\[\BGG_c(\F^{(\kb,w)})\hra \DR_c(\F^{(\kb,w)}).\]
\end{theorem}

\begin{remark}\label{R:BGG}
	\begin{itemize}
\item[(1)] It is possible to give an explicit formula for the operator $\Theta_{\tau_0,k_{\tau_0}-1}$ appearing in \eqref{Equ:diff-BGG}. Let $f$ be a local section of $\omegab^{(s_{J}\cdot\kb,w)}$ with $q$-expansion 
\[
f=\sum_{\xi} a_{\xi} q^{\xi}
\]
at a cusp of $\bfSh^{\tor}_{K,R}$, where $\xi$ runs through $0$ and the set of totally positive elements in a lattice of $F$. Using the complex uniformization, one can show that
 \begin{align*}
 \Theta_{\tau_0,k_{\tau_0}-1}(f)%&=\frac{(-1)^{k_{\tau_0}-2}}{(k_{\tau_0}-2)!}(q_{\tau_0}\frac{\partial}{\partial_{q_{\tau_0}}})^{k_{\tau_0}-1}(f)\\
 &=\frac{(-1)^{k_{\tau_0}-2}}{(k_{\tau_0}-2)!}\sum_{\xi}\tau_0(\xi)^{k_{\tau_0}-1}a_{\xi}q^{\xi}.
 \end{align*}
The denominator $(k_{\tau_0}-2)!$ explains the assumption that $(k_{\tau}-2)!$ is invertible in $R$ for every $\tau$. The main results of this paper do not use this formula on $q$-expansions.

\item[(2)] The  embedding $\BGG(\F^{(\kb,w)})\hra \DR(\F^{(\kb,w)})$ is constructed using reprensentation theory, and the morphisms in each degree are given by differential operators rather than morphisms of $\cO_{\bfSh_{K,R}^{\tor}}$-modules (cf. \cite[\S 3]{Fa} and \cite[Chap. VI \S5]{FC}). 
When $k_{\tau}=3$ for all $\tau\in \Sigma_{\Sigma}$, the embedding $\bgg^\bullet(\F^{(\kb,w)})\hra \F^{(\kb,w)}$ splits the Hodge filtration on $\calH^1$ globally as abelian sheaves over $\bfSh_{K,R}^{\tor}$, and it is certainly not $\cO_{\bfSh_{K,R}^{\tor}}$-linear.

\item[(3)] Assume $R=\C$. Let
$\LL^{(\kb,w)}$ denote the local system $\F^{(\kb,w)}(\C)^{\nabla=0}$  on the complex manifold $\bfSh_{K}(\C)$, and  $j:\bfSh_{K}(\C)\hra \bfSh^{\tor}_{K}(\C)$ be the open immersion. Then by the Riemann-Hilbert-Deligne correspondence,  $\DR(\F^{(\kb,w)})$ resolves  $Rj_*(\LL^{(\kb,w)})$, and $\DR_c(\F^{(\kb,w)})$ resolves the sheaf $j_!(\LL^{(\kb,w)})$ \cite[Chap. VI. 5.4]{FC}.
\end{itemize}
\end{remark}

\section{Overconvergent Hilbert Modular Forms}
\label{Section:ocgt HMF}

\subsection{Notation}\label{S:HMF-notation}
We fix a  number field $L\subset \C$ containing  $\tau(F)$ for all $\tau\in \Sigma_{\infty}$. The fixed isomorphism $\iota_p:\C\xra{\simeq}{\Qpb}$ determines a $p$-adic place $\wp$ of $L$. We denote by $L_{\wp}$ the completion,  $\calO_{\wp}$   the ring of integers, and $k_0$  the residue field.  The isomorphism $\iota_p$ also identifies $\Sigma_{\infty}$  with  the set of $p$-adic embeddings $\Hom_{\Q}(F,\Qpb)=\Hom_{\Z}(\cO_F,k_0)$. 
The natural action of the Frobenius on $\Hom_\ZZ(\cO_{F},k_0)$ defines, via the identification above, a natural action on $\Sigma_{\infty}$: $\tau\mapsto \sigma\circ\tau$. 
We have a natural partition: $\Sigma_{\infty}=\coprod_{\gothp\in \Sigma_p}\Sigma_{\infty/\gothp}$, where $\Sigma_{\infty/\gothp}$ consists of the $\tau$'s such that $\iota_p\circ\tau$ induces the place $\gothp$. 
For any $\calO_\wp$-scheme $S$ and a coherent $(\cO_S\otimes\cO_F)$-sheaf $M$, we have a canonical decomposition 
$
M=\bigoplus_{\tau\in \Sigma_{\infty}} M_{\tau},
$
where $M_{\tau}$ is the direct summand of $M$ on which $\cO_F$ acts via $\tau: \cO_F\ra \cO_\wp\ra \cO_S$.

Unless stated otherwise, we  take the open compact subgroups $K=K^pK_p\subset G(\AAA^{\infty})$ so that $K_p=\GL_2(\cO_F\otimes \Z_p)$ and that $K^p$ satisfies Hypothesis~\ref{H:fine-moduli}. Then the corresponding Shimura variety $\bfSh_K(G)$ is a fine moduli space of abelian varieties over $\Z_{(p)}$. We choose a toroidal compactification $\bfSh^{\tor}_K(G)$, and let $\bfSh^*_{K}(G)$ be the minimal compactification as in  Subsection \ref{Subsection:compactification}.
 To simplify notation, let $\bfX_K$,  $\bfX^{\tor}_K$,  and $\bfX^*_K$ denote the base change to $W(k_0)$ of $\bfSh_{K}(G)$,  $\bfSh^{\tor}_{K}(G)$, and $\bfSh^*_{K}(G)$, respectively.
Let $X_K$, $X^{\tor}_K$, and $X^*_K$  be respectively their special fibers. Denote by  $\X^{\tor}_K$ the formal completion of $\bfX^{\tor}_K$ along $X_K^\tor$,
 and $\X^\tor_{K,\rig}$ the base change to $L_{\wp}$ of the associated rigid analytic space over $W(k_0)[1/p]$.  For a locally closed subset  $U_0\subset X_K$,  we use $]U_0[$ to denote the inverse image of $U_0$ under the specialization map
$\mathrm{sp}: \X_{K,\rig}^\tor\ra X^{\tor}_K$. 
Similarly, we have the evident variants $\X^*_K$, $\X^*_{K,\rig}$ for the minimal compactification $\bfX^*$.
If there is no risk of confusion, we will use the same notation $\D$ to denote the toroidal boundary in various settings: $\bfX^{\tor}_K-\bfX_K$ and  $X^{\tor}_K-X_K$.

\subsection{Hasse invariant and ordinary locus}\label{Subsection:Hasse}
%Let $K=K^pK_p\subseteq G(\AAA^{\infty})$ be an open compact subgroup as in the previous subsection. 
Let $\calA^\sm_{k_0}$ be the semi-abelian scheme over $X^{\tor}_{K}$ that extends the universal abelian variety $\univA_{k_0}$ over $X_{K}$.
 The Verschiebung homomorphism $\mathrm{Ver}:(\calA^{\sm}_{k_0})^{(p)}\ra \calA^\sm_{k_0}$ induces an $\cO_F$-linear  map on the module of invariant differential $1$-forms:
\[
h:\omegab\ra \omegab^{(p)},
\]
 which  induces, for each $\tau\in \Sigma_{\infty}$ (identified with the set of $p$-adic embeddings of $F$),   a map $h_{\tau}:\omegab_{\tau}\ra \omegab_{\sigma^{-1}\circ\tau}^{p}$. This defines, for each $\tau\in \Sigma_{\infty}$, a  section
 $$
  h_\tau\in H^0(X_{K}^\tor, \omegab_{\sigma^{-1}\circ\tau}^{p}\otimes \omegab_{\tau}^{-1}).
 $$
 We put $h=\otimes_{\tau\in \Sigma_{\infty}}h_{\tau}\in \Gamma(X_K^\tor, \det(\omegab)^{p-1})$.
  We call $h$ and $h_{\tau}$ respectively the \emph{(total) Hasse invariant}, and the \emph{partial Hasse invariant at $\tau$}. 
  
  Let $Y_K$ and $Y_{K,\tau}$ be the closed subschemes of $X^\tor_K$ defined by the vanishing locus of $h$ and $h_{\tau}$. Each $Y_{K,\tau}$ is reduced and smooth, and $Y_K=\bigcup_{\tau}Y_{K,\tau}$ is a normal crossing divisor in $X^{\tor}_K$ \cite{goren-oort, goren}. The stratification given by taking the intersections of $Y_{K, \tau}$ are called the \emph{Goren-Oort stratification} (or \emph{GO-stratification} for short).  We call the complement $X^{\tor, \ord}_{K}=X^{\tor}_K-Y_K$ the \emph{ordinary locus}. This is the open subscheme of the  moduli space $X^{\tor}_K$ where the semi-abelian scheme $\calA^{\sm}_{k_0}$ is ordinary. We point out that $Y_K$ does not intersect the toroidal boundary $\D=X^{\tor}_K-X_K$.
  
   Similarly, for the minimal compactification $X^*_K$, we put $X^{*,\ord}_K=X^*_K-Y_K$. Since $\det(\omegab)$ is an  ample line bundle on $X^{*}_K$ (see Subsection~\ref{Subsection:compactification}), $X^{*,\ord}_K$ is affine.

\subsection{Overconvergent Cusp Forms}\label{Subsection:ovt-cusp}

Let $j:\;]X^{\tor,\ord}_K[\hra \X^{\tor}_{K,\rig}$ be the natural inclusion of rigid analytic spaces. When it is necessary, we write $j_{K}$ instead to emphasize the level $K$. For a coherent sheaf $\calF$ on $\X^\tor_{K,\rig}$, following Berthelot \cite{berthelot2, berthelot}, we define $j^{\dagger}\calF$ to be the sheaf on $\X^{\tor}_{K,\rig}$ such that, for all admissible open subset $U\subset \X^\tor_{K,\rig}$, we have
\[
\Gamma(U,j^{\dagger}\calF)=\varinjlim_{V}\Gamma(V\cap U, \calF),
\]
where $V$ runs through a fundamental system of strict neighborhoods of $]X_K^{\tor,\ord}[$ in $\X^{\tor}_{K,\rig}$.  An explicit fundamental system  of strict neighborhoods of $]X_K^{\tor,\ord}[$ in $\X^{\tor}_{K,\rig}$ can be constructed as follows. Let $\tilde{E}$ be a lift to characteristic $ 0$ of a certain power of the Hasse invariant $h$. For any rational number $r>0$, we denote by $]X^{\tor,\ord}_K[_r$  the admissible open subset of $\X^{\tor}_{K,\rig}$ defined by $|\tilde{E}|\leq p^{-r}$. Then  the admissible open subsets $]X_K^{\tor,\ord}[_r$ with $r\ra 0^+$ form a fundamental system of strict neighborhoods of $]X_K^{\tor,\ord}[$ in $\X_{K,\rig}^{\tor}$. For the minimal compactification $\bfX_K^{*}$, we can define similarly admissible open subsets $]X^{*,\ord}_K[_{r}$, which also form a fundamental system of strict neighborhoods of $]X^{*,\ord}_K[$ in $\X^{*}_{K, \rig}$.
 We again point out that $]X_K^{*,\ord}[_r$ are  affinoid subdomains of $\X^*_{K,\rig}$, while  $]X^{\tor,\ord}_K[_r$  are not.

For a multiweight $(\kb, w)\in \Z^{\Sigma_\infty}\times \Z$, we define the space of \emph{cuspidal overconvergent Hilbert modular forms} (\emph{overconvergent cusp forms} for short) with coefficients in $L_{\wp}$ to be
\[
S^{\dagger}_{(\kb,w)}(K,L_{\wp}):= H^0(\fX^{\tor}_{K,\rig},j^{\dagger}\omegab^{(\kb,w)}(-\D)).
\]
%where $t=\sum_{\tau\in \Sigma_{\infty}}\tau$, and $\omegab^{-2t}(-\D)$ is canonically isomorphic to $\Omega^{g}_{\X^{\tor}_{K,\rig}/{L_{\wp}}}$ via the Kodaira-Spencer isomorphism \eqref{Equ:Kod-j-isom}.
This space does not depend on the choice of the toroidal compactification. 
When there is no risk of confusions, we write $S^{\dagger}_{(\kb,w)}$ for $S^{\dagger}_{(\kb,w)}(K,L_\wp)$.

%Note that in the definition of $S_{(\kb,w)}^{\dagger}$, we do not require  $(\kb,w)$ to be . 
If the weight $(\kb,w)$ is cohomological (i.e. $w\geq k_{\tau}\geq 2$), the space of overconvergent cusp forms  contains all classical cusp forms with Iwahori level structure at $p$. 
We denote by   $\Iw_p=\prod_{\gothp\mid p}\Iw_{\gothp}\subset \GL_2(\cO_{F}\otimes_{\Z}\Z_p)$ the Iwahori subgroup, where 
\begin{equation}\label{E:Iwahori}
\Iw_{\gothp}=\Big\{g=\begin{pmatrix}a&b\\c&d\end{pmatrix}\in \GL_2(\cO_{F_{\gothp}})\;\Big|\; c\equiv 0\mod \gothp\Big\}.
\end{equation}
Let $S_{(\kb,w)}(K^p\Iw_p,L_{\wp})$ denote the space of classical Hilbert cusp forms  of multiweight $(\kb,w)$ and of prime-to-$p$ level $K^p$ and Iwahori level at all places above $p$. 
By the theory of canonical subgroups, there is a natural injection
\[
\iota:S_{(\kb,w)}(K^p\Iw_p,L_{\wp})\hookrightarrow S^{\dagger}_{(\kb,w)}(K,L_{\wp}).
\]
An overconvergent cusp form $f\in S^{\dagger}_{(\kb,w)}(K,L_{\wp})$ is called \emph{classical}, if it lies in the image of $\iota$.

%For each subset $J\subseteq \Sigma_\infty$, recall from Subsection~\ref{S:BGG} that $s_J\in W_G=\{\pm 1\}^{\Sigma_{\infty}}$ is the element which is $1$ at $\tau\in J$ and $-1$ at $\tau\notin J$.  We  put 
%\begin{equation}
%\label{Defn:S_J-dagger}
%S_{(s_J\cdot \kb,w)}^{\dagger}(K,L_{\wp})=\varinjlim_{U} H^0(U, \omegab^{(s_J\cdot\kb,w)}),
%\end{equation}
%where $\omegab^{ (s_J\cdot\kb,w)}$ is defined in \eqref{Equ:Kod-j-isom} and $\underline \Omega^J $ is defined just below \eqref{Equ:Kod-j-isom}.
%The space $S^\dagger_{\epsilon_{J}(\kb,w)}$ does not depend on the choice of the toroidal compactification.
 %As usual, when the context is clear, we write $S_{\epsilon_{J}(\kb,w)}^{\dagger}=S_{\epsilon_{J}(\kb,w)}^{\dagger}(K,L_\wp)$. 
%In particular, when $J=\Sigma_{\infty}$ we have  $S^{\dagger}_{\epsilon_{\Sigma_{\infty}}\cdot(\kb,w)}=S^{\dagger}_{(\kb,w)}$.

%We remark that, by Kodaira-Spencer isomorphism
%\eqref{Equ:Kod-isom}, we can identify $S^\dagger_{\epsilon_J (\underline k,w)}$ with $S^\dagger_{(\underline k',w)}$, where $k'_\tau = k_\tau$ if $\tau \in J$ and $k'_\tau =2-k_\tau$ if $\tau \notin J$.
%But we prefer to use \eqref{Defn:S_J-dagger} because keeping the differential forms reminds us the sheaf is part of the dual BGG complex.

 Recall that the dual BGG-complex $\BGG_c(\F^{(\kb,w)})$ is  quasi-isomorphic to the de Rham complex $\DR_c(\F^{(\kb,w)})$ (Theorem~\ref{Theorem:BGG}). By applying $j^{\dagger}$ to $\BGG_c(\F^{(\kb,w)})$ and taking global sections, we get a complex $\scrC^\bullet_K$ of  overconvergent cusp forms concentrated in degrees $[0,g]$ with
\begin{equation}\label{Equ:complex}
\scrC^{j}_{K}: =\bigoplus_{\substack{J\subseteq \Sigma_{\infty}\\ \#J=j}} S^{\dagger}_{(s_J\cdot\kb,w)}(K,L_{\wp})e_{J}.
\end{equation}
Here, $e_{J}$ is the symbol introduced in \eqref{Equ:Kod-j-isom} in order to get the correct signs, and the differential map $d^j:\scrC^{j}_K\ra \scrC^{j+1}_K$  is given by the formula \eqref{Equ:diff-BGG}.

\subsection{Rigid cohomology of the ordinary locus.}\label{subsection:rigid-coh}
  We denote by $j^{\dagger}\DR_{c}(\F^{(\kb,w)})$ the complex of sheaves on $\X^{\tor}_{K,\rig}$ obtained by applying $j^{\dagger}$ to each component of $\DR_c(\F^{(\kb,w)})$. We define
\[
R\Gamma_{\rig}(X^{\tor,\ord}_K,\D; \F^{(\kb,w)}): =R\Gamma(\X^{\tor}_{K,\rig}, j^\dagger\DR_c(\F^{(\kb,w)}))\]
as an object in the derived category of $L$-vector spaces, and its cohomology groups will be denoted by
\[
H^\star_{\rig}(X^{\tor,\ord}_K,\D; \F^{(\kb,w)}):=\mathbb{H}^\star(\X^{\tor}_{K,\rig}, j^{\dagger}\DR_c(\F^{(\kb,w)})),
\]
where the left hand side denotes the hypercohomology of the complex $j^{\dagger}\DR_c(\F^{(\kb,w)})$. In Section~4, we will interpret $H^\star_{\rig}(X^{\tor,\ord}_K,\D; \F^{(\kb,w)})$  as the rigid
cohomology of a certain isocrystal over the ordinary locus $X^{\tor,\ord}_{K}$ and with compact support in $\D\subset X^{\tor}_K$.

\begin{theorem}\label{Theorem:overconvergent}
 The object $R\Gamma_{\rig}(X^{\tor,\ord}_K, \D;\F^{(\kb,w)})$ in the derived category of $L$-vector spaces  is represented by the complex $\scrC^{\bullet}_K$ defined in \eqref{Equ:complex}.
 In particular, we have an isomorphism
\[
H^g_{\rig}(X^{\tor,\ord}_K,\D;\F^{(\kb,w)})\cong S^{\dagger}_{(\kb,w)}/\sum_{\tau\in \Sigma_{\infty}}\Theta_{\tau, k_{\tau}-1}(S^{\dagger}_{(s_{\Sigma_{\infty}\backslash\{\tau\}}\cdot\kb,w)}).
\]
\end{theorem}

The following Lemma is due to  Kai-Wen Lan.

\begin{lemma}[{\cite[Theorem 8.2.1.3]{lan3}}]\label{Lemma:vanishing}
Let $\pi: \bfX^{\tor}_{K,L}\ra \bfX^{*}_{K,L}$ be the canonical projection. Then for any  multi-weight $(\kb,w )\in \ZZ^{\Sigma_{\infty}}\times \Z$,   we have $R^{q}\pi_{*}\big( \omegab^{(\kb,w)}(-\D)\big)=0$ for $q>0$.
\end{lemma}

\begin{proof}[Proof of Theorem \ref{Theorem:overconvergent}] Since the complex $\BGG_c(\F^{(\kb,w)})$ is quasi-isomorphic to the compactly supported de Rham complex $\DR_c(\F^{(\kb,w)})$, we have
\begin{align*}
R\Gamma_{\rig}(X^{\tor,\ord}_K,\D; \F^{(\kb,w)})&\cong R\Gamma(\X_{K,\rig}^{\tor}, j^{\dagger}\BGG_c(\F^{(\kb,w)}))\\
&\cong R\Gamma(\X^{*}_{K,\rig}, R\pi_*j^{\dagger}\BGG_c(\F^{(\kb,w)})).
\end{align*}
Since  the boundary $\D\subset \X^{\tor}_{K,\rig}$ is contained in the ordinary locus  $]X_K^{\tor,\ord}[$, we deduce $R\pi_*j^\dagger=j^{\dagger}R\pi_*$. 
By Lemma~\ref{Lemma:vanishing}, we have $R\pi_{*}\omega^{(s_J\cdot \kb, w)}(-\D)=\pi_{*}\omega^{(s_J\cdot \kb, w)}(-\D)$ for any $J\subseteq \Sigma_{\infty}$, and hence
$R\pi_*\BGG_c(\F^{(\kb,w)})=\pi_*\BGG_c(\F^{(\kb,w)})$. Let $]X_K^{*,\ord}[_r$
for rational $r>0$ be the strict neighborhoods of $]X_K^{*,\ord}[$ considered in 
Subsection~\ref{Subsection:ovt-cusp}. Since the $]X_K^{*,\ord}[_r$'s are affinoid and form a fundamental system of strict neighborhoods of $]X_K^{*,\ord}[$ in $\X^*_{K,\rig}$, we deduce that
\begin{align*}
&{\ }\quad H^{n}\big(\X_{K,\rig}^{*},  j^{\dagger}\pi_*\omegab^{(s_J\cdot\kb,w)}  (-\D)\big)\\
&=\varinjlim_{r\ra 0} H^{n}\big(\,]X_K^{*,\ord}[_r, \pi_*\omegab^{(s_J\cdot\kb,w)} (-\D)\big)=
\begin{cases}
S^{\dagger}_{(s_J\cdot\kb,w)} & \text{for } n=0,\\ 0 & \text{for } n\neq 0.
\end{cases}
\end{align*}
It follows that
$$
R\Gamma(\X_{K,\rig}^*, j^{\dagger}\pi_*\omegab^{(s_J\cdot \kb,w)}(-\D))=S^{\dagger}_{(s_J\cdot \kb,w)},
$$
and hence
$$
R\Gamma_{\rig}(X_K^{\tor,\ord},\D; \F^{(\kb,w)})=R\Gamma(\X^*_{K,\rig}, j^\dagger\pi_*\BGG_c(\F^{(\kb,w)}))=\scrC^{\bullet}_K.
$$
This completes the proof of the theorem.
\end{proof}

\subsection{Prime-to-$p$ Hecke actions} Let $\scrH(K^p,L_\wp)=L_\wp[K^p\backslash G(\AAA^{\infty, p})/K^p]$ be the abstract prime-to-$p$ Hecke algebra of level $K^p$. We will define  actions of $\scrH(K^p,L_\wp)$ on  $H^{\star}_{\rig}(X^{\texttt{}\tor,\ord}_K,\D; \F^{(\kb,w)})$ and on the complex $\scrC^{\bullet}_K$ such that the isomorphisms described in Theorem \ref{Theorem:overconvergent} are equivariant of the actions.

 Consider the double coset $[K^pg K^p]$.  We put $K'^p=K^p\cap g K^pg^{-1}$ and $K'=K'^p K_p$. By choosing suitable rational polyhedral cone decomposition data, we have  the following Hecke correspondence \eqref{E:Hecke-tor}:
\[\xymatrix{
&\bfX^{\tor}_{K'}\ar[rd]^{\pi_2=[g]^\tor}\ar[ld]_{\pi_1=[1]^\tor}\\
\bfX^{\tor}_{K}&&\bfX^{\tor}_{K}.
}
\]
Using the isomorphisms \eqref{Equ:Kod-j-isom} and \eqref{E:isom-g-F}, one has a  map  of complexes of sheaves
\[
\pi_2^*\colon  \pi_2^{-1}\DR_c(\F^{(\kb,w)})\ra \DR_c(\F^{(\kb,w)}),
\]
which is compatible with the $\tF$-filtration. 
For each $J\subseteq \Sigma_{\infty}$, the sheaf $\omegab^{(s_J\cdot \kb,w)}(-\D)$  appears as a direct summand of $\Gr_{\tF}^\bullet\DR_c(\F^{(\kb,w)})$ by Theorem~\ref{Theorem:BGG}. 
The  morphism above induces a map of abelian sheaves
\[
\pi_2^*\colon \pi_2^{-1}(\omegab^{ (s_J\cdot \kb,w)}(-\D))\ra\omegab^{ (s_J\cdot\kb,w)}(-\D_{K'}),
\]
where $\D_{K'}$ is the boundary of $\bfX^{\tor}_{K'}$, 
and hence a morphism of the BGG-complexes 
 $$\pi_2^*: \pi_2^{-1}\BGG_c(\F^{(\kb,w)})\ra \BGG_c(\F^{(\kb,w)}).
 $$ 
 It is  clear that the two morphisms $\pi^*_2$ are compatible with the natural quasi-isomorphic inclusion   $\BGG_c(\F^{(\kb,w)})\hra \DR_c(\F^{(\kb,w)})$ in Theorem~\ref{Theorem:BGG}.
%In order to define the action of $[K^pg K^p]\in \scrH_{K}$ on $R\Gamma(X^{\tor,\ord}_{K},\D;\F^{(\kb,w)})=R\Gamma(\X^{\tor}_{K,\rig}, j^{\dagger}_{K}\DR_c(\F^{(\kb,w)}))$, we need the following 
\begin{lemma}\label{L:trace-map}
 Under  the above notation,  we have $R^q\pi_{1,*}\cO_{\bfX^{\tor}_{K'}}=0$ for $q>0$, and $\pi_{1,*}(\cO_{\bfX_{K'}^{\tor}})$ is finite flat over $\cO_{\bfX^{\tor}_{K}}$.
\end{lemma}
\begin{proof}
The statement is clear over $\bfX_{K}$, since $\pi_1$ is finite \'etale there. Therefore, it is enough to prove the lemma after changing base of $\pi_{1}$ to  the completion of $\bfX_{K}^{\tor}$ along $\D_K=\bfX^{\tor}_{K}-\bfX_K$. The morphism $\pi_1$ over this completion  is \'etale locally given by equivariant morphisms between toric varieties. So the result follows from  similar arguments as in \cite[Ch.~I~\S3]{KKMS}. 
\end{proof}

\begin{cor}\label{C:trace}
There exist natural trace maps 
\begin{align*}
\Tr_{\pi_1}\colon& R\pi_{1,*}\DR_c(\F^{(\kb,w)})\ra \DR_c(\F^{(\kb,w)}), \quad \textrm{and}\\
\Tr_{\pi_1}\colon& R\pi_{1,*}\big(\omegab^{(s_J\cdot\kb,w)} (-\D_{K'})
\big)\ra\omegab^{(s_J\cdot\kb,w)}(-\D)
\end{align*}
for each $J\subseteq \Sigma_{\infty}$, such that the induced map $\Tr_{\pi_1}$ on $\BGG_c(\F^{(\kb,w)})$ is compatible with that on   $\DR_c(\F^{(\kb,w)})$ via the quasi-isomorphism of Theorem~\ref{Theorem:BGG}.
\end{cor}
\begin{proof}
By \eqref{E:isom-g-F}, each term  $M'=\mathrm{DR}^j_{c}(\F^{(\kb,w)})$ or $\omegab^{(s_J\cdot \kb,w)}(-\D_{K'})$ on $\bfX^{\tor}_{K'}$ is the pullback via $\pi_1$ of the corresponding object on $\bfX_{K}^{\tor}$, i.e. $M'$ has the form $M'=\pi_1^*(M)$. By  the projection formula and Lemma~\ref{L:trace-map} above, we deduce an isomorphism \[R\pi_{1,*}(M')\cong M\otimes_{\cO_{\bfX_{K}^{\tor}}} R\pi_{1,*}(\cO_{\bfX_{K'}^{\tor}})=M\otimes_{\cO_{\bfX^{\tor}_K}}\pi_{1,*}\cO_{\bfX_{K'}^{\tor}}.\] The existence of  the  trace map $\pi_{1,*}(\cO_{\bfX_{K'}^{\tor}})\ra \cO_{\bfX_{K}^\tor}$ follows from the finite flatness of $\pi_{1,*}(\cO_{\bfX_{K'}^{\tor}})$.  
\end{proof}

We can describe now the action of the double coset $[K^pg K^p]$ on $H^{\star}_{\rig}(X^{\tor,\ord}_K,\D;\F^{(\kb,w)})$ and $\scrC^{\bullet}_K$. Since the partial Hasse invariants depend only on the $p$-divisible group associated with the universal abelian scheme, it is clear that  the inverse image of $X^{\tor,\ord}_K$ via both $\pi_1$ and $\pi_2$ are identified with  $X^{\tor,\ord}_{K'}$.
We define the action of $[K^pg K^p]$ on $R\Gamma(\X^{\tor}_{K,\rig}, j^{\dagger}_K \DR_c(\F^{(\kb,w)}))$ to be the composite map:
\[
\xymatrix{R\Gamma\big(\X^{\tor}_{K,\rig}, j^{\dagger}_K \DR_c(\F^{(\kb,w)})\big)\ar[r]^-{\pi_2^{*}}\ar[rd]_{[K^pg K^p]_*}
&R\Gamma\big(\X^{\tor}_{K',\rig}, j^{\dagger}_{K'}\DR_c(\F^{(\kb,w)})\big)\ar[d]^{\Tr_{\pi_1}}\\
 &R\Gamma\big(\X_{K,\rig}^{\tor}, j^{\dagger}_K\DR_c(\F^{(\kb,w)})\big),
}
\]
where $\Tr_{\pi_1}$ is induced by the trace map $\Tr_{\pi_1}\colon R\pi_{1,*}(\DR_c(\F^{(\kb,w)}))\ra \DR_c(\F^{(\kb,w)})$ in Corollary~\ref{C:trace}. Taking the cohomology, one gets the action of $[K^pg K^p]$ on the cohomology groups $H^{\star}_{\rig}(X^{\tor,\ord}_K, \D; \F^{(\kb,w)})$, hence the action of $\scrH(K^p,L_\wp)$ by linear combinations.

Similarly, for each  $J\subseteq \Sigma_{\infty}$, we define the action of $[K^pg K^p]$ on $S^{\dagger}_{(s_{J}\cdot \kb,w)}(K,L_\wp)=H^0(\X_{K,\rig}^{\tor}, j^{\dagger}_K\omegab^{(s_J\cdot \kb,w)}(-\D))$ as 
\[
[K^pg K^p]_*\colon S^{\dagger}_{(s_J\cdot\kb,w)}(K,L_\wp)\xra{\pi_2^*} S^{\dagger}_{(s_J\cdot\kb,w)}(K',L_\wp)\xra{\Tr_{\pi_1}}S^{\dagger}_{(s_J\cdot\kb,w)}(K,L_\wp).
\]
 Putting together all $S^{\dagger}_{(s_J\cdot\kb,w)}(K,L_\wp)$, one gets the action of $[K^pg K^p]$ on the complex $\scrC^\bullet_{K}$. 
 It is clear that the action on $\scrC^{\bullet}_K$ is compatible with the action on $H^{\star}_{\rig}(X^{\tor,\ord}_{K}, \D; \F^{(\kb,w)})$ via Theorem~\ref{Theorem:overconvergent}.

\subsection{The operator $S_{\gothp}$} 
\label{S:operator S_gothp}
We now define the Hecke actions at $p$. We start with the operator $S_\gothp$ for $\gothp \in \Sigma_p$. In the classical adelic language, the operator $S_{\gothp}$ will be the Hecke action given by the element in $\GL_2(\AAA_F^{\infty,p})$ which is $\big(\begin{smallmatrix}\varpi_{\gothp}^{-1} &0\\0&\varpi^{-1}_{\gothp}\end{smallmatrix}\big)$ at $\gothp$ and is the identity matrix at all other places.

%, $\varpi_{\gotha}\in\A^{(\infty),\times}_{F}$ is a finite idele with whose $\gothl$-component is given by $\varpi_{\gothl}^{v_{\}}$ ideal $\gotha$. Put $g_{\gotha}=\begin{pmatrix}\varpi_{\gotha} &0\\0&\varpi_{\gotha}\end{pmatrix}$.
%:
Let $\pi_{S_p}: \bfX^{\tor}_K\ra \bfX^{\tor}_K$  be the endomorphism whose effect at non-cusp points is given by
\[
\pi_{S_{\gothp}}\colon (A,\iota,\bar\lambda, \bar\alpha_{K^p})\mapsto (A\otimes_{\cO_F}\gothp^{-1}, \iota', \bar\lambda',\bar\alpha_{K^p}'),
\]
where the induced structures on  $A\otimes_{\cO_F} \gothp^{-1}\cong A/A[\gothp]$ are given as follows: 
 The action $\iota'$ of $\cO_F$ on $A\otimes_{\cO_F} \gothp^{-1}$ is evident. 
 If $A$ is $\gothc$-polarized, i.e. $\lambda$ induces an isomorphism $\gothc\xra{\sim} \Hom_{\cO_F}^{\Sym}(A,A^{\vee}) $ preserving the positivity, the polarization $\lambda'$ is given by 
\[
\lambda':\gothc \gothp^{2}\xra{\lambda\otimes 1} \Hom^{\Sym}_{\calO_F}(A,A^\vee)\otimes_{\cO_F}\gothp^{2}=\Hom_{\cO_F}^\Sym(A\otimes_{\cO_F}\gothp^{-1},(A\otimes_{\cO_F}\gothp^{-1})^{\vee}),
\]
which sends the positive cone $\gothc^+\gothp^2$ to the positive cone of the polarizations of $A\otimes_{\cO_F} \gothp^{-1}$. 
Finally,   the level-$K^p$ structure $\alpha_{K^p}$ on $A$ induces naturally a level $K^p$-structure $\alpha'_{K^p}$ on $A\otimes_{\calO_F} \gothp^{-1}$, since $A$ and $A\otimes_{\cO_F}\gothp^{-1}$ have naturally isomorphic prime-to-$p$ Tate modules. This canonically defines a point $(A\otimes_{\cO_F}\gothp^{-1},\iota',\bar\lambda',\bar\alpha'_{K^p})$ on $\bfX_K$ with the convention in Remark~\ref{R:polarization}.
 The automorphism $\pi_{S_{\gothp}}$ preserves the ordinary locus $X^{\tor,\ord}_{K}$, since $A$ and $A\otimes \gothp^{-1}$ have isomorphic $p$-divisible groups. 
 
 We have a canonical isogeny
\[
[\varpi_{\gothp}]\colon \univA^{\sm}\ra \pi_{S_{\gothp}}^*(\univA^{\sm})=\univA^{\sm}\otimes_{\cO_F} \gothp^{-1},
\]
with kernel $\univA^\sm[\gothp]$.
It induces a map on the relative de Rham cohomology
$[\varpi_{\gothp}]^*\colon \pi_{S_{\gothp}}^*\calH^1\ra\calH^1$,
hence a morphism of vector bundles
\[
[\varpi_{\gothp}]^*\colon \pi_{S_{\gothp}}^*\F^{(\kb,w)}\ra\F^{(\kb,w)}
\]
compatible with the Hodge filtration, $\cO_F$-action, and the Gauss-Manin connection. 
Hence, it induces a maps of de Rham complexes:
\[
 [\varpi_{\gothp}]^*\colon  \DR_c(\pi_{S_{\gothp}}^*\F^{(\kb,w)}) \ra \DR_c(\F^{(\kb,w)})
\]

We define the action of $S_{\gothp}$ on $H^{\star}_{\rig}(X^{\tor,\ord}_{K}, \D; \F^{(\kb,w)})$ as the composite
\begin{equation}
\label{Equ:S-operator}
\xymatrix{
 H^{\star}\big(\X^{\tor}_{K,\rig}, j^{\dagger}\DR_c(\F^{(\kb,w)})\big)\ar[r]^-{\pi_{S_\gothp}^{*}}\ar[rd]_{S_{\gothp}} & H^{\star}\big(\X^{\tor}_{K,\rig}, j^{\dagger}\DR_c(\pi_{S_{\gothp}}^{*}\F^{(\kb,w)})\big)\ar[d]^{[\varpi_{\gothp}]^*}\\
&H^{\star}\big(\X^{\tor}_{K,\rig}, j^{\dagger}\DR_c(\F^{(\kb,w)})\big).
}
\end{equation}

Similarly, for each subset $J\subseteq \Sigma_{\infty}$, the morphism $[\varpi_{\gothp}]^*$ on $\calH^1$ induces a morphism of modular line bundles 
\[[\varpi_{\gothp}]^*\colon \pi_{S_\gothp}^*\omegab^{(s_J\cdot\kb,w)}(-\D)\ra \omegab^{(s_{J}\cdot\kb,w)}(-\D).
\]
By the functoriality of BGG-complex,  the maps $[\varpi_{\gothp}]^*$ on various $\omegab^{(s_J\cdot\kb,w)}(-\D)$'s commute with the differentials in $\BGG_c(\F^{(\kb,w)})$. 
Thus one has  a map of of the BGG-complexes 
\[
[\varpi_{\gothp}]^*: \BGG_c([\varpi_\gothp]^*\F^{(\kb,w)})\ra \BGG_c(\F^{(\kb,w)}),
\]
which is compatible with the $[\varpi_{\gothp}]^*$ on $\DR_c(\F^{(\kb,w)})$ via the quasi-isomorphic inclusion $\BGG_c(\F^{(\kb,w)})\hra \DR_c(\F^{(\kb,w)})$ in Theorem~\ref{Theorem:BGG}, since the canonical Kodaira-Spencer isomorphism \eqref{Equ:Kod-isom} is Hecke equivariant.    
Taking overconvergent sections, one gets thus an endomorphism on overconvergent cusp forms
\[
\xymatrix{
 H^{0}\big(\X^{\tor}_{K,\rig}, j^{\dagger}\omegab^{(s_J\cdot \kb,w)}(-\D)\big)\ar[r]^-{\pi_{S_\gothp}^{*}}\ar[rd]_{S_{\gothp}} & H^{0}\big(\X^{\tor}_{K,\rig}, j^{\dagger}\pi_{S_{\gothp}}^{*}\omegab^{(s_J\cdot \kb, w)}\big)\ar[d]^{[\varpi_{\gothp}]^*}\\
&H^{\star}\big(\X^{\tor}_{K,\rig}, j^{\dagger}\omegab^{(s_J\cdot \kb,w)}(-\D))\big).
}
\]
Putting all $J$'s together, one obtains an endomorphism $S_\gothp$ on  the complex $\scrC^{\bullet}$ which is compatible with that on $H^{\star}_{\rig}(X^{\tor,\ord}_K,\D;\F^{(\kb,w)})$ when taking cohomology.

 %Similarly as $T_{\gotha}$, the endomorphism $S_{\gotha}$ and the isogeny $A\ra A\otimes \gotha^{-1}$ induces an action on $H^{\star}_{\rig}(\Sh_{k_0}^{\tor,\ord},\D;\F^{(\kb,w)})$ and on the complex $\scrC^{\bullet}$ such that the two actions are compatible under the isomorphism given by Theorem \ref{Theorem:overconvergent}.

%We denote by $\TT^{\tame}$ the algebra over $\Z$ generated by the formal symbols $T_{\gotha}$ for $\gotha$ prime to $Np$, $U_{\gothb}$ for $\gothb$ not coprime to $N$ but still prime to $p$, and $S_{\gothc}$ for $\gothc\subset \cO_F$. We call $\TT^{\tame}$ the tame Hecke algebra. According to the discussion above, we have a natural action of $\TT^{\tame}$ on $\scrC^{\bullet}$ and the cohomology group $H^{\star}_{\rig}(\Sh^{\tor,\ord}_{k_0}, \D; \F^{(\kb,w)})$ in a compatible way under the isomorphism given in Theorem \ref{Theorem:overconvergent}.

\subsection{The $\gothp$-canonical subgroup}\label{S:p-can-subgroup}
% Let $\Sigma_{p}$ be the set of places of $F$ above $p$. For each $\gothp\in \Sigma_p$, let $\Sigma_{\infty/\gothp}$ be the set of archimedean embeddings such that induces the place $\gothp$. 

For a rigid point $x\in \X^{\tor}_{K,\rig}$ and $\tau\in \Sigma_{\infty}$, %(identified with the set of $p$-adic embeddings via $\iota_p:\C\simeq \Qpb$)
Goren and Kassaei defined in \cite[4.2]{goren-kassaei} the \emph{$\tau$-valuation} of $x$, denoted by $\nu_{\tau}(x)$, as follows.
In a small enough affine chart $\mathfrak U$ of $\X_{K}^\tor$ containing the rig-point $x$, we can lift the partial Hasse invariant $h_\tau$ to a section $\tilde h_\tau$ of $\omegab^p_{\sigma^{-1}\circ \tau} \otimes \omegab_\tau^{-1}$ which trivializes over $\mathfrak U$. Then we define $\nu_\tau(x)$ to be $\min\{ \val_p(\tilde h_\tau(x)), 1\}$.
This gives a well-defined rational number in $[0,1]$, independent of the affine chart, the lift $\tilde h_\tau$, or the trivialization of the line bundle. 
Moreover, the point $x$ belongs to the ordinary locus $]X_{K}^{\tor,\ord}[$ if and only if $\nu_{\tau}(x)=0$ for all $\tau \in \Sigma_{\infty}$. We write $\rb$ for a tuple $(r_{\gothq})\in [0,p)^{\Sigma_{p}}$ with $r_{\gothq}\in \Q$. Following \cite[5.3]{goren-kassaei}, we
put
\[
]X^{\tor,\ord}_{K}[_{\rb}=\big\{x\in \X^{\tor}_{K,\rig}\; \big| \;\nu_{\tau}(x)+p\nu_{\sigma^{-1}\circ\tau}(x)\leq r_{\gothq}, \; \forall \tau\in \Sigma_{\infty/\gothq}\big\}.
\]
Then we have $]X^{\tor,\ord}_{K}[_{\rb}=]X^{\tor,\ord}_{K}[$ if  $\rb=0$, and $]X^{\tor,\ord}_{K}[_{\rb}$ form a fundamental system of strict neighborhoods of $]X^{\tor,\ord}_{K}[$ in $\X^{\tor}_{K, \rig}$ as   $r_{\gothq}\ra 0^+$ for all $\gothq\in \Sigma_p$. 
%Let $\univA^{\sm}$ be the family of semi-abelian schemes over $\bfX^{\tor}_K$.
We put \[
]X_{K}^{\ord}[_{\rb}\,=\,]X_{K}[\;\cap\; ]X_{K}^{\tor,\ord}[_{\rb}.\]

Now we fix a prime ideal $\gothp\in \Sigma_{p}$, and choose $\rb=(r_{\gothq})_{\gothq\in \Sigma_{p}}$ as above with $0<r_{\gothp}<1$.  Goren-Kassaei proved that there exists a finite flat subgroup scheme $\calC_{\gothp}\subset \univA^{\sm}[\gothp]$ over $]X^{\tor,\ord}_{K}[_{\rb}$, called the \emph{universal $\gothp$-canonical subgroup}, satisfying the following properties \cite[5.3, 5.4]{goren-kassaei}:
\begin{itemize}
\item[(1)] Locally for the \'etale topology on $]X^{\tor,\ord}_{K}[_{\rb}$, we have $\univC_{\gothp}\simeq \cO_F/\gothp$.

\item[(2)] The restriction of $\univC_{\gothp}$ to the ordinary locus $]X^{\tor,\ord}_{K}[$ is the multiplicative part of $\univA^{\sm}[\gothp]$.

\item[(3)] If we equip $\univA^{\sm}/\univC_{\gothp}$ with the induced action of $\cO_F$, polarization, and $K^p$-level structure, then the quotient isogeny
$
\pi_{\gothp}\colon \univA^{\sm}\ra \univA^{\sm}/\univC_{\gothp}
$
 over $]X_{K}^{\tor,\ord}[_{\rb}$ induces a finite flat map
\begin{equation}
\label{Equ:varphi}
\varphi_{\gothp}\colon  ]X^{\tor,\ord}_{K}[_{\rb}\longrightarrow\, ]X_{K}^{\tor,\ord}[_{\rb'}
\end{equation}
such that $\varphi^*_\gothp\univA^{\sm}$ is isomorphic to $ \univA^{\sm}/\univC_{\gothp}$ together with all induced structures, where $\rb'\in [0,p)^{\Sigma_p}$ is given by   $r'_{\gothp}=p r_{\gothp}$ and  $r'_{\gothq}=r_{\gothq}$ for $\gothq\neq \gothp$. The restriction of $\varphi_{\gothp}$ to the non-cuspidal part $]X_{K}^{\ord}[_{\rb}$ is finite \'etale of degree $N_{F/\Q}(\gothp)$.
\end{itemize}

In the sequel, for any point $(A, \iota, \bar\lambda,\bar\alpha_{K^p})$ lying in the locus $]X^{\tor,\ord}_{K}[_{\rb'}$, we denote by $C_{\gothp}
\subset A[\gothp]$ the $\gothp$-canonical subgroup of $A$.

The isogeny  $\pi_{\gothp}$ induces a map on the relative de  Rham cohomology
\begin{equation}\label{Equ:pi-frob}
\pi_{\gothp}^*\colon \varphi_{\gothp}^*\calH^1=\calH^1_{\dR}(\univA^{\sm, (\varphi_{\gothp})}/]X_{K}^{\tor,\ord}[_{\rb})\ra \calH^1_{\dR}(\univA^{\sm}/]X_{K}^{\tor,\ord}[_{\rb})=\calH^1.
\end{equation}
compatible with the Hodge filtration,  the action of $\cO_F$, and the connections $\nabla$ on both sides.

\subsection{Partial Frobenius $\Fr_{\gothp}$}\label{S:operator-Phi}
 Let $(\kb, w)$ be a cohomological multi-weight.  The morphisms $\varphi_{\gothp}$ and $\pi_{\gothp}^*:\varphi_{\gothp}^*(\calH^1)\ra \calH^1$   induce a map of vector bundles on $]X_K^{\tor,\ord}[_{\rb}$:
\[
\pi^*_{\gothp}\colon \varphi^*_{\gothp}\F^{(\kb,w)}\ra \F^{(\kb,w)}
\]
compatible with all structures on both sides, and hence a morphism of the de Rham complexes:
\[
\pi^*_{\gothp}\colon \DR_{c}(\varphi_{\gothp}^*\F^{(\kb,w)})\ra \DR_c(\F^{(\kb,w)}).
\]
We define $\Fr_{\gothp}$ to be  the composite map on the cohomology groups
\begin{equation*}
\xymatrix{
H^{\star}\big(]X_{K}^{\tor,\ord}[_{\rb'}, \DR_{c}(\F^{(\kb,w)})\big)\ar[r]^{\varphi^*_{\gothp}}\ar[rd]_{\Fr_{\gothp}}
&H^{\star}\big(]X_K^{\tor,\ord}[_{\rb},\DR_{c}(\varphi_{\gothp}^*\F^{(\kb,w)})\big)\ar[d]^{\pi_{\gothp}^*}\\
&H^{\star}\big(]X_K^{\tor,\ord}[_{\rb},\DR_c(\F^{(\kb,w)})\big).}
\end{equation*}
 Taking the direct limit as $\rb\ra 0^+$, one gets
\begin{equation}\label{Equ:Fr_q}
\Fr_{\gothp}\colon H^{\star}\big(\X^{\tor}_{K, \rig}, j^{\dagger} \DR_{c}(\F^{(\kb,w)})\big)\ra
 H^{\star}\big(\X^{\tor}_{K, \rig}, j^{\dagger}\DR_c(\F^{(\kb,w)})\big).
 \end{equation}
We call $\Fr_{\gothp}$ the {\it partial Frobenius} at $\gothp$.

%To define the action of $\Fr_{\gothp}$ on overconvergent cusp forms, 
% we note that each $\omegab^{(s_J\cdot \kb,w)}(-\D) e_{J}$ for $J\subseteq \Sigma_{\infty}$ is a direct summand of $\Gr_{\tF}^{\bullet}\DR_c(\F^{(\kb,w)})$.

 Since \eqref{Equ:pi-frob} is compatible with Hodge filtration and the $\cO_F$-action, it induces by functoriality a map of modular  line bundles 
 \[
\pi_{\gothp}^*\colon \varphi_{\gothp}^*\omegab^{(s_J\cdot\kb,w)}(-\D)\ra \omegab^{(s_J\cdot\kb,w)}(-\D),
\]
for any $J\subseteq \Sigma_{\infty}$. 
Taking direct sums over $J$, one gets  a map $\pi_{\gothp}^*: \pi_{\gothp}^*\BGG_c(\F^{(\kb,w)})\ra \BGG_c(\F^{(\kb,w)})$.
 This  is compatible with the $\pi_{\gothp}^*$ on $\DR_c(\F^{(\kb,w)})$ via the quasi-isomorphic inclusion $\BGG_c(\F^{(\kb,w)})\hra \DR_c(\F^{(\kb,w)})$.
 Taking overconvergent sections, one gets a composite map
\begin{equation}
\label{E:Frp on rigid cohomology}
\xymatrix{
H^0\big(]X_{K}^{\tor,\ord}[_{\rb'}, \omegab^{(s_{J}\cdot\kb,w)}(-\D)\big) \ar[r]^{\varphi_{\gothp}^*}\ar[dr]_{\Fr_\gothp} &H^0\big(]X_{K}^{\tor,\ord}[_{\rb}, \omegab^{(s_J\cdot\kb,w)}(-\D)\big)\ar[d]^{\pi_{\gothp}^*}\\
&H^0\big(]X_{K}^{\tor,\ord}[_{\rb}, \omegab^{(s_J\cdot\kb,w)}(-\D)\big)
}
\end{equation}
Letting $\rb\ra 0^+$, one gets the action of $\Fr_{\gothp}$ on the space of overconvergent cusp  forms:
\[
\Fr_{\gothp}\colon S^{\dagger}_{(s_J\cdot\kb,w)}(K,L_{\wp})\xra{\varphi_{\gothp}^*}S^{\dagger}_{(s_J\cdot\kb,w)}(K,L_{\wp})\xra{\pi_{\gothp}^*}S^{\dagger}_{(s_J\cdot\kb,w)}(K,L_{\wp}).
\]
Taking direct sum on $J$, we get  an endomorphism of complexes $\Fr_{\gothp}: \scrC^\bullet_{K}\ra \scrC^\bullet_{K}$, which is compatible with the $\Fr_{\gothp}$-action on $H^{\star}(\X^{\tor}_{K, \rig}, j^{\dagger}\DR_c(\F^{(\kb,w)}))$ via the isomorphism in Theorem~\ref{Theorem:overconvergent}.

\subsection{Study of $\varphi_{\gothp}$ over the ordinary locus}\label{S:Frobenius-ordinary}
 The ordinary locus $]X_{K}^{\tor, \ord}[$ is stable under $\varphi_{\gothp}$. The restriction of $\varphi_{\gothp}$ to $]X_{K}^{\tor,\ord}[$ can be defined over the formal model $\X_{K}^{\tor,\ord}$. 
 % The morphism $\varphi_{\gothp}$ induces a map on the differential $1$-forms
%\begin{equation}\label{Equ:varphi-diff}
%\varphi^*_{\gothp}: \varphi_{\gothp}^*(\Omega^1_{\X_K^{\tor,\ord}}(\log \D))\ra \Omega^1_{\X_K^{\tor,\ord}}(\log \D).
%\end{equation}
%By \eqref{Equ:Kod-isom}, we have $\Omega^1_{\X_{K}^{\tor,\ord}}(\log \D)\cong \bigoplus_{\tau\in \Sigma_{\infty}} \underline \Omega_{\tau} e_{\tau}. $ For any $\gothq\in \Sigma_{p}$, we put
%\[
%\Omega^1_{\X_K^{\tor,\ord}}(\log \D)[\gothq]=\bigoplus_{\tau\in \Sigma_{\infty/\gothq}}\underline\Omega_{\tau} e_{\tau}.
%\]
% This is the direct summand of $\Omega^1_{\X_{K}^{\tor,\ord}}(\log \D)$, where the action of $\cO_F$ factors through $\cO_{F_{\gothq}}$. The action of $\varphi_{\gothp}^*$ preserves $\Omega^1_{\X_K^{\tor,\ord}}(\log \D)[\gothq]$ for all $\gothq\in \Sigma_p$.

%The following Lemma will be needed in the main result result of 
 \begin{lemma}\label{Lemma:varphi-omega}
 If we regard $\cO_{\X^{\tor,\ord}_{K}}$ as a finite flat algebra over $\varphi^*_{\gothp}(\cO_{\X^{\tor,\ord}_K})$, then we have
$$
\Tr_{\varphi_{\gothp}}\big(\cO_{\X^{\tor,\ord}_{K}}\big)\subseteq p^{[F_{\gothp}:\Q_p]} \varphi^*_{\gothp}\big(\cO_{\X^{\tor,\ord}_{K}}\big).
$$
 \end{lemma}
 
 To prove this Lemma, we need some preliminary on the Serre-Tate local moduli.

 Let $\xb:\Spec(\overline \FF_p)\ra X_{K}^{\tor,\ord}$ be a geometric point in the ordinary locus, and $A_{\xb}$ the HBAV at $\xb$. We denote by $\widehat{\cO}_{\xb}$ the completion of the local ring $\cO_{\bfX_{K,W(\overline \FF_p)}^{\tor,\ord},\xb}$ at $\xb$.
 Let $\Def_{\cO_{F}}(A_{\xb}[p^{\infty}])$ be the deformation space of $A_{\xb}[p^{\infty}]$, i.e. the formal scheme over $W(\overline \FF_p)$ that classifies the $\cO_F$-deformations of $A_{\xb}[p^\infty]$ to noetherian complete local $W(\overline \FF_p)$-algebras with residue field $\overline \FF_p$. By the Serre-Tate's theory, we have a canonical isomorphism of formal schemes
\begin{equation}\label{Equ:Serre-Tate}
\Spf(\widehat{\cO}_{\xb})\cong \Def_{\cO_F}(A_{\bar x}[p^{\infty}]).
\end{equation}
 The $p$-divisible group $A_{\xb}[p^{\infty}]$ has a canonical decomposition
\[
A_{\xb}[p^{\infty}]=\prod_{\gothq\in \Sigma_p}A_{\xb}[\gothq^{\infty}],
\]
where each $A_{\xb}[\gothq^{\infty}]$ is an ordinary Barsotti-Tate $\cO_{F_{\gothq}}$-group of height $2$  and dimension $d_{\gothq}=[F_{\gothq}:\Q_p]$. This induces a canonical decomposition of the deformation spaces
\begin{equation}\label{Equ:decomp-local-moduli}
\Def_{\cO_F}(A_{\xb}[p^{\infty}])\cong \prod_{\gothq\in \Sigma_{p}}\Def_{\cO_{F_{\gothq}}}(A_{\xb}[\gothq^{\infty}]),
\end{equation}
where $\Def_{\cO_{F_{\gothq}}}(A_{\xb}[\gothq^{\infty}])$ denotes the deformation space of $A_{\xb}[\gothq^\infty]$ as a Barsotti-Tate  $\cO_{F_{\gothq}}$-modules, and the product is in the category of formal $W(\overline \FF_p)$-schemes. Since $A_{\xb}$ is ordinary, for each $\gothq\in \Sigma_p$, we have a canonical exact sequence
\[
0\ra A_{\xb}[\gothq^{\infty}]^{\mu}\ra A_{\xb}[\gothq^{\infty}]\ra A_{\xb}[\gothq^\infty]^{\et}\ra 0,
 \]
where $A_{\xb}[\gothq^{\infty}]^{\mu}$ and $A_{\xb}[\gothq^{\infty}]^{\et}$ denote respectively the multiplicative part and the \'etale part of $A_{\xb}[\gothq^{\infty}]$.
By  Serre-Tate theorem, the deformation space $\Def_{\cO_{F_\gothq}}(A_{\xb}[\gothq^{\infty}])$ has a natural formal group structure, canonically isomorphic to the formal group associated to the $p$-divisible group
\[
\Hom_{\cO_{F_{\gothq}}}(T_p(A_{\xb}[\gothq^{\infty}]^{\et}), A_{\xb}[\gothq^{\infty}]^{\mu})\cong \mu_{p^{\infty}}\otimes_{\Z_p} \cO_{F_{\gothq}}.
\]
Here,  the last step used the fact that both $A_{\xb}[\gothq^{\infty}]^{\mu}$ and $A_{\xb}[\gothq^{\infty}]^{\et}$ have both height $1$ as Barsotti-Tate $\cO_{F_{\gothq}}$-modules. Therefore, we have
$$\Def_{\cO_{F_{\gothq}}}(A_{\xb}[\gothq^{\infty}])\cong \widehat{\G}_m\otimes_{\Z_p} \cO_{F_{\gothq}}\simeq \widehat{\G}_m^{d_{\gothq}}.$$
We choose an isomorphism
\begin{equation}\label{Equ:coordinates}
\Def_{\cO_{F_\gothq}}(A_{\xb}[\gothq^{\infty}])\simeq \Spf(W(\overline \FF_p))[[t_{\gothq, 1},\dots, t_{\gothq, d_{\gothq}}]],
\end{equation}
  so that the multiplication by $p$ on $\Def_{\cO_{F_\gothq}}(A_{\xb}[\gothq^{\infty}])$ is given by
   \[[p](t_{\gothq, i})=(1+t_{\gothq, i})^p-1.\]
Therefore, 
$\frac{d t_{\gothq, i}}{1+t_{\gothq, i}}$ $(1\leq i\leq d_{\gothq})$ are invariant differential $1$-forms, and they form a basis of  $\widehat{\Omega}^1_{\Def_{\cO_{F_{\gothq}}}(A_{\xb}[\gothq^{\infty}])/W(\overline \FF_p)}$. 
   By \eqref{Equ:Serre-Tate}, we have
\[
\widehat{\cO}_{\xb}\simeq W(\overline \FF_p)[[\{t_{\gothq, i}:\gothq\in \Sigma_{p}, 1\leq i\leq d_{\gothq}\}]].
\]
Furthermore, we remark that the direct summand  $\widehat{\Omega}^1_{\widehat{\cO}_{\xb}/W(\overline \FF_p)}[\gothq]$ of the  differential module
$\widehat{\Omega}^1_{\widehat{\cO}_{\xb} / W(\overline \FF_p)}$
 is generated over $\widehat{\cO}_{\xb}$ by
 $\{\frac{d t_{\gothq, i}}{1+t_{\gothq,i}}: 1\leq i\leq d_{\gothq}\}.$

\begin{proof}[Proof of Lemma \ref{Lemma:varphi-omega}]
 The problem is local.
 Let $\xb$ be a geometric point of $X^{\tor,\ord}_K$, and let $\varphi_{\gothp}(\xb)$ be  its image under $\varphi_{\gothp}$. It suffices to show that
 \begin{equation*}\label{Equ:trace-varphi-local}
 \Tr_{\varphi_{\gothp}^*}(\cO_{\xb})\subseteq p^{d_{\gothp}}\varphi_{\gothp}^*(\cO_{\varphi_{\gothp}(\xb)}).
 \end{equation*}
 We always use \eqref{Equ:Serre-Tate} to identify $\Spf(\widehat{\cO}_{\xb})$ with $\Def_{\cO_F}(A_{\xb}[p^{\infty}])$.  

Let $\univA_{\xb}$ be  the base change of the universal HBAV to $\Spf(\widehat{\cO}_{\xb})$. 
Then $\univA_{\xb}[\gothp^{\infty}]$ is the universal deformation of $A_{\xb}[\gothp^\infty]$ over $\Def_{\cO_{F_{\gothp}}}(A_{\xb}[\gothp^{\infty}])$. 
It is an ordinary Barsotti-Tate  $\cO_{F_\gothp}$-modules, i.e. an extension of its \'etale part by its multiplicative part. 
 The isogeny $\pi_{\gothp}:\univA_{\xb}\ra \univA_{\varphi_{\gothp}(\xb)}=\univA_{\xb}/\univC_{\gothp, \xb}$ induces a morphism between two exact sequences of $\gothp$-divisible groups
\[
\xymatrix{0\ar[r]& \univA_{\xb}[\gothp^{\infty}]^{\mu}\ar[r]\ar[d]^{\pi_{\gothp}^{\mu}}&\univA_{\xb}[\gothq^{\infty}]\ar[d]^{\pi_{\gothp}}\ar[r]
&\univA_{\xb}[\gothp^\infty]^\et\ar[r]\ar[d]^{\pi_{\gothp}^{\et}}&0\\
0\ar[r] &\univA_{\varphi_{\gothp}(\xb)}[\gothp^{\infty}]^{\mu}\ar[r]&\univA_{\varphi_{\gothp}(\xb)}[\gothp^{\infty}]\ar[r]&\univA_{\varphi_{\gothp}(\xb)}[\gothp^{\infty}]^{\et}
\ar[r]&0.
}
\]
Since the $\gothp$-canonical subgroup $\univC_{\gothp,\xb}$ coincides with the $p$-torsion of $\univA_{\xb}[\gothp^{\infty}]^{\mu}$,
% ($p$ being supposed unramified in $F$), 
the isogeny $\pi_{\gothp}^{\mu}$ is  the multiplication by $p$ up to isomorphism, and $\pi_{\gothp}^{\et}$ is an isomorphism. 
This implies that, there exists an isomorphism
$$
\phi\colon
\Def_{\cO_{F_\gothp}}(A_{\xb}[\gothp^{\infty}])\ra\Def_{\cO_{F_\gothp}}(A_{\varphi_\gothp(\xb)}[\gothp^{\infty}])
$$
such that $\varphi_{\gothp}=p\cdot \phi.$

 Let $\varphi'_{\gothp}$ denote the endomorphism on $\Def_{\cO_F}(A_{\xb}[p^{\infty}])$ that gives the multiplication by $p$ on $\Def_{\cO_{F_{\gothp}}}(A_{\xb}[\gothp^{\infty}])$ and the  identity on $\Def_{\cO_{F_{\gothq}}}(A_{\xb}[\gothq^{\infty}])$ with $\gothq\neq \gothp$.
By the discussion above, there exists an isomorphism $\phi: \Def_{\cO_F}(A_{\xb}[p^{\infty}])\ra \Def_{\cO_F}(A_{
\varphi_{\gothp}(\xb)}[p^{\infty}])$ such that $\varphi_{\gothp}=\phi\circ \varphi'_{\gothp}$.
 Thus it suffices to prove that $\Tr_{\varphi'^*_{\gothp}}$ is divisible by $p^{d_{\gothp}}$.
 We may further reduce the problem to showing that the trace map of the multiplication by $p$ on $\Def_{\cO_{F_{\gothp}}}(A_{\xb}[\gothp^{\infty}])$ is divisible by $p^{d_{\gothp}}$. 
  This follows from an easy computation using the Serre-Tate coordinates $\{t_{\gothp, i}: 1\leq i\leq d_{\gothp}\}$ in \eqref{Equ:coordinates}.
 \end{proof}

 \subsection{$U_{\gothp}$-correspondence}
\label{S:Up correspondence}
 Let $\rb=(r_{\gothq})_{\gothq}\in ((0,p)\cap \Q)^{\Sigma_{p}}$ be a tuple with $r_{\gothp}<1$ as in Subsection~\ref{S:p-can-subgroup}, and 
$\rb'=(r'_{\gothq})_{\gothq}\in (0,p)^{\Sigma_p}$ be such that  $r'_{\gothp}=p\,r_{\gothp}$ and $r'_{\gothq}=r_{\gothq}$ with $\gothq\neq \gothp$.
 Let $\calA^{\sm}$ be the family of semi-abelian schemes over $]X^{\tor,\ord}_{K}[_{\rb'}$, and  $]X^{\tor,\ord}_{K}[_{\rb'}^{\gothp}$ be the rigid analytic space that  classifies the $\cO_F$-stable finite flat group schemes $\calD\subseteq \calA^\sm[\gothp]$ which is disjoint from the $\gothp$-canonical subgroup $\calC_{\gothp}$, i.e. outside the toroidal boundary, $]X^{\tor,\ord}_{K}[^{\gothp}_{\rb'}$ parametrizes the tuples $(A,\iota,\lambda,\alpha_{K^p},H)$
\begin{itemize}
\item $(A,\iota,\bar\lambda,\bar\alpha_{K^p})$ is a  point of $]X^{\tor,\ord}_{K}[_{\rb'}$,

\item $D\subset A[\gothp]$ is a subgroup stable under $\cO_F$, \'etale locally isomorphic to $\cO_F/\gothp$ and disjoint from the $\gothp$-canonical subgroup $C_{\gothp}\subset A[\gothp]$.
\end{itemize}
We have two projections
\begin{equation}\label{Equ:tor-p-corresp}
  \xymatrix{
&]X^{\tor,\ord}_{K}[^{\gothp}_{\rb'}\ar[rd]^{\pr_2}\ar[ld]_{\pr_1}\\
]X^{\tor,\ord}_{K}[_{\rb'}&& ]X^{\tor,\ord}_{K}[_{\rb},
}
\end{equation}
whose effect on non-cuspidal points are given by
\begin{align*}
\pr_1(A,\iota,\bar\lambda, \bar\alpha_{K^p}, D)
&\mapsto (A,\iota, \bar\lambda,\bar\alpha_{K^p})\\
\pr_2(A,\iota,\bar\lambda, \bar\alpha_{K^p}, D)&\mapsto (A/D,{\iota}', \bar\lambda',\bar{\alpha}'_{K^p}).
\end{align*}
Here, $(A/D, {\iota}',\bar\lambda', \bar{\alpha}'_{K^p})$ denotes the quotient rigid analytic HBAV $A/D$ with the induced polarization and $K^p$-level structure. 
In the terminology of \cite{goren-kassaei}, the subgroups $H$ are \emph{anti-canonical} at $\gothp$, and  \cite[Theorem~5.4.4(4)]{goren-kassaei} implies that image of $\pr_2$ lies in $]X^{\tor,\ord}_{K}[_{\rb}
$.

\begin{lemma}\label{Lemma:p2-isom}
The morphism $\pr_1$ is finite \'etale of degree  $N_{F/\Q}(\gothp)$. The map $\pr_2$  is an isomorphism of rigid analytic spaces, 
with the inverse  map $\tilde{\pi}_{S_{\gothp^{-1}}}\circ \tilde\varphi_{\gothp}$, where
\begin{align*}
\tilde{\varphi}_{\gothp}&\colon(A,\iota, \bar\lambda,\bar\alpha_{K^p})\mapsto (A/C_{\gothp}, \iota', \bar\lambda', \bar{\alpha}'_{K^p}, A[\gothp]/C_{\gothp}), \textrm{ and}\\
\tilde{\pi}_{S_{\gothp^{-1}}}=\tilde{\pi}_{S_{\gothp}}^{-1}&\colon (A, \iota, \bar\lambda, \bar\alpha_{K^p}, D)\mapsto (A\otimes_{\cO_F}\gothp, \iota'', \bar\lambda'', \bar\alpha''_{K^p}, D\otimes_{\cO_F} \gothp).
\end{align*}
Here, $(\iota', \bar\lambda', \bar{\alpha}'_{K^p})$ and $(\iota'', \bar\lambda'', \bar\alpha''_{K^p})$ denote the natural induced structures on the corresponding objects. In particular, we have
\begin{equation}\label{Equ:p1-p2}
\pr_1={\pi}_{S_{\gothp}}^{-1}\circ\varphi_{\gothp}\circ\pr_2,
\end{equation}
 where $\varphi_{\gothp}$ is defined in \eqref{Equ:varphi}, and  $\pi_{S_{\gothp}}^{-1}$ denotes the inverse of  the automorphism  $\pi_{S_{\gothp}}$ on $]X^{\tor,\ord}_{K}[_{\rb'}$.
  \end{lemma}
 \begin{proof}
 The statement for $\pr_1$ is clear. 
 To see $\pr_2$ is an isomorphism, we take a point $(A, D):=(A, \iota, \bar\lambda, \bar\alpha_{K^p}, D)$ in $]X^{\tor,\ord}_{K}[^{\gothp}_{\rb'}$. 
 We have $\pr_2(A, D)=A/D$, and $A[\gothp]/D$ is the $\gothp$-canonical subgroup of $A'=A/D$.  So  we have
 \[
 \tilde{\varphi}_{\gothp}(A')=(A'/(A[\gothp]/D), A'[\gothp]/(A[\gothp]/D))=(A/A[\gothp], \bar{D})=(A\otimes_{\cO_F}\gothp^{-1}, D\otimes_{\cO_F} \gothp^{-1}).\]
 with all the induced structures. The Lemma now follows immediately.
  \end{proof}

% We put $\D^{\gothp}=\pr^{-1}_2\D$. Then the correspondence \eqref{Equ:tor-p-corresp} induces  isomorphisms of differential forms
% \begin{align*}
% \pr_1^*&\colon \pr_1^*\Omega^1_{]X^{\tor,\ord}_{K}[_{\rb'}}(\log \D)\xra{\cong}\Omega^1_{]X^{\tor,\ord}_{K}[^{\gothp}_{\rb'}}(\log \D^{\gothp})\\
 %\pr^*_2&\colon \pr_2^*\Omega^1_{]X^{\tor,\ord}_{K}[_{\rb}}(\log \D)\xra{\cong}\Omega^1_{]X^{\tor,\ord}_{K}[^{\gothp}_{\rb'}}(\log \D^{\gothp}),
% \end{align*}
 %which preserves the natural action of $\cO_F$ on both sides induced from the extended Kodaira-Spencer isomorphism \eqref{Equ:Kod-isom}.
%In particular, for each subset $J\subseteq \Sigma_{\infty}$,  these isomorphisms induce an isomorphism 
 %\begin{equation}\label{Equ:comparison-omega}
 %\phi_{12}\colon \pr_2^*\underline \Omega^J e_J \xrightarrow{\cong} \pr_1^*\underline \Omega^J e_J.
 %\end{equation}

By the  definition of the maps $\pr_1$ and $\pr_2$ in Subsection~\ref{S:Up correspondence},  there is a natural isogeny of semi-abelian schemes over $]X^{\tor,\ord }_{K}[^{\gothp}_{\rb'}$:
\[
\check{\pi}_{\gothp}\colon \pr_1^{*}\univA^{\sm}\ra\pr_1^*\univA^{\sm}/\calD =  \pr_2^*\univA^{\sm},
\]
where  $\calD\subset \pr_1^*\univA^{\sm}[\gothp]$ is the tautological  subgroup scheme over $]X^{\tor,\ord }_{K}[^{\gothp}_{\rb'}$ disjoint from  $\univC_{\gothp}$.
 It induces a morphism  on the relative de Rham cohomology
\[
\check{\pi}^*_{\gothp}\colon \pr_2^*\calH^1\ra \pr_1^*\calH^1
\]
compatible with all the structures on both sides. In particular, for each $\tau\in \Sigma_{\infty}$, it induces a morphism
$\check{\pi}_{\gothp}\colon \pr_2^*\calH^1_{\tau}\ra\pr_1^* \calH^1_{\tau}$
compatible with the Hodge filtration $0\ra \omegab_{\tau}\ra \calH^1_{\tau}\ra \Lie((\univA^{\sm})^\vee)_{\tau}\ra 0$. 

\begin{lemma}\label{Lemma:check-pi}
Let $x=(A,\iota, \lambda, \alpha_{K^p}, D)$ be a rigid point on $]X^{\tor,\ord}_{K}[^{\gothp}_{\rb'}$ defined over the ring of integers $\cO_{\wp'}$ of a finite extension $L'_{\wp'}$ of $L_\wp$, and let $\check{\pi}_{\gothp, x}:A\ra A':=A/D $ be the canonical isogeny. Assume that $A$ has ordinary good reduction.  Let  $\omega_{\tau}$ and $\eta_{\tau}$ (resp. $\omega_{\tau}'$ and $\eta_{\tau}'$) be a basis of $\calH^1_{\tau}(A/\cO_{\wp'})$ (resp.  $\calH^1_{\tau}(A'/\cO_{\wp'})$) over $\cO_{\wp'}$ adapted to the Hodge filtration, and write
$$
\check{\pi}_{\gothp,x}^*(\omega_{\tau}')=a_{\tau}\omega_{\tau},\quad \textrm{and}\quad \check{\pi}_{\gothp,x}^*(\eta_{\tau}')\equiv b_{\tau}\eta_{\tau}\pmod {\omegab_{\tau}}.
$$
 Then we have $\val_p(a_{\tau})=0$ for all $\tau \in \Sigma_{\infty}$,  $\val_p(b_{\tau})=0$ if $\tau\notin \Sigma_{\infty/\gothp}$ and $\val_p(b_{\tau})=1$ if $\tau\in \Sigma_{\infty/\gothp}$. In particular, we have
 $$\check{\pi}_{\gothp,x}^*(\omega_{\tau}'\wedge\eta_{\tau}')=a_{\tau}b_{\tau} \omega_{\tau}\wedge\eta_{\tau},$$
 with $\val_p(a_{\tau}b_{\tau})=0$ if $\tau\notin \Sigma_{\infty/\gothp}$  and $\val_p(a_{\tau}b_{\tau})=1$ if $\tau\in \Sigma_{\infty/\gothp}$.
\end{lemma}
\begin{proof}
The problem 
depends only on the $p$-divisible group $A[p^{\infty}]$. The isogeny $\check{\pi}_{\gothp}$ induces an isomorphism of the $p$-divisible groups $A[\gothq^{\infty}]\xra{\sim} A'[\gothq^{\infty}]$ over $\cO_{\wp'}$ for $\gothq\neq \gothp$. Thus, the statements for  $\tau\notin \Sigma_{\infty/\gothp}$ are evident. 
The subgroup $D\subset A[\gothp]$ with $D\neq C_{\gothp}$ is necessarily \'etale, since $A$ has good ordinary reduction. Therefore, $\check{\pi}_{\gothp,x}$ is \'etale and induces an isomorphism
\[
\omegab_{A'[\gothp^{\infty}]}=\bigoplus_{\tau\in \Sigma_{\infty/\gothp}}\omegab_{A',\tau}\xra{\cong}\bigoplus_{\tau\in \Sigma_{\infty/\gothp}} \omegab_{A,\tau}=\omegab_{A[\gothp^{\infty}]}.
\]
It follows immediately that $a_{\tau}$ are units in $\cO_{\wp'}$ for $\tau\in \Sigma_{\infty/\gothp}$. To show that $\val_p(b_{\tau})=1$, we consider the dual isogeny
$\check{\pi}_{\gothp,x}^\vee\colon A'^\vee \ra A^\vee$. Let $A'^\vee[\gothp^{\infty}]^{\mu}$ and $A^\vee[\gothp^\infty]^{\mu}$ be respectively the multiplicative part of $A'^\vee[\gothp^{\infty}]$ and $A^\vee[\gothp^{\infty}]$.  We have an induced isogeny
\[
(\check{\pi}_{\gothq,x}^\vee)^{\mu}: A'^\vee[\gothp^{\infty}]^{\mu}\ra A^\vee[\gothp^\infty]^{\mu}.
\]
The kernels of $\check{\pi}_{\gothp,x}^\vee$ and $(\check{\pi}_{\gothp,x}^\vee)^{\mu}$ are both $D^\vee$, the Cartier dual of $D$,  which coincides  with the $p$-torsion of $A'^\vee[\gothp^{\infty}]^{\mu}$ (since $\gothp$ is unramified).  Hence, the induced map on $\Lie(A'^\vee[\gothp^\infty]^{\mu})\ra \Lie(A^\vee[\gothp^\infty]^{\mu})$ is given by the multiplication by $p$ up to units, whence $\val_p(b_{\tau})=1$ for all $\tau \in \Sigma_{\infty/\gothp}$. Now the Lemma follows from the fact that $\Lie(A'^\vee)_{\tau}=\Lie(A'^\vee[\gothp^\infty]^{\mu})_{\tau}$ for $\tau\in \Sigma_{\infty/\gothq}$, since $A'$ is ordinary.
 \end{proof}

   \subsection{$U_{\gothp}$-operator} We now define the $U_{\gothp}$-operator on $H^{\star}_{\rig}(X^{\tor,\ord},\D; \F^{(\kb,w)})$ and on the complex $\scrC^\bullet_K$.
    The map $\check{\pi}_{\gothp}^*:\pr_2^{*}\calH^1\ra \pr_1^*\calH^1$  induces a map
$\check{\pi}^*_{\gothp}: \pr_2^*\F^{(\kb,w)}\ra \pr_1^*\F^{(\kb,w)}$
and hence a map of de Rham complexes
\begin{equation}\label{Equ:check-pi-dR}
\check{\pi}^*_{\gothp}\colon \DR_c(\pr_2^*\F^{(\kb,w)})\ra \DR_c(\pr_1^*{\F^{(\kb,w)}})
\end{equation}
compatible with the $\tF$-filtrations on both sides defined in Subsection~\ref{S:DR-Hodge}.

We  define $U_{\gothp}$-operator to be the composite map on the cohomology groups
\begin{equation}
\label{E:Up on rigid cohomology}
\xymatrix{H^{\star}(]X_K^{\tor,\ord}[_{\rb}, \DR_c(\F^{(\kb,w)}))\ar@{-->}^{U_{\gothp}}[rr]\ar[d]^{\pr_2^*}&& H^{\star}(]X_K^{\tor,\ord}[_{\rb'}, \DR_c(\F^{\kb,w}))\\
H^{\star}(]X_K^{\tor,\ord}[^{\gothp}_{\rb'}, \DR_c(\pr_2^*\F^{(\kb,w)}))\ar[rr]^{\check{\pi}_{\gothp}^*}&&H^{\star}(]X_{K}^{\tor,\ord}[^{\gothp}_{\rb'}, \DR_c(\pr_1^*\F^{(\kb,w)}))\ar[u]^{\Tr_{\pr_1}}},
\end{equation}
where the existence of the trace map $\Tr_{\pr_1}$ follows from similar arguments as in Corollary~\ref{C:trace}.

By letting $\rb\ra 0^+$ (so $\rb'\ra 0^+$ as well), we get a map
\begin{equation}\label{Equ:U-p}
U_{\gothp}\colon H^{\star}(\fX^{\tor}_{K,\rig}, j^{\dagger}\DR_{c}(\F^{(\kb,w)}))\ra H^{\star}(\fX^{\tor}_{K,\rig}, j^{\dagger}\DR_c(\F^{(\kb,w)})).
\end{equation}

 Similarly as the discussion for for  $S_{\gothp}$ and $\Fr_{\gothp}$, the map $\check{\pi}_{\gothp}: \pr_2^*\calH^1\ra \pr_1^*\calH^1$ induces by functoriality a morphism 
 \[
 \check{\pi}^*_{\gothp}: \pr_2^*\omegab^{(s_J\cdot \kb, w)}(-\D)\ra \pr_1^* \omegab^{(s_J\cdot\kb,w)}(-\D)
 \]
 for each subset $J\subseteq \Sigma_{\infty}$. Taking direct sum over  all $J$, one gets a  morphism $\check{\pi}_{\gothp}^*$ on the dual BGG-complex such that the following diagram 
 \begin{equation}\label{diag:bgg-dR}
\xymatrix{
{\BGG_c(\pr_2^*\F^{(\kb,w)})}\ar@{^{(}->}[d]\ar[rr]^{\check{\pi}^*_{\gothp}} &&\BGG_c(\pr_1^*\F^{(\kb,w)})
\ar@{^{(}->}[d]\\
\DR_c(\pr^*_2(\F^{(\kb,w)}))\ar[rr]^{\check{\pi}^*_{\gothp}} && \DR_c(\pr_1^*\F^{(\kb,w)}).
}
\end{equation}
is commutative and compatible with the $\tF$-filtrations. 
Here, the vertical arrows are the quasi-isomorphic inclusions as in Theorem~\ref{Theorem:BGG}.

Let  $U_1\subset\, ]X_{K}^{\tor,\ord}[_{\rb'}$ and $U_2\subset\, ]X_K^{\tor,\ord}[_{\rb}$ be quasi-compact admissible open subsets such that $\pr_1^{-1}(U_1)\subset \pr^{-1}_2(U_2)$. We denote by $\res_{12}: \pr_1^{-1}(U_1) \to \pr_2^{-1}(U_2)$ the natural inclusion.
  For every $J\subseteq \Sigma_{\infty}$, we have a composite map $U_{\gothp}$
\[
\xymatrix{
\Gamma(U_2, \omegab^{(s_J\cdot\kb,w)}(-\D))\ar@{-->}[rr]^{U_{\gothp}}\ar[d]^{\pr_2^*}&& \Gamma(U_1, \omegab^{(s_J\cdot\kb,w)}(-\D))\\
\Gamma(\pr_2^{-1}(U_2), \pr_2^*\omegab^{(s_J\cdot\kb,w)}(-\D))\ar[rr]^{\check{\pi}^{*}_{\gothp}\circ \res_{12}^*} &&
\Gamma(\pr_1^{-1}(U_1), \pr_1^*\omegab^{(s_J\cdot\kb,w)}(-\D))\ar[u]^{\Tr_{\pr_1}}.
}
\]
Taking $U_1=]X_{K}^{\tor,\ord}[_{\rb'}$ and $U_2=]X_K^{\tor,\ord}[_{\rb}$ and making $\rb\ra 0^+$, one gets an endomorphism
\[
U_{\gothp}\colon  S^{\dagger}_{(s_J\cdot\kb,w)}(K,L_\wp)\ra S^{\dagger}_{ (s_J\cdot\kb,w)}(K,L_\wp).
\]
%Since $]X_K^{\tor,\ord}[_{\rb'}$ strictly contains $]X_K^{\tor,\ord}[_{\rb}$, the endomorphism $U_{\gothp}$ on $S^{\dagger}_{\epsilon_J(\kb,w)}(K,L_\wp)$ is completely continuous.
% As in the case of $\Fr_{\gothp}$, $U_{\gothp}$ commutes with the differential $d^j$ of the complex $\scrC^\bullet_K$ by its very construction. 
Putting  all $J\subseteq \Sigma_{\infty}$ together, one obtains  an endomorphism of  complexes $U_{\gothp}\colon \scrC^{\bullet}_K\ra \scrC^{\bullet}_K$.
   When taking cohomology, it follows from digram \eqref{diag:bgg-dR} that  the ${U}_{\gothp}$ on $H^{\star}(\scrC^{\bullet}_K)$ is canonically identified with the one defined in  \eqref{Equ:U-p} via Theorem~\ref{Theorem:overconvergent}.

\begin{remark}\label{remark:U_p}
 Our definition of the $U_{\gothp}$-operator on $S^{\dagger}_{(s_J\cdot\kb,w)} (K,L_\wp)$ for all $J\subseteq \Sigma_{\infty}$ coincides with the  normalized $U_{\gothp}$-operator defined in \cite[(1.11.7), (4.2.7)]{KL}. 
In \emph{loc. cit.}, the authors worked over the fine moduli spaces  $\M_K=\coprod_{\gothc}\M^\gothc_K$'s and omitted $\wedge^2\calH^1_\tau$'s in their notation.
 However,  in order for the $U_\gothp$-operator descend to forms on $\bfSh_K$, they have to add carefully a normalizing factor as explained in \cite[(1.11.6)]{KL}. 
Here, since  our Kodaira-Spencer isomorphism \eqref{Equ:Kod-isom} is Hecke equivariant, this factor is automatically taken into account by the various powers of $\wedge^2\calH^1_{\tau}$'s.
%However, for $J\subset \Sigma_{\infty}$ with $\Sigma_{\infty/\gothp}\cap J\neq \Sigma_{\infty/\gothp}$, our ${U}_{\gothp}$-operator on $S^{\dagger}_{\epsilon_J(\kb,w)}(K,L_\wp)$, which is induced from the action of $U_{\gothp}$-correspondence on $\DR_{c}(\F^{(\kb,w)})$, is not the usual $U_{\gothp}$ studied by Kisin-Lai \cite{KL}. 
 % Indeed, it is easy to check that Kisin-Lai's $U_{\gothp}$-operators does not commute with $d^j:\scrC^j_K\ra \scrC^{j+1}_K$. If $p$ is inert in $F$, our definition of $U_{\gothp}$ on $S^{\dagger}_{\epsilon_J(\kb,w)}(K,L_\wp)$ coincides with $p^{\sum_{\tau\notin  J}(k_{\tau}-1)}$ times Kisin-Lai's $U_{\gothp}$.
     \end{remark}
     
  There exists a simple relationship between the partial Frobenius $\Fr_{\gothp}$ and the operator $U_{\gothp}$:
  \begin{lemma}\label{Lemma:Frob-U_p}
  As operators on the cohomology groups $H^{\star}_{\rig}(X^{\tor,\ord}_{K}, \D;\F^{(\kb,w)})$ or on $\scrC^\bullet_{K}$, we have
  \[
  U_{\gothp}\Fr_{\gothp}=N_{F/\Q}(\gothp) S_{\gothp},
  \]
  where the action of $S_{\gothp}$ is defined in \eqref{Equ:S-operator}.
  \end{lemma}
\begin{proof}
By the definitions of $U_\gothp$ and $\Fr_\gothp$ in \eqref{E:Up on rigid cohomology} and \eqref{E:Frp on rigid cohomology}, we have
\[
U_{\gothp}\Fr_{\gothp}=\Tr_{\pr_1} \circ\check{\pi}^{*}_{\gothp}\circ\pr_2^*\circ \pi^*_{\gothp}\circ\varphi_{\gothp}^*=\Tr_{\pr_1} \circ \check{\pi}^{*}_{\gothp}\circ \pi^*_{\gothp}\circ\pr_2^*\circ \varphi_{\gothp}^*.
\]
Here, the second step is because the morphism induced by isogeny commutes with base change. We note that for a point $(A, D)\in ]\M^{ \tor,\ord}_{K,k_0}[^{\gothp}_{\rb'}$, the composite isogeny
\[
A\xra{\check{\pi}_{\gothp}} A/D\xra{\pi_{\gothp}}(A/D)/C_{\gothp}=A/A[\gothp]=A\otimes_{\cO_F} \gothp^{-1}
\]
is by definition the isogeny $[\varpi_{\gothp}]$. Hence, we have $\check{\pi}^{*}_{\gothp}\circ \pi^*_{\gothp}=[\varpi_{\gothp}]^*$, and
\[
U_{\gothp}\Fr_{\gothp}=\Tr_{\pr_1}\circ [\varpi_{\gothp}]^*\circ\pr_{2}^*\circ\varphi_{\gothp}^*.
\]
By \eqref{Equ:p1-p2}, we have $\pr_2^*\circ\varphi_{\gothp}^*=\pr_1^*\circ \pi_{S_{\gothp}}^*$. It follows that
\[
U_{\gothp}\Fr_{\gothp}=\Tr_{\pr_1}\circ[\varpi_{\gothp}]^*\circ\pr_1^*\circ \pi_{S_{\gothp}}^*=N_{F/\Q}(\gothp)[\varpi_{\gothp}]^*\pi_{S_{\gothp}}^*\stackrel{\eqref{Equ:S-operator}}=N_{F/\Q}(\gothp)S_{\gothp}.
\]
\end{proof}
%For an analytic function $g$ defined over a quasi-compact open subset $U\subset\, ]X_{K}^{\tor,\ord}[^{\gothp}_{\rb'}$, we define the norm
%$$\|g\|_{U}:=\|(\pr^*_{2})^{-1}g\|_{\pr_2(U)}.$$
%Here, the norm $\|\bullet\|_{\pr_2(U)}$ is defined by using the integral model $\X_K^{\tor}$.

 \subsection{Norms}
\label{S:Norms}
We recall the construction of $p$-adic norms on rigid analytic varieties. Suppose we are given  an admissible formal scheme $\fZ$ over $\cO_{L_\wp}$, and  a vector bundle  $\scrE$ on $\fZ$. Let $\fZ_{\rig}$ denote the rigid analytic space over $L_\wp$ associated to $\fZ$, and $\scrE_{\rig}$ denote the associated vector bundle on $\fZ_{\rig}$. We denote by $|\cdot|$ the non-archimedean norm on $\C_p$ normalized by $|p|=p^{-1}$.  For a quasi-compact open subset $U\subseteq \fZ_{\rig}$,  one can define a norm $\|\cdot\|_U$ on $\Gamma(U, \scrE_{\rig})$ such that $\|\lambda\cdot s\|_U=|\lambda|\cdot \|s\|_U$ for $\lambda\in \C_p$ and $s\in \Gamma(U,\scrE_{\rig})$ as follows.  Recall that a  point $x\in \fZ_{\rig}$ defined over an extension $L'_{\wp'}/L_\wp$ is equivalent to  a morphism of $\cO_{L_\wp}$-formal schemes $x:\Spf(\cO_{L'_{\wp'}})\ra \fZ$. We write $\scrE_x$ for the pullback $x^*\calE$. Given a section $s\in \Gamma(U,\scrE_{\rig})$ and a point $x\in U$ defined over $L'_{\wp'}$,  we denote by  $x^*(s)\in \scrE_{x}\otimes_{\cO_{\wp'}} L'_{\wp'}$ the inverse image of $s$ under $x$. We define  $|s(x)|$ to be the minimal of  $|\lambda|^{-1}$ where $\lambda\in L'_{\wp'}$ such that $\lambda \cdot s\in \scrE_{x}$, and put 
 \[\|s\|_U=\max_{x\in U}|s(x)|.\] 

We apply the construction above to the integral model $\X_K^{\tor}$  and the modular line bundle $\omegab^{(s_J\cdot \kb,w)}$  over it for any subset $J\subseteq \Sigma_{\infty}$. 
 For a quasi-compact admissible open subset $U\subset\,]X^{\tor,\ord}_{K}[_{\rb'}$, we have a well-defined norm $\|\cdot\|_U$ on the space of sections $\Gamma(U, \omegab^{(s_J\cdot\kb,w)})$. 
 For a section $s$ of $\pr_2^*(\omegab^{(s_J\cdot\kb,w)}(-\D))$ over a quasi-compact open subset $V\subseteq\, ]X^{\tor,\ord}_{K}[^{\gothp}_{\rb'}$, we put 
\begin{equation}
\label{E:norm on X tor p}
\|s\|_V=\|(\pr_{2}^{*})^{-1}(s)\|_{\pr_2(V)}
\end{equation}
since $\pr_2\colon ]X^{\tor,\ord}_{K}[^{\gothp}_{\rb'}\ra ]X^{\tor,\ord}_K[_{\rb}$ is an isomorphism of rigid spaces.

\begin{lemma}\label{Lemma:trace}
Let $U\subset\, ]X^{\tor, \ord}_{K}[$ be a quasi-compact admissible open subset, and $g$ be a section of $\cO_{\pr_1^{-1}(U)}$. We have
\[
\|\Tr_{\pr_1}(g)\|_{U}\leq p^{-d_{\gothp}}\|g\|_{\pr_1^{-1}(U)},
\]
where $d_{\gothp}=[F_{\gothp}:\Q_p]$, and $\Tr_{\pr_1}: \Gamma(\pr_1^{-1}(U), \cO_{\pr^{-1}(U)})\ra \Gamma(U, \cO_{U})$ is the trace map.
\end{lemma}
\begin{proof}
Since $\pr_2$ is an isomorphism, one may write $g=\pr_2^*(h)$. Then, by definition \eqref{E:norm on X tor p}, we have $\|g\|_{\pr_1^{-1}(U)}=\|h\|_{\pr_2(\pr_1^{-1}(U))}$.
% By scaling, we may assume that $\|h\|_{]X^{\tor, \ord}_K[}=1$ so that $h$ is defined over the integral formal model $\X^{\tor,\ord}_{K}$. \liang{I don't think one can extend $h$ from $\pr_2\pr_1^{-1}(U)$ to the entire $]X^{\tor, \ord}_K[$ in general. But we don't need this, right?}
By \eqref{Equ:p1-p2}, we have $\pr_1=\pi_{S_{\gothp}}^{-1}\circ \varphi_{\gothp}\circ \pr_2$ and hence $\Tr_{\pr_1}=\pi^{*}_{S_{\gothp}}\Tr_{\varphi_{\gothp}}\Tr_{\pr_2}$. Note that $\Tr_{\pr_2}$ is the inverse of $\pr_2^*$, since $\pr_2$ is an isomorphism. 
Thus, we have $\Tr_{\pr_1}(g)=\pi^{*}_{S_{\gothp}}(\Tr_{\varphi_{\gothp}}(h))$. 
Since $\pi_{S_{\gothp}}$ is an automorphism of the integral model $\X_K^{\tor}$,  we have
\[
\|\Tr_{\pi^{-1}_{S_{\gothp}}}\circ\Tr_{\varphi_{\gothp}}(h)\|_{U}=\|\Tr_{\varphi_{\gothp}}(h)\|_{S_{\gothp}(U)}.
\]
It thus  suffices to show that $\|\Tr_{\varphi_{\gothp}}(h)\|_{V}\leq p^{-d_{\gothp}} \| h\|_{\varphi_\gothp^{-1}(V)}$ for $V = \pi_{S_\gothp}(U)$. But this follows from Lemma \ref{Lemma:varphi-omega}.
\end{proof}

\begin{prop}\label{prop:norm-U_p}
Let $U_1, U_2\subset\, ]X_{K}^{\tor,\ord}[$ be quasi-compact admissible open subsets in the ordinary locus such that $\pr_1^{-1}(U_1)\subseteq \pr^{-1}_2(U_2)$, and let $f$ be a section of $\omegab^{(s_J\cdot\kb,w)}(-\D)$ over $U_2$. We have
\[\|U_{\gothp}(f)\|_{U_1}\leq  p^{-\big(\sum_{\tau\in \Sigma_{\infty/\gothp}}\frac{w-k_{\tau}}{2}+\sum_{\tau \in(\Sigma_{\infty/\gothp}-J) }(k_{\tau}-1)\big)}\|f\|_{U_2}.\]
\end{prop}
\begin{proof}
After shrinking $U_1$ and $U_2$, we may assume that, for each $\tau\in \Sigma_{\infty}$, there exists
   a basis $(\omega_{\tau, i},\eta_{\tau, i})$ of $\calH^1_{\tau}$ over $U_i$ adapted to the Hodge filtration $0\ra \omegab_{\tau}\ra \calH^1_{\tau}\ra \wedge^2(\calH^1_{\tau})\otimes \omegab^{-1}_{\tau}\ra 0$ and satisfying
     $$
     \|\omega_{\tau, i}\|_{U_i}=\|\eta_{\tau,i}\|_{U_i}=1.
     $$
Put $\varepsilon_{\tau,i}=\omega_{\tau,i}\wedge \eta_{\tau,i}$, which is a basis of $\wedge^2\calH_\tau^1$ over $U_i$ with $\|\varepsilon_{\tau,i}\|_{U_i}=1$. 
%Then $\varepsilon_{\tau,i}\otimes \omega_{\tau,i}^{-1}$  is a basis of  $\wedge^2(\calH^1_{\tau})\otimes \omegab^{-1}_{\tau}$.
 By definition of $\omegab^{(s_J\cdot \kb,w)}$ \eqref{E:omega-s-j}, we  may write
    \[
    f=g \bigotimes_{\tau\notin J }\bigl(\varepsilon_{\tau,2}^{\frac{w+k_{\tau}}{2}-2}\otimes \omega_{\tau,2}^{2-k_{\tau}} \bigr)
   \bigotimes_{\tau\in J} \bigl( \varepsilon_{\tau,2}^{\frac{w-k_{\tau}}{2}-1}\otimes\omega_{\tau,2}^{k_{\tau}} \bigr)\quad \text{with }g\in \Gamma(U_2,\cO_{U_2})
    \]
    By definition, we have
\begin{small}
\[
U_{\gothp}(f)=\Tr_{\pr_1}\biggl(\pr_2^*(g)|_{\pr_1^{-1}(U_1)}\cdot \check{\pi}_{\gothp}^*\pr_2^*\biggl(\bigotimes_{\tau\notin J }\bigl(\varepsilon_{\tau,2}^{\frac{w+k_{\tau}}{2}-2}\otimes \omega_{\tau,2}^{2-k_{\tau}} \bigr)
   \bigotimes_{\tau\in J} \bigl( \varepsilon_{\tau,2}^{\frac{w-k_{\tau}}{2}-1}\otimes\omega_{\tau,2}^{k_{\tau}} \bigr)
\biggr)\biggr).
\]
\end{small}
      There exist rigid analytic functions $a_{\tau}, b_{\tau}$ on $\pr_1^{-1}(U_1)$ such that
    \[
    \begin{cases}
    \check{\pi}^*_{\gothp}\pr_2^*(\omega_{\tau,2})=a_{\tau} \pr_1^*(\omega_{\tau,1}),\\
    \check{\pi}^*_{\gothp}\pr_2^*(\etab_{\tau,2})=b_{\tau} \pr_1^*(\etab_{\tau,1}),\\
    \check{\pi}^*_{\gothp}\pr_2^*(\varepsilon_{\tau,2})=a_{\tau}b_{\tau} \pr_1^*(\varepsilon_{\tau,1}).
    \end{cases}
    \]
    By Lemma \ref{Lemma:check-pi}, we have $\|a_{\tau}\|_{\pr_1^{-1}(U_1)}=1$ for all $\tau\in \Sigma_{\infty}$,  $\|b_{\tau}\|_{\pr_1^{-1}(U_1)}=1$ for $\tau\notin \Sigma_{\infty/\gothp}$ and $\|b_{\tau}\|_{\pr_1^{-1}(U_1)}=p^{-1}$ for $\tau\in \Sigma_{\infty/\gothp}$.
    % Similarly, there exists a rigid analytic function $c_{J}$ such that $\pr_2^*(d z_{J,2})=c_{J} \pr_1^*(d z_{J,1})$. By Lemma \ref{Lemma:norm-omega}, we have $\|c_{J}\|_{\pr_1^{-1}(U_1)}=p^{|J\cap \Sigma_{\infty/\gothp}|}$.
     So we obtain
    \[
U_{\gothp}(f) =\Tr_{\pr_1}\big(\pr_2^*(g)h\big)\bigg(\bigotimes_{\tau\notin J }\bigl(\varepsilon_{\tau,1}^{\frac{w+k_{\tau}}{2}-2}\otimes \omega_{\tau,1}^{2-k_{\tau}} \bigr)
   \bigotimes_{\tau\in J} \bigl( \varepsilon_{\tau,1}^{\frac{w-k_{\tau}}{2}-1}\otimes\omega_{\tau,1}^{k_{\tau}} \bigr)\bigg),
    \]
where $h=\big(\prod_{\tau\notin J}(a_{\tau}b_{\tau})^{\frac{w+k_{\tau}}{2}-2}a_{\tau}^{2-k_{\tau}} \big)\cdot\big(\prod_{\tau\in J}(a_{\tau}b_{\tau})^{\frac{w-k_{\tau}}{2}-1} a_{\tau}^{k_{\tau}}\big)$.
    Now it follows from Lemma \ref{Lemma:trace} that
    \begin{align*}
    \|U_{\gothp}(f)\|_{U_1}&=\|\Tr_{\pr_1}\bigl(\pr_2^*(g)h\bigr)\|_{U_1}\\
    &\leq p^{-d_{\gothp}-\sum_{\tau\in ( \Sigma_{\infty/\gothp}-J)}(\frac{w+k_{\tau}}{2}-2)-\sum_{\tau\in \Sigma_{\infty/\gothp}\cap J}\frac{w-k_\tau}{2}-1}
    \|g\|_{\pr_2(\pr_1^{-1}(U_1))}\\
    &\leq p^{-\sum_{\tau\in \Sigma_{\infty/\gothp}}\frac{w-k_{\tau}}{2}-\sum_{\tau\in ( \Sigma_{\infty/\gothp}-J)}(k_{\tau}-1)}\|f\|_{U_2}.
    \end{align*}
\end{proof}

We deduce immediately from Proposition \ref{prop:norm-U_p} the following
  \begin{cor}\label{Prop:slopes-ocv}
Let  $f\in S^{\dagger}_{(s_J\cdot\kb,w)}(K,L_{\wp})$ be a generalized eigenform for ${U}_{\gothp}$ with eigenvalue $\lambda_{\gothp}\neq 0$. Then we have
  \[
  \val_p(\lambda_{\gothp})\geq \sum_{\tau\in \Sigma_{\infty/\gothp}}\frac{w-k_{\tau}}{2}+\sum_{\tau\in (\Sigma_{\infty/\gothp}-J)} (k_{\tau}-1).
  \]
  \end{cor}

\section{Formalism of Rigid Cohomology}
\label{Section:spectral sequence}
In this section, we will relate the cohomology group $H^{\star}_{\rig}(X^{\tor,\ord}_K,\D; \F^{(\kb,w)})$ to the rigid cohomology of the Goren-Oort strata of the Hilbert modular variety.

\subsection{A brief recall of rigid cohomology}\label{S:rigid-coh}
We recall what we need  on the rigid cohomology. For more details, we refer the reader to \cite{berthelot2, berthelot} and \cite{tsuzuki}.
Let $L_{\wp}$ be a finite extension of $\Qp$, $\calO_{\wp}$ the  ring of integers and $k_0$ the residue field. Let $\calP$ be a proper smooth formal scheme over  $W(k_0)$, $P$  its special fiber,  and  $\calP_\rig$  the associated rigid analytic space.  We have a natural specialization map $\spe\colon \calP_\rig \to  P$. 
For a locally closed subscheme $Z \subseteq P$, we put $]Z[_{\calP}=\spe^{-1}(Z)$. 
 When it is clear, we omit the subscript $\calP$ from the notation.

For $\overline X$ a locally closed subscheme of $P$, $j_X: X \to \overline X$  an open subset, and $\calE$  a  sheaf of abelian groups defined over some strict neighborhood of   $]X[$ in $]\overline {X}[$, we put
\[
j_X^\dagger \calE = \varinjlim_V j_{V*}j^*_V \calE,
\]
where $V$ runs through a fundamental system of  strict neighborhoods of $]X[$ inside $]\overline X[$ on which $\calE$ is defined,  and $j_V: V \to ]\overline X[$  is the natural inclusion.

An \emph{overconvergent $F$-isocrystal} $\scrE$ on $X/L_{\wp}$ can be viewed as a locally free coherent sheaf defined over some strict neighborhood $V$ of $]X[$ inside $]\overline X[$, equipped with
 an integrable connection $\nabla: \scrE\ra \scrE\otimes_{\cO_V} \Omega^1_{V}$ satisfying certain (over)convergence conditions \cite[Chap. 2]{berthelot}, and, Zariski locally, with an isomorphism $F^*\scrE \to \scrE$ where $F$ is a Zariski local lift of the absolute Frobenius to $\calP$. Let $\DR(\scrE)=\scrE \otimes \Omega^\bullet_{V}$ denote the associated de Rham complex. 
The \emph{rigid cohomology} of $\scrE$ is defined  to be 
\[
R\Gamma_\rig(X/L_{\wp}, \scrE) : = R\Gamma\big(]\overline X[, j_X^\dagger\DR(\scrE)\big).
\]
When $\scrE$ is the constant $F$-isocrystal, we simply put  $R\Gamma_{\rig}(X/L_{\wp})=R\Gamma_{\rig}(X/L_{\wp},\scrE)$.
For a sheaf $\calE$ of abelian groups over a strict neighborhood of $]X[$, we define a sheaf on $]\overline X[$ by 
\[
\Gammab_{]X[}(\calE): =\Ker (j^{\dagger}_{X}\calE\ra i_*i^*\calE),
\] 
where $i:\;]\overline X-X[\ra ]\overline X[$ is the canonical immersion. 
Following Berthelot \cite{berthelot2},  the \emph{rigid cohomology with compact support} of $X$ with values in $\scrE$ is given by 
\[
R\Gamma_{c,\rig}(X/L_{\wp},\scrE):=R\Gamma\big(]\overline X[, \Gammab_{]X[}(\DR(\scrE))\big).
\]
There is a natural map $R\Gamma_{c,\rig}(X/L_{\wp},\scrE)\ra R\Gamma_{\rig}(X/L_{\wp},\scrE)$ in the derived category, which induces maps on cohomology groups $H^{\star}_{c,\rig}(X/L_{\wp},\scrE)\ra H^{\star}_{\rig}(X/L_{\wp},\scrE)$. 
By the main theorem of \cite{kedlaya}, these cohomology group are finite $L_\wp$-vector spaces.

Similarly, if $Z$ is a closed subscheme of $X$, we define the functor $\Gammab^{\dagger}_{]Z[}$ by   
\[\Gammab^\dagger_{]Z[}(\calE):= \Ker (j_{X}^\dagger \calE\ra  j^{\dagger}_{X-Z}\calE)\]
for any sheaf of abelian groups $\calE$ defined over a strict neighborhood of $]X[$ on $]\overline X[$. The functor  $\Gammab^{\dagger}_{]Z[}$ is exact. The \emph{rigid cohomology with support in $Z$} of the $F$-isocrystal $\scrE$ is defined to be 
\[
R\Gamma_{ Z,\rig}(X/L_{\wp}, \scrE): = R\Gamma\big( ]\overline X[,  \Gammab^\dagger_{]Z[}(\DR(\scrE))\big).
\]
There is a canonical distinguished triangle 
\[
R\Gamma_{Z,\rig}(X/L_{\wp},\scrE|_Z)\ra R\Gamma_{\rig}(X/L_{\wp},\scrE)\ra R\Gamma_{\rig}(X-Z/L_{\wp},\scrE|_{X-Z})\xra{+1}
\]
In particular, one has canonical maps of cohomology groups
\begin{equation}\label{E:rig-coh-Z-X}
H^{\star}_{Z,\rig}(X/L_{\wp},\scrE)\ra H^{\star}_{\rig}(X/L_{\wp},\scrE).
\end{equation}
If $Z$ is closed in $\overline X$ (equivalently, $Z$ is proper over $k_0$), then this map factor through $H^{\star}_{Z,\rig}(X/L_{\wp},\scrE)\ra H^{\star}_{c,\rig}(X/L_{\wp},\scrE)$.  
It is standard that $H^{\star}_{\rig}(X/L_{\wp},\scrE)$ and  $H^{\star}_{Z,\rig}(X/L_{\wp},\scrE)$ are independent of the embedding $X\hra P$ and the choice of formal model $\calP$.
We remark that if $U$ is an open subscheme of $X$ containing $Z$, then we have a natural isomorphism $H^*_{Z,\rig}(X/L_{\wp}, \scrE) \cong H^*_{ Z,\rig}(U/L_{\wp}, \scrE|_U)$ \cite[Proposition 2.1.1]{tsuzuki}. 

If $Z$ is a smooth closed subvariety of a smooth variety $X$ of codimension $r$, and if $\scrE$ is an overconvergent $F$-isocrystal $\scrE$ on $X$, then \cite[Th\'eor\`em~3.8]{berthelot3} says that we have a canonical quasi-isomorphism, called the \emph{Gysin isomorphism}
\begin{equation}
\label{E:Gysin-isom}
G_{Z, \scrE}: R\Gamma_\rig(Z/L_\wp, \scrE|_Z) \xrightarrow\cong \R\Gamma_{Z, \rig}(X/L_\wp, \scrE)[2r](r),
\end{equation}
where $\cdot [2r](r)$ means to shift up the cohomological degree by $2r$ and to multiply the action of the arithmetic Frobenius by $p^r$.
In particular, this induces a canonical isomorphism on the corresponding cohomology groups: $G_{Z, \scrE}: H^\star_\rig(Z/L_\wp, \scrE|_Z) \cong H^{\star+2r}_{Z, \rig}(X/L_\wp, \scrE)$.
When $\scrE$ is the constant isocrystal, we write $c(Z) \in H^{2r}_{c, \rig}(X/L_\wp)$ for the image of $1 \in H^0_\rig(Z/L_\wp)$ under the Gysin map in the compactly supported rigid cohomology. It is the \emph{rigid cycle class} of $Z$ in $X$.
 
 \subsection{Formalism of dual \v Cech complex}\label{S:cech-symbol}
Let $\Sigma$ denote a finite set.
Assume that, to each subset $\ttT \subseteq \Sigma$, there is an associated  $\QQ$-vector space $M_\ttT$ such that for each inclusion of subsets $\ttT_1 \subseteq \ttT_2$, we have an (ordering reversing)  $\QQ$-linear map $i_{\ttT_2, \ttT_1}: M_{\ttT_2} \to M_{\ttT_1}$ satisfying the natural cocycle condition.
We consider some formal symbols  $e_\tau$, called the \emph{\v Cech symbols}, indexed by elements $\tau \in \Sigma$, and their formal wedge products in the sense that $e_\tau \wedge e_{\tau'} = -e_{\tau'} \wedge e_\tau$ for $\tau, \tau' \in \Sigma$.
For a subset $\ttT = \{\tau_1, \dots, \tau_i\}$ of $\Sigma$, we fix an order for it and write $e_\ttT$ for $e_{\tau_1}\wedge \cdots \wedge e_{\tau_i}$.
%We put $\ttS^c = \Sigma - \ttS$.
The \emph{dual \v Cech complex} associated to $M_\ttT$ is then given by
\[
M_\Sigma  e_\Sigma \to \cdots \to \bigoplus_{\#\ttT = 2}M_{\ttT}  e_{\ttT} \to  \bigoplus_{\#\ttT = 1}M_\ttT e_\ttT \to M_\emptyset,
\]
where the connecting homomorphisms are given by, for $\ttT = \{\tau_1, \dots, \tau_i\}$,
\[
m_\ttT  e_{\tau_1} \wedge \cdots \wedge e_{\tau_i} \mapsto \sum_{j=1}^i
(-1)^j
 i_{\ttT, \ttT-\{\tau_j\}} (m_\ttT) e_{\tau_1} \wedge \cdots e_{\tau_{j-1}} \wedge e_{\tau_{j+1}}
\wedge \cdots \wedge e_{\tau_i}.
\]
It is clear from the construction that this is a complex. Note that when  $M_{\ttT}=M$ for all $\ttT\subseteq \Sigma$ and $i_{\ttT_2,\ttT_1}=\Id_M$ for all  $\ttT_{1}\subset \ttT_2$, the dual \v Cech complex associated to $M_{\ttT}$ is acyclic (if $\Sigma$ is non-empty).

\begin{lemma}\label{L:resolution}
Let  the notation be  as in Subsection~\ref{S:rigid-coh}. Let $Y=\bigcup_{\tau\in \Sigma}Y_\Sigma$ be a finite union of closed subschemes of $X$. For any subset $\ttT\subseteq \Sigma$, we put $Y_{\ttT}=\bigcap_{\tau\in \ttT}Y_{\tau}$, and let $j_{\ttT}: \overline X-Y_{\ttT}\ra \overline X$ denote the natural immersion. 
For any  sheaf $\calE$ of abelian groups defined on a strict neighborhood of $]\overline X-Y[$, the sequence  
\begin{equation}\label{E:resolution}
0\ra j_{\Sigma}^\dagger \calE \;e_{\Sigma}
\to \bigoplus_{\tau \in \Sigma}
j_{\Sigma\backslash\{\tau\}}^\dagger \calE\; e_{\Sigma\backslash \{\tau\}} \to \cdots \to \bigoplus_{\tau\in \Sigma }j^\dagger_{\{\tau\}}\calE \;e_{\tau}\ra j^{\dagger}_{\overline X-Y} \calE\; e_{\emptyset}\ra 0.
\end{equation}
given with dual \v Cech complex is exact. Here, we place $j^{\dagger}_{\overline X-Y} \calE\; e_{\emptyset}$ at degree $0$, and the $(-i)$-th term is a direct sum, over all subsets $\ttT \subseteq \Sigma$ with $\# \ttT = i$, of $j^\dagger_{\ttT}\calE \, e_{\ttT}$, and all the morphisms are natural restriction maps. 
\end{lemma}
\begin{proof}
This is a standard property of \v Cech covering. We prove by induction on $\#\Sigma$. When $\#\Sigma=1$, the statement is trivial. Assume now that the Lemma holds for $\#\Sigma=n-1$, and we need to prove it for $\#\Sigma=n$. For each $\tau\in \Sigma$, let $V_\tau$ be a strict neighborhood of $]\overline X-Y_{\tau}[$. Then $V_\tau$ for all $\tau \in \Sigma$ and $]Y_{\Sigma}[$ form an admissible covering of $]\overline X[$. The restriction of \eqref{E:resolution} to $]Y_\Sigma[$ is identically zero, it suffices to prove its exactness when restricted to each $V_\tau$. By standard arguments of direct limits, it is enough to prove the exactness of \eqref{E:resolution} after applying $j^\dagger_{\{\tau\}}$. Note that $j^{\dagger}_{\{\tau\}}j^\dagger_{\ttT}=j^{\dagger}_{\{\tau\}}$ if $\tau\in \ttT$, and $j^{\dagger}_{\{\tau\}}j^{\dagger}_{\overline X-Y}=j^{\dagger}_{\overline X-Y}$. It is easy to see that after applying $j^{\dagger}_{\{\tau\}}$, the resulting complex is the mapping fiber of a morphism from 
\begin{itemize}
\item
the dual \v Cech complex concentrated in degrees $[-n, -1]$ with constant group $j^{\dagger}_{\{\tau\}}\calE$, to
\item
a complex of type \eqref{E:resolution} but with $X$ replaced by $X'=X-Y_{\tau}$ and $Y$ replaced by $Y'=\bigcup_{\tau'\in \Sigma\backslash\{\tau\}} (Y_{\tau'} - Y_\tau)$.
\end{itemize}
By the last remark of the previous Subsection, the latter is acyclic. Hence, the desired exactness follows from the induction hypothesis.
\end{proof}

 \subsection{Setup of Hilbert modular varieties}\label{S:rig-setup}

Let $L$, $L_{\wp}$, $\cO_{\wp}$ and $k_0$ be as in Subsection~\ref{S:HMF-notation}.
We fix an open subgroup $K=K_pK^p$ such that $K_p=\GL_2(\cO_{F}\otimes_{\Z}\Z_p )$, and $K^p$ satisfies Hypothesis~\ref{H:fine-moduli}. 
To simplify notation, let $\bfX$ denote the base change to $W(k_0)$ of the integral model of the Shimura variety  $\bfSh_{K}(G)$ considered  in Subsection~\ref{Subsection:HMV}. Let $\bfX^{\tor}$ be a toroidal compactification of $\bfX$ as in Subsection~\ref{Subsection:compactification}.
We use $X$ and $X^\tor$ to denote their special fibers over $k_0$. Let $\X^{\tor}$ be the formal completion of $\bfX^\tor$ along its special fiber, and let $\X_\rig^\tor$  denote the base change to $L_{\wp}$ of  the rigid analytic spaces associated to $\bfX^\tor$. Let $\fX\subset \fX^{\tor}$ denote the open formal subscheme corresponding to $X$.
 For a subvariety $Z\subseteq X^{\tor}$, we denote by $]Z[=]Z[_{\X^{\tor}}$ the tube  of $Z$ in $\X^{\tor}_{\rig}$.

For $\tau \in \Sigma_\infty$, let $Y_\tau$ denote the vanishing locus of the partial Hasse invariant $h_{\tau}$ at $\tau \in \Sigma_\infty$ defined in Subsection~\ref{Subsection:Hasse}. Recall that these $Y_\tau$ are smooth divisors with simple normal crossings. Note that $Y_{\tau}$ has no intersection with the toroidal boundary $\D$. We put $Y=\bigcup_{\tau\in \Sigma_{\infty}}Y_{\tau}$, and  $X^{\tor,\ord} = X^\tor - Y$ and $X^\ord = X^{\tor, \ord} \cap X$.
For a subset $\ttT \subseteq \Sigma_\infty$, we put $Y_\ttT = \cap _{\tau \in \ttT} Y_\tau$. It is a smooth closed subvariety of $X^{\tor}$ of codimension $\#\ttT$, and we call it a \emph{closed Goren-Oort stratum} (or GO-stratum for short) of codimension $\#\ttT$. As a convention, we put $Y_{\emptyset}=X$.

\subsection{Isocrystals on the Hilbert modular varieties}
%We will also consider the associated isocrystals over the special fiber of Shimura varieties.  For the definition and basic facts on $F$-isocrystals and their cohomology theory, we refer to  \cite{berthelot}.  
%For a variety $S$ over $k_0$, we denote by $(S/W(k_0))_{\cris}$ the small crystalline site of $S$ relative to $W(k_0)$, and $\cO_{S/W(k_0)}$ the crystalline structure sheaf. 

Let  $\calA^{\sm}$ denote the family of semi-abelian varieties over $X^{\tor}$ which extends the universal HBAV $\calA$ on $X$. 
Let $(X/W(k_0))_{\cris}$ denote the crystalline site of $X$ relative to the natural divided power structure on $(p)\subset W(k_0)$.
 Then the relative crystalline cohomology $\calH^1_{\cris}(\calA/X)$ is an $F$-crystal over $(X/W(k_0))_{\cris}$.
 The evaluation of $\calH^1_{\cris}(\calA/X)$ at the divided power embedding $X\ra \fX$ is canonically identified with the relative de Rham cohomology $\calH^1_{\dR}(\calA/\fX)$, where $\calA$ also denotes the universal HBAV over $\fX$ by abuse of notation.
  We denote by $\scrD(\calA)$ the (overconvergent) $F$-isocrystal on $X/W(k_0)[1/p]$ (hence also an $F$-isocrystal over $X/L_{\wp}$ by base change) associated to $\calH^1_{\cris}(\calA/X)$.
  % By using log-crystalline cohomology, we can extend $\scrD^1(\calA)$ to a log-isocrystal $\scrD^1(\calA^{\sm})$ over $X^{\tor}/W(k_0)$. 
 % The evaluation of  $\scrD^1(\calA)$ on $X\hra\fX$  is canonically identified with  the relative log-de Rham cohomology $\calH^1$ \eqref{Subsection:dR-cohomology} on the associated rigid analytic space $\fX^{\tor}_{\rig}$ over $L_0$.
The action of $\cO_F$ on $\calA$ induces an action of $\cO_F$ on $\scrD(\calA)$, which gives rise to a natural decomposition
\[
\scrD(\calA) =\oplus_{\tau\in \Sigma_{\infty}}\scrD(\calA)_{\tau},
\]
where each $\scrD(\calA)_{\tau}$ is a isocrystal of rank $2$. 

For a multiweight $(\kb, w)$, we put
\[
\scrD^{(\kb,w)}: =\bigotimes_{\tau\in \Sigma_{\infty}}(\wedge^2\scrD(\calA)_{\tau})^{\frac{w-k_\tau}{2}}\otimes \Sym^{k_{\tau}-2}\scrD(\calA)_{\tau}.
\]
This is an $F$-isocrystal over $X/L_{\wp}$, and its evaluation on $\fX$ is  the vector bundle $\F^{(\kb,w)}$ defined in Subsection~\ref{S:descent} on the rigid analytic variety  $\fX_{\rig}$. The isocrystal $\scrD^{(\kb,w)}$ extends to the vector bundle $\F^{(\kb, w)}$ over  $\fX^{\tor}_{\rig}$ equipped with an integrable connection with logarithmic poles along $\D$ (Subsection~\ref{S:descent}).
 For a subvariety $Z\subset X$ disjoint from $\D$, the rigid cohomology  of $Z$ with values in $\scrD^{(\kb,w)}$ can be computed as 
 \[H^{\star}_{\rig}(Z/L_{\wp},\scrD^{(\kb,w)})=H^{\star}\big(]Z[, j_{Z}^{\dagger}\DR(\F^{(\kb,w)})\big),\]
 where $j_Z$ denote the canonical inclusion $ ]Z[\hra \fX^{\tor}_{\rig}$.

\subsection{ Partial Frobenius on $X$}
\label{S:Frobenius-HMV}
%We start by discussing the partial Frobenius action on the Hilbert modular variety $X:= \bfSh_K(G)_{\FF_p}$.
 Let $S$ be a locally noetherian $\FF_p$-scheme, and  $x=(A,\iota, \bar\lambda, \bar\alpha_{K^p})$ an $S$-valued point of $X$. For each $\gothp\in \Sigma_p$, we construct a new point $\varphi_{\gothp}(x)=(A',\iota', \bar\lambda',\bar\alpha'_{K^p})$ of $X$  as follows: 
 \begin{itemize}
 \item Let $\Ker_{\gothp}$ denote the $\gothp$-component of the kernel of the relative Frobenius homomorphism $\Fr_{A}: A\ra A^{(p)}$. We put $A'=A/\Ker_{\gothp}$, and equip it with the induced action $\iota'$ of $\cO_F$. Let $\pi_{\gothp}: A\ra A'$ denote the canonical isogeny. 
 
 \item If $\lambda$ is a $\gothc$-polarization on $A$, then it induces a natural $\gothc\gothp$-polarization on $A'$ determined by the commutative diagram:
 \[
\xymatrix{ A'\otimes_{\cO_F}\gothc\gothp\ar[r]^-{\check\pi_{\gothp}}\ar[d]^{\cong}_{\lambda'} & A\otimes_{\cO_F} \gothc\ar[d]_{\lambda}^{\cong} \\
A'^\vee\ar[r]^{\pi_{\gothp}^\vee} & A^\vee.
}
\] 
Here, $\check\pi_{\gothp}$ is the unique map such that the composite 
$A\otimes_{\cO_{F}}\gothc\gothp\xra{\pi_{\gothp}} A'\otimes_{\cO_F}\gothc\gothp\xra{\check\pi_{\gothp}}A\otimes_{\cO_F}\gothc
$ 
is the canonical quotient map by $A[\gothp] \otimes_{\calO_F} \gothc \gothp$.
  
  \item The $K^p$-level structure $\alpha'_{K^p}$ on $A'$ is the unique one induced by the isomorphism $\pi_{\gothp,*}\colon T^{(p)}(A)\xra{\sim} T^{(p)}(A')$ of prime-to-$p$ Tate modules. 
 \end{itemize}
 With the convention in Remark~\ref{R:polarization}, $(A',\iota',\bar\lambda',\bar\alpha_{K^p})$ defines a point on $X$.
 We denote by $\varphi_{\gothp}: X\ra X$ the induced endomorphism of the Hilbert modular variety.
 It is a finite and flat morphism of degree $p^{[F_{\gothp}:\Q_p]}$. By choosing appropriate cone decompositions, one may assume that $\varphi_{\gothp}$ extends to an endomorphism of $X^\tor$. 
 It is clear that the restriction of $\varphi_{\gothp}$ to the ordinary locus $X^{\tor, \ord}$ coincides with the reduction of $\varphi_{\gothp}:\X^{\tor,\ord}\ra \X^{\tor,\ord}$ considered in Subsection~\ref{S:Frobenius-ordinary}, since the $\gothp$-canonical subgroups there lift  $\Ker_{\gothp}$. (But $\varphi_\gothp$ does not lift to $\X^{\tor}$ in general.) 
 Note that $\varphi_{\gothp}$ and $\varphi_{\gothq}$ with $\gothp\neq \gothq$ commute with each other, and the product $F_{X/\FF_p}=\prod_{\gothp\in\Sigma_p}\varphi_{\gothp}:X\ra X$ is the Frobenius endomorphism of $X$ relative to $\FF_p$.
  We call $\varphi_{\gothp}$ the \emph{$\gothp$-partial Frobenius}. 

Let $\sigma_{\gothp}: \Sigma_{\infty}\ra \Sigma_{\infty}$ be the  map defined by  
  \[
  \sigma_{\gothp}(\tau)=\begin{cases}\tau&\text{ if }\tau\notin \Sigma_{\infty/\gothp},\\
   \sigma\tau& \text{ if }\tau\in \Sigma_{\infty/\gothp}.
   \end{cases}
  \] 
For a subset $\ttT\subseteq\Sigma_{\infty}$, we denote by $\sigma_{\gothp}\ttT$ its image under $\sigma_{\gothp}$. 
  
\begin{lemma}\label{L:partial-Frob-Hasse}
Let $x=(A,\iota,\bar\lambda, \bar\alpha_{K^p})$ be a point of $X$ with values in a locally noetherian $k_0$-scheme $S$, and let $\varphi_{\gothp}(x)=(A',\iota',\bar\lambda',\bar\alpha'_{K^p})$ be its image under $\varphi_\gothp$. Then $\tau$-partial Hasse invariant $h_{\tau}(\varphi_{\gothp}(x))$ is canonically identified with  $h_{\tau}(x)$ if $\tau\notin\Sigma_{\infty/\gothp}$, and with $h_{\sigma^{-1}\tau}(x)^{\otimes p}$ if $\tau\in \Sigma_{\infty/\gothp}$; in particular, if $S$ is the spectrum of a perfect field, then $h_{\tau}(\varphi_{\gothp}(x))=0$ if and only if $h_{\sigma_{\gothp}^{-1}\tau}(x)=0$.
   \end{lemma}
  \begin{proof}
  The statement is clear for $\tau\notin \Sigma_{\infty/\gothp}$. Now suppose that $\tau\in \Sigma_{\infty/\gothp}$. %Let $\cD(A[\gothp^{\infty}])$ be the evaluation at the trivial embedding $S\hra S$ of the  contravariant Dieudonn\'e  crystal of $A[\gothp^{\infty}]$. Then we have a canonical isomorphism $\cD(A[\gothp^{\infty}])\simeq \oplus_{\tau\in \Sigma_{\infty/\gothp}}H^{1}_{\dR}(A/S)_{\tau}$. 
  As $\gothp$ is unramified, $A'[\gothp^{\infty}]$ is the quotient of $A[\gothp^{\infty}]$ by its kernel of Frobenius, hence there exists an isomorphism of $p$-divisible groups $A'[\gothp^{\infty}] \cong (A[\gothp^{\infty}])^{(p)}.$
   There exists thus an isomorphism 
  $$
  \omega_{A'/S,\tau}=\omega_{A'[\gothp^{\infty}]/S,\tau}\cong \omega_{A/S,\sigma^{-1}\tau}^{(p)}
  $$ 
  compatible with the morphism induced by the Verschiebung. It follows that $h_{\tau}(A')$ can be identified with the base change of $h_{\sigma^{-1}\tau}(A)$ via the absolute Frobenius on $S$, whence the Lemma.
    \end{proof}
    
    \begin{cor}\label{C:partial-Frob}
    For a subset $\ttT\subseteq \Sigma_{\infty}$, the restriction of the partial Frobenius $\varphi_{\gothp}$ to $Y_{\ttT}$ defines a finite flat map $\varphi_{\gothp}: Y_{\ttT}\ra Y_{\sigma_{\gothp}\ttT}$ of degree $p^{\#(\Sigma_{\infty/\gothp}-\ttT_{\gothp})}$,
     with $\ttT_{\gothp}=\Sigma_{\infty/\gothp}\cap\ttT$.
      If $\varphi_{\gothp}^{-1}(Y_{\sigma_{\gothp}\ttT})$ is the fiber product of $\varphi_{\gothp}: X\ra X$ with the closed immersion $Y_{\sigma_{\gothp}\ttT}\hra X$,
       then we have an equality $[\varphi_{\gothp}^{-1}(Y_{\sigma_{\gothp}\ttT})]=p^{\#\ttT_{\gothp}}[Y_{\ttT}]$ in the group of algebraic cycles on $X$ of codimension $\#\ttT$.
    \end{cor}
    \begin{proof}
Lemma~\ref{L:partial-Frob-Hasse} implies that $\varphi_{\gothp}$ sends $Y_{\ttT}$ to $Y_{\sigma_{\gothp}\ttT}$. We note that  $\prod_{\gothp\in \Sigma_{p}} \varphi_{\gothp}: Y_{\ttT}\ra Y_{\ttT}^{(p)}$ is the relative Frobenius of $Y_{\ttT}$, which is finite flat. The flatness criterion by fibers implies the finite flatness of  $\varphi_{\gothp}|_{Y_{\ttT}}$.    
     By the Lemma,  $\varphi_{\gothp}^{-1}(Y_{\sigma_{\gothp}\ttT})$ is the closed subscheme of $X$ defined by vanishing of $h_{\tau}$'s for $\tau\in \ttT-\ttT_{\gothp}$ and $h_{\tau}^{\otimes p}$'s  for $\tau\in \ttT_{\gothp}$. Hence,   $Y_{\ttT}$ is the closed subscheme of $\varphi_{\gothp}^{-1}(Y_{\sigma_{\gothp}\ttT})$ defined by the vanishing of $h_{\tau}$'s for $\tau\in \ttT_{\gothp}$. Since $Y_{\ttT}$ is non-singular,  the equality $[\varphi_{\gothp}^{-1}Y_{\sigma_{\gothp}\ttT}]=p^{\#\ttT_{\gothp}}[Y_{\ttT}]$ follows immediately.   
Note that  $\varphi^{-1}_{\gothp}(Y_{\sigma_{\gothp}\ttT})$ is a finite and flat of degree $p^{[F_{\gothp}:\Q_p]}$ over $Y_{\sigma_{\gothp}\ttT}$. Hence, the flat map $\varphi_{\gothp}|_{Y_{\ttT}}$ must have degree $p^{[F_{\gothp}:\Q_p]}/p^{\#\ttT_{\gothp}}=p^{\#(\Sigma_{\infty/\gothp}-\ttT_{\gothp})}$.
        \end{proof}
  
   We have the isogeny $\pi_{\gothp}: \univA^{\sm}\ra \varphi_{\gothp}^*\univA^{\sm}$ obtained by taking the quotient by the subgroup $\Ker_{\gothp}$ of $\univA^{\sm}$. This induces an isomorphism of $F$-isocrystals:
 \[
 \pi^*_{\gothp}: \varphi_{\gothp}^*\scrD(\calA)=\scrD(\varphi^*_{\gothp}\calA)\xra{\cong} \scrD(\calA),
 \]
 and hence an isomorphism $\pi^*_{\gothp}: \varphi_{\gothp}^*\scrD^{(\kb,w)}\cong \scrD^{(\kb,w)}$.
This gives rise to an  operator $\Fr_{\gothp}$ on the rigid cohomology for each $Y_{\ttT}$ with $\ttT\subseteq \Sigma_{\infty}$:
\begin{align*}
\Fr_{\gothp} \colon H^{\star}_{c,\rig}\big(Y_{\ttT}/L_{\wp},\scrD^{(\kb,w)}|_{Y_{\ttT}}\big)&\xrightarrow{\varphi^*_{ \gothp}} H^{\star}_{c,\rig}\big(Y_{\sigma_{\gothp}^{-1}\ttT}/L_{\wp},\varphi_{\gothp}^*\scrD^{(\kb,w)}|_{Y_{\sigma_{\gothp}^{-1}\ttT}}\big)
\\
&\xrightarrow{\pi_{\gothp}^*}
 H^{\star}_{c,\rig}\big(Y_{\sigma_{\gothp}^{-1}\ttT}/L_{\wp},\scrD^{(\kb,w)}|_{Y_{\sigma^{-1}_{\gothp}\ttT}}\big).
\end{align*}
 Here  $H^{j-2i}_{c,\rig}\big( Y_\ttT/L_{\wp}, \scrD^{(\kb,w)}|_{Y_\ttT} \big)$ is  the same as the usual rigid cohomology without compact support if $\ttT\neq \emptyset$. 
 Similarly, we have an operator $\Fr_{\gothp}$ on $H^{\star+2\#\ttT}_{Y_{\ttT},\rig}(X, \scrD^{(\kb,w)})$
   such that the following diagram is commutative:
 \begin{equation}
 \xymatrix{
 H^{\star}_{c,\rig}(Y_{\ttT}/L_{\wp},\scrD^{(\kb,w)}|_{Y_{\ttT}})\ar[rr]^-{\cong}_-{\mathrm{Gysin}}\ar[d]_{p^{\#\ttT_{\gothp}}\Fr_{\gothp}} && H^{\star+2\#\ttT}_{Y_{\ttT}}(X/L_{\wp},\scrD^{(\kb,w)})\ar[d]^{\Fr_{\gothp}}\\
 H^{\star}_{c,\rig}(Y_{\sigma^{-1}_{\gothp}\ttT}/L_{\wp},\scrD^{(\kb,w)}|_{Y_{\sigma^{-1}_{\gothp}\ttT}}) \ar[rr]^-{\cong}_-{\mathrm{Gysin}} &&H^{\star+2\#\ttT}_{Y_{\sigma^{-1}_{\gothp}\ttT}}(X/L_{\wp},\scrD^{(\kb,w)}).
 }
 \end{equation}
 Here, $p^{\#\ttT_{\gothp}}$ appears on the left vertical arrow, because the inverse image of cycle class  $c(Y_{\ttT})\in H^{2\#\ttT}_{Y_{\ttT},\rig}(X/L_{\wp})$ under $\varphi_{\gothp}^*$ is  the class
 \[
\varphi_{\gothp}^*c(Y_{\ttT})=c(\varphi_{\gothp}^{-1}(Y_{\ttT}))=p^{\#\ttT_{\gothp}}c(Y_{\sigma_{\gothp}^{-1}\ttT}) \in H^{2\#\ttT}_{Y_{\sigma_{\gothp}^{-1}\ttT},\rig}(X/L_{\wp}),
 \]
where we used Corollary~\ref{C:partial-Frob} and basic properties of the rigid class cycle map \cite{petregin}.

Recall  that we have, for each $\gothp\in \Sigma_p$, an automorphism $S_\gothp$ on $X^\tor$   defined in Subsection~\ref{S:operator S_gothp}. We have an natural isogeny 
\[[\varpi_{\gothp}]: \calA\ra S_{\gothp}^*\calA=\calA\otimes_{\cO_F}\gothp^{-1},
\]
which induces an isomorphism of isocrystals $[\varpi_{\gothp}]^*:S_{\gothp}^*\scrD^{(\kb,w)}\cong \scrD^{(\kb,w)}$.
Since $Y_{\ttT}$ is stable under $S_{\gothp}$ for each $\ttT\subseteq \Sigma_{\infty}$, the morphism   $S_{\gothp}$ induces an automorphism 
\[
S_{\gothp}: H^{\star}_{c,\rig}(Y_{\ttT}/L_{\wp},\scrD^{(\kb,w)}|_{Y_{\ttT}})\ra H^{\star}_{c,\rig}(Y_{\ttT}/L_{\wp},\scrD^{(\kb,w)}|_{Y_{\ttT}}).
\]

\subsection{Twisted partial Frobenius}
In order to compare with the unitary setting later, we define the \emph{twisted partial Frobenius} to be
\[
\gothF_{\gothp^2}: =\varphi_\gothp^2 \circ S_\gothp^{-1}: X^{\tor}\ra X^{\tor}.
\]
Note that $\gothF_{\gothp^2}$ sends a point $(A,\iota,\bar\lambda,\bar\alpha_{K^p})$ to $((A/\Ker_{\gothp^2})\otimes_{\cO_F}\gothp,\iota',\bar\lambda',\bar\alpha'_{K^p})$, where $\Ker_{\gothp^2}$ is the $\gothp$-component of the kernel of the relative $p^2$-Frobenius $A\ra A^{(p^2)}$. It is clear that $\gothF_{\gothp^2}$ send a GO-stratum $Y_{\ttT}$ to $Y_{\sigma^2_{\gothp}\ttT}$.
We use 
$
\eta: \calA \rightarrow
\gothF_{\gothp^2}^*\calA
$
to denote the canonical quasi-isogeny 
\[
\calA\ra \calA/\Ker_{\gothp^2}\leftarrow \gothF^*_{\gothp^2}\calA=(\calA/\Ker_{\gothp^2})\otimes_{\cO_F}\gothp.
\]
It induces an isomorphism of $F$-isocrystals $\eta^*_{\gothp}: \gothF^*_{\gothp^2}\scrD(\calA)=\scrD(\gothF^*_{\gothp^2}\calA)\xra{\cong}\scrD(\calA)$, and hence an isomorphism
$
\eta^*_{\gothp}\colon \gothF^*_{\gothp^2}\scrD^{(\kb,w)}\xra{\cong} \scrD^{(\kb,w)}
$.
For $\ttT\subseteq \Sigma_{\infty}$ we define the  operator $\Phi_{\gothp^2}$ on the rigid cohomology to be
\begin{eqnarray}\label{E:partial-Frob-Hilbert}
\Phi_{\gothp^2}: 
H^{\star}_{c,\rig}(Y_{\ttT}/L_{\wp},\scrD^{(\kb,w)}|_{Y_{\ttT}})&\xrightarrow{\gothF^*_{ \gothp^2}} H^{\star}_{c,\rig}(Y_{\sigma_{\gothp}^{-2}\ttT}/L_{\wp},\gothF_{\gothp^2}^*\scrD^{(\kb,w)}|_{Y_{\sigma_{\gothp}^{-2}\ttT}})\\
\nonumber
&\xrightarrow{\eta_{\gothp}^*} 
 H^{\star}_{c,\rig}(Y_{\sigma_{\gothp}^{-2}\ttT}/L_{\wp},\scrD^{(\kb,w)}|_{Y_{\sigma^{-2}_{\gothp}\ttT}}).
\end{eqnarray}
It is clear that $\Phi_{\gothp^2}=\Fr_{\gothp}^2\cdot S_{\gothp}^{-1}$, and $\Phi_{\gothp^2}$ commute with $\Phi_{\gothq^2}$ for $\gothp,\gothq\in \Sigma_p$. Similar to the case for $\Fr_{\gothp}$, we also have an operator 
$\Phi_{\gothp^2}$ on  $H^{\star+2\#\ttT}_{Y_{\ttT},\rig}(X,\scrD^{(\kb,w)})$ such that the following diagram is commutative:
\begin{equation}\label{E:Phi-Gysin}
\xymatrix{
 H^{\star}_{c,\rig}(Y_{\ttT}/L_{\wp},\scrD^{(\kb,w)}|_{Y_{\ttT}})\ar[rr]^-{\cong}_-{\mathrm{Gysin}}\ar[d]_{p^{2\#\ttT_{\gothp}}\Phi_{\gothp^2}} && H^{\star+2\#\ttT}_{Y_{\ttT}}(X/L_{\wp},\scrD^{(\kb,w)})\ar[d]^{\Phi_{\gothp^2}}\\
 H^{\star}_{c,\rig}(Y_{\sigma^{-2}_{\gothp}\ttT}/L_{\wp},\scrD^{(\kb,w)}|_{Y_{\sigma^{-2}_{\gothp}\ttT}}) \ar[rr]^-{\cong}_-{\mathrm{Gysin}} &&H^{\star+2\#\ttT}_{Y_{\sigma^{-2}_{\gothp}\ttT}}(X/L_{\wp},\scrD^{(\kb,w)}).}
\end{equation}

%For $\ttS' \subseteq \Sigma_\infty \cup \Sigma_p$, we put $\bfX_{\ttS'} = \bfSh_{K_{\ttS', p}}(G_\ttS)\otimes W(k_0)$, where $K_{\ttS', p}$ is the level structure at places above $p$ as specified in Subsection~\ref{S:level-structure}.
%Let $X_{\ttS'}$ denote its special fiber.

Recall that we have defined the  cohomology group  $H^{\star}_{\rig}(X^{\tor,\ord}, \D; \F^{(\kb, w)})$ in Subsection~\ref{subsection:rigid-coh}. Its relation with the rigid cohomology of the strata $Y_{\ttT}$  is given by the following

 \begin{prop}\label{P:spectral-sequence}
\begin{itemize}
	\item[(1)]
 There exists a  spectral sequence in the second quadrant 
 \begin{equation}
\label{E:spectral sequence for ordinary cohomology}
E_1^{-i,j}  = \bigoplus_{\#\ttT = i}  H^{j-2i}_{c,\rig}\big( Y_\ttT/L_{\wp}, \scrD^{(\kb,w)}|_{Y_\ttT} \big) e_{\ttT}\Rightarrow H^{j-i}_{\rig}(X^{\tor, \ord}, \D; \scrF^{(\kb,w)}). 
\end{equation}
Here, the $e_{\ttT}$'s are the \v Cech symbols from
Subsection~\ref{S:cech-symbol}, and the  transition maps $d_{1}^{-i,j}\colon E^{-i,j}_1\ra E^{1-i,j}_1$ are direct
sums of the Gysin maps $H^{j-2i}_{c,\rig}(Y_{\ttT}/L,\scrD^{(\kb,w)}|_{Y_{\ttT}})\ra H^{j-2i+2}_{c,\rig}(Y_{\ttT'}/L, \scrD^{(\kb,w)}|_{Y_{\ttT'}})$ with $\ttT'\subseteq \ttT$ and $\#\ttT'=\#\ttT-1=i-1$ using the dual \v Cech complex formalism in Subsection~\ref{S:cech-symbol}.

\item[(2)] The spectral sequence is equivariant under the natural action of tame Hecke algebra $\scrH(K^p,L_{\wp})=L_{\wp}[K^p\backslash \GL_2(\AAA^{\infty,p})/K^p]$,  and the actions of $\Fr_{\gothp}$ for each $\gothp\in \Sigma_{p}$, if we let $\Fr_{\gothp}$ act on $H^{\star}_{\rig}(X^{\tor,\ord},\D;\F^{(\kb,w)})$ as in Subsection~\ref{S:operator-Phi}, and  on the spectral sequence $(E^{-i,j}_1, d_1^{-i,j})$  as follows: for $\ttT=\{\tau_1,\cdots, \tau_{\#\ttT}\}$, we define
\[
\xymatrix@R=0pt{
\Fr_\gothp : H^{j-2i}_{c,\rig}(Y_\ttT/L_{\wp}, \scrD^{(\kb,w)}|_{Y_{\ttT}}) e_{\ttT} \ar[r] & H^{j-2i}_{c,\rig}(Y_{\sigma^{-1}_\gothp\ttT}/L_{\wp},\scrD^{(\kb,w)}|_{Y_{\sigma^{-1}_{\gothp}\ttT}}) e_{\sigma^{-1}_\gothp\ttT}\\
m_\ttT \cdot e_{\tau_1} \wedge \cdots \wedge e_{\tau_{\#\ttT}} \ar@{|->}[r] &
\Fr_\gothp(m_\ttT) \cdot p^{\#\ttT_{\gothp}}\cdot  e_{\sigma_\gothp^{-1}\tau_1} \wedge \cdots\wedge e_{\sigma_\gothp^{-1} \tau_{\#\ttT}},
}
\]
where $\ttT_{\gothp}=\ttT\cap\Sigma_{\infty/\gothp}$.
Similarly, the spectral sequence is equivariant for the action of $S_{\gothp}$ and $\Phi_{\gothp^2}=\Fr^2_{\gothp}\cdot S^{-1}_{\gothp}$ on both sides. 

\end{itemize}
 \end{prop}
 
\begin{proof}
This is a standard excision of Weil cohomology for simple normal crossing divisors on a variety, with some special attention to the sign convention and the action of Frobenius. 
We put $\F=\F^{(\kb,w)}$ to simplify the notation.
For an open subset $U\subseteq X^{\tor}$, let $j_{U}: U\hra X^\tor$ denote the natural immersion.
By Lemma~\ref{L:resolution}, we have the  following resolution of $j^{\dagger}_{X^{\tor,\ord}}\DR_c(\F)$:
\begin{small}
\begin{equation}
\label{E:cech resolution}
j_{X^{\tor} - Y_{\Sigma_\infty}}^\dagger \DR_c(\F) e_{\Sigma_\infty}
\to \bigoplus_{\tau \in \Sigma_\infty}
j_{X^{\tor}- Y_{\Sigma_\infty\backslash \tau}}^\dagger \DR_c(\F) e_{\Sigma_\infty\backslash \tau} \to \cdots \to \bigoplus_{\tau \in \Sigma_\infty} j^\dagger_{X^{\tor} - Y_{\tau}}\DR_c(\F) e_{\tau}.
\end{equation}
\end{small}
Here,  the $(-i)$-th term is a direct sum of $j^\dagger_{X^{\tor} - Y_\ttT}\DR_c(\F) e_{\ttT}$, over a subset $\ttT \subseteq \Sigma_\infty$ such that $\#\ttT = i$; and 
all connecting maps are natural restrictions. Hence, using the  remark at the end of Subsection~\ref{S:cech-symbol}, we see that  the following  double complex, denoted by $K^{\bullet,\bullet}$,
\begin{small}
\[
\xymatrix@C=15pt{
\DR_c(\F) e_{\Sigma_\infty} \ar[r]\ar[d] &
\displaystyle\bigoplus_{\tau \in \Sigma_\infty}
\DR_c(\F) e_{\Sigma_\infty \backslash \tau} \ar[r] \ar[d]
&
\cdots \ar[r] &
\displaystyle\bigoplus_{\tau \in \Sigma_\infty}
\DR_c(\F)
e_{\tau} \ar[d]
\\
j_{X^{\tor} - Y_{\Sigma_\infty}}^\dagger \DR_c(\F) e_{\Sigma_\infty} \ar[r]
&
\displaystyle \bigoplus_{\tau \in \Sigma_\infty}
j_{X^{\tor} - Y_{\Sigma_\infty \backslash \tau}}^\dagger \DR_c(\F)e_{\Sigma_\infty \backslash\tau} \ar[r]
&
 \cdots
 \ar[r]&
 \displaystyle \bigoplus_{\tau \in \Sigma_\infty} j^\dagger_{X^{\tor}- Y_{\tau}}\DR_c(\F) e_{\tau}
}
\]
\end{small}
is quasi-isomorphic to 
\begin{equation*}
\Cone \big[\DR_c(\F)\to
j^\dagger_{X^{\tor,\ord}} \DR_c(\F)
\big][-1].
\end{equation*}
In other words, if $\underline{s}(K^{\bullet,\bullet})$ denote the simple complex associated to $K^{\bullet,\bullet}$, then we have a quasi-isomophism
\begin{equation}\label{E:resolution-of-j-ordinary}
\biggl( \underline{s}(K^{\bullet,\bullet})\ra \DR_{c}(\F)\biggl)\xra{\sim} j^{\dagger}_{X^{\tor,\ord}}\DR_c(\F),
\end{equation}
where $\underline{s}(K^{\bullet,\bullet})\ra \DR_{c}(\F)$ is induced by the sum of identity maps $\bigoplus_{\tau\in \Sigma_{\infty}}\DR_c(\F)e_{\tau}\ra \DR_c(\F)$. Taking global sections on $\fX^{\tor}_{\rig}$, one obtains a spectral sequence in the second quadrant:
\begin{equation}\label{E:spectral-seq-proof}
E_{1}^{-i,j}\Rightarrow H^{j-i}(\fX^{\tor}_{\rig}, j^{\dagger}_{X^{\tor,\ord}}\DR_c(\F))=H^{j-i}_{\rig}(X^{\tor,\ord},\D;\F),
\end{equation}
where $E_1^{0,j}=H^j(\fX^{\tor}_{\rig}, \DR_c(\F))$ and 
\[
E_1^{-i,j}= \bigoplus_{
\#\ttT=i} H^{j}(\fX^{\tor}_{\rig},  \Cone \big[\DR_c(\F)\to j^{\dagger}_{X^{\tor}-Y_{\ttT}} \DR_c(\F)\big][-1]) e_{\ttT},\quad\text{for $i\geq 1$.} 
\]

We observe that, for each non-empty subset  $\ttT\subseteq \Sigma_{\infty}$, we have a quasi-isomorphism of complexes
$$
\Cone \big[\DR_c(\F)\to j^{\dagger}_{X^{\tor}-Y_{\ttT}} \DR_c(\F)\big][-1]\cong \Cone \big[j_{X}^{\dagger} \DR(\F)\to j^{\dagger}_{X-Y_{\ttT}}\DR(\F)\big][-1]
$$
 by the excision.   
 After taking global sections over $\X^\tor_{\rig}$, one obtains the rigid cohomology with support in $Y_{\ttT}$:
 $$
 R\Gamma_{Y_{\ttT},\rig}(X/L_{\wp},\scrD^{(\kb,w)})= R\Gamma(\X_{\rig}^{\tor}, \Cone \big[\DR_c(\F)\to j^{\dagger}_{X^{\tor}-Y_{\ttT}} \DR_c(\F)\big][-1]).
 $$
 Therefore, by the Gysin isomorphism \eqref{E:Gysin-isom}, the term $E_{1}^{-i,j}$ in \eqref{E:spectral-seq-proof} for $i\geq 1$ is isomorphic to the direct sum of $H^{j-2i}_{Y_{\ttT},\rig}(X,\scrD^{(\kb,w)})$ for all $\ttT\subseteq \Sigma_{\infty}$ with $\#\ttT=i$.

 Now statement (1) of the Proposition follows from \eqref{E:spectral-seq-proof}  and  Lemma~\ref{L:rigid cohomology with log pole} below.  By functoriality of the construction,   the spectral sequence is clearly equivariant under the action of $\scrH(K^p,L_\wp)$. For the equivariance under the actions of $\Fr_{\gothp}$ and $\Phi_{\gothp^2}$, it suffices to note that the action of $\Fr_{\gothp}$ on the spectral sequence \eqref{E:spectral sequence for ordinary cohomology} has already taken account of the Frobenius twist given by the Gysin isomorphism.
\end{proof}

\begin{lemma}
\label{L:rigid cohomology with log pole}
The cohomology $H^\star(\X^{\tor}_{\rig}, \DR_c(\scrF^{(\kb,w)}))$ is canonically isomorphic to the rigid cohomology with compact support $H^\star_{c,\rig}(X/L_{\wp}, \scrD^{(\kb,w)})$.
\end{lemma}
\begin{proof}
This is well known to the experts, but unfortunately not well recorded in the literature.
To sketch a  proof, we pass to the dual. As above, we set $\scrF = \scrF^{(\kb,w)}$.
By \cite[2.6]{baldassarri-chiarellotto}, there is an isomorphism $H^\star(\X^{\tor}_{\rig}, \DR(\scrF^\vee)) \cong H^\star_\rig(X/L_{\wp}, \scrD^{(\kb,w),\vee})$.
By \cite[Theorem~1.2.3]{kedlaya},
$ H^\star_{c, \rig}(X/L_{\wp}, \scrD^{(\kb,w)})$ is in natural Poincar\'e duality  with $H^{2g-\star}_\rig(X/L_{\wp}, \scrD^{(\kb,w),\vee})$. 

It then suffices to show that
$H^{\star}(\X^{\tor}_{\rig},\DR_c(\F))$ 
and $H^{2g-\star}(\X_{\rig}^{\tor}, \DR(\F^\vee))$ are in natural Poincar\'e duality.  By the rigid GAGA theorem, this is equivalent to proving that the algebraic de Rham cohomology
$H^{\star}(\bfX^{\tor}_{L_{\wp}}, \DR_c(\F))$ and $H^{2g-\star}(\bfX^{\tor}_{L_{\wp}}, \DR(\F^\vee))$ are in Poincar\'e duality.
Unfortunately, this is only available in the literature \cite{baldassarri-log-duality} when $\F$ equals to the  constant sheaf.
One can either modify  the proof of \emph{loc. cit.} for the general case; or alternatively, using the  embedding $L_{\wp}\hra \overline{\Q}_p\xleftarrow{\sim} \C$, one reduces to show that $H^{\star}(\bfX^{\tor}(\C), \DR_c(\F))$ and $H^{2g-\star}(\bfX^{\tor}(\C), \DR(\F^\vee))$ are in Poincar\'e duality.
 Let $\LL=(\F|_{\bfX(\C)})^{\nabla=0}$ denote the local system of horizontal sections of $\F$ on $\bfX(\C)$.  By the Riemann-Hilbert-Deligne correspondence and classical GAGA, $H^{2g-\star}(\bfX^{\tor}(\C), \DR(\F^\vee))$ is canonically isomorphic to the singular cohomology $H^{2g-\star}(\bfX(\C), \LL^{\vee})$.
  By the same arguments as in \cite[Chap. VI 5.4]{FC}, one proves that $H^{\star}(\bfX^{\tor}(\C), \DR_c(\F))$ is the same as $H^{\star}(\bfX^{\tor}(\C), j_{!}\LL)=H^{\star}_{c}(\bfX(\C), \LL)$, where  $j: \bfX(\C)\ra \bfX^\tor(\C)$ denotes the natural immersion. The desired duality now follows from the classical Poincar\'e theory for  manifolds.
\end{proof}

\subsection{\'Etale Cohomology}\label{S:etale-coh}
To compute the cohomology groups $H^{\star}_{c,\rig}(Y_{\ttT}/L_{\wp}, \scrD^{(\kb,w)})$, we compare them with their \'etale analogues. Let $l\neq p$ be a fixed prime, and fix an isomorphism $\iota_{l}: \C\xra{\simeq} \overline\Q_{l}$. This defines an $l$-adic place $\gothl$ of the number field $L$; denote by $L_{\gothl}$  its completion. Post-composition with $\iota_l$ identifies $\Sigma_{\infty}$ with the set of $l$-adic embeddings of $F$.  Let $a:\calA\ra \bfSh_{K}(G)$ be the structural morphism of the universal abelian scheme. The relative \'etale  cohomology $R^1a_{*}(L_{\gothl})$  has a canonical decomposition:
\[
R^1a_{*}(L_{\gothl})=\bigoplus_{\tau\in \Sigma_{\infty}} R^1a_{*}(L_{\gothl})_{\tau},
\]
where $R^1a_{*}(L_{\gothl})_{\tau}$ is the direct summand on which $F$ acts via $\iota_{l}\circ \tau$.
   For a multiweight $(\kb,w)$, we put 
\[
\scrL^{(\kb,w)}_{\gothl}: =\bigotimes_{\tau\in \Sigma_{\infty}}\Big( (\wedge^2R^1a_{*}(L_{\gothl})_{\tau})^{\frac{w-k_{\tau}}{2}}\otimes \Sym^{k_{\tau}-2} R^1a_{*}(L_{\gothl})_{\tau}\Big).
\]
Note that $\scrL_{\gothl}^{(\kb,w)}|_{X}$ is a lisse $L_{\gothl}$-sheaf pure of weight $g(w-2)$. We have a natural action of the prime-to-$p$ Hecke operators $\scrH(K^p,L_{\gothl})$ on $H^{\star}_{\et}(Y_{\ttT}, \scrL^{(\kb,w)}_{\gothl})$ for each $\ttT\subseteq \Sigma_\infty$.
For each $\gothp\in \Sigma_p$, consider $\varphi_{\gothp}: X\ra X$. The isogeny $\pi_{\gothp}: \calA\ra \varphi^*_{\gothp}\calA$ induces an isomorphism 
$
\pi_{\gothp}^*\scrL^{(\kb,w)}_{\gothl}\xra{\cong}\scrL^{(\kb,w)}_{\gothl}.
$
This gives rise to an action of the partial Frobenius on cohomology groups:
\begin{equation}\label{E:partial-Frob-l}
\Fr_{\gothp}: H^{\star}_{c, \et}(Y_{\ttT,\Fpb}, \scrL^{(\kb,w)}_{\gothl})\xra{\varphi^*_{\gothp}} H^{\star}_{c,\et}(Y_{\sigma^{-1}_{\gothp}\ttT,\Fpb}, \varphi_{\gothp}^*\scrL_{\gothl}^{(\kb,w)})\xra{\pi_{\gothp}^*}H^{\star}_{c,\et}(Y_{\sigma^{-1}_{\gothp}\ttT,\Fpb}, \scrL_{\gothl}^{(\kb,w)}).
\end{equation}
As usual, we put $Y_{\emptyset}=X$.
Similarly, we have morphisms  $\Phi_{\gothp^2}$ and $S_{\gothp}$ on $H^{\star}_{c,\et}(Y_{\ttT,\Fpb},\scrL^{(\kb,w)}_{\gothl})$, as in the case of rigid cohomology. 
For simplicity, we use $Y^{(i)}$ to denote the \emph{disjoint union} of the GO-strata $Y_\ttT$ for $\#\ttT = i$.  Then, for a fixed integer $i\geq 0$, $H^{\star}_{c,\et}(Y^{(i)}_{\Fpb}, \scrL^{(\kb,w)}_{\gothl}) = \bigoplus_{\#\ttT =i}H^{\star}_{c,\et}(Y_{\ttT,\Fpb}, \scrL^{(\kb,w)}_{\gothl})$ is stable under the action of $\Fr_\gothp$ for each $\gothp$.

\begin{prop}\label{P:comparison-rig-et}
We identify both $\overline\Q_{l}$ and $\Qpb$ with $\C$ using the isomorphisms $\iota_{l}$ and $\iota_p$. Then for an integer $i\geq 0$, we have an equality  in the Grothendieck group of finite-dimensional $\scrH(K^p,\C)[\Fr_{\gothp},S_{\gothp}, S_{\gothp}^{-1}; \gothp\in \Sigma_{p}]$-modules:
\[
\sum_{n=0}^{2g-2i}(-1)^{n}\big[\bigoplus_{\#\ttT=i}H^{n}_{c,\et}(Y_{\ttT,\Fpb}, \scrL^{(\kb,w)}_{\gothl})\otimes_{L_{\gothl}}\overline\Q_l\big]=\sum_{n=0}^{2g-2i} (-1)^n\big[\bigoplus_{\#\ttT=i}H^{n}_{c,\rig}(Y_{\ttT}/L_{\wp}, \scrD^{(\kb,w)})\otimes_{L_{\wp}}\Qpb\big].
\]
Moreover, if $i\neq 0$, we have an equality for each $n$:
\[
\big[\bigoplus_{\#\ttT=i}H^{n}_{c,\et}(Y_{\ttT,\Fpb}, \scrL^{(\kb,w)}_{\gothl})\otimes_{L_\gothl}\overline \Q_{l}\big]=\big[\bigoplus_{\#\ttT=i}H^{n}_{c,\rig}(Y_{\ttT}/L_{\wp}, \scrD^{(\kb,w)})\otimes_{L_{\wp}}\overline\Q_p\big].\]
\end{prop}
\begin{proof}
Note that $\scrL^{(\kb,w)}_{\gothl}$ and $\scrD^{(\kb,w)}$  on $X$ are pure of weight $g(w-2)$ in the sense of Deligne and \cite{abe-caro} respectively. As $Y_{\ttT}$ is proper and smooth for $\ttT\neq\emptyset$,  $H^n_{\et}(Y_{\ttT,\Fpb},\scrL^{(\kb,w)}_{\gothl})$ and $H^{n}_{\rig}(Y_{\ttT}/L_{\wp}, \scrD^{(\kb,w)})$ are both pure of weight $g(w-2)+n$ by Deligne's Weil II and its rigid analogue  (\emph{loc. cit.}). Since $\prod_{\gothp\in\Sigma_{p}}\Fr_{\gothp}$ is the  Frobenius endomorphism of $X$, 
the weight can be detected by the action of $\prod_{\gothp\in \Sigma_p}\Fr_{\gothp}$, the second part of the Proposition follows immediately from the first part.

To prove the first part, we follow the strategy of \cite[\S6]{saito}.
We consider  $ F\otimes_{\Q}L\cong\prod_{\tau\in \Sigma_{\infty}}L_{\tau}$, where $L_{\tau}$ is the copy of $L$ with embedding $\tau:F\hra L$. Let $e_{\tau}\in F\otimes_{\Q} L$ denote the projection onto $L_{\tau}$. Since $F\otimes_{\Q}L$ is generated over $L$ by $1+p\cO_F$, we may write $e_{\tau}$ as a linear combination of elements in $1+p\cO_F$. Hence, $e_{\tau}$ is a linear combination of endomorphisms of $\calA$ over $\bfX$ of degrees prime to $p$.

Using the Vandermonde determinant, one can find easily a $\Q$-linear combination $e^1$ of multiplications by prime-to-$p$ integers   on $\calA$ such that the induced action of $e^1$ on $R^1a_{*}(\Q_l)$ is the identity map, and is $0$ on $R^qa_{*}\Q_{l}$ for $q\neq 1$. 
 Consider the fiber product $a^{w-2}: \calA^{w-2}\ra \bfX$. Then $e_{\tau}^{\otimes w-2}\cdot (e^1)^{\otimes w-2}$ acts as an idempotent on $R^{q}a^{w-2}_{*}(L_{\gothl})$, we get $(R^1a_{*}(L_{\gothl})_{\tau})^{\otimes (w-2)}$ if $q=w-2$, and $0$ if $q\neq w-2$. One finds also easily an idempotent  $e^{(k_{\tau},w)}\in \Q[\gothS_{w-2}]$ in the group algebra of the symmetric group with $w-2$ letters  such that 
$$
e^{(k_{\tau},w)}\cdot (R^1a_{*}(L_{\gothl})_{\tau})^{\otimes (w-2)}=(\wedge^2R^1a_*(L_{\gothl})_\tau)^{\frac{w-k_{\tau}}{2}}\otimes \Sym^{k_{\tau}-2}R^1a_{*}(L_{\gothl})_{\tau}.
$$  
Note the action of $\gothS_{w-2}$ on $R^1a_{*}(L_{\gothl})_{\tau}^{\otimes (w-2)}$ is induced by its action on $\calA^{w-2}$.

Consider the fiber product $a^{(w-2)g}: \calA^{ (w-2)g}\ra \bfX$. Taking the product of the previous constructions, we get a $L$-linear combination $e^{(\kb,w)}$ of algebraic correspondences on $\calA^{(w-2)g}$ satisfying the following properties:

\begin{itemize}
\item[(1)] It is an $L$-linear combination of permutations in $\gothS_{(w-2)g}$ and endomorphisms of $\calA^{(w-2)g}$ as an abelian scheme over $\bfX$ whose degrees are prime to $p$,

\item[(2)] The action of $e^{(\kb,w)}$ on the cohomology $R^qa^{(w-2)g}_{*}(L_{\gothl})$ is the projection onto the direct summand $\scrL^{(\kb,w)}_{\gothl}$ if $q=(w-2)g$ and is equal to $0$ if $q\neq (w-2)g$.
 % Similarly, let $R^{q}_{\cris}a^{(w-2)g}(\cO_{\calA^{(w-2)g}/W(k_0)})$ be the relative crystalline cohomology of $(\calA^{(w-2)g}/W(k_0))_{\cris}$ over $(X/W(k_0))_{\cris}$, and $\calH^{q}_{\rig}(\calA^{})$ 
\end{itemize}

The algebraic correspondence $e^{(\kb,w)}$ with coefficients in $L$ acts also on  $H^{q}_{c,\et}(\calA_{\Fpb}^{g(w-2)}, L_{\gothl})$. 
Using the Leray spectral sequence for $a^{(w-2)g}$, one sees easily that, for any locally closed subscheme  $Z\subseteq X$, we have
\begin{equation}\label{E:projection}
e^{(\kb,w)}\cdot H^{n+(w-2)g}_{c,\et}(\calA^{g(w-2)}\times_{X}Z_{\Fpb}, L_{\gothl})=H^{n}_{c,\et}(Z_{\Fpb}, \scrL^{(\kb,w)}_{\gothl}).
\end{equation}
Similarly, let $(\calA^{(w-2)g}/W(k_0))_{\cris}$ over $(X/W(k_0))_{\cris}$ denote respectively the small crystalline sites of $\calA^{(w-2)g}$ and $X$ with respect to $W(k_0)$, and  $R^{q}_{\cris}a^{(w-2)g}_*(\cO_{\calA^{(w-2)g}/W(k_0)})$ be the relative crystalline cohomology. This is an $F$-crystal over $(X/W(k_0))_{\cris}$, and we denote by  $\calH^{q}_{\rig}(\calA^{(w-2)g}/X)$ the associated overconvergent isocrystal on $X/L_{\wp}$. The algebraic correspondence  $e^{(\kb,w)}$ acts on $\calH^{q}_{\rig}(\calA^{(w-2)g}/X)$ as an idempotent, and we have
 \[
 e^{(\kb,w)}\cdot\calH^{q}_{\rig}(\calA^{(w-2)g}/X)=\begin{cases}\scrD^{(\kb,w)}&\text{if }q=(w-2)g,\\
0&\text{otherwise}.
\end{cases}
\]
Consequently, we have
\[
e^{(\kb,w)}\cdot H^{n+(w-2)g}_{c,\rig}(\calA^{g(w-2)}\times_{X}Z/L_{\wp})=H^{n}_{c,\rig}(Z/L_{\wp}, \scrD^{(\kb,w)})
\]
 for any subscheme $Z\subseteq X$.
 
As $K^p$ varies, the Hecke action of  $\GL_2(\AAA^{\infty,p})$ on $\bfSh_{K}(G)$ \eqref{E:Hecke-tor} lifts to an equivariant action on  $\calA$. 
Then, for each double coset $[K^pgK^p]$ with $g\in \GL_2(\AAA^{\infty,p})$, there exists a finite flat algebraic correspondence on $\calA^{(w-2)g}$ such that, 
after composition with $e^{(\kb,w)}$, its induced actions   on $H^{(w-2)g+n}_{c,\et}(\calA^{(w-2)g}\times_{X}Y_{\ttT,\Fpb},L_{\gothl})$ and  $H^{n+(w-2)g}_{c,\rig}(\calA^{g(w-2)}\times_{X}Y_{\ttT}/L_{\wp})$
 give the Hecke action of $[K^pgK^p]$ on $H^{n}_{c,\et}(Y_{\ttT,\Fpb}, \scrL^{(\kb,w)}_{\gothl})$ and $H^{n}_{c,\rig}(Y_{\ttT}/L_{\wp}, \scrD^{(\kb,w)})$, respectively. 

Consider the partial Frobenius $\varphi_{\gothp}: X\ra X$ for each $\gothp\in \Sigma_p$. The quasi-isogeny $\pi_{\gothp}^{(w-2)g}: \calA^{(w-2)g}\ra \varphi_{\gothp}^*(\calA^{(w-2)g})=\varphi_{\gothp}^*(\calA)^{(w-2)g}$ defines an algebraic correspondence on $\calA^{(w-2)g}$ whose composition with $e^{(\kb,w)}$ induces the action of $\Fr_{\gothp}$ on $H^{n}_{c,\et}(Y^{(i)}_{\Fpb}, \scrL^{(\kb,w)}_{\gothl})$ and on $H^{n}_{c,\rig}(Y^{(i)}/L_{\wp}, \scrD^{(\kb,w)})$.
 Similarly, the actions of $S_{\gothp}$ and $S_{\gothp}^{-1}$ on the \'etale and rigid cohomology groups of $Y^{(i)}$ are also induced by algebraic correspondences on $\calA^{(w-2)g}$. 

In summary, the action of $\scrH(K^p,\CC)[\Fr_{\gothp}, S_{\gothp},S_{\gothp}^{-1};\gothp\in \Sigma_p]$  on the \'etale  and rigid cohomology groups are linear combinations of actions induced by algebraic correspondences on $\calA^{(w-2)g}$. Therefore, in order to prove the first part of the Proposition, it suffices to show, for any algebraic correspondence $\Gamma$ of $\calA^{(w-2)g}$ and an integer $i\geq 0$, we have the following equality:
\[
\sum_{n}(-1)^n\Tr(\Gamma^*, H^{n}_{c,\et}(\calA^{(w-2)g}\times_{X}Y^{(i)}_{\Fpb},L_\gothl))
=\sum_{n}(-1)^n\Tr(\Gamma^*,H^n_{c,\rig}(\calA^{(w-2)g}\times_X Y^{(i)}/L_{\wp})).
\]
If $i\geq 1$, $\calA^{(w-2)g}\times_{X}Y^{(i)}$ is proper and smooth over $k_0$. Since the cycle class map is well defined for \'etale and rigid cohomology \cite{petregin} and  the Lefschetz formula is valid,  the two  sides  above are both equal to the intersection number $(\Gamma, \Delta)$, where $\Delta$ is the diagonal of  $(\calA^{(w-2)g}\times_{X}Y^{(i)})\times (\calA^{(w-2)g}\times_{X}Y^{(i)})$. If $i=0$, the desired equality still holds thanks to \cite[Corollary 3.3]{mieda}, whose proof uses Fujiwara's trace formula \cite{fujiwara} and its rigid analogue due to Mieda.
\end{proof}
\begin{remark}
Even though it will not be used in this paper, it is interesting to consider 
the   \'etale counterpart of  $H^{\star}_{\rig}(X^{\tor,\ord}, \D; \F^{(\kb,w)})$. Let $t: X\hra X^{\tor}$ and $j: X^{\tor,\ord}\ra X^{\tor}$ be the natural open immersions. 
 Then $X^{\tor,\ord}$  can be viewed as a partial compactification of the ordinary locus $X^{\ord}$.
 We consider the cohomology group $H^{\star}_{\et}\big(X^{\tor,\ord}_{\Fpb}, t_{!}(\scrL^{(\kb,w)}_{\gothl}|_{X^{\ord}})\big)=H^{\star}_{\et}\big(X^{\tor}_{\Fpb}, Rj_*t_{!}(\scrL^{(\kb,w)}_{\gothl}|_{X^{\ord}})\big)$. 
 Similar to the rigid case, it is  equipped with a natural action of the algebra $\scrH(K^p,L_{\gothl})[\Fr_{\gothp}, S_{\gothp},S_{\gothp}^{-1};\gothp\in \Sigma_{p}]$. 
  Using the cohomological purity for smooth pairs [SGA 4, XVI Th\'eor\`eme 3.3], one sees easily  that 
  \[
  R^bj_{*}t_{!}(\scrL^{(\kb,w)}|_{X^{\ord}})=
  \begin{cases}
 t_!(\scrL^{(\kb,w)}_{\gothl}) &\text{if } b=0,\\
 \bigoplus_{\#\ttT=b}\scrL^{(\kb,w)}_{\gothl}|_{Y_{\ttT}}(-b), &\text{if }b\geq 1.
  \end{cases}
  \]
  One deduces immediately a  spectral sequence 
  \begin{align*}
  E_2^{a,b}=H^{a}_{\et}(X^{\tor}_{\Fpb}, R^bj_*t_!(\scrL^{(\kb,w)}_{\gothl}))=\bigoplus_{\#\ttT=b} H^a_{c,\et}(Y_{\ttT,\Fpb}&,\scrL^{(\kb,w)}_{\gothl }|_{Y_{\ttT}})(-b)
  \\
  &\Longrightarrow H^{a+b}_{\et}(X^{\tor,\ord}_{\Fpb}, t_!(\scrL^{\kb,w}_{\gothl})),
  \end{align*}
which is $\scrH(K^p,L_{\gothl})[\Fr_{\gothp}, S_{\gothp},S_{\gothp}^{-1};\gothp\in \Sigma_{p}]$-equivariant if we define the actions of $\scrH(K^p,L_{\gothl})$,  $\Fr_{\gothp}$, and $S_{\gothp}$ on the left hand side in a similar way as its rigid analogue \eqref{E:spectral sequence for ordinary cohomology}. Then by Proposition~\ref{P:comparison-rig-et}, we have an equality in the Grothendieck group of finite-dimensional $\scrH(K^p,\C)[\Fr_{\gothp}, S_{\gothp},S_{\gothp}^{-1};\gothp\in \Sigma_{p}]$-modules:
\[
\sum_{n}(-1)^n\big[H^{n}_{\et}(X^{\tor,\ord}_{\Fpb}, t_!(\scrL^{(\kb,w)}_{\gothl}))\otimes_{L_{\gothl}}\overline\Q_l\big]=\sum_{n}(-1)^n\big[H^{n}_{\rig}(X^{\tor,\ord},\D; \F^{(\kb,w)})\otimes_{L_{\wp}}\overline\Q_{p}\big].
\]
As usual, we identify both $\overline\Q_l$ and $\Qpb$ with $\C$ via $\iota_{l}$ and $\iota_p$, respectively.
 \end{remark}

\section{Quaternionic Shimura Varieties and Goren-Oort Stratification}
\label{Section:GO-stratification}

As shown in Proposition~\ref{P:spectral-sequence}, the cohomology group $H^\star_\rig(X^{\tor, \ord}, \D; \scrF^{(\underline k, w)})$ can be computed by a spectral sequence consisting of rigid cohomology on the GO-strata of the (special fiber of) Hilbert modular variety; computing this is  further equivalent to computing the \'etale counterparts, as shown in Proposition~\ref{P:comparison-rig-et}.
The aim of this section is to compute the corresponding \'etale cohomology groups together with the actions of various operators.
The first step is to relate the \'etale cohomology of the GO-strata to the \'etale cohomology of analogous GO-strata of the Shimura variety for the group $G''_\emptyset =\GL_{2,F} \times_{F^\times} E^\times$ for certain CM extension $E$ of $F$ (Proposition~\ref{P:cohomology of Sh(G*E) vs Sh(G)}).
The next step is to apply the main theorem in the previous paper \cite{TX-GO} of this sequel which states that  each such a GO-stratum  is isomorphic to a $\PP^1$-power bundle over some other Shimura varieties (Theorem~\ref{T:GO-Hilbert}).
Combining these two, we arrive at an explicit description of those \'etale cohomology groups appearing in Proposition~\ref{P:comparison-rig-et} (which contains similar information as each term of the spectral sequence \eqref{E:spectral sequence for ordinary cohomology} does);
this is done in Propositions~\ref{T:total Frobenius action} and \ref{C:cohomology of GO-strata}.
One subtlety is that, in general, we do not have full information on the action of twisted partial Frobenius on these spaces (Conjecture~\ref{Conj: partial frobenius}).
This is why a complete description is only available when $p$ is inert.  Nonetheless, we can still prove our main theorem on classicality, as shown in the next section.

This section will start with a long digression on introducing quaternionic Shimura varieties and certain unitary-like Shimura varieties; the reason for this detour is that the description of the GO-strata does naturally live over the special fiber of Hilbert modular varieties but rather the unitary-like ones.

\subsection{Quaternionic Shimura variety}\label{S:QSV} We shall only recall the quaternionic Shimura varieties that we will need. For more details, see \cite[\S 3]{TX-GO}. Let $\ttS$ be an even subset of places of $F$. Put $\ttS_{\infty}=\ttS\cap \Sigma_{\infty}$.  We denote by $B_{\ttS}$ be the quaternionic algebra over $F$ ramified exactly at $\ttS$, and $G_{\ttS}=\Res_{F/\Q}(B_{\ttS}^{\times})$ be the associated $\Q$-algebraic group. We consider the Deligne homomorphism 
$$
h_{\ttS}: \C^{\times}\ra G_{\ttS}(\R)\simeq (\HH^{\times})^{\ttS_{\infty}}\times \GL_2(\R)^{\Sigma_{\infty}-\ttS_{\infty}}
$$ 
given by $h_{\ttS}(x+yi)=(z^{\tau}_{G_{\ttS}})_{\tau\in \Sigma_{\infty}}$ such that $z^{\tau}_{G_{\ttS}}=1$ for $\tau\in \ttS_{\infty}$ and $z_{G_{\ttS}}^\tau=\begin{pmatrix}x &y\\ -y &x\end{pmatrix}$ for $\tau\in \Sigma_{\infty}-\ttS_{\infty}$. Then the $G_{\ttS}(\R)$-conjugacy class of $h_{\ttS}$ is isomorphic to $\gothH_{\ttS}=(\gothh^{\pm})^{\Sigma_{\infty}-\ttS_{\infty}}$, where  $\gothh^\pm=\PP^1(\C)-\PP^1(\RR)$. For an open compact subgroup $K_{\ttS}\subset G_{\ttS}(\AAA^{\infty})$, we put  
\[
\Sh_{K_{\ttS}}(G_{\ttS})(\C)=G_{\ttS}(\Q)\backslash \gothH_{\ttS}\times G_{\ttS}(\AAA^{\infty})/K_\ttS.
\]
The Shimura variety $\Sh_{K_{\ttS}}(G_{\ttS})$ is defined over its reflex field $F_{\ttS}$, which is a subfield of the Galois closure of $F$ in $\C$. We have a natural action of the group $G_{\ttS}(\AAA^{\infty})$ on  $\Sh(G_{\ttS})=\varprojlim_{K_{\ttS}} \Sh_{K_{\ttS}}(G_{\ttS})$.  When $\ttS=\emptyset$,  this is the  Hilbert modular varieties $\Sh(G)$ considered in Section \ref{Section:HMV}. %In the sequel, we will omit $\emptyset$ from the notation when $\ttS=\emptyset$.

For each $\gothp$, we put $\ttS_{\infty/\gothp}=\Sigma_{\infty/\gothp}\cap \ttS$. In this paper, we will consider only $\ttS$ satisfying 
\begin{hypo}\label{H:subset-S}
We have $\ttS\subseteq \Sigma_{\infty}\cup \Sigma_p$. (Put $\ttS_p = \ttS \cap \Sigma_p$.)

Moreover, for a  $p$-adic place  $\gothp\in \Sigma_p$,
\begin{enumerate}
\item[(1)]  if $\gothp\in \ttS$, then  the degree $[F_{\gothp}:\Q_p]$ is odd and $\Sigma_{\infty/\gothp}\subseteq \ttS$;

\item[(2)] if  $\gothp\notin \ttS$, then $\ttS_{\infty/\gothp}$ has even cardinality.
\end{enumerate}
\end{hypo}
We fix an isomorphism $G_{\ttS}(\AAA^{\infty,p})\simeq \GL_2(\AAA^{\infty,p}_{F})$, so that the prime-to-$p$ component $K^p_{\ttS}$ will be viewed as an open subgroup of $\GL_2(\AAA^{\infty,p}_{F})$. In this paper, we will only encounter primes $\gothp\in\Sigma_{p} $ and open compact subgroups $K_{\ttS,\gothp}\subset B_{\ttS}^{\times}(F_{\gothp})$ of the following types:
\begin{itemize}
\item \textbf{Type $\alpha$ and $\alpha^\#$} The degree $[F_{\gothp}:\Q_p]$ is even, so $B_{\ttS}$ splits at $\gothp$ by Hypothesis~\ref{H:subset-S}. Fix an isomorphism $B^\times_{\ttS}(F_{\gothp})\simeq \GL_2(F_{\gothp})$. We will only consider  $K_{\ttS,\gothp}$ to be either $\GL_2(\cO_{F_{\gothp}})$ or the Iwahori subgroup $\Iw_{\gothp}\subset \GL_2(\cO_{F_\gothp})$ \eqref{E:Iwahori}, and the latter case may only happen when  $\Sigma_{\infty/\gothp}=\ttS_{\infty/\gothp}$.  We will say $\gothp$ is of \emph{type $\alpha$} if $K_{\ttS,\gothp} = \GL_2(\calO_{F_\gothp})$ is hyperspecial, and is of \emph{type $\alpha^\#$ if$K_{\ttS,\gothp}$} is Iwahori. Note that when  $\Sigma_{\infty/\gothp}\neq\ttS_{\infty/\gothp}$, $\gothp$ is necessarily of type $\alpha$, but when $\Sigma_{\infty/\gothp}=\ttS_{\infty/\gothp}$, both types are possible.

\item \textbf{Type $\beta$ and $\beta^\#$} The degree $[F_{\gothp}:\Q_p]$ is odd.  There are two cases:
\begin{itemize}
\item When $\ttS_{\infty/\gothp}\neq \Sigma_{\infty/\gothp}$,  $B_{\ttS}$ splits at $\gothp$ by Hypothesis~\ref{H:subset-S}. We fix an isomorphism $B^\times_{\ttS}(F_{\gothp})\simeq \GL_2(F_{\gothp})$. We consider only the case $K_{\ttS,\gothp}=\GL_2(\cO_{F_{\gothp}})$. We say $\gothp$ is of \emph{type $\beta$}.

\item When $\ttS_{\infty/\gothp}=\Sigma_{\infty/\gothp}$, $B_{\ttS}$ is ramified at $\gothp$. Then $B_{\ttS,\gothp}:=B_{\ttS}\otimes_F F_{\gothp}$ is the  quaternion division algebra over $F_{\gothp}$. Let $\cO_{B_{\ttS,\gothp}}$ be the maximal order of $B_{\ttS,\gothp}$. We will only allow $K_{\ttS,\gothp}=\cO_{B_{\ttS,\gothp}}^{\times}$. We say $\gothp$ is of \emph{type $\beta^\#$}.
\end{itemize}
\end{itemize}

Let $K_{\ttS}=K_{\ttS}^p\cdot \prod_{\gothp}K_{\ttS,\gothp}\subset G_{\ttS}(\AAA^\infty)$ be an open compact subgroup of the types considered above. The isomorphism $\iota_p:\C\simeq \Qpb$ determines a $p$-adic place $\wp$ of the reflex field $F_{\ttS}$. Let $\cO_\wp$ be the valuation ring of $F_{\ttS,\wp}$, and $k_\wp$ its residue field.  %Let $k_0$ be a finite field that contains all the residue fields of $\cO_F$ of characteristic $p$, $W(k_0)$ be the ring of Witt vectors. Then   $L_0=W(k_0)[1/p]$ can be viewed as a subfield of $\Qpb$, containing ${F_{\ttS,\scrP}}$. 

\begin{theorem}[{\cite[Corollary~3.18]{TX-GO}}]\label{T:integral-model}
For  $K_{\ttS}^p$ sufficiently small, there exists a smooth quasi-projective scheme   $\bfSh_{K_{\ttS}}(G_{\ttS})$ over $\cO_\wp$ satisfying the extension property in \cite[\S2.4]{TX-GO} such that 
 $$
 \bfSh_{K_{\ttS}}(G_{\ttS})\times_{\cO_\wp} F_{\ttS,\wp}\cong \Sh_{K_{\ttS}}(G_{\ttS})\times_{F_{\ttS}}F_{\ttS,\wp}.
 $$  
%When   $p\geq 3$,  the projective limit $\bfSh_{K_{\ttS,p}}(G_{\ttS}):=\varprojlim_{K_{\ttS}^p}\bfSh_{K^p_{\ttS}K_{\ttS,p}}(G_{\ttS})$  satisfies the following extension property: For a smooth $\cO_{F_{\ttS,\wp}}$-scheme $S$, any morphism $f: S\otimes_{\cO_{F_{\ttS,\wp}}}F_{\ttS,\wp}\ra \bfSh_{K_{\ttS,p}}(G_{\ttS})$ extends uniquely to a  morphism $\tilde f: S\ra \bfSh_{K_{\ttS,p}}(G_{\ttS})$.
 If $\ttS=\emptyset$, then $\bfSh_{K}(G)$ (we omit $\emptyset$ from the notation) is isomorphic to the integral model of the Hilbert modular variety considered in Subsection~\ref{Subsection:HMV}; if $\ttS\neq \emptyset$, $\bfSh_{K_\ttS}(G_{\ttS})$ is projective. 
 \end{theorem}

The construction of $\bfSh_{K_{\ttS}}(G_{\ttS})$  in \emph{loc. cit.} makes uses of an auxiliary choice of CM extension $E/F$ such that $B_\ttS \otimes_F E$ is isomorphic to $\rmM_2(E)$.  (We fix such an isomorphism from now on.)
When $p\geq 3$, the integral model $\bfSh_{K_{\ttS,p}}(G_{\ttS})$ satisfies certain extension property (see \cite[\S 2.4]{TX-GO}) and hence does not depend on the choice of such $E$.
The basic idea of the construction follows the method of ``mod\`eles \'etranges'' of Deligne \cite{deligne1} and  Carayol \cite{carayol}, but we allow certain $p$-adic places of $F$ to be inert in $E$ so that the construction may be used to describe Goren-Oort strata.

\subsection{Auxiliary CM extension}
\label{S:auxiliary CM}
Let $E/F$ be a CM extension
that is split over all $\gothp \in \Sigma_p$ of type $\alpha$ or $\alpha^\#$ and is inert over all $\gothp \in \Sigma_p$ of type $\beta$ or $\beta^\#$.
Denote by $\Sigma_{E,\infty}$ the set of archimedean embeddings of $E$, and $\Sigma_{E, \infty} \to \Sigma_\infty$ the natural two-to-one map given by restriction to $F$. 
Our construction depends on a choice of subset $\tilde \ttS_\infty$ consisting of, for each $\tau \in \ttS_\infty$, a choice of exactly one $\tilde \tau$ extending the archimedean embedding of $F$ to $E$.
This is equivalent to the collection of numbers $s_{\tilde \tau} \in \{0,1,2\}$  for each $\tilde \tau \in \Sigma_{E, \infty}$ such that
\begin{itemize}
\item
if $\tau \in \Sigma_\infty - \ttS_\infty$, we have $s_{\tilde \tau} = 1$ for all lifts $\tilde \tau $ of $\tau$;
\item
if $\tau \in \ttS_\infty$ and $\tilde \tau$ is the lift in $\tilde \ttS_\infty$, we have $s_{\tilde \tau} = 0$ and $s_{\tilde \tau^c} = 2$, where $\tilde \tau^c$ is the  conjugate of $\tilde \tau$.
\end{itemize}

Put $T_{E, \tilde \ttS} = \Res_{E/\Q}\GG_m$, where the subscript $\tilde \ttS = (\ttS, \tilde \ttS_\infty)$ indicates our choice of the Deligne homomorphism:
\[
\xymatrix@R=0pt@C=50pt{
h_{E, \tilde \ttS}\colon
\SSS(\RR) = \CC^\times \ar[r] &
T_{E, \tilde \ttS}(\RR) = \bigoplus_{\tau \in \Sigma_\infty} (E \otimes_{F, \tau}\RR)^\times \simeq \bigoplus_{\tau\in\Sigma_\infty} \CC^\times\\
z\ar@{|->}[r] & (z_{E, \tau})_\tau.
}
\]
Here $z_{E, \tau} = 1$ if $\tau \notin \ttS_\infty$ and $z_{E, \tau} = z$ if $\tau \in \ttS_\infty$, in which case,  the isomorphism $(E \otimes_{F, \tau}\RR)^\times \simeq \CC^\times$ is given by the complex conjugation $\tilde \tau^c$ of the lift $\tilde \tau \in \tilde \ttS_\infty$.
The reflex field $E_{\tilde \ttS}$ is the field of definition of the  Hodge cocharacter 
\[
\mu_{E,\tilde\ttS}\colon \G_{m,\C}\xra{z\mapsto(z,1)} \SSS(\C)\cong \C^{\times}\times \C^{\times}\xra{h_{E,\tilde\ttS}} T_{E, \tilde \ttS}(\C),
\]
where the first copy $\C^{\times}$ in $\SSS(\C)$ is given by the identity character of $\C^{\times}$, and the second by complex conjugation.
  More precisely, $E_{\tilde \ttS}$ is
 the subfield of $\CC$ corresponding to the subgroup of $\Aut(\C/\Q)$ which stabilizes the set $\tilde \ttS_\infty$. It contains $F_\ttS$ as a subfield.
The isomorphism $\iota_p: \CC \simeq \overline \QQ_p$ determines a $p$-adic place $\tilde \wp$ of $E_{\tilde \ttS}$. We use $\calO_{\tilde \wp}$ to denote the valuation ring and $k_{\tilde \wp}$ the residue field.%  We point out that $[k_{\tilde \wp}: \FF_p]$ is always even.

We take the level structure $K_E$ to be $K_E^pK_{E,p}$, 
where $K_{E, p} = (\calO_E \otimes_\ZZ \ZZ_p)^\times$, and $K_E^p$ is an open compact subgroup of $\AAA_E^{\infty,p,\times}$.
This then gives rise to a Shimura variety $\Sh_{K_E}(T_{E, \tilde \ttS})$ and its limit $\Sh_{K_{E,p}}(T_{E, \tilde \ttS}) = \varprojlim_{K_E^p}\Sh_{K_{E,p}K_E^p}(T_{E, \tilde \ttS})$. They have obvious integral models $\bfSh_{K_E}(T_{E, \tilde \ttS})$ and $\bfSh_{K_{E,p}}(T_{E, \tilde \ttS})$ over $\calO_{\tilde \wp}$.
The set of $\CC$-points of the limit is given by
\[
\bfSh_{K_{E,p}}(T_{E,\tilde\ttS})(\C) =T_{E, \tilde \ttS}(\Q)^\cl\backslash T_{E, \tilde \ttS}(\AAA^{\infty})/K_{E,p}=\AAA^{\infty, p,\times}_{E}/\calO_{E, (p)}^{\times, \cl},
\]
where the superscript cl denotes the closure in the appropriate topological groups.
The geometric Frobenius $\Frob_{\tilde \wp}$ in the Galois group $\Gal_{k_{\tilde \wp}} = \Gal(\calO_{\tilde\wp}^\ur/\calO_{\tilde\wp})$ acts on $\bfSh_{K_{E,p}}(T_{E,\tilde\ttS})$ by multiplication by the image of local uniformizer at $\tilde \wp$ in the ideles of $E_{\tilde \ttS}$ under the following reciprocity map
\begin{align*}
\gothRec_E\colon
\AAA^{\infty,\times}_{E_{\tilde \ttS}}/E_{\tilde \ttS}^{\times} \times \calO_{\tilde \wp}^\times&\xra{\mu_{E,\ttS}} 
T_{E, \tilde \ttS}(\AAA^{\infty}_{E_{\tilde \ttS}})/T_{E, \tilde \ttS}(E_{\tilde \ttS})T_{E, \tilde \ttS}(\calO_{\tilde \wp})\\
&\xra{N_{E_{\tilde \ttS}/\Q}} T_{E,\tilde \ttS}(\AAA^{\infty})/T_{E,\tilde \ttS}(\Q)
T_{E, \tilde \ttS}(\Zp)
\xrightarrow{\cong} \AAA^{\infty, p,\times}_{E}/\calO_{E, (p)}^{\times, \cl}.
\end{align*}

For later use, we record the notation on the action of $\sigma_\gothp$ for $\gothp \in \Sigma_p$ on $\Sigma_{E, \infty}$: for $\tilde \tau \in\Sigma_{E, \infty}$, we put
\begin{equation}\label{E:defn-sigma-gothp}
\sigma_{\gothp}\tilde\tau=\begin{cases}\sigma \circ\tilde\tau & \text{if }\tilde \tau\in \Sigma_{E,\infty/\gothp},\\
\tilde\tau &\text{if }\tilde\tau\notin \Sigma_{E,\infty/\gothp},
\end{cases}
\end{equation}
where $\Sigma_{E,\infty/\gothp}$ denotes the lifts of places in $\Sigma_{\infty/\gothp}$.  This $\sigma_\gothp$-action is compatible with the one on $\Sigma_\infty$ given in Subsection~\ref{S:Frobenius-HMV}.
Set $\sigma_\gothp\tilde \ttS = (\ttS_{p}\cup \sigma_{\gothp}\ttS_{\infty},\sigma_\gothp \tilde \ttS_\infty)$ with $\tilde \ttS_\infty$ the image of $\sigma_\gothp\tilde \ttS_\infty$ under $\sigma_{\gothp}$.
The product $\sigma = \prod_{\gothp \in \Sigma_p} \sigma_\gothp
$ is the usual Frobenius action,  and we define $\sigma \tilde\ttS$ similarly.

\subsection{Auxiliary Shimura varieties}
\label{S:auxiliary-Sh-var}
We also consider the product group $ G_\ttS \times T_{E, \tilde \ttS}$ with the product of Deligne homomorphism
\[
\tilde h_{\tilde \ttS} =h_\ttS \times h_{E, \tilde \ttS} \colon \SSS(\R) = \CC^\times \longrightarrow (G_\ttS\times T_{E, \tilde \ttS})(\RR).
\]
This gives rise to the product Shimura variety:
\[
\bfSh_{K_{\ttS,p} \times K_{E,p}}(G_\ttS \times T_{E,\tilde \ttS}) = \bfSh_{K_{\ttS,p}}(G_\ttS) \times_{\calO_{\wp}}
\bfSh_{K_{E,p}}(T_{E, \tilde \ttS}) .
\]

Let $Z = \Res_{F/\QQ}\GG_m$ denote the center of $G_\ttS$.
Put $G''_{\tilde \ttS} = G_\ttS \times_Z T_{E, \tilde \ttS}$ which is the quotient of $G_\ttS \times T_{E, \tilde \ttS}$ by $Z$ embedded anti-diagonally as $z \mapsto (z, z^{-1})$.
We consider the homomorphism $h''_{\tilde \ttS}: \SSS(\RR) \to G''_{\tilde \ttS}(\R)$ induced by $\tilde h_{\tilde \ttS}$.
We will consider open compact subgroups $K''_{\ttS} \subset G''_{\tilde \ttS}(\AAA^\infty)$ of the form $K''^p_{\ttS}K''_{\ttS,p}$, where $K''^p_\ttS$ is an open compact subgroup of $G''_{\tilde \ttS}(\AAA^{\infty,p})$
and $K''_{\ttS,p}$ is the image of $K_{\ttS,p}\times K_{E,p}$ under  the natural projection $G_\ttS(\Q_p) \times T_{E, \tilde \ttS}(\Qp)\ra G''_{\tilde\ttS}(\Q_p)$.
(Note here that the level structure at $p$ only depends on $\ttS$ but not its lift $\tilde \ttS$; this is why we suppress the tilde from the notation.) 
For $K''^p_{\ttS}\subset G''_{\tilde \ttS}(\AAA^{\infty,p})$ sufficiently small, the corresponding Shimura variety admits a smooth canonical  integral model $\bfSh_{K''_{\ttS}}(G''_{\tilde\ttS})$ over $\cO_{\tilde \wp}$ \cite[Corollary 3.19]{TX-GO}. Taking the limit over prime-to-$p$ levels, we get $\bfSh_{K''_{\ttS,p}}(G''_{\tilde \ttS}) = \varprojlim_{K''^p_\ttS}\bfSh_{K''^p_{\ttS}K''_{\ttS,p}}(G''_{\tilde\ttS})$.
    
For $\gothp\in \Sigma_p$, we use $S_\gothp$ denote the Hecke action on $\bfSh_{K''_{\ttS,p}}(G''_{\tilde \ttS})$
 given by multiplication by $(1, \varpi_\gothp^{-1})$, where $\varpi_\gothp$ is the uniformizer of $\calO_{F, \gothp}$ embedded in $(\calO_E \otimes_{\calO_F} \calO_\gothp)^\times \subset T_{E, \tilde \ttS}(\AAA^\infty)$.
 
Let $ \alpha: G_{\ttS}\times T_{E, \tilde \ttS}\ra G''_{\tilde\ttS}$ to denote the natural projection. The homomorphisms of algebraic groups induces a diagram  of (projective systems of) Shimura varieties:
   \begin{equation}\label{E:diag-Sh-Var}
   \xymatrix{\bfSh_{K_{\ttS,p}}(G_{\ttS}) &\bfSh_{K_{\ttS,p}\times K_{E,p}}(G_\ttS \times T_{E,\tilde\ttS})\ar[l]_-{\pr_1}\ar[d]^-{\pr_2} \ar[r]^-{\bbalpha}  &\bfSh_{K''_{\ttS,p}}(G''_{\tilde\ttS}) \\
   &\bfSh_{K_{E,p}}(T_{E,\tilde\ttS})
}
\end{equation}
Note that the corresponding Deligne homomorphisms are compatible for all morphisms of the algebraic groups.

\subsection{Automorphic sheaves on Shimura varieties}\label{S:quaternion-automorphic-bundle}

Let $L$ be the number  field as in Subsection~\ref{S:HMF-notation}. 
%We fix an isomorphism $\iota_l: \C\cong \overline \QQ_l$ so that $\Sigma_{\infty}$ are identified with the $l$-adic embeddings of $F$ into $\Qlb$.
Note that
$
G_\ttS \times_\QQ  L = \prod_{\tau \in \Sigma_\infty} \GL_{2,L}
$, where $F^\times$ acts on the $\tau$-component through $\tau$.
Given a  multiweight $(\underline{k}, w)$, we  consider
 the following algebraic representation of $G_\ttS \times_{\QQ} L$:
\[
\rho_\ttS^{(\underline{k}, w)} = \bigotimes_{\tau \in \Sigma_\infty} \rho_\tau^{(k_\tau, w)} \circ \check\pr_\tau \quad \textrm{with} \quad
\rho_\tau^{(k_\tau, w)} =
\Sym^{k_\tau-2} \otimes
\det{}^{\frac{w-k_\tau}{2}}
,
\]
where  $\check{\pr}_\tau$ is the contragredient of the natural projection to the $\tau$-component of $G_\ttS \times_{\QQ} L$.
%, and $\rho_\tau^{(k_\tau, w)}$ is the corresponding representation on the $\tau$-tensorant of $\rho^{(\underline{k}, w)}$.
The representation $\rho^{(\underline k, w)}$ factors through the quotient group $G^c_\ttS = G_\ttS / \Ker(\Res_{F/\QQ}\GG_m \to \GG_m)$.
By \cite[Ch.~III,~\S~7]{milne}, the representation $\rho^{(\underline k, w)}$ gives rise to an $\Qlb$-lisse sheaf $\scrL^{(\kb,w)}_{\ttS,l}$ on $\Sh_{K_{\ttS}}(G_{\ttS})$ which naturally extends to its integral model $\bfSh_{K_{\ttS}}(G_\ttS)$. 

For the $l$-adic local systems on $\bfSh_{K''_{\ttS}}(G''_{\tilde \ttS})$, we need to fix a section of the natural map $\Sigma_{E, \infty} \to \Sigma_\infty$, that is to fix a lift $\tilde \tau$ for each $\tau \in \Sigma_\infty$. We use $\widetilde \Sigma = \widetilde \Sigma_{\infty}$ to denote the image of the section so that $\Sigma_{E, \infty} = \widetilde \Sigma \coprod \widetilde \Sigma^c$. (The choice of $\widetilde \Sigma$ is independent of the choice of $\tilde \ttS_\infty$. Our main result is insensitive to the choice of $\widetilde \Sigma$. Any choice works.)

Consider the injection
\[
G''_{\tilde\ttS} \times_\QQ L =
(B_\ttS^\times \times_{F^\times} E^\times) \times_\QQ L \hookrightarrow
\Res_{E/\QQ}(B_{\ttS}\otimes_{F}E)^{\times} \times_\QQ L
 \cong
\prod_{\tilde \tau \in \widetilde \Sigma}
\GL_{2, L, \tilde \tau} \times
\GL_{2, L, \tilde \tau^c},
\]
where $E^\times$ acts on $\GL_{2, L, \tilde \tau}$ (resp. $\GL_{2, L, \tilde \tau^c}$) through $\tilde \tau$ (resp. complex conjugate of $\tilde \tau$).
For a multiweight $(\underline k, w)$, we consider the following representation of $G''_{\tilde \ttS} \times_\QQ L$:
\[
\rho''^{(\underline{k}, w)}_{\tilde \ttS, \widetilde \Sigma} = \bigotimes_{\tilde \tau \in \widetilde \Sigma} \rho_\tau^{(k_\tau, w)} \circ \check\pr_{\tilde \tau} \quad  \textrm{ with} \quad
\rho_\tau^{(k_\tau, w)} =
\Sym^{k_\tau-2} \otimes
\det{}^{\frac{w-k_\tau}{2}}
,
\]
where $\check\pr_{\tilde \tau}$  is the contragredient of the natural projection to the $\tilde \tau$-component of $G''_{\tilde \ttS} \times_\QQ L \hookrightarrow \Res_{E/\QQ} D_\ttS^\times \times_\QQ L$. By \cite[Ch.~III,~\S~7]{milne}, the representation $\rho''^{(\kb,w)}_{\tilde \ttS,\tSigma}$ gives rise to 
 an $\Qlb$-lisse sheaf  $\scrL''^{(\underline{k}, w)}_{\tilde \ttS, \tSigma,l}$ on $\bfSh_{K''_{\ttS,p}}(G''_{\tilde \ttS})$.

We also consider the following one-dimensional representation of  $\Res_{E/\QQ}\GG_m\times_\QQ L \cong \prod_{\tilde \tau\in \widetilde \Sigma} \GG_{m, \tilde \tau} \times \GG_{m, \tilde \tau^c}$:
\[
\rho_{E,  \tilde \Sigma}^w =  \bigotimes_{\tilde \tau \in \widetilde \Sigma}
x^{2-w} \circ {\pr}_{E, \tilde \tau},
\]
where $ \pr_{E, \tilde\tau}$ is the projection to the $\tilde \tau$-component and $x^{2-w}$ is the character of $\CC^\times$ given by raising to $(2-w)$th power.
These representations give rise to a lisse $\Qlb$-sheaf $\scrL_{E, \widetilde \Sigma,l}^w$ on $\bfSh_{K_{E,p}}(T_{E,\tilde\ttS})$.

By definition, there is a natural isomorphism of representations of $G_\ttS \times T_{E, \tilde \ttS}$
\[
\rho''^{(\underline k, w)}_{\tilde \ttS, \widetilde \Sigma} \circ \alpha \cong \big( \rho^{(\underline k, w)}_{\ttS} \circ \pr_1 \big) \otimes  \big( \rho^w_{E, \widetilde \Sigma} \circ \pr_2 \big) \qquad (\textrm{for any }\widetilde \Sigma),
\]
and hence a natural isomorphism of lisse $\Qlb$-sheaves  on $\bfSh_{K_{\ttS,p}\times K_{E,p}}(G_\ttS \times T_{E,\tilde\ttS})$:
\begin{equation}
\label{E:comparing auto sheaves from unit to quad}
\tL^{(\kb,w)}_{\tilde \ttS,\tSigma,l}
:=\bbalpha^* \scrL''^{(\underline k, w)}_{\tilde \ttS, \widetilde \Sigma,l} \cong \pr_1^* \scrL^{(\underline k, w)}_{\ttS,l} \otimes \pr_2^* \scrL^w_{E, \tSigma,l} \qquad (\textrm{for  any }\widetilde \Sigma).
\end{equation}

\begin{remark}
The lisse $\Qlb$-sheaves $\scrL_{\ttS,l}^{(\kb,w)}$,  $\scrL''^{(\kb,w)}_{\tilde\ttS,\tilde\Sigma, l}$, $\scrL^w_{E, \tSigma,l}$,  and $\widetilde\scrL^{(\kb,w)}_{\tilde \ttS,\tSigma,l}$ are base change of $L_{\gothl}$-sheaves on the corresponding Shimura varieties, where $\gothl$ is the $l$-adic place of $L$ determined by the isomorphism $\iota_l: \CC \simeq \overline \QQ_l$. For the computation of cohomology in terms of automorphic forms, it is more convenient to work with $\Qlb$-coefficients.
\end{remark}

\subsection{Family of Abelian varieties}
\label{S:family of AV over Sh}
We summarize the basic properties of certain abelian varieties over $\bfSh_{K''_{\ttS,p}}(G''_{\tilde\ttS})$ constructed in \cite{TX-GO}.

\begin{itemize}

\item[(1)] \cite[\S 3.20]{TX-GO}
There is a  natural family of  abelian varieties $\bfA'' = \bfA''_{\tilde \ttS}$ of dimension $4g$ over $\bfSh_{K''_{\ttS,p}}(G''_{\tilde\ttS})$ equipped with a natural action of $ \rmM_2(\cO_E)$ and satisfying certain Kottwitz's determinant condition.
There is a (commutative) equivariant action of $G''_{\tilde\ttS}(\AAA^{\infty,p})$  on $\bfA''$ so that for sufficiently small $K''^p_{\ttS} \subset G''_{\tilde\ttS}(\AAA^{\infty,p})$, the abelian scheme $\bfA''$ descends to $\bfSh_{K''^p_{\ttS}K''_{\ttS,p}}(G''_{\tilde\ttS})$.
%{\color{blue}[I am keen to kill the rest of this property and (2); we are not going to use these properties later....]} such that for each $b\in \cO_E$, the characteristic polynomial of the action of $\tilde\iota(b)$ on $\Lie(\bfA'')$ is given by 
 %\[
 %\mathrm{det}_{
%\cO_{\bfSh_{K''_{\ttS,p}}(G''_{\tilde\ttS})}}(T-\tilde\tau(b)|\Lie(\tilde\bfA))=\prod_{\tilde\tau\in \Sigma_{E,\infty}}(T-\tilde\tau(b))^{2s_{\tilde\tau}}.
 %\]
 
 %\item The abelian scheme $\bfA''$ is equipped with a polarization ${\lambda}: \bfA''\ra \bfA''^\vee$ compatible with the action of $\rmM_2(\cO_E)$ such that the $p$-torsion $\Ker({\lambda})[p^{\infty}]$ is a  finite flat closed subgroup scheme contained in $\prod_{\gothp \textrm{ of type } \beta^\#}A[\gothp]$ of rank $\prod_{\gothp \textrm{ of type } \beta^\#} (\#k_{\gothp})^{4}$.

\item[(2)] \cite[\S 3.21]{TX-GO} Let $a'': \bfA''\ra \bfSh_{K_{\ttS}''}(G''_{\tilde\ttS})$ be the structural morphism. The direct image $R^1a''_*(\Qlb)$ has a canonical decomposition:
\[
R^1a''_*(\Qlb)= \bigoplus_{\tilde \tau\in \widetilde \Sigma}\big(R^1a''_{*}(\Qlb)_{\tilde\tau}\oplus R^1a''_{*}(\Qlb)_{\tilde\tau^c}\big),
%=\bigoplus_{\tau\in \Sigma_{\infty}}\big(R^1a''_{*}(L_{\gothl})_{\tilde\tau}^{\circ,\oplus 2}\oplus R^1a''_{*}(L_{\gothl})_{\tilde\tau^c}^{\circ,\oplus 2}\big)
\]
where $R^1a''_{*}(\Qlb)_{\tilde\tau}$ (resp. $R^1a''_{*}(\Qlb)_{\tilde\tau^c}$) is the direct summand where $\cO_E$ acts via $\tilde\tau: \calO_E \to \overline \QQ_l$ (resp. via $\tilde\tau^c: \calO_E \to \overline \QQ_l$). Let $\gothe = \big(\begin{smallmatrix}
1&0\\0&0
\end{smallmatrix}\big) \in \rmM_2(\calO_E)$. We put  $R^1a''_{*}(\overline \QQ_l)_{\tilde\tau}^{\circ}=\gothe  R^1a''_{*}(\Qlb)_{\tilde\tau}$ for $\tilde \tau\in  \widetilde \Sigma$. This is  a $\Qlb$-lisse sheaf over $\bfSh_{K''_{\ttS,p}}(G''_{\tilde \ttS})$ of rank $2$.
For a multiweight $(\underline k, w)$, there is a natural isomorphism of lisse $\Qlb$-sheaves:
\[
\scrL''^{(\underline k, w)}_{\tilde \ttS, \widetilde \Sigma, l} \cong \bigotimes_{\tilde \tau \in \widetilde \Sigma} \bigg(\Sym^{k_\tau -2} R^1a''_{*}(\Qlb)_{\tilde\tau}^{\circ} \otimes
(\wedge^2
R^1a''_{*}(\Qlb)_{\tilde\tau}^{\circ})^{\frac{w-k_\tau}{2}}
\bigg).
\]

\item[(3)]  \cite[Proposition~3.23]{TX-GO}
For each $\gothp \in \Sigma_p$, we have a natural $G''_{\tilde \ttS}(\AAA^{\infty, p})$-equivariant \emph{twisted partial Frobenius morphism} and an quasi-isogeny of family of abelian varieties:
\begin{equation}
\label{E:partial Frobenius ''}
\gothF''_{\gothp^2}\colon
\bfSh_{K''_{\ttS,p}}(G''_{\tilde \ttS})_{k_{\tilde \wp}} \longrightarrow
\bfSh_{K''_{\ttS,p}}(G''_{\sigma_\gothp^2\tilde \ttS})_{k_{\tilde \wp}} \quad \textrm{and} \quad
\eta''_{\gothp^2}
\colon
\bfA''_{\tilde \ttS, k_{\tilde \wp}}
\longrightarrow \gothF''^{*}_{\gothp^2}(\bfA''_{\sigma^2_{\gothp}\tilde \ttS, k_{\tilde \wp}}).
\end{equation}
Here the level structures at $p$: $K''_{\sigma_\gothp^2\ttS, p}$ and $K''_{\ttS, p}$ are equal by definition.
The two morphisms in \eqref{E:partial Frobenius ''} induce a natural $G''_{\tilde \ttS}(\AAA^{\infty, p})$-equivariant homomorphism of \' etale cohomology groups:
\[
\Phi_{\gothp^2}\colon
H^*_\et\big(
\bfSh_{K''_{\ttS,p}}(G''_{\sigma_\gothp^2\tilde \ttS})_{\overline \FF_p}, \scrL''^{(\underline k, w)}_{\sigma_\gothp^2\tilde \ttS, \widetilde \Sigma, l}
\big) \longrightarrow
H^*_\et\big(
\bfSh_{K''_{\ttS,p}}(G''_{\tilde \ttS})_{\overline \FF_p}, \scrL''^{(\underline k, w)}_{\tilde \ttS, \widetilde \Sigma, l}\big),
\]
where $\sigma_\gothp$ is as defined at the end of Subsection~\ref{S:auxiliary CM}.
Moreover, we have an equality of morphisms
\[
\prod_{\gothp \in \Sigma_p} \Phi_{\gothp^2} = S_p^{-1} \circ F^2
\colon
H^*_\et\big(
\bfSh_{K''_{\ttS,p}}(G''_{\sigma^2\tilde \ttS})_{\overline \FF_p}, \scrL''^{(\underline k, w)}_{\sigma^2\tilde \ttS, \widetilde \Sigma,l}\big) \longrightarrow
H^*_\et\big(
\bfSh_{K''_{\ttS,p}}(G''_{\tilde \ttS})_{\overline \FF_p}, \scrL''^{(\underline k, w)}_{\tilde \ttS, \widetilde \Sigma,l}\big),
\]
where $F^2$ is  the relative $p^2$-Frobenius, $\sigma$ is as defined at the end of Subsection~\ref{S:auxiliary CM}, and  $S_p$ is the Hecke action given by the central element $\underline p^{-1}\in \AAA^{\infty,\times}_{E}\subset G''(\AAA^{\infty})$. Here, $\underline p$ is the idele which is $p$ at all $p$-adic places of $E$ and $1$ at all other places.

\item[(4)] When $\ttS=\emptyset$, let $\calA$ denote the universal abelian variety over the Hilbert modular variety $\bfSh_{K_{p}}(G)$. One has an isomorphism of abelian schemes over $\bfSh_{K_{\ttS,p}\times K_{E,p}}(G_\ttS \times T_{E,\tilde\ttS})$ \cite[Corollary~3.26]{TX-GO}:
  \begin{equation}\label{E:AV for HMV and unitary}
  \bbalpha^*\bfA''\cong  (\pr_1^*\calA\otimes_{\cO_F}\cO_E)^{\oplus2},
  \end{equation}
  compatible with $\rmM_{2}(\cO_E)$-actions, where $\bbalpha$ is defined in \eqref{E:diag-Sh-Var}.
Moreover, the morphism $\bbalpha$ and the isomorphism \eqref{E:AV for HMV and unitary} are compatible with the action of twisted partial Frobenius \cite[Corollary~3.28]{TX-GO}.

\item[(5)]
\cite[4.6, 4.7, 4.9]{TX-GO}
Let $k_0$ be a finite extension of $\FF_p$ containing all residue fields of $E$ of characteristic $p$.
The special fiber $\bfSh_{K''_{\emptyset, p}}(G''_\emptyset)_{k_0}$ admits a GO-stratification, that is a \emph{smooth} $G''_\emptyset(\AAA^{\infty,p})$-stable subvariety $\bfSh_{K''_{\emptyset,p}}(G''_\emptyset)_{k_0, \ttT}$ for each subset $\ttT \subseteq \Sigma_\infty$ (given by the vanishing locus of certain variants of partial Hasse invariants).  
We refer to {\it loc. cit.} for the precise definition.
The twisted partial Frobenius morphism $\gothF''_{\gothp^2}$ in \eqref{E:partial Frobenius ''} takes the subvariety $\bfSh_{K''_{\ttS,p}}(G''_{\tilde \ttS})_{k_0, \ttT}$ to $\bfSh_{K''_{\ttS,p}}(G''_{\sigma_\gothp^2\tilde \ttS})_{k_0, \sigma_\gothp^2 \ttT}$. Here, we view $K''_{\ttS}$ also as a subgroup of $G''_{\sigma_\gothp^2\tilde \ttS}(\AAA^{\infty,p})$ via a fixed isomorphism $G''_{\tilde\ttS}(\AAA^{\infty,p})\simeq G''_{\sigma_\gothp^2\tilde \ttS}(\AAA^{\infty,p})$.

The GO-stratification on $\bfSh_{K''_{\emptyset,p}}(G''_{\emptyset})_{k_0}$ is compatible with the Goren-Oort stratification on the Hilbert modular variety $\bfSh_{K_p}(G)_{k_0}$ recalled in Subsection~\ref{Subsection:Hasse} in the sense that
\begin{equation}\label{E:compatibility-GO}
\bbalpha^{-1}(\bfSh_{K''_p}(G''_\emptyset)_{k_0, \ttT}) \cong \pr_1^{-1}(
\bfSh_{K_p}(G)_{k_0, \ttT}) \quad\text{for all }\ttT\subseteq \Sigma_{\infty}.
\end{equation}
Moreover, this isomorphism is compatible with the action of twisted Frobenius.

  \end{itemize}

\subsection{Tensorial induced representations}
\label{S:tensorial-induction}
We recall the definition of tensorial induced representations.
Let $G$ be a group and $H$ a subgroup of finite index.
Let $(\rho,V)$ be a finite dimensional representation of $H$.
Let $\Sigma \subseteq G/H$ be a finite subset.
Consider the left action of $G$ on the set of left cosets $G/H$.
Let $H'$ be the subgroup of $G$ that fixes the subset $\Sigma$ of $G/H$.
Fix representatives $s_1, \dots, s_n$ of $G/H$ and we may assume that $\Sigma = \{s_1H, \dots, s_rH\}$ for some $r$.

The \emph{tensorial induced representation}, denoted by $\otimes_\Sigma\textrm-\Ind_H^G V$, is defined to be $\otimes_{i=1}^r V_i$, where $V_i$ is a copy of $V$.
The action of $H'$ is given as follows:
for a given $h' \in H'$ and for each $i \in \{1, \dots, r\}$, there exists a unique $j(i) \in \{1, \dots, r\}$ such that $h' s_{j(i)} \in s_i H$;
then we define
\[
h'(v_1 \otimes \cdots \otimes v_r)
= \rho(s_1^{-1} h' s_{j(1)})(v_{j(1)}) \otimes \cdots \otimes \rho(s_r^{-1} h' s_{j(r)})(v_{j(r)}).
\]
One can easily check that this definition of $\otimes_\Sigma\textrm-\Ind_H^G V$
does not depend on the choice of coset representatives.

Typical examples of the above construction we will need later are: (1)  $G=\Gal_{\Q}$, $H=\Gal_{F}$, and $\Sigma=\Sigma_{\infty}-\ttS_{\infty}\subseteq \Sigma_{\infty} \simeq G/H $. Then the subgroup $H'$ is $\Gal_{F_{\ttS}}$;
(2) $G=\Gal_{\Q}$, $H=\Gal_{E}$, and $\Sigma = \tilde \ttS_\infty$ (see Subsection~\ref{S:family of AV over Sh}), viewed as a subset of $\Sigma_{E,\infty} \simeq G/H$. Then the subgroup $H'$ contains $\Gal_{E_{\ttS}}$.

\subsection{Automorphic representations of $\GL_{2,F}$}

For a multiweight $(\kb,w)$, let $\scrA_{(\kb,w)}$ denote the set of irreducible cuspidal automorphic representations $\pi$ of $\GL_{2}(\AAA_{F})$ whose archimedean component $\pi_{\tau}$ for each $\tau\in\Sigma_\infty$ is a discrete series of weight $k_{\tau}-2$ with central character $x\mapsto x^{w-2}$.
It follows that the central character $\omega_\pi: \AAA_F^\times / F^\times \to \CC^\times$ for such $\pi$ can be written as 
$\omega_\pi = \varepsilon_\pi |\cdot|_F^{w-2}$, with $\varepsilon_\pi$ a finite character Hecke character  trivial on $(F\otimes\RR)^{\times}$.

Given $\pi \in \scrA_{(\kb,w)}$, the finite part $\pi^{\infty}$ of $\pi$ can be defined over a number field (viewed as a subfield of $\C$).
%In the sequel, we will suppose the number field $L$ sufficiently large so that all $\pi^{\infty}$ with $(\pi^{\infty})^K\neq 0$ is defined over $L$.
For an even subset $\ttS \subseteq \Sigma$, we use $\pi_\ttS$ to denote the Jacquet-Langlands transfer of $\pi$  to an automorphic representation over $B_\ttS^\times(\AAA_F)$. It is zero if $\pi$ does not transfer. 

%Let $\pi\in \scrA^{(\kb,w)}$ defined over a number field $L$, and $\gothl$ be an $l$-adic place of $L$.
 Thanks to the work of many people \cite{carayol-hilbert, taylor, blasius-rogawski}, we can associate to $\pi$ a 2-dimensional  Galois representation $\rho_{\pi, l}:\Gal_F\ra \GL_2(\Qlb)$. % for either a $p$-adic or an $l$-adic place $\lambda$ of $L$.
We normalize $\rho_{\pi,l}$ so that $\det(\rho_{\pi,l})=\varepsilon_{\pi}^{-1}\cdot \chi_{\mathrm{cyc}}^{1-w}$, where $\chi_{\mathrm{cyc}}$ is the $l$-adic  cyclotomic character and $\varepsilon_\pi$ is the finite character above, viewed as a Galois representation  with values in $\Qlb$ via class field theory and the isomorphism $\iota_l: \C\simeq \Qlb$. If $\pi^{\infty}$ is defined over a number field $L\subseteq \C$, then $\rho_{\pi,l}$ is rational over  $L_{\gothl}$, where $\gothl$ denotes the $l$-adic place of $L$ determined by $\iota_l$.

When $k_{\tau}=2$ for all $\tau$, $w\geq 2$ is an even integer.  We denote by $\scrB_{w}$ the set of irreducible  automorphic representations $\pi$ of $\GL_2(\AAA_F)$ which factor  as
$$
 \GL_2(\AAA_F)\xra{\det} \AAA_{F}^{\times}/F^{\times}\xra{\chi} \C^{\times},
 $$
where $\chi$ is an algebraic  Hecke character whose restriction to $F_\tau^{+}$ for each $\tau \in \Sigma_\infty$ is $x\mapsto x^{\frac{w}{2}-1}$. They are the one-dimensional automorphic representations.
  % We suppose that $L$ contains the image of the map $x\mapsto \chi(x)N_{F/\QQ}(x_\infty)^{1-\frac w2}$ for all $\chi$ which is trivial on  $K$. 
  With the fixed  isomorphism $\iota_{l}:\C\simeq \overline \Q_{l}$, we define an $l$-adic character on $\AAA_{F}^{\times}$ given by
\[
\chi_{l}: x \mapsto \big(
\chi(x) \cdot N_{F/\QQ}(x_\infty)^{1-\frac w2} \big)
\cdot 
N_{F/\QQ}(x_l)^{\frac w2-1}\in \Qlb^{\times},
\]
where $x_\infty \in (F\otimes \RR)^\times$ (resp. $x_l \in (F \otimes \QQ_l)^\times$) denotes the archimedean components (resp. $l$-components of $x$). Note also that $\chi_l$ is trivial on $F^\times$, and hence by class field theory gives  rise to a $l$-adic Galois character  on $\Gal_F$.
 We put $\rho_{\pi,l}=\chi_l^{-1}$.
 Also note that the map $x\mapsto \chi(x)N_{F/\QQ}(x_\infty)^{1-\frac w2}$ on $\AAA^{\times}_F$ has values  in a number field. We may choose a number field $L\subseteq \C$ large enough so that  $\rho_{\pi,l}$ is rational over  $L_{\gothl}$.
 Given $\pi\in \scrB_{w}$ and an even subset $\ttS$, we denote by $\pi_{\ttS}$  the one-dimensional automorphic representation of $G_{\ttS}$ that factors as $G_{\ttS}(\AAA)\xra{\nu_{\ttS}}\AAA_F^{\times}/F^{\times}\ra \C^{\times}$, where $\nu_{\ttS}$ is the reduced norm map. 

%  For an even subset $\ttS$ of places of $F$, we denote by $\pi_{\ttS}$ the one-dimensional automorphic representation of $G_{\ttS}(\AAA)$ that factors as $G_{\ttS}(\AAA)\xra{\nu} \AAA_{F}^{\times}/F^{\times}\xra{\chi} \C^{\times}$, where $\nu$ is the reduced norm. 

\subsection{Cohomology of $\bfSh_{K_\ttS}(G_\ttS)$}\label{S:coh-automophic}
Let  $\ttS$ be an even subset of places of $F$ satisfying Hypothesis~\ref{H:subset-S}. Let $K_{\ttS}\subset G_{\ttS}(\AAA^{\infty})$ be an open compact subgroup. We fix an isomorphism $G_{\ttS}(\AAA^{\infty,p})\simeq \GL_2(\AAA^{\infty,p})$. For a field $M$ of characteristic $0$, let  $\scrH(K^p_{\ttS},M)$ denote the prime-to-$p$ Hecke algebra $M[K^p_{\ttS}\backslash \GL_2(\AAA^{\infty,p})/K^p_{\ttS}]$. 

%In computing the cohomology of \'etale sheaves in terms of automorphic forms, it it convenient to use $\Qlb$-sheaves. For an $L_{\gothl}$-lisse sheave 

For each $\pi\in \scrA_{(\kb,w)}$ or $\pi\in \scrB_{w}$, let $(\pi_{\ttS}^{\infty})^{K_{\ttS}}=(\pi_{\ttS}^{\infty,p})^{K_{\ttS}^p}\otimes (\pi_{\ttS,p})^{K_{\ttS,p}}$ be the $K_{\ttS}$-invariant subspace of $\pi_{\ttS}^{\infty}$. We consider it as a   $\scrH(K^p_{\ttS}, \C)$-module with the natural Hecke action of $\scrH(K^p_{\ttS}, \C)$ on $(\pi^{\infty,p}_{\ttS})^{K_{\ttS}^p}$ and the trivial action on $(\pi_{\ttS,p})^{K_{\ttS,p}}$. The following result is well known. %We have the following description of the cohomology of automorphic sheaves.

%{\color{blue}[Yichao, you need to trim the following to just what we need.]}

\begin{theorem}
\label{T:cohomology of Shimura variety}
%Let the notation be as above, and suppose $L$ sufficiently large. 
 %We have the following description of the cohomology of Shimura varieties.
In the Grothendieck group of finite-dimensional $\scrH(K_{\ttS}, \Qlb)[\Gal_{F_{\ttS}}]$-modules, we have an equality
\begin{align*}
\big[
H^\star_{c,\et}(\Sh_{K_{\ttS}}&(G_\ttS)_{\overline \QQ}, \scrL^{(\underline{k}, w)}_{\ttS,l} )\big]
=(-1)^{g-\#\ttS_\infty}
\sum_{\pi \in \scrA_{(\kb,w)}}\big[
 (\pi_\ttS^\infty)^{K_{\ttS}}\otimes \rho_{\pi, l}^\ttS
\big] \\
&+ \delta_{\kb,2} \sum_{\pi\in \scrB_{w}}\big[(\pi_{\ttS}^{\infty})^{K_{\ttS}}\otimes \rho^{\ttS}_{\pi,l}\big]\otimes
\bigg(\big[(\Qlb\oplus \Qlb(-1))^{\otimes (\Sigma_{\infty}-\ttS_{\infty})}\big]-\delta_{\ttS,\emptyset}\big[ \Qlb\big]\bigg).
\end{align*}
% +\big[ \bigoplus_{\pi \in \scrA_{\ttS, \res}^{(\underline k, w)}} (\pi^\infty)^K \otimes \rho_\pi^\ttS \big].
Here, for each $\pi$ in $\scrA_{(\kb,w)}$ or $\scrB_{w}$, we put $\rho_{\pi,l}^\ttS: = \bigotimes_{\Sigma_{\infty}-\ttS_\infty}\textrm{-} \mathrm{Ind}_{\Gal_F}^{\Gal_\QQ}(\rho_{\pi, l})$ where $\rho_{\pi,l}$ denotes the $l$-adic representation of $\Gal_{F}$ defined above,  $\delta_{\kb,2}$ is equal to $1$ if $\kb=(2,\dots, 2)$ and $0$ otherwise, and $\delta_{\ttS,\emptyset}=1$ if $\ttS=\emptyset$ and $0$ otherwise.
\end{theorem}

\begin{proof}

  %Actually, one has a canonical decomposition 
% \[
 %H^{n}_{\dR}(\Sh_{K_{\ttS}}(G_{\ttS})_\C, \F^{(\kb,w)}_{\ttS} )=\bigoplus_{\pi}(\pi_{\ttS}^{\infty})^{K_{\ttS}} \otimes H^{n}(\gothg_{\ttS}, \gothk_{\ttS}; \pi_{\ttS,\infty}\otimes \rho^{(\kb,w)}),
 %\] 
%where $\pi$ runs through all automorphic representations of $\GL_{2}(\AAA_F)$ that transfer to $G_{\ttS}$, and $H^{n}(\gothg_{\ttS}, \gothk_{\ttS}; \bullet)$ denotes the $(\gothg,\gothk)$-cohomology for the group $G_{\ttS}(\R)$ and its maximal compact subgroup studied in \cite{borel-wallach}. By standard arguments of comparison of cohomology theories, one is reduced to compute the dimension of $H^{n}(\gothg_{\ttS}, \gothk_{\ttS}; \pi_{\ttS,\infty}\otimes \rho^{(\kb,w)})$.
%\begin{itemize}
%\item When $\kb\neq (2,\cdots,2)$, then $H^{n}(\gothg_{\ttS}, \gothk_{\ttS}; \pi_{\ttS,\infty}\otimes \rho^{(\kb,w)})$ is nonzero  if and only if $\pi\in \scrA^{(\kb,w)}$ and $n=g-\ttS_{\infty}$; in that case its dimension is $2^{g-\#\ttS_{\infty}}$  by \cite{borel-wallach}. 

%\item When $\kb=(2,\cdots,2)$, then besides the contribution of $\scrA^{(\kb,w)}$ as above, there are contributions from $\scrB^{w}$. More precisely, one has 
%\[
%\dim H^{n}(\gothg_{\ttS}, \gothk_{\ttS}; \pi_{\ttS,\infty}\otimes \rho^{(\kb,w)})=\begin{cases}\binom{g-\#\ttS_{\infty}}{i} &\text{if $n=2i$ with }0\leq i\leq 2(g-\#\ttS_{\infty});\\
%0 &\text{otherwise}.
% \end{cases}
%\]
%\end{itemize}
%The statement for  $\ttS\neq \emptyset$ follows immediately. 

For $\ttS\neq\emptyset$, this is  proved in \cite[\S 3.2]{brylinski-labesse}. When $\ttS=\emptyset$, the contributions from the cuspidal representations and the one-dimensional representations are computed in the same way as above in \emph{loc. cit.}. The subtraction by $\Qlb$ when $\kb=(2,\dots, 2)$ comes from the fact that $H^0_{c, \et}(\Sh_{K}(G)_{\Qb},\Qlb)=0$.  We explain now why there is  no contributions from Eisenstein series in the Grothendieck group. 
According to \cite{MSYZ},  the Eisenstein spectra appear  in $H^{i}_c$ only when $\kb$ is of parallel weight. In that case, each possible Eisenstein series will only appear in $H^{i}_{c, \et}$ with multiplicity $\binom{g-1}{i}$ for $1\leq i\leq g$, and none in $H^{i}_{c, \et}$ with $i=0$ or $i\geq g+1$. Hence, their contributions cancel out when passing to the Grothendieck group.
\end{proof}

\subsection{Cohomology of $\bfSh_{K_{E,p}}(T_{E,\tilde\ttS})$}
Let $w \in \ZZ$ and $\widetilde \Sigma \subset  \Sigma_{E,\infty}$ be as in Subsection~\ref{S:quaternion-automorphic-bundle}.
Let $\scrA_{E, \widetilde \Sigma}^w$ denote the set of Hecke characters $\chi$ of $\AAA_E^\times / E^\times$ such that $\chi|_{E_{ \tilde \tau}^{\times}} : x \mapsto x^{w-2}$ for all $ \tilde \tau \in \tSigma$ and $\chi$ is unramified at places above $p$. Here, the isomorphism  $E_{\tilde\tau}=E\otimes_{F,\tau}\R\xra{\sim} \C$ is defined using the embedding $\tilde\tau :E\hra \C$.  
%We fix an isomorphism $\CC \cong \overline \QQ_l$. 
%We suppose the coefficient field $L$ large enough so that  $L_\gothl$  contains all $l$-adic embeddings of $E$.  Hence we can identify each $\tilde \tau \in \Sigma_{E,\infty}$ with an  embedding $\tilde \tau_\gothl$ of $E$ into $L_\gothl$. 

We fix an isomorphism $\iota_l: \C\simeq \Qlb$ as before. Then we can identify each $\tilde\tau\in \Sigma_{E,\infty}$ with an embedding  $\tilde\tau_l$ of $E$ into $\Qlb$. Define an  $l$-adic character on $\AAA_E^\times$ associated to $\chi$ by
\[
\chi_{l} \colon x \mapsto \big( \chi(x)\cdot \prod_{\tilde \tau \in \widetilde \Sigma} \tilde \tau(x)^{2-w} \big) \cdot \prod_{\tilde \tau \in \widetilde \Sigma}\tilde \tau_l(x)^{w-2}\in \Qlb^{\times}.
\]
This character factors through $E^\times\backslash\AAA_E^\times / E_\RR^\times$ and hence induces a Galois representation $\chi_{l}\colon \Gal_E \to \Qlb^\times$ via class field theory. We put $\rho_{\chi, l}=\chi_{l}^{-1}$.

Given $\ttS$ and $\tilde \ttS_\infty$ as in Subsection~\ref{S:family of AV over Sh}, we view $\tilde \ttS_\infty^c$ as a subset of $ \Sigma_{E, \infty} \cong \Gal_\QQ / \Gal_E$, where $\Gal_\QQ$ acts on the left by postcomposition.
The construction in Subsection~\ref{S:tensorial-induction} gives rise to a representation of $\Gal_{E_{\tilde \ttS}}$
\[
\rho_{\chi, \tilde \ttS, l} : = \bigotimes_{\tilde \ttS_{\infty}}\textrm{-}\Ind_{\Gal_E}^{\Gal_\Q} \rho_{\chi, l}.
\]

%The following is an easy consequence of the main theorem of complex multiplication.
\begin{lemma}
We have a $\Qlb[\AAA^{\infty,\times}_E]\times \Gal_{k_{\tilde\wp}}$-equivariant isomorphism:
\begin{equation}
\label{E:cohomology of Sh(T_E)}
H^0_\et\big(\bfSh_{K_{E,p}}(T_{E, \tilde \ttS})_{\overline \FF_p}, \scrL_{\tilde \ttS, E, l}^w\big)
\simeq \bigoplus_{\chi \in \scrA_{E, \widetilde \Sigma}^w} \chi \otimes \rho_{\chi, \tilde \ttS, l}|_{\Gal_{k_{\tilde \wp}}},
\end{equation}
where $\AAA^{\infty,\times}_E$ acts on the right hand side via (the finite part) of $\chi$, and $\Gal_{k_{\tilde\wp}}$ acts via $\rho_{\chi,\tilde\ttS,l}$. 
\end{lemma}
\begin{proof}
According to Deligne's  definition of Shimura varieties for tori, the action of $\Frob_\wp$ is the same as the Hecke action of the element $\gothRec_{E, \tilde \ttS}(\varpi_{\tilde \ttS})$, where the map $\gothRec_{E, \tilde \ttS}$ is the reciprocity map defined in Subsection~\ref{S:auxiliary CM}. It follows that the Galois action on $H^0_\et\big(\bfSh_{K_{E,p}}(T_{E, \tilde \ttS})_{\overline \FF_p}, \scrL_{\tilde \ttS, E, l}^w\big)
$ is as described.
 % This not only applies to the action on the geometric points of the CM Shimura variety, but also on the universal CM abelian varieties. 
\end{proof}

The following lemma will be used later.

\begin{lemma}
\label{L:rho S of Frob}
Keep the notation as above.
Put $d_{\tilde \wp} = [k_{\tilde \wp}: \FF_p]$.  Let $d_\gothq$ denote the inertia degree of a $p$-adic place $\gothq \in \Sigma_{E,p}$ in $E/\QQ$.
Let $\tilde \ttS^c_{\infty, \gothq}$ denote the set of places in $\tilde \ttS_{\infty}^c$ inducing the place $\gothq$ of $E$ via $\iota_p:\CC \simeq \overline \QQ_p  $.
Let $\Frob_{\tilde \wp}$ denote the geometric Frobenius for $k_{\tilde \wp}$.
Then
$\rho_{\chi, \tilde \ttS, l}(\Frob_{\tilde \wp}) = 
\prod_{\gothq \in \Sigma_{E,p}}\chi_ l(\varpi_{\gothq})^{-d_{\tilde \wp}\#\tilde \ttS_{\infty/\gothq}^c / d_\gothq}
$.
\end{lemma}
\begin{proof}
This is a straightforward calculation.
For each $\gothq \in \Sigma_{E,p}$, let $\Frob_\gothq$ denote a geometric Frobenius of $\Gal_E$ at $\gothq$.
Then we have
\begin{align*}
\rho_{\chi, \tilde \ttS, l}(\Frob_{\tilde \wp})&
=
\bigotimes_{\gothq \in \Sigma_{E,p}}
\Big(
\bigotimes_{\tilde \ttS_{\infty/\gothq}^c} \textrm{-}\Ind_{\Gal_{E_\gothq}}^{\Gal_{\Qp}}\rho_{\chi,l}|_{\Gal_{E_\gothq}}\Big)(\Frob_{\tilde \wp})
 \\
&=
\prod_{\gothq \in \Sigma_{E,p}}
\rho_{\chi, l}(\Frob_\gothq)^{d_{\tilde \wp}\#\tilde \ttS_{\infty/ \gothq}^c / d_\gothq}
 = 
\prod_{\gothq \in \Sigma_{E,p}}\chi_l(\varpi_{\gothq})^{-d_{\tilde \wp}\#\tilde \ttS_{\infty/ \gothq}^c / d_\gothq}.
\end{align*}
%{\color{blue}[I don't know what more I can explain, it is really trivial once you know how to prove it.]}
\end{proof}

When relating the \'etale cohomology of $\bfSh_{K''_{ \ttS,p}}(G''_{\tilde \ttS})$ and $\bfSh_{K_p}(G)$, we need the following.

\begin{prop}\label{P:cohomology of Sh(G*E) vs Sh(G)}
Let $\chi\in \scrA_{E,\widetilde\Sigma}^w$, and $\chi_{F}$ its restriction to $\AAA^{\infty,\times}_F$.

\begin{itemize}
	\item[(1)]
 We have
a natural $
G_\ttS(\AAA^{\infty,p})\times \Gal_{E_{\tilde \ttS}}$-equivariant  isomorphism

\begin{equation}\label{E:cohomology G vs G''}
H^\star_{c,\et}(\bfSh_{K_{\ttS,p}}(G_\ttS)_{\overline \FF_p}, \scrL_{\ttS,
l}^{(\underline k, w)})^{\AAA_F^{\infty, \times}
 =
\chi_F} \otimes \rho_{\chi, \tilde \ttS, l}
\cong
H^\star_{c,\et} (\bfSh_{K''_{\ttS,p}}(G''_{\tilde \ttS})_{\overline \FF_p},
\scrL''^{(\underline k, w)}_{\tilde \ttS, \tSigma, l} )^{\AAA_E^{\infty, \times}
= \chi},
\end{equation}
 where the superscripts mean to take the subspaces where the Hecke
actions are given as described.

\item[(2)] When $\ttS = \emptyset$, we have analogous $\GL_2(\AAA^{\infty,p})$-equivariant isomorphisms for all $\ttT \subseteq \Sigma_\infty$:
\begin{equation}
\label{E:cohomology GO-strata for HMV v.s that of G''}
H^\star_{c,\et}(\bfSh_{K_p}(G)_{ \overline \FF_p,\ttT}, \scrL_{\ttS,
l}^{(\underline k, w)})^{\AAA_F^{\infty, \times} =
\chi_F}
\cong
H^\star_{c,\et} (\bfSh_{K''_{p}}(G''_\emptyset)_{\overline \FF_p, \ttT},
\scrL''^{(\underline k, w)}_{\emptyset, \tSigma, l} )^{\AAA_E^{\infty, \times}
= \chi}.
\end{equation}
Moreover, \eqref{E:cohomology GO-strata for HMV v.s that of G''} is equivariant for the actions of $\Phi_{\gothp^2}$ on both sides (see Subsections~\ref{S:Frobenius-HMV} and \ref{S:family of AV over Sh}(3) for the definition of the actions).
\end{itemize}
\end{prop}

\begin{proof}

(1)
We first claim that the quotient  $\alpha: G_\ttS \times T_{E,\tilde \ttS} \to G''_{\tilde \ttS}$ induces an isomorphism of Shimura varieties
\begin{equation}\label{E:quotient-Sh-var}
\big(\bfSh_{K_{\ttS,p}}(G_\ttS) \times \bfSh_{K_{E,p}}(T_{E, \tilde \ttS} )\big) \big/ \AAA_F^{\infty, p,\times} \cong \bfSh_{K''_{\ttS,p}}(G''_{\tilde \ttS}),
\end{equation}
where $\AAA_F^{\infty,p,\times}$ acts by the anti-diagonal Hecke action.
For this, it is enough to show that the product
\[
\AAA_F^{\infty, p, \times} \cdot \big( G_\ttS(\QQ)_+^{(p), \cl} \times \calO_{E, (p)}^{\times, \cl}  \big)
\]
is already closed in the $G_\ttS(\AAA^{\infty,p}) \times \AAA_E^{\infty,p, \times}$,
 where the superscript means to take closure inside the corresponding adelic group, $\AAA_F^{\infty, p, \times}$ embeds in the product anti-diagonally, and $G_\ttS(\QQ)_+^{(p)}$ denote the subgroup consisting of  $p$-integral elements of $G_\ttS(\Q)$ with totally positive determinant.
For this, we take an open compact subgroup $U_\ttS^p$ of $G_\ttS(\AAA^{\infty,p})$ and intersect the product above with $U_\ttS^p \times \widehat{\calO}_E^{(p),\times}$. We are  left to prove that the product
\[
\widehat{\calO}_F^{(p),\times} \cdot \big( (
G_\ttS(\QQ)_+^{(p)} \cap U_\ttS^p)^\cl \times \calO_{E}^{\times,\cl}
\big)
\]
is closed in $U_\ttS^p \times \widehat{\calO}_E^{(p),\times}$.  But Dirichlet's unit  Theorem implies that $\calO_F^\times$ is a finite index subgroup of $\calO_E^\times$. It follows that the above product is a finite union of 
$ \widehat\calO_F^{(p),\times}\cdot(G_\ttS(\QQ)_+^{(p)} \cap U_\ttS^p)^\cl $, 
which is obviously closed in $U_\ttS^p \times \widehat{\calO}_E^{(p),\times}$.  This proves the claim.

This claim in particular implies that for any $\Qlb$-lisse sheaf $\scrL''$ on $\bfSh_{K''_{\ttS ,p}}(G''_{\tilde \ttS})_{k_{\tilde \wp}}$, we have a natural isomorphism
\begin{equation}
\label{E:quotient by E}
H^\star_{c,\et}(\bfSh_{K''_{\ttS,p}}(G''_{\tilde \ttS})_{\overline \FF_p}, \scrL'') \cong H^\star_{c,\et}(\bfSh_{K_{\ttS,p} \times K_{E,p}}(G_{\ttS} \times T_{E,\tilde \ttS})_{\overline \FF_p}, \bbalpha^*\scrL'')^{\textrm{anti-diag }\AAA_F^{\infty,p,\times}=1},
\end{equation}
where the superscript means to take the subspace where the anti-diagonal $\AAA_F^{\infty,p,\times}$ acts trivially.

Applying this to \eqref{E:comparing auto sheaves from unit to quad}, and further taking the subspace where $\AAA_E^{\infty,p,\times}$ acts via $\chi$ (note that the restriction of a Hecke character to finite ideles away from $p$ already determines its value at places above $p$), we get
\begin{align*}
&\quad \quad H^\star_{c,\et} (\bfSh_{K''_{\tilde \ttS,p}}(G''_{\tilde \ttS})_{\overline \FF_p},
\scrL''^{(\underline k, w)}_{\tilde \ttS, \widetilde \Sigma, l} )^{\AAA_E^{\infty, \times} = \chi}
\\
&\xra{\cong, \eqref{E:quotient by E}}
H^\star_{c,\et} \big(
\bfSh_{K_{\ttS,p}}(G_\ttS)_{\overline \FF_p} \times_{\overline \FF_p}
\bfSh_{K_{E,p}}(T_{E,\tilde \ttS})_{\overline \FF_p}, \bbalpha^*
\scrL''^{(\underline k, w)}_{\tilde \ttS, \widetilde \Sigma,l} \big)^{\AAA_F^{\infty, \times}
\times \AAA_E^{\infty, \times} = \chi_F \times \chi}\\
&\xra{\cong, \eqref{E:comparing auto sheaves from unit to quad}}H^\star_{c,\et}(\bfSh_{K_{\ttS,p}}(G_\ttS)_{\overline \FF_p}, \scrL_{\ttS,
l}^{(\underline k, w)})^{\AAA_F^{\infty, \times} =
\chi_F} \otimes
H^\star_{\et}\big(\bfSh_{K_{E,p}}(T_{E, \tilde \ttS})_{\overline \FF_p}, \scrL_{E,\widetilde \Sigma,
l}^w\big)^{ \AAA_E^{\infty, \times} = \chi}\\
&\xra{\cong, \eqref{E:cohomology of Sh(T_E)}}H^\star_{c,\et}(\bfSh_{K_{\ttS,p}}(G_\ttS)_{\overline \FF_p}, \scrL_{\ttS,
l}^{(\underline k, w)})^{\AAA_F^{\infty, \times} =
\chi_F} \otimes
\rho_{\chi, \tilde \ttS, l}.
\end{align*}
This proves \eqref{E:cohomology G vs G''}.

(2) Assume now $\ttS=\emptyset$. Consider the base change to $k_0$ of the isomorphism  \eqref{E:quotient-Sh-var}. Since the GO-stratification on $\bfSh_{K_p}(G)_{k_0}$ and that on $\bfSh_{K''_{\emptyset,p}}(G''_{\emptyset})_{k_0}$ are compatible as shown in \eqref{E:compatibility-GO}, we have an isomorphism 
\[
\big(\bfSh_{K_{p}}(G)_{k_0,\ttT} \times \bfSh_{K_{E,p}}(T_{E, \emptyset} )_{k_0}\big) \big/ \AAA_F^{\infty, p,\times} \cong \bfSh_{K''_{\emptyset,p}}(G''_{\emptyset})_{k_0,\ttT},
\]
for any subset $\ttT\subseteq \Sigma_{\infty}$. Then   the same argument as above  applies to the cohomology of  $\bfSh_{K''_{\emptyset,p}}(G''_{\emptyset})_{\Fpb,\ttT}$.
This then proves \eqref{E:cohomology GO-strata for HMV v.s that of G''}. Here, note that the Galois representation $\rho_{\chi,\tilde \ttS, l}$ is trivial, since the Deligne homomorphism $h_{E,\emptyset}$ is trivial.
The compatibility with the twisted partial Frobenius follows from Subsection~\ref{S:family of AV over Sh}(5).
\end{proof}

\begin{notation}
\label{N:A(k,w)[chi]}
Let  $\chi \in \scrA_{E, \widetilde \Sigma}^w$ be a Hecke character, and put $\chi_F= \chi|_{\AAA^{\infty,\times}_F}$.
  We denote by  $\scrA_{(\underline k,w)}[\chi_F]$ the subset of cuspidal automorphic representations  $\pi \in \scrA_{(\underline k, w)}$ for which the central character $\omega_\pi = \chi_F$. 
When $\kb=(2,\dots, 2)$, $w\geq 2$ is an even integer.
 We also denote by $\scrB_{w}[\chi_F]$ be the set of one-dimensional automorphic representations $\pi$  of $\GL_{2,F}$ such that  $\omega_{\pi}=\chi_F$.

We remark that every Hecke character  $\chi_F$ unramified at places above $p$ and whose archimedean component is $x_{\infty}\mapsto N_{F/\Q}(x_{\infty})^{w-2}$ extends to a Hecke character $\chi \in \scrA_{E, \widetilde \Sigma}^w$.
Indeed, we may first fix an arbitrary Hecke character $\chi' \in \scrA_{E, \widetilde \Sigma}^w$, and then $\omega_0 = \chi'|^{-1}_{\AAA_F^\times}\cdot \chi_F$ is a \emph{finite} character trivial on $(F\otimes\R)^{\times}$ (and on $(\calO_F \otimes \ZZ_p)^\times$. Since $\AAA_F^{\infty,\times}/F^\times(\calO_F\otimes\ZZ_p)^\times$ injects into $\AAA_E^{\infty,\times}/E^\times(\calO_E\otimes\ZZ_p)^\times$, we may always find a finite character $\chi_0$ of $\AAA_E^{\infty,\times}/E^\times(\calO_E\otimes\ZZ_p)^\times$ extending $\omega_0$.  Then $\chi' \chi_0$ is a Hecke character of $\AAA_E^\times / E^\times$ in $\chi \in \scrA_{E, \widetilde \Sigma}^w$ extending $\chi_F$.
\end{notation}

Note that the Galois representation $\rho_{\pi,l}^\ttS = \bigotimes_{\Sigma_{\infty}-\ttS_\infty}\textrm{-} \mathrm{Ind}_{\Gal_F}^{\Gal_\QQ}(\rho_{\pi, l})$  appearing in the cohomology of quaternionic Shimura variety (see Theorem~\ref{T:cohomology of Shimura variety}), when restricted to the local Galois group $\Gal_{F_{\ttS, \wp}}$, decomposes into the tensor product of Galois representations
\[
\rho_{\pi,l}^\ttS \big|_{\Gal_{F_{\ttS, \wp}}} \cong \bigotimes
_{\gothp \in \Sigma_p}
\Big(
\bigotimes_{\Sigma_{\infty/\gothp}-\ttS_{\infty/\gothp}} \textrm{-} \mathrm{Ind}_{\Gal_{F_\gothp}} ^{\Gal_{\Qp}} \big(\rho_{\pi, l}|_{\Gal_{F_\gothp}}\big)
\Big).
\] 
On the one hand, the usual trace formula proof only detects the action of the total Frobenius, namely the action of the Frobenius on the tensor product representation. On the other hand, it is natural to expect that there are certain ``partial Frobenius" actions that account for the Frobenius actions on each factor above.
The following is a folklore conjecture regarding the action of these twisted partial Frobenius actions.

\begin{conj}[Partial Frobenius]
\label{Conj: partial frobenius}
For each $\gothp \in \Sigma_p$, let $n_\gothp$ be the smallest positive number $n$ such that $\sigma_\gothp^n  \ttS_\infty =  \ttS_\infty$.
Let $d_\gothp$ denote the inertia degree of $\gothp$ in $F/\QQ$. Assume that
\begin{align}
\label{E:Sq = Sq bar}
&\textrm{for any $p$-adic place $\gothp\in \Sigma_p$ that splits into two primes $\gothq$ and $\bar\gothq$ in $E$,}
\\
\nonumber &\textrm{we have } \#\tilde\ttS_{\infty/\gothq}^c=\#\tilde\ttS_{\infty/\bar\gothq}^c.
\end{align}
 Then, for any Hecke character $\chi \in \scrA_{E, \widetilde \Sigma}^w$,
we have the following equality in the Grothendieck group of admissible modules over $\Qlb[\GL_2(\AAA^{\infty,p})][(\Phi_{\gothp^2})^{n_\gothp}; \gothp \in \Sigma_p]$.
\begin{small}
\begin{align}
\label{E:conjecture}
\big[
H^\star_{c,\et}(\bfSh_{K''_{\ttS,p}}(G''_{\tilde \ttS})_{\overline \FF_p}, &\scrL''^{(\underline{k}, w)}_{\tilde \ttS, \widetilde \Sigma, l})^{\AAA_E^{\infty, \times} = \chi}\big]
=(-1)^{g-\#\ttS_\infty}
\big[\bigoplus_{\pi \in \scrA_{(\underline k,w)}[\chi_F]}
 (\pi_\ttS^\infty)^{K_{\ttS,p}} \otimes \tilde \rho_{\pi, l}^\ttS
\big] \\ \nonumber
&+\delta_{\kb,2}\sum_{\pi\in \scrB_{w}[\chi_F]}\big[(\pi_{\ttS}^{\infty})^{K_{\ttS,p}}\otimes \tilde\rho^{\ttS}_{\pi,l}\big]\otimes
 \bigg(\big[  (\Qlb\oplus \Qlb(-1))^{\otimes(\Sigma_{\infty}-
 \ttS_{\infty})}\big]-\delta_{\ttS,\emptyset}[\Qlb]\bigg). 
\end{align}
\end{small}Here, for each $\pi$ in $\scrA_{(\kb,w)}[\chi_F]$ or $\scrB_{w}[\chi_F]$, we take $\tilde{\rho}^{\ttS}_{\pi, l}$ to have the same underlying $\Qlb$-vector space as 
\[
 \rho_{\pi,l}^\ttS := \bigotimes
_{\gothp \in \Sigma_p}
\Big(
\bigotimes_{\Sigma_{\infty/\gothp}-\ttS_{\infty/\gothp}} \textrm{-} \mathrm{Ind}_{\Gal_{F_\gothp}} ^{\Gal_{\Qp}} \big(\rho_{\pi, l}|_{\Gal_{F_\gothp}} \big)
\Big),
\] 
on which each  $(\Phi_{\gothp^2})^{n_{\gothp}}$
acts as the (geometric) $p^{2n_{\gothp}}$-Frobenius
 $\Frob_{p^{2n_{\gothp}}}$ times  the number $\omega_\pi( \varpi_\gothp)^{n_\gothp(1-\#\ttS_{\infty/\gothp}/d_{\gothp})}$ on the factor indexed by $\gothp$,  and  acts   trivially on the other factors. Here, $\varpi_{\gothp}\in \AAA^{\times}_F$ is the idele  which is a uniformizer at $\gothp$ and is $1$ at other places.
 The action of  $(\Phi_{\gothp^2})^{n_{\gothp}}$   on the $\Qlb(-1)$'s indexed by $\Sigma_{\infty/\gothp}-\ttS_{\infty/\gothp}$ is  the multiplication by $p^2$, and  is trivial  on the other $\Qlb(-1)$'s.
  \end{conj}

\begin{remark}
This Conjecture provides certain refinement of Langlands' philosophy on describing Galois representations appearing in the cohomology of Shimura varieties.
Unfortunately, to our best knowledge of literature, only the action of ``total Frobenius" was addressed using trace formula. It might be possible to modify the proof to understand the action of partial Frobenius; but this would go beyond the scope this paper.  We leave it as a conjecture for interested readers to pursue.

An alternative way to establish such a result is to generalize the Eichler-Shimura relations to our case.
In \cite{nekovar}, Nekov\'a\v r made progress when $p$ splits completely in $F/\QQ$ (see \cite[A4.3.1]{nekovar}).
The general case might benefit from using the work by Helm \cite{helm} on the integral model of unitary Shimura varieties with Iwahori level structure.
\end{remark}

Conjecture~\ref{Conj: partial frobenius} holds when we take the product of all twisted partial Frobenii.

\begin{prop}
\label{T:total Frobenius action}
Put $d_{ \wp} = [k_{ \wp}:\FF_p]$ and let
 $\Phi_{\wp^2}$ denote the product $\prod_{\gothp \in \Sigma_p} \Phi_{\gothp^2}^{d_{ \wp}}$. We fix a Hecke character $\chi \in \scrA_{E, \widetilde \Sigma}^w$.
\begin{enumerate}
\item[(1)] Then the equality \eqref{E:conjecture} holds in the Grothendieck group of admissible modules of $\Qlb[\GL_2(\AAA^{\infty,p})][\Phi_{ \wp^2}]$ (even without the assumption \eqref{E:Sq = Sq bar}).
Here, for each $\pi$  in $ \scrA_{(\kb,w)}[\chi_F]$ or  in $ \scrB_{w}[\chi_F]$, $\Phi_{\wp^2}$ acts  on $\tilde \rho^{\ttS}_{\pi,l}$ 
as $\rho_{\pi, l}^\ttS(\Frob_{ \wp}^2)$ multiplied by the number
\begin{equation}
\label{E:the number}
\omega_\pi(\underline p)^{d_{\wp}} \cdot 
\prod_{\gothq \in \Sigma_{E,p}} \chi(\varpi_\gothq)^{-2d_{\wp} \#\tilde \ttS^c_{\infty/\gothq} /d_\gothq},
\end{equation}
 where $\underline p\in \AAA^{\times}_{F}$ denotes the idele which is $p$ at all $p$-adic places and $1$  at all other places, and $\varpi_\gothq \in \AAA_E^\times$ denotes the idele which is a uniformizer at $\gothq$ and is $q$ at all other places.

\item[(2)] If the assumption \eqref{E:Sq = Sq bar} holds, then the number \eqref{E:the number} is equal to
 \[
 \omega_\pi(\underline p)^{d_{\wp}} \cdot \prod_{\gothp\in \Sigma_{p}}\omega_{\pi}(\varpi_{\gothp})^{-d_{\wp}\#\ttS_{\infty/\gothp}/d_{\gothp}}=u\cdot p^{d_{\wp}(w-2)(\#\ttS_{\infty}-g)},
 \]
 for $u$ some root of unity.
  %{\color{blue}[Again, check numbers.]}

\item[(3)]
Conjecture~\ref{Conj: partial frobenius} holds if $p$ is inert in $F$.

\end{enumerate}
\end{prop}
\begin{proof}

(1)
Combining \eqref{E:cohomology G vs G''} and Theorem~\ref{T:cohomology of Shimura variety}, we get an equality in the Grothendieck group of admissible modules over  $\Qlb[\GL_2(\AAA^{\infty,p})][\Phi_{ \wp^2}]$:
\begin{align*}
&\quad \quad \big[
H^\star_{c,\et}(\bfSh_{K''_p}(G''_{\tilde \ttS})_{\overline \FF_p}, \scrL''^{(\underline{k}, w)}_{\tilde \ttS, \widetilde \Sigma, l})^{\AAA_E^{\infty,\times} = \chi}\big]\\
&=
\big[
H^\star_{c,\et}(\bfSh_{K_{\ttS,p}}(G_\ttS)_{\overline \FF_p}, \scrL_{\ttS,
l}^{(\underline k, w)})^{\AAA_F^{\infty, \times}= \chi_F} \otimes  \rho_{\chi, \tilde \ttS, l}\big]
\\
&=
(-1)^{g-\#\ttS_\infty}
\sum_{\pi \in \scrA_{(\underline k,w)}[\chi_F]}\big[
 (\pi_\ttS^\infty)^{K_{\ttS,p}} \otimes \rho_{\pi, l}^\ttS\otimes  \rho_{\chi, \tilde \ttS, l}
\big]\\
&\quad \quad+ 
\delta_{\kb,2} \sum_{\pi\in \scrB_{w}[\chi_F]}\big[(\pi_{\ttS}^{\infty})^{K_{\ttS,p}}\otimes \rho^{\ttS}_{\pi,l}\otimes  \rho_{\chi,\tilde\ttS, l}\big]\otimes\
\bigg(\big[(\Qlb\oplus \Qlb(-1))^{ \otimes(\Sigma_{\infty}-\ttS_{\infty})}\big]-\delta_{\ttS,\emptyset}\big[ \Qlb\big]\bigg).
\end{align*}
Note that $\Phi_{\wp^2}$ acts on the cohomology as $\Frob_{\wp}^{2}\cdot S_p^{-d_{\wp}}$ by Subsection~\ref{S:family of AV over Sh}(3).
Let $\Frob_{\tilde\wp}$ denote the geometric Frobenius element of the residue field $k_{\tilde\wp}$ of $\cO_{\tilde\wp}$.
  We have either $\Frob_{\tilde\wp}=\Frob_{\wp}$ or $\Frob_{\tilde\wp}=\Frob_{\wp}^2$. 
  In either case, it follows from Lemma~\ref{L:rho S of Frob} that
  \[
  \rho_{\chi,\tilde\ttS,l}(\Frob^2_{\wp})=\prod_{\gothq\in \Sigma_{E,p}}\chi(\varpi_{\gothq})^{-{2d_{\wp}\#\tilde\ttS^c_{\infty/\gothq}}/{d_{\gothq}}}.
  \]
  Since the action of $S_p$
   is the given by the central idele  $\underline p^{-1}\in \AAA_F^{\times}$, 
 $\Phi_{\wp^2}$ acts on $(\pi_\ttS^\infty)^{K_{\ttS,p}} \otimes \rho_{\pi, l}^\ttS\otimes  \rho_{\chi, \tilde \ttS, l}$ for each $\pi\in \scrA_{(\kb,w)}[\chi_F]$ as $\rho_{\pi,l}^{\ttS}(\Frob^2_{ \wp})$ multiplied by
\[
\omega_{\pi}(\underline p^{d_{\wp}}) \prod_{\gothq\in \Sigma_{E,p}}\chi(\varpi_{\gothq})^{-{2d_{\wp}\#\tilde\ttS^c_{\infty/\gothq}}/{d_{\gothq}}}.
\] 
Similarly, one proves the statement for $\pi\in \scrB_w[\chi_F]$.
%It remains to prove that the number above write as $u\cdot p^{(w-2)g(\#\ttS_{\infty}-g)}$.
%We note that $\chi$ and $\omega_{\pi}$ write as   $\varepsilon_{\chi}$ a finite Hecke character on $E$ trivial at $E\otimes_{\Q}\R$.

(2) First fix $\gothp\in \Sigma_p$. If there is a unique prime $\gothq$ of $E$ above $\gothp$, we have $d_{\gothq}=2d_{\gothp}$ and $\chi(\varpi_{\gothq})=\omega_{\pi}(\varpi_{\gothp})$ since $\omega_{\pi}=\chi_F$. If $\gothp$ splits into $\gothq$ and $\bar\gothq$, under the assumption of the statement,  we have $\#\ttS^c_{\infty/\gothq}=\#\ttS^c_{\infty/\bar\gothq}=\frac{1}{2}\#\ttS_{\infty/\gothp}$. It follows immediately that the number \eqref{E:the number} is equal to
\[
\omega_{\pi}(\underline p^{d_{ \wp}}) \prod_{\gothp\in \Sigma_p} \omega_{\pi}(\varpi_{\gothp})^{-d_{\wp}\#\ttS_{\infty/\gothp}/d_{\gothp}}=\bigg(\varepsilon_{\pi}(\underline p)^{d_{\wp}}\prod_{\gothp\in\Sigma_p}\varepsilon_{\pi}(\varpi_{\gothp})^{-d_{\wp}\#\ttS_{\infty/\gothp}/d_{\gothp}}\bigg)p^{d_{\wp}(w-2)(\#\ttS_{\infty}-g)}.
\]
Here,  $\varepsilon_{\pi}:=\omega_{\pi}|\cdot|^{2-w}_F$ is a Hecke character of finite order. In particular, the expression in the bracket is a root of unity.

(3) When $p$ is inert, Conjecture~\ref{Conj: partial frobenius} is the same as (1), with the number \eqref{E:the number} computed in (2).
\end{proof}

\subsection{Description of the GO-stratification of $\bfSh_{K''_{\emptyset,p}}(G''_\emptyset)_{\overline \FF_p}$}\label{S:description-GO-strata}
Let $k_0$ be a finite field containing all residue fields of $\calO_E$ of characteristic $p$.
The main result  of  \cite{TX-GO} says that
the GO-stratum $\bfSh_{K''_{\emptyset,p}}(G''_\emptyset)_{k_0, \ttT}$, for $\ttT\subseteq \Sigma_\infty$,
is naturally isomorphic to a $\PP^1$-power bundle over the special fiber of another Shimura variety: $\bfSh_{K''_{\ttS(\ttT),p}}(G''_{\tilde \ttS(\ttT)})_{k_0}$, for some appropriate $\tilde \ttS(\ttT)$.
We now recall this result in more details as follows.

We recall first the definition of  $\ttS(\ttT)\subseteq \Sigma_{\infty}\cup \Sigma_p$ given in \cite[\S 5.1]{TX-GO} for our case (i.e., $\ttS = \emptyset$ using the notation from {\it loc. cit.})
It suffices to specify $\ttS(\ttT)_{/\gothp}=\ttS(\ttT)\cap (\Sigma_{\infty/\gothp}\cup \{\gothp\})$ for each $\gothp\in \Sigma_{p}$, since $\ttS(\ttT)=\coprod_{\gothp\in \Sigma_p}\ttS(\ttT)_{/\gothp}$. For each $\gothp \in \Sigma_p$, we put $\ttT_\gothp = \ttT \cap \Sigma_{\infty/\gothp}$.
According to the convention of \emph{loc. cit.}, we have   several cases:

\begin{itemize}
\item[(Case $\alpha 1$)] $[F_{\gothp}:\Q_p]$ is even, and $\ttT_{\gothp}\subsetneq \Sigma_{\infty/\gothp}$. In this case, we write $\ttT_{\gothp}=\coprod C_i$ as a disjoint union of chains. Here,  by a chain, we mean there exists a $\tau_i\in \Sigma_{\infty/\gothp}$ and an integer $r_i\geq 0$ such that  $C_i=\{\sigma^{-a}\tau_i: 0\leq a\leq r_{i}\}$ is a subset of $\ttT_\gothp$, but $\sigma\tau_i, \sigma^{-r_i-1}\tau_i\notin \ttT_{\gothp}$. For each cycle $C_i$, if $\#C_i$ is even, we put $C_i'=C_i$; if $\#C_i$ is odd, we put $C_i'=C_i\cup \{\sigma^{-r_{i}}\tau_i\}$.  Then, we define $\ttS(\ttT)_{/\gothp}=\coprod C_i'$.

For example, if $\Sigma_{\infty / \gothp} = \{\tau_0, \sigma\tau_0, \dots, \sigma^5 \tau_0\}$ and $\ttS_{/\gothp} = \{\sigma\tau_0, \sigma^3\tau_0, \sigma^4\tau_0\}$, then we have $\ttS(\ttT)_{/\gothp} = \{\tau_0,\sigma\tau_0, \sigma^3\tau_0, \sigma^4\tau_0\}$.

\item[(Case $\alpha2$)] $[F_{\gothp}:\Q_p]$ is even and $\ttT_{\gothp}=\Sigma_{\infty/\gothp}$. In this case, we put $\ttS(\ttT)_{/\gothp}=\Sigma_{\infty/\gothp}$. 

\item[(Case $\beta1$)] $[F_{\gothp}:\Q_p]$ is odd and $\ttT_{\gothp}\subsetneq \Sigma_{\infty/\gothp}$. In this case, we define $\ttS(\ttT)_{/\gothp}$ using the same rule as Case $\alpha1$.

\item[(Case $\beta 2$)] $[F_{\gothp}:\Q_p]$ is odd and $\ttT_{\gothp}=\Sigma_{\infty/\gothp}$. We put $\ttS(\ttT)_{/\gothp}=\Sigma_{\infty/\gothp}\cup \{\gothp\}$.
\end{itemize}
It is clear from the definition that $\sigma_\gothp(\ttS(\ttT)) = \ttS(\sigma_\gothp(\ttT))$.

{\it We do not recall the precise choice of the lifts $\tilde \ttS(\ttT)_\infty$, as
it is combinatorially complicated.  We refer interested readers to \emph{loc. cit.} for the construction. In this paper, we only need to know that  $\tilde\ttS(\ttT)_{\infty}$  satisfies the assumption~\eqref{E:Sq = Sq bar} in Conjecture~\ref{Conj: partial frobenius}, i.e. if a prime $\gothp\in \Sigma_p$ splits into two places  $\gothq$ and $\bar\gothq$ in $E$, then $\#\tilde\ttS(\ttT)^c_{\infty/\gothq}=\#\tilde\ttS(\ttT)^c_{\infty/\bar\gothq}$.}

We now specify the subgroup
 $K_{\ttS(\ttT), p}=\prod_{\gothp\in\Sigma_p}K_{\ttS(\ttT),\gothp}\subset G_{\ttS(\ttT)}(\Q_p)=\prod_{\gothp\in\Sigma_p}(B_{\ttS}\otimes_FF_{\gothp})^{\times}$. For this, we fix  an isomorphism $(B_{\ttS(\ttT)}\otimes_{F}F_{\gothp})^{\times}\simeq \GL_2(F_{\gothp})$ for each $\gothp\notin \ttS(\ttT)$,
    i.e. if we are in cases $\alpha 1$, $\alpha2$ and $\beta1$.
\begin{itemize}
\item In Case $\alpha1$ and $\beta1$, we take $K_{\ttS(\ttT),\gothp}=\GL_2(\cO_{F_\gothp})$.

\item In Case $\alpha2$, we take $K_{\ttS(\ttT),\gothp}=\Iw_{\gothp}\subset \GL_2(\cO_{F_{\gothp}})$, which is the Iwahori subgroup  \eqref{E:Iwahori}.

\item In Case $\beta2$, $B_{\ttS(\ttT)}$ is ramified at $\gothp$, and we take $K_{\ttS(\ttT),\gothp}=\cO_{B_{\ttS(\ttT),\gothp}}^{\times}$. Here, $\cO_{B_{\ttS(\ttT),\gothp}} $ denotes the unique maximal order of $B_{\ttS(\ttT),\gothp} = B_{\ttS(\ttT)} \otimes_F F_\gothp$. 
\end{itemize}
Finally, we put 
$$K''_{\ttS(\ttT),p}=K_{\ttS(\ttT),p}\times_{(\cO_F\otimes\Z_p)^{\times}}(\cO_E\otimes\Z_p)^{\times}\subset G''_{\ttS(\ttT)}(\Q_p).
$$

Now we can recall the main result of \cite{TX-GO} in the case of GO-strata for $\bfSh_{K''_{\emptyset,p}}(G''_{\emptyset})_{k_0}$.

\begin{theorem}[{\cite[Corollary 5.9]{TX-GO}}]\label{T:GO-Hilbert}
\begin{enumerate}
\item[(1)] The GO-stratum $ \bfSh_{K''_{\emptyset,p}}(G''_{\emptyset})_{k_0, \ttT}$ is isomorphic to a $(\PP^1)^{I_{\ttT}}$-bundle over $\bfSh_{K''_{\ttS(\ttT),p} }(G''_{\tilde \ttS(\ttT)})_{k_0}$, where the index set is 
\[
I_{\ttT}=\ttS(\ttT)_{\infty}-\ttT=\bigcup_{\gothp\in \Sigma_{p}} \ttS(\ttT)_{\infty/\gothp}-\ttT_{\gothp}. 
\]

\item[(2)] The natural projection  $\pi_{\ttT}: \bfSh_{K''_{\emptyset,p}}(G''_{\emptyset})_{k_0, \ttT} \to \bfSh_{K''_{\ttS(\ttT),p} }(G''_{\tilde \ttS(\ttT)})_{k_0} $ given  in \emph{(1)} is  equivariant for the action of  $G''_\emptyset(\AAA^{\infty, p}) \simeq G''_{\tilde \ttS(\ttT)}(\AAA^{\infty,p})$. In particular, for a sufficiently small compact subgroup $K''^p_{\emptyset}\subset G''_{\emptyset}(\AAA^{\infty,p})$, the projection $\pi_{\ttT}$ descends to a $(\PP^1)^{I_{\ttT}}$-fiber bundle $\pi_{\ttT}\colon\bfSh_{K''_{\emptyset}}(G''_{\emptyset})_{k_0, \ttT} \to \bfSh_{K''_{\ttS(\ttT)} }(G''_{\tilde \ttS(\ttT)})_{k_0}$, where $K''_{\emptyset}=K''^p_{\emptyset}K''_{\emptyset,p}$ and $K''_{\ttS(\ttT)}=K''^p_{\emptyset}K''_{\ttS(\ttT),p}$. Here, we have fixed an isomorphism between $G''_{\emptyset}(\AAA^{\infty,p})$ and $ G''_{\ttS(\ttT)}(\AAA^{\infty,p})$, and view $K''^p_{\emptyset}$ also as a subgroup of the latter.

\item[(3)] Let $\bfA''_{\emptyset,k_0}$ $($resp. $\bfA''_{\tilde \ttS(\ttT),k_0})$ denote the family of abelian varieties over $\bfSh_{K''_{\emptyset,p}}(G''_{\emptyset})_{k_0}$ $($resp. $\bfSh_{K''_{\ttS(\ttT),p} }(G''_{\tilde \ttS(\ttT)})_{k_0})$ discussed in Subsection~\ref{S:family of AV over Sh}. Then the restriction of $\bfA''_{\emptyset, k_0}$ to $\bfSh_{K''_{\emptyset,p}}(G''_{\emptyset})_{k_0,\ttT}$ is naturally isogenous to $\pi^*_{\ttT}(\bfA''_{\tilde \ttS(\ttT), k_0})$.

\item[(4)] For each $\gothp\in \Sigma_{p}$, we have a commutative  diagram:
\[
\xymatrix@C=50pt{
\bfSh_{K''_{\emptyset,p}}(G''_{\emptyset})_{k_0, \ttT} \ar[r]_-{\xi^\mathrm{rel}} \ar@/^15pt/[rr]^-{\gothF_{\gothp^2,\emptyset}} \ar[dr]_{\pi_{\ttT}}
&
\gothF_{\gothp^2, \tilde \ttS(\ttT)}^*(\bfSh_{K''_{\emptyset,p}}(G''_{\emptyset})_{k_0, \sigma_{\gothp}^2 \ttT}) \ar[r]_-{\gothF_{\gothp^2, \tilde \ttS(\ttT)}^*} \ar[d]
& \bfSh_{K''_{\emptyset,p}}(G''_{\emptyset})_{k_0,\sigma_{\gothp}^2 \ttT} \ar[d]^{\pi_{\sigma^2_{\gothp}\ttT}}
\\
&
\bfSh_{K''_{\ttS(\ttT),p} }(G''_{\tilde \ttS(\ttT)})_{k_0}
\ar[r]^-{\gothF_{\gothp^2, \tilde \ttS(\ttT)}} &
\bfSh_{K''_{\sigma_\gothp^2\ttS(\ttT),p} }(G''_{\sigma_\gothp^2\tilde \ttS(\ttT)})_{k_0},
}
\]
where the square is Cartesian, $\gothF''_{\gothp^2, \emptyset}$ $($resp. $\gothF''_{\gothp^2, \tilde \ttS(\ttT)})$ is the twisted partial Frobenius on $\bfSh_{K''_{\emptyset,p}}(G''_{\emptyset})_{k_0}$ $($resp. $\bfSh_{K''_{\tilde \ttS(\ttT),p} }(G''_{\tilde \ttS(\ttT)})_{k_0})$ \cite[\S 3.22]{TX-GO}, and  $\xi^\mathrm{rel}$ is a morphism whose restriction to a fiber $\pi_{\ttT}^{-1}(x)=(\PP^1_x)^{I_{\ttT}}$ is the product of the relative $p^2$-Frobenius of the $\PP^1_x$'s indexed by $I_{\ttT}\cap \Sigma_{\infty/\gothp}=\ttS(\ttT)_{\infty/\gothp}-\ttT_{\gothp}$, and the identity map on the other $\PP^1_x$'s.
\end{enumerate}
\end{theorem}

We list a few special cases of the theorem for the convenience of the readers.
\begin{example}

The prime-to-$p$ level of all Shimura varieties below are taken to be the same as $K''^p_\emptyset$ (they can be naturally identified). Unless specified otherwise, the level structure $K''_p$ at $p$ is taken to be ``maximal".
To simplify notation, we use $\tilde X_{\ttT}$ to denote the GO-stratum $\bfSh_{K''_{\emptyset,p}}(G''_{\emptyset})_{k_0, \ttT}$ and  $\overline{Sh}_{\ttS}$ to denote the Shimura variety
$\bfSh_{K''_{\ttS,p} }(G''_{\tilde \ttS})_{k_0}$ (note that we have suppressed the choice of signature here.)

\begin{itemize}

\item[(1)] When $F$ is a real quadratic field in which $p$ splits into two places $\gothp_1$ and $\gothp_2$, the chosen isomorphism $\iota_p: \CC \xrightarrow \simeq \overline \QQ_p$ associates to each place $\gothp_i$ an archimedean place $\infty_i$ of $F$.
Then the non-trivial closed GO-strata are  $\tilde X_{\{\infty_1\}}$, $\tilde X_{\{\infty_2\}}$, and $\tilde X_{\{\infty_1, \infty_2\}}$.
Then Theorem~\ref{T:GO-Hilbert} says that each $\tilde X_{\{\infty_i\}}$ is isomorphic to $\overline{Sh}_{\{\gothp_i, \infty_i\}}$, and $\tilde X_{\{\infty_1, \infty_2\}}$ is isomorphic to
$\overline{Sh}_{\{\gothp_1, \gothp_2, \infty_1, \infty_2\}}$.

\item[(2)] When $F$ is a real quadratic field in which $p$ is inert, we label the two archimedean places of $F$ to be $\infty_1$ and $\infty_2$.
Then Theorem~\ref{T:GO-Hilbert} says that each $\tilde X_{\{\infty_i\}}$ is isomorphic to a $\PP^1$-bundle over $\overline{Sh}_{\{\infty_1, \infty_2\}}$,
and $\tilde X_{\{\infty_1, \infty_2\}}$ is isomorphic to $\overline{Sh}_{\{\infty_1, \infty_2\}}$ with an \emph{Iwahori level structure at $p$}.

\item[(3)] When $F$ is a real cubic field in which $p$ is inert, the chosen isomorphism $\iota_p: \CC \xrightarrow \simeq \overline \QQ_p$ makes $\Sigma_\infty = \{\infty_0, \infty_1, \infty_2\}$ into a cycle under the action of the Frobenius $\sigma$, i.e. $\infty_0 \stackrel\sigma\mapsto \infty_1 \stackrel\sigma\mapsto \infty_2 \stackrel\sigma\mapsto \infty_3 = \infty_0$.
The stratum $\tilde X_{\{\infty_i\}}$ is isomorphic to a $\PP^1$-bundle over $\overline{Sh}_{\{\infty_{i-1}, \infty_{i}\}}$; the stratum $\tilde X_{\{\infty_{i-1}, \infty_{i}\}}$ is isomorphic to $\overline{Sh}_{\{\infty_{i-1}, \infty_{i}\}}$; and the stratum $\tilde X_{\{\infty_1, \infty_2, \infty_3\}}$ is isomorphic to $\overline{Sh}_{\{p,\infty_1, \infty_2, \infty_3\}}$.

\item[(4)] When $F$ is a totally real field of degree $4$ over $\QQ$ in which $p$ is inert, we may label the archimedean places of $F$ by $\infty_1, \dots, \infty_4$ such that the Frobenius $\sigma$ takes each $\infty_i$ to $\infty_{i+1}$, where $\infty_i = \infty_{i\bmod 4}$.
We have the following description of the GO-strata.

\renewcommand{\arraystretch}{1.5}
\begin{tabular}{|c|c|}
\hline
Strata & Description\\

\hline
$\tilde X_{\{\infty_i\}}$ for each $i$ & $\PP^1$-bundle over $\overline{Sh}_{\{\infty_{i-1}, \infty_i\}}$\\

\hline
$\tilde X_{\{\infty_{i-1},\infty_i\}}$ for each $i$ &  $\overline{Sh}_{\{\infty_{i-1}, \infty_i\}}$\\

\hline
$\tilde X_{\{\infty_1, \infty_3\}}$ and $\tilde X_{\{\infty_2, \infty_4\}}$ & $(\PP^1)^2$-bundle over
$\overline{Sh}_{\{\infty_1,\dots, \infty_4\}}$\\
\hline
$\tilde X_{\ttT}$ with $\#\ttT = 3$ & $\PP^1$-bundle over $\overline{Sh}_{\{\infty_1,\dots, \infty_4\}}$\\
\hline
$\tilde X_{\{\infty_1, \dots, \infty_4\}}$ & $\overline{Sh}_{\{\infty_1,\dots, \infty_4\}}$ with Iwahori level at $p$\\
\hline
\end{tabular}
\end{itemize}
\end{example}

We fix an open compact subgroup $K=K^pK_p\subset \GL_2(\AAA^{\infty})$ with $K_p=\GL_2(\cO_F\otimes_{\Z}\widehat{\Z}_p)$ and $K^p$ sufficiently small. For $\ttT\subseteq \Sigma_{\infty}$, we denote by $Y_{\ttT}$ the closed GO-stratum $\bfSh_{K}(G)_{k_0, \ttT}$  as in Subsection~\ref{S:rig-setup}. We put $K_{\ttS(\ttT)}=K^pK_{\ttS(\ttT),p}$ with $K_{\ttS(\ttT),p}$ defined in Subsection~\ref{S:description-GO-strata}. Here, we fix an isomorphism $G_{\ttS(\ttT)}(\AAA^{\infty,p})\simeq \GL_2(\AAA^{\infty,p})$ and regard $K^p$ as a subgroup of $G_{\ttS(\ttT)}(\AAA^{\infty,p})$.
We now combine all the results in this section to compute the cohomology of $Y_{\ttT}$.

\begin{prop}
\label{C:cohomology of GO-strata}
Let 
$F_{\ttS(\ttT)}$ be the reflex  field of $\bfSh_{K_{\ttS(\ttT)}}(G_{\ttS(\ttT)})$,  $k_{\wp}$ the residue field of $\cO_{F_{\ttS(\ttT)}}$ at the  $p$-adic place given by the isomorphism $\iota_p:\C\simeq \Qpb$, and  $d_{\wp} = [k_{\wp}: \FF_p]$. 
In the Grothendieck group of finite-dimensional $\scrH(K^p,\overline \QQ_l)[S_\gothp, S_\gothp^{-1}: \gothp \in \Sigma_p][\Phi_{\wp^2}]$-modules, we have an equality
\begin{small}
\begin{align}
\label{E:cohomology of YT}
\big[ H^\star_{c,\et}(Y_{\ttT, \overline \FF_p}&, \scrL_l^{(\underline k, w)} )\big]
=(-1)^{g-\#\ttS(\ttT)_\infty}
\big[\bigoplus_{\pi \in \scrA_{(\kb,w)}}
 (\pi_{\ttS(\ttT)}^\infty)^{K_{\ttS(\ttT)}}\otimes \tilde \rho_{\pi, l}^{\ttS(\ttT)} \otimes (\overline \QQ_l \oplus\overline \QQ_l(-1))^{\otimes I_\ttT}
\big]
\\
\nonumber
&+ \delta_{\kb,2}\big[ \bigoplus_{\pi\in \scrB_{w}}(\pi_{\ttS(\ttT)}^{\infty})^{K_{\ttS(\ttT)}}\otimes \tilde \rho_{\pi,l}^{\ttS(\ttT)}\big]
 \otimes\bigg( \big[ (\overline \QQ_l \oplus\overline \QQ_l(-1))^{\otimes(\Sigma_{\infty}-\ttT)}\big] - \delta_{\ttT, \emptyset}\big[ \overline \QQ_l \big] \bigg),
\end{align}
\end{small}
where
\begin{itemize}
\item
$\delta_{\underline k, 2} $ is equal to $1$ if all $k_\tau=2$, and $0$ otherwise;

\item
$\delta_{\ttT,\emptyset} $ is equal to $1$ if $\ttT = \emptyset$, and $0$ otherwise;

\item
on each $ \overline \QQ_l(-1)$, $\scrH(K^p,\overline \QQ_l)[S_\gothp, S_\gothp^{-1}; \gothp \in \Sigma_p]$ acts trivially, and $\Phi_{ \wp^2}$ acts by multiplication by $p^{2d_\wp}$;

\item
$\Phi_{\wp^2}$ acts on the left hand side of \eqref{E:cohomology of YT} by $\prod_{\gothp \in \Sigma_p}(\Phi_{\gothp^2})^{d_{\wp}}$ with $\Phi_{\gothp^2}=\Fr_{\gothp}^2\cdot S_{\gothp}^{-1}$ as  defined in Subsection~\ref{S:etale-coh} (using the formula similar to \eqref{E:partial-Frob-Hilbert});

\item for $\pi\in \scrA_{(\kb,w)}$ or $\pi\in \scrB_{w}$, 
$\tilde  \rho_{\pi, l}^{\ttS(\ttT)}$  is isomorphic to $\rho_{\pi, l}^{\ttS(\ttT)}$  as a vector space, and is equipped with a $\Phi_{\wp^2}$-action given by $\rho_{\pi, l}^{\ttS(\ttT)}(\Frob_\wp^2)$ multiplied by the number
\[
\omega_\pi(\underline p)^{d_{\wp}} \cdot \prod_{\gothp\in \Sigma_{p}}\omega_{\pi}(\varpi_{\gothp})^{-d_{\wp}\#\ttS(\ttT)_{\infty/\gothp}/d_{\gothp}}=u\cdot p^{(w-2)(\#\ttS(\ttT)_{\infty}-g)d_{\wp}},
\]
with $u$ a root of unity. Note that when $\pi\in \scrB_{w}$, $\Phi_{\wp^2}$ acts trivially on $\tilde\rho^{\ttS(\ttT)}_{\pi,l}$.
\end{itemize}

Moreover, if Conjecture~\ref{Conj: partial frobenius} holds, the equality \eqref{E:cohomology of YT} holds in the Grothendieck group of finite-dimensional $\scrH(K^p, \overline \QQ_l)[\Phi_{\gothp^2}^{n_\gothp},S_\gothp, S_\gothp^{-1}; \gothp \in\Sigma_p]$-modules, where 
\begin{itemize}
\item
$n_\gothp$ is the smallest positive number $n$ such that $\sigma_\gothp^n\ttT = \ttT$,
\item
in $(\Qlb\oplus \Qlb(-1))^{\otimes I_{\ttT}}$ and $(\Qlb\oplus \Qlb(-1))^{\otimes (\Sigma_{\infty}-\ttT)}$, $\Phi_{\gothp^2}$ acts trivially on $\Qlb$ and on  the copies of $\Qlb(-1)$ which are labeled with elements not in $\Sigma_{\infty/\gothp}$; on the copies of $\Qlb(-1)$'s labeled by elements in $\Sigma_{\infty/\gothp}$,  the action of 
$\Phi_{\gothp^2}^{n_\gothp}$ is the  multiplication by $p^{2n_\gothp}$;

\item
$\Phi_{\gothp^2}^{n_\gothp}$ acts on  $\tilde \rho_{\pi, l}^{\ttS(\ttT)}$
by the action of the $p^{2n_\gothp}$-Frobenius on $\rho_{\pi, l}^{\ttS(\ttT)}$ times the number $\omega_\pi(\varpi_\gothp)^{-n_\gothp(1-\# \ttS_{\infty/\gothp} /d_\gothp)}$.
\end{itemize}
\end{prop}

\begin{proof}
We first remark that the Hecke action of $F^\times$ (viewed as a subgroup of  the center $\AAA_F^{\infty, \times} \subset \GL_2(\AAA_F^{\infty})$) on $H^\star_{c,\et}(Y_{\ttT, \overline \FF_p}, \scrL_l^{(\underline k, w)})$  is given by $(2-w)$th power of the norm. %{\color{blue}{(I doubt this is the correct statement.)}} {\color{green}[It first of all acts trivially on the base Shimura varieties and it changes the level structure by whatever is written their.]}
Hence, to prove the equality in the Proposition, we may consider the submodules on both sides on which $\AAA_F^{\infty,\times}$ acts via the restriction of a fixed Hecke character $\chi_F$ of $F$ whose all archimedean components are given by $x \mapsto x^{w-2}$.
By the discussion in Notation~\ref{N:A(k,w)[chi]}, there exists a Hecke character $\chi \in \scrA_{E, \widetilde \Sigma}^w$ whose restriction to $\AAA_F^\times$  is just $\chi_F$.
Then it follows from  \eqref{E:cohomology GO-strata for HMV v.s that of G''} that 
\begin{align}
 H^\star_{c,\et}(Y_{\ttT, \overline \FF_p}, \scrL_l^{(\underline k, w)} )^{\AAA_F^{\infty,\times} = \chi_F}&=H^{\star}_{c,\et}(\bfSh_{K_{p}}(G)_{\Fpb,\ttT}, \scrL_l^{\underline k, w})^{\AAA_F^{\infty,\times} = \chi_F, K^p=\mathrm{trivial}}\label{E:cohomology-Y_T}
\\ 
=&
 H^\star_{c,\et} (\bfSh_{K''_{\emptyset, p}}(G''_{\emptyset})_{\Fpb, \ttT},
\scrL''^{(\underline k, w)}_{\emptyset, \tSigma, l} )^{\AAA_E^{\infty, \times}
= \chi, K^p=\textrm{trivial}}.\nonumber
\end{align}
By Theorem~\ref{T:GO-Hilbert}(1) and (2), we see that
$
\scrL''^{(\kb,w)}_{\emptyset, \tSigma, l}|_{\bfSh_{K''_{\emptyset, p}}(G''_{\emptyset})_{\Fpb,\ttT}}\simeq \pi_{\ttT}^*(\scrL''^{(\kb,w)}_{\tilde\ttS(\ttT),\tSigma, l}),
$
and hence
\begin{equation}\label{E:direct-image}
R^n\pi_{\ttT,*}\big(\scrL''^{(\kb,w)}_{\emptyset, \tSigma, l}|_{\bfSh_{K''_{\emptyset, p}}(G''_{\emptyset})_{\Fpb,\ttT}}\big)=
\begin{cases}
\scrL''^{(\kb,w)}_{\tilde\ttS(\ttT),\tSigma, l}\otimes \Qlb(-i)^{\oplus \binom{\#I_{\ttT}}{i}} &\text{if $n=2i$ with }0\leq i\leq \#I_{\ttT};\\
0 &\text{otherwise}.
\end{cases}
\end{equation}
Therefore, in the Grothendieck group of finite-dimensional $\scrH(K^p,\overline \QQ_l)[S_\gothp, S_\gothp^{-1}: \gothp \in \Sigma_p][\Phi_{\wp^2}]$-modules, we have
\begin{align*}
& [H^{\star}_{c,\et}(Y_{\ttT, \overline \FF_p}, \scrL_l^{(\underline k, w)} )^{\AAA_F^{\infty,\times} = \chi_F}]\\
=\ & \big[ H^\star_{c,\et} \big(\bfSh_{K''_{\ttS(\ttT),p}}(G''_{\tilde \ttS(\ttT)})_{\overline \FF_p},
R\pi_{\ttT,*}\big(\scrL''^{(\underline k, w)}_{\emptyset, \tSigma, l}|_{\bfSh_{K''_{\emptyset, p}}(G''_{\emptyset})_{\Fpb,\ttT}} \big)\big)^{\AAA_E^{\infty, \times}
= \chi, K^p=\textrm{trivial}}\big]\\
=\
&\big[H^{\star}_{c,\et}\big(\bfSh_{K''_{\ttS(\ttT),p}}(G''_{\tilde \ttS(\ttT)})_{\overline \FF_p},
\scrL''^{(\underline k, w)}_{\tilde\ttS(\ttT), \tSigma, l}\big)^{\AAA_{E}^{\infty,\times}=\chi, K^p=\mathrm{trivial}}\big]
\otimes 
\big[
(\overline \QQ_l \oplus \overline \QQ_l(-1))^{\otimes I_\ttT} \big]
\\
=\ &
(-1)^{g-\#\ttS(\ttT)_\infty}
\big[\bigoplus_{\pi \in \scrA_{(\kb,w)}[\chi_F]}
 (\pi_{\ttS(\ttT)}^\infty)^{K_{\ttS(\ttT)}}\otimes \tilde \rho_{\pi, l}^\ttS \otimes (\overline \QQ_l \oplus\overline \QQ_l(-1))^{\otimes I_\ttT}
\big]
\\
&\ + \delta_{\kb,2}\big[ \bigoplus_{\pi\in \scrB_{w}[\chi_F]}(\pi_{\ttS(\ttT)}^{\infty})^{K_{\ttS(\ttT)}}\otimes \tilde \rho^{\ttS(\ttT)}_{\pi,l}
\big] \otimes\bigg( \big[ (\overline \QQ_l \oplus\overline \QQ_l(-1))^{\otimes(\Sigma_{\infty}-\ttS(\ttT)_{\infty})}\big]-\delta_{\ttS(\ttT),\emptyset}\big[\Qlb\big]\bigg)\\
&\qquad \otimes \big[(\Qlb\oplus \Qlb(-1))^{\otimes I_{\ttT}}\big] .
\end{align*}
Here, we used Theorem~\ref{T:cohomology of Shimura variety} in the last equality.
We remark that $\ttS(\ttT)=\emptyset$ if and only if $\ttT=\emptyset$, and in which case $I_{\ttT}=\emptyset$. 
Now it is clear that the expression above is exactly the $\chi_F$-component of \eqref{E:cohomology of YT}. 
The description of the action of $\Phi_{\wp^2}$ is immediate from the fact that $\Phi_{\wp^2}=\Frob_{\wp}^2S_{p}^{-d_{\wp}}$ and Proposition~\ref{T:total Frobenius action}.

The second part of the Theorem follows from exactly the same argument by using Conjecture~\ref{Conj: partial frobenius} in place of Proposition~\ref{T:total Frobenius action}.
\end{proof}

\section{Computation of the Rigid Cohomology I}
\label{Section:classicality-I}
  
  We will use the same notation as in Subsection~\ref{S:rig-setup} and Proposition~\ref{C:cohomology of GO-strata}.
  We consider  the cohomology group  $H^{\star}_{\rig}(X^{\tor,\ord},\D; \scrF^{(\kb,w)})$ as defined in Subsection~\ref{subsection:rigid-coh}. Note that we have  fixed an open compact subgroup $K=K^pK_p\subseteq\GL_2(\AAA_F^{\infty})$ with $K_p=\GL_2(\cO_F\otimes \Z_p)$, and omitted $K$ from the notation.
   In this section, we will first use the results in Section~4 and 5 to compute it as an element in a certain Grothendieck group.    Then, combining the results in Subsection~3, especially Corollary~\ref{Prop:slopes-ocv}, we prove our main theorem on the classicality of overconvergent cusp  forms.  
   
The second part of the following theorem will not be used later, but it indicates what to expect and it is also a baby version of the computation in the next section.  So we keep it here.

\begin{theorem}[Weak cohomology computation]
\label{T:rigid coh of ordinary}
Let $(\underline{k}, w)$ be a  multiweight.
%We identify $\CC$ with $\overline \QQ_p$ with $\iota_p$ as usual.
 %Let $\scrH(K^p,\Qpb)$ denote the prime-to-$p$ Hecke algebra $\Qpb[K^p\backslash \GL_2(\AAA^{\infty,p}_F)/K^p]$. 
\begin{itemize}
\item[(1)] For each integer $n$, in the Grothendieck group of finite-dimensional $\scrH(K^p,\Qpb)$-modules,   $[H^{n}_{\rig}(X^{\tor,\ord},\D;\F^{(\kb,w)})\otimes_{L_{\wp}}\Qpb]$ is a sum of Hecke modules coming from classical automorphic representations of $\GL_{2}(\AAA_F)$ (including cuspidal representations, one-dimensional representations and Eisenstein series) whose central character is an algebraic Hecke character with archimedean component $N_{F/\Q}^{w-2}$. 

\item[(2)] We have the following equality in the Grothendieck group of finite-dimensional modules over $\scrH(K^p, \Qpb)$:
\begin{equation}\label{E:tame Hecke action on rigid cohomology}
\big[H^\star_\rig(X^{\tor,\ord}, \D; \scrF^{(\underline{k}, w)})\otimes_{L_{\wp}}\Qpb \big] =
(-1)^g\cdot [S_{(\underline{k},w)}(K^p\Iw_p,\Qpb)].
\end{equation}
\end{itemize}
\end{theorem}
\begin{proof}

(1) By the Hecke equivariant  spectral sequence \eqref{E:spectral sequence for ordinary cohomology}, each $H^n_{\rig}(X^{\ord},\D;\F^{(\kb,w)})$ is a sub-quotient of the rigid cohomology groups of GO-strata $Y_{\ttT}$'s. It suffices to prove that, for all $\ttT\subseteq \Sigma_{\infty}$,  each individual rigid cohomology group of $Y_{\ttT}$ is a sum of Hecke modules coming from classical automorphic representations. 
 For $\ttT=\emptyset$, this is clear by standard comparison between rigid and de Rham cohomology and classical theory. 
For $\ttT\neq \emptyset$, we may reduce to a similar problem for  \'etale cohomology of $Y_{\ttT}$ by Proposition~\ref{P:comparison-rig-et}.
Then the required statement follows from  Theorem~\ref{T:GO-Hilbert} and the proof of Proposition~\ref{C:cohomology of GO-strata}. This proves statement (1).

(2) We also identify $\CC$ with $\overline \QQ_l$ via a fixed $\iota_l: \C\simeq \Qlb$.
Computing the tame Hecke action on the ordinary locus is straightforward:
\begin{small}
\begin{align}
\nonumber
&
\big[H^\star_\rig(Y^{\tor,\ord}, \D; \scrF^{(\underline{k}, w)})\otimes_{L_{\wp}}\Qpb \big] 
\\
\nonumber
\xlongequal{\eqref{E:spectral sequence for ordinary cohomology}}&
\sum_{\ttT\subseteq \Sigma_\infty}
(-1)^{\#\ttT}
\big[
H^{\star}_{c,\rig}(Y_\ttT/L_\wp, \scrD^{(\underline k, w)}) \otimes_{L_\wp}\Qpb
\big]
\\
\nonumber
\xlongequal{\textrm{Prop }\ref{P:comparison-rig-et}}
&
\sum_{\ttT\subseteq \Sigma_\infty}
(-1)^{\#\ttT}
\big[
H^{\star}_{c,\et}(Y_{\ttT,\overline\FF_p}, \scrL_l^{(\underline k, w)})\otimes_{L_{\gothl}}\Qlb \big]
\\
\nonumber
\xlongequal{\textrm{Prop. }\ref{C:cohomology of GO-strata}}
&\sum_{\ttT\subseteq \Sigma_\infty}
(-1)^{\#\ttT} \Big(
(-1)^{g-\#\ttS(\ttT)_\infty}
2^{g-\#\ttT}
\big[\bigoplus_{\pi \in \scrA_{(\kb,w)}}
 (\pi_{\ttS(\ttT)}^\infty)^{K_{\ttS(\ttT)}}
\big]
\\ \nonumber
&\qquad\qquad\qquad
+ \delta_{\kb,2}(2^{g-\#\ttT} - \delta_{\ttT, \emptyset} )\big[ \bigoplus_{\pi\in \scrB_{w}}(\pi_{\ttS(\ttT)}^{\infty})^{K_{\ttS(\ttT)}}
\big] \Big)
\\
\label{E:cohomology factor into p}
=&
\sum_{\pi \in \scrA_{(\underline k, w)}}
\big[
(\pi^{\infty,p})^{K^p}
\big]
\prod_{\gothp \in \Sigma_p}
\sum_{\ttT_{\gothp} \subseteq \Sigma_{\infty/\gothp}}
(-1)^{\#\ttT_{\gothp}} (-1)^{\#\Sigma_{\infty/\gothp} - \#\ttS(\ttT)_{
\infty/\gothp}} 2^{\#\Sigma_{\infty/\gothp} -\#\ttT_{\gothp}}
\big[(\pi_{\ttS(\ttT),\gothp} )^{K_{\ttS(\ttT),\gothp}}\big] 
\\
\nonumber
&\qquad +\delta_{\kb,2}
\sum_{\pi \in \scrB_w}
\big[ (\pi^{\infty})^{K}
\big] \big( -1 + \sum_{\ttT \subseteq \Sigma_{\infty}} (-1)^{\#\ttT}
2^{g-\#\ttT} \big).
\end{align}
\end{small}Here when citing Proposition~\ref{C:cohomology of GO-strata}, we have ignored the $\Phi_{\wp^2}$-action and only kept the dimension of the space of Galois representation.
The last equality is just to separate out the contribution from each prime $\gothp \in \Sigma_p$ (note that $(\pi^{\infty,p}_{\ttS(\ttT)})^{K_{\ttS(\ttT)}^p}$ is isomorphic to $(\pi^{\infty,p})^{K^p}$).

We separate the computation for each $\pi \in \scrA_{(\underline k, w)} \cup \scrB_w$.

For $\pi \in \scrB_w$, note that $\sum_{\ttT \subseteq \Sigma_{\infty}} (-1)^{\#\ttT}
2^{g-\#\ttT} = (2-1)^g =1$ by binomial expansion.  So the contribution of $\pi \in \scrB_w$ to $\big[H^\star_\rig(Y^{\tor,\ord}, \D; \scrF^{(\underline{k}, w)}) \big]$ is trivial.  This agrees with the right hand side of \eqref{E:tame Hecke action on rigid cohomology}, where none of such $\pi$ appears.

For $\pi \in \scrA_{(\underline k, w)}$, we need to compute each factor of the product over $\gothp$.
\begin{itemize}
\item
When $\pi_\gothp$ is ramified, $(\pi_{\ttS(\ttT), \gothp})^{K_{\ttS(\ttT),\gothp}}$ is nonzero if and only if $\ttT_{\gothp} = \Sigma_{\infty/\gothp}$ and $\pi_\gothp$ is Steinberg.
In this case, $(\pi_{\ttS(\ttT), \gothp})^{K_{\ttS(\ttT),\gothp}}$ is one-dimensional.
So the factor for $\gothp$ in the product \eqref{E:cohomology factor into p} is nontrivial only when $\ttT_{\gothp} = \Sigma_{\infty/\gothp}$.
It has total contribution of multiplicity $(-1)^{\#\Sigma_{\infty/ \gothp}}$ in this case.

\item
When $\pi_\gothp$ is unramified,
$(\pi_{\ttS(\ttT), \gothp})^{K_{\ttS(\ttT),\gothp}}$ is one-dimensional, unless $\ttT_{\gothp} = \Sigma_{\infty/\gothp}$.
In the latter case, it is zero if $\#\Sigma_{\infty/\gothp}$ is odd and is $2$ if $\#\Sigma_{\infty/\gothp}$ is even, i.e. it is $1+(-1)^{\#\Sigma_{\infty/\gothp}}$.
Also note that $\#\ttS(\ttT)_{\infty/ \gothp}$ is always even unless $\ttS(\ttT)_{\infty/\gothp} = \Sigma_{\infty / \gothp}$ and it is an odd set.
But the latter case is exactly when $(\pi_{\ttS(\ttT), \gothp})^{K_{\ttS(\ttT),\gothp}}$ vanishes.
So we may ignore the term $(-1)^{\#\ttS(\ttT)_{\infty/ \gothp}}$ in computation.
In summary, the factor for $\gothp$ in the product \eqref{E:cohomology factor into p} has total contribution
\begin{align*}
&\sum_{\ttT_{\gothp} \subsetneq \Sigma_{\infty/\gothp}}
(-1)^{\#\ttT_{\gothp}} (-1)^{\#\Sigma_{\infty/\gothp}} \cdot 2^{\#\Sigma_{\infty/\gothp} -\#\ttT_{\gothp}} + \big(1+ (-1)^{\#\Sigma_{\infty/ \gothp}}\big)
\\
=&
\big( (2-1)^{\#\Sigma_{\infty/\gothp}} -1 \big)
+ \big(1+ (-1)^{\#\Sigma_{\infty/ \gothp}}\big)  = 2 \times (-1)^{\#\Sigma_{\infty/ \gothp}}.
\end{align*}
\end{itemize}
This means that the end contribution of $\pi_\gothp$ from \eqref{E:spectral sequence for ordinary cohomology} agrees with its contribution to $S_{(\underline{k}, w)} (K^p\Iw_p,\C)$.
Putting all places above $p$ together proves the Theorem.
\end{proof}

\begin{remark}
(1) By a careful check of cancellations in the spectral sequence \eqref{E:spectral sequence for ordinary cohomology}, it is possible to show that each individual cohomology group $[H^n_\rig(X^{\tor,\ord},\D;\F^{(\kb,w)})]$ does not contain one-dimensional automorphic representations. However, it may indeed contain the tame Hecke spectrum of some Eisenstein series, because the Eisenstein spectrum in $H^n_{c,\rig}(X,\scrD^{(\kb,w)})$ can not be canceled out by cohomology groups of $Y_{\ttT}$ with $\ttT\neq \emptyset$.

(2) In the proof above, we have dropped the action of twisted partial Frobenius.
We will get to a more delicate computation in the next subsection which involves matching the action of partial Frobenius with the action of $U_\gothp$-operators.
\end{remark}

We  arrive at the following very weak version of the classicality of cuspidal overconvergent Hilbert modular forms, and we do not attempt to optimize the bound on slopes for the moment.
The sole purpose is to prove that the  classical cusp forms are Zariski dense in the Kisin-Lai eigencurve.

\begin{prop}[Weak classicality]
\label{P:weak classicality}
Let $f\in S^{\dagger}_{(\kb,w)}(K,\Qpb)$ be an overconvergent eigenform for $\overline \QQ_p[U_\gothp, S_{\gothp}, S_{\gothp}^{-1}: \gothp\in \Sigma_{\gothp}]$.
For $\gothp \in \Sigma_p$, let $\lambda_{\gothp}$ denote the eigenvalue of $f$ for the operator $U_{\gothp}$.
%Put $k_\gothp = \sum_{\tau \in \Sigma_{\infty/\gothp} (k_\tau -1 )}.
Assume that
\begin{equation}
\label{E:weak slope bound}
\sum_{\gothp \in \Sigma_p} \Big(
\val_p(\lambda_\gothp) - \sum_{\tau \in \Sigma_{\infty/\gothp}} \frac{w-k_\tau}{2} \Big)
+g
 < \frac{1}{g}\min_{\tau \in \Sigma_\infty}(k_\tau -1).
\end{equation}
Then $f$ is a classical cusp  form of level $K^p\Iw_{p}$.
\end{prop}

%{\color{blue}{When $(\kb,w)=(2,\cdots, 2;2)$, new forms of Iwahori level at $p$ has slope $0$. Is your proof shows in this case that (the spectrum of) such an eigenform $f$ as in the Proposition may have Iwahori level?}}
%{\color{red}[You are very right.  I think I missed the Frobenius twists on Gysin-embeddings in the spectral sequence 4.9.1; I made very minimal changes in the proof at  appropriate places.  See red changes]}

\begin{proof}
The basic idea is that, when the slope is very small comparing to the weights, there is essentially one term in the spectral sequence \eqref{E:spectral sequence for ordinary cohomology} which can possibly contribute to the corresponding slope.
Moving between various normalizations unfortunately makes the proof appear to be complicated.

%We introduce some notation.  We write superscript $\scrH^{\sph} = \lambda_f$ to denote the subspace where $\scrH^{\sph}$ acts via the same \emph{generalized} eigensystems as $f$.
We use superscript $\prod_\gothp U_\gothp\textrm{-slope}=\sum_{\gothp} \val_p(\lambda_p)$ to denote the subspace where the eigenvalues of $\prod_{\gothp \in \Sigma_p} U_\gothp$ all have $p$-adic valuation $\sum_{\gothp \in \Sigma_p} \val_p(\lambda_p)$.
We aim to show that, under the weight-slope condition \eqref{E:weak slope bound},  the natural embedding
\begin{equation}
\label{E:classical embeds in overconvergent}
S_{(\underline k, w)}(K^p\Iw_{p}, \overline \QQ_p)^{\prod_\gothp U_\gothp\textrm{-slope}=\sum_{\gothp} \val_p(\lambda_p)} \hookrightarrow
S^\dagger_{(\underline k, w)}(K, \overline \QQ_p)^{\prod_\gothp U_\gothp\textrm{-slope}=\sum_{\gothp} \val_p(\lambda_p)}
\end{equation}
is in fact an isomorphism.
It suffices to show that both sides have the same dimension. The Proposition then follows from this.

Let $\scrC^{\bullet}$ be the complex \eqref{Equ:complex} of overconvergent cusp forms.  Consider its subcomplex formed by taking the isotypical part with  $\prod_\gothp U_\gothp\textrm{-slope}=\sum_{\gothp} \val_p(\lambda_p)$. By Corollary~\ref{Prop:slopes-ocv}, only the last term $(S^\dagger_{(\underline k, w)})^{ \prod_\gothp U_\gothp \textrm{-slope} = \sum_{\gothp} \val_p(\lambda_\gothp)}$ is nonzero.
 Hence,   it follows from Theorem~\ref{Theorem:overconvergent} that  in the Grothendieck group of finite-dimensional $\overline \QQ_p[U_\gothp, S_\gothp, S_\gothp^{-1}: \gothp\in \Sigma_p]$-modules, we have 
 \begin{align}
 (S^\dagger_{(\underline k, w)}&)^{ \prod_\gothp U_\gothp \textrm{-slope} = \sum_{\gothp} \val_p(\lambda_\gothp)}=\label{E:overconvergent-grothendieck}
(-1)^g\big[ H^\star(\scrC^\bullet)^{(\prod_\gothp U_\gothp) \textrm{-slope} = \sum_\gothp  \val_p(\lambda_\gothp)} \big]\\
&\qquad \qquad =(-1)^g\big[
H^\star_{\rig}(X^{\tor, \ord},  \D;\scrF^{(\underline k, w)})^{(\prod_\gothp U_\gothp) \textrm{-slope} = \sum_\gothp  \val_p(\lambda_\gothp)}\big].\nonumber
\end{align}
We need to show that the dimension of this (virtue) $\overline \QQ_p$-vector space is the same as the left hand side of \eqref{E:classical embeds in overconvergent} times $(-1)^g$.

Put $N = g!$ (a very divisible number).  We put $\Phi: = \prod_{\gothp \in \Sigma_p} \Phi_{\gothp^2}^N$ with $\Phi_{\gothp^2} = \Fr_\gothp^2/S_\gothp$.
Then the slope condition above on $\prod_{\gothp}U_{\gothp}$ is equivalent to $\Phi \textrm{-slope} = Nwg-2N\sum_\gothp  \val_p(\lambda_\gothp)$, by  Lemma~\ref{Lemma:Frob-U_p}.
Now we argue as in Theorem~\ref{T:rigid coh of ordinary}(2): fixing an isomorphisms $\Qlb\simeq \C\simeq \Qpb$,   we have equalities in the Grothendieck group of finite-dimensional $\overline \QQ_p[\Phi]$-modules:
\begin{small}
\begin{align}
\nonumber
&
\big[H^\star_\rig(X^{\tor,\ord}, \D; \scrF^{(\underline{k}, w)})^{\Phi \textrm{-slope} = Nwg-2N\sum_\gothp  \val_p(\lambda_\gothp)} \big] 
\\
\nonumber
\xlongequal{\eqref{E:spectral sequence for ordinary cohomology}}&
\sum_{\ttT\subseteq \Sigma_\infty}
(-1)^{-\#\ttT}
\big[
H^{\star}_{c,\rig}(Y_\ttT/L_\wp, \scrD^{(\underline k, w)})^{\Phi \textrm{-slope} = Nwg -2N\#\ttT-2N\sum_\gothp \val_p(\lambda_\gothp)} \otimes_{L_{\wp}}\Qpb\big]
\\
\label{E:ord-coh}
\xlongequal{\textrm{Prop }\ref{P:comparison-rig-et}}&
\sum_{\ttT\subseteq \Sigma_\infty}
(-1)^{-\#\ttT}
\big[
H^{\star}_{c,\et}(Y_{\ttT,\overline\FF_p}, \scrL_l^{(\underline k, w)})^{\Phi \textrm{-slope} = Nwg -2N\#\ttT-2N\sum_\gothp \val_p(\lambda_\gothp)} \big].
\end{align}
\end{small}Here, we have to  modify the $\Phi$-slope starting from the first equality by  $-2N\#\ttT$ in order to take account of the action of $\Phi_{\gothp^2}$'s on the \v Cech symbols as described in Proposition~\ref{P:spectral-sequence}(2), which in turn came from the commutation relation between $\Phi_{\gothp^2}$ and Gysin isomorphisms in \eqref{E:Phi-Gysin}. 

We first claim that all terms in \eqref{E:ord-coh} with $\ttT \neq \emptyset$ vanishes.
Note that the slope condition \eqref{E:weak slope bound} implies that $\kb\neq (2,\dots, 2)$. Thus, Proposition \ref{C:cohomology of GO-strata} says that, in the Grothendieck group of finite-dimensional $\overline \QQ_l[\Phi]$-modules, we have
\[
[ H^\star_{c,\et}(Y_{\ttT, \overline \FF_p}, \scrL_l^{(\underline k, w)} )]
=\sum_{\pi \in \scrA_{(\kb,w)}}
(-1)^{g-\#\ttS(\ttT)_\infty}
\big[
 (\pi_{\ttS(\ttT)}^\infty)^{K_{\ttS(\ttT)}}\otimes \tilde \rho_{\pi, l}^{\ttS(\ttT)} \otimes (\overline \QQ_l \oplus\overline \QQ_l(-1))^{\otimes I_\ttT}
\big],
\]
where 
the action of $\Phi$ on each of $\overline \QQ_l(-1)$ is the multiplication by $p^{2N}$, and the action of $\Phi$ on $\tilde\rho_{\pi, l}^{\ttS(\ttT)}$ is given by $\rho_{\pi,l}^{\ttS(\ttT)}(\Frob_{p^{2N}})$ times $u\cdot p^{N(w-2)(\#\ttS(\ttT)_{\infty}-g)}$
with $u$  a root of unity. We will show that, for each $\pi\in \scrA_{(\kb,w)}$,  the slope of $\Phi$ on  $\tilde \rho_{\pi, l}^{\ttS(\ttT)} \otimes (\overline \QQ_l \oplus\overline \QQ_l(-1))^{\otimes I_\ttT}$ is always strictly smaller than
\[
N \sum_{\tau \in \Sigma_\infty} k_\tau 
+2N\#I_\ttT
- \frac{2N}{g} \min_{\tau \in \Sigma_\infty} (k_\tau-1),
\]
which is easily seen to be strictly smaller than  $Nwg-2N\#\ttT
-2N\sum_\gothp \val_p(\lambda_\gothp)$ under our assumption~\eqref{E:weak slope bound} (and with the fact that $ \#\ttT+\#I_{\ttT}=\#\ttS(\ttT)_{\infty}\leq g$). This would then imply that all terms in \eqref{E:ord-coh} with $\ttT \neq \emptyset$ is zero. 
Since the $p$-adic valuation of the number $u\cdot p^{N(w-2)(\#\ttS(\ttT)_{\infty}-g)}$ is $N (w-2)(\#\ttS(\ttT)_\infty - g)$,
it remains to show that the $\rho_{\pi, l}^{\ttS(\ttT)}(\Frob_{p^{2N}})$ has slope strictly smaller than
\begin{equation}\label{E:slope-Frob-2N}
N(w-2)(g- \#\ttS(\ttT)_\infty ) + N \sum_{\tau \in \Sigma_\infty} k_\tau -\frac{2N}{g}\min_{\tau \in \Sigma_\infty} (k_\tau-1).
\end{equation}

We claim that, for each $\gothp \in \Sigma_p$ with $\ttS(\ttT)_{\infty/\gothp}\neq \Sigma_{\infty/\gothp}$, the slope of $\Frob_{p^{2N}}$ on the unramified $l$-adic Galois representation $\bigotimes_{\Sigma_{\infty/\gothp}-\ttS(\ttT)_{/\gothp}}\textrm{-}\Ind_{\Gal_{F_\gothp}}^{\Gal_{\Qp}}(\rho_{\pi,l}|_{\Gal_{F_{\gothp}}})$ of $\Gal_{F_{\ttS,\wp}}$ is
less than or equal to
\[
N\sum_{\tau \in \Sigma_{\infty/\gothp}}(w+k_\tau-2) \big(1-\frac{\#\ttS(\ttT)_{\infty/\gothp}}{d_{\gothp}} \big).
\]
%Here $\WD_{\gothp}(\rho_{\pi,l})$ denotes the Weil-Deligne representation at $\gothp $ of $\rho_{\pi,l}$.
Indeed, let $\alpha_\gothp, \beta_\gothp$ denote the eigenvalues of $\Frob_\gothp$ on $\rho_{\pi,l}$ with $\val_p(\alpha_\gothp) \leq \val_p(\beta_\gothp)$.
Using the admissibility condition of the corresponding $p$-adic representation of $\Gal_{F_{\gothp}}$, we have
\begin{equation}\label{E:weak-adm-bound}
\val_p(\beta_\gothp) \leq \sum_{\tau \in \Sigma_{\infty/\gothp} }\frac{w+k_\tau-2}{2}.
\end{equation}
Therefore, the slope of $\Frob_{p^{2N}}$ acting on $\otimes_{\Sigma_{\infty/\gothp}-\ttS(\ttT)_{\infty/\gothp}}\textrm{-}\Ind_{\Gal_{F_\gothp}}^{\Gal_{\Qp}} \rho_{\pi,l}$ is less than or equal to
$$
2N\frac{d_\gothp - \#\ttS(\ttT)_\gothp}{d_\gothp}
\val_p(\beta_\gothp)\leq N\sum_{\tau \in \Sigma_{\infty/\gothp}}(w+k_\tau-2)\big(1-\frac{\#\ttS(\ttT)_{\infty/\gothp}}{d_{\gothp}} \big).
$$
Note that the expression above is automatically zero if $\ttS_{\infty/\gothp}=\Sigma_{\infty/\gothp}$. Hence, summing over all $\gothp\in \Sigma_p$, we see that the eigenvalues of $\rho_{\pi, l}^{\ttS(\ttT)}(\Frob_{p^{2N}})$ have slopes smaller than or equal  to 
\[
N(w-2)(g-\#\ttS(\ttT)_{\infty}) +N\sum_{\tau\in \Sigma_{\infty}}k_{\tau}-N\sum_{\gothp\in \Sigma_p}\frac{\#\ttS(\ttT)_{\infty/\gothp}}{d_{\gothp}}\sum_{\tau\in \Sigma_{\infty/\gothp}}k_{\tau},
\]
which is strictly smaller than \eqref{E:slope-Frob-2N} due to the very loose inequality
\[
\sum_{\gothp\in \Sigma_p}\frac{N \#\ttS(\ttT)_{\infty/\gothp}}{d_\gothp} \sum_{\tau \in \Sigma_{\infty/\gothp}}k_\tau>  \frac{2N}{g} \min_{\tau \in \Sigma_{\infty/\gothp}}(k_\tau-1).
\]
Therefore, all terms in \eqref{E:ord-coh} with $\ttT \neq \emptyset$ are zero.
Hence, in view of \eqref{E:overconvergent-grothendieck}, we get
\[
\big[(S^\dagger_{(\underline k, w)})^{ \prod_\gothp U_\gothp \textrm{-slope} = \sum_{\gothp} \val_p(\lambda_\gothp)}\big]
= 
(-1)^g\big[
H^{\star}_{c,\et}(X_{\overline\FF_p}, \scrL_l^{(\underline k, w)})^{\Phi \textrm{-slope} = Nwg-2N\sum_\gothp \val_p(\lambda_\gothp)} \big], 
\]
where the action of $\Phi$ on the left hand side is  given by $p^{2gN}(\prod_{\gothp} S_{\gothp}/U^2_{\gothp})^N$.

Similar to the argument above, Proposition~\ref{C:cohomology of GO-strata} implies that (note that $\underline k \neq(2,\dots,2)$)
\begin{align}
\label{E:ord-coh=total coh}
(-1)^g\big[
\  &H^{\star}_{c,\et}(X_{\overline\FF_p}, \scrL_l^{(\underline k, w)})^{\Phi \textrm{-slope} = Nwg-2N\sum_\gothp \val_p(\lambda_\gothp)} \big]\nonumber \\
& = 
\big[
\bigoplus_{\pi \in \scrA_{(\kb,w)}}
 (\pi^{\infty,p})^{K^p}\otimes \pi_p^{K_p}\otimes (\tilde \rho_{\pi, l}^{\emptyset})^{\Phi \textrm{-slope} = Nwg-2N\sum_\gothp \val_p(\lambda_\gothp)} \big],
\end{align}
where $\tilde \rho_{\pi, l}^{\emptyset}$ is isomorphic to $ \bigotimes_{\Sigma_\infty}\textrm{-}\Ind_{\Gal_F}^{\Gal_\QQ}(\rho_{\pi,l})$ but with $\Phi$-action given by $\Frob_{p^{2N}}$ times a number $u \cdot p^{-(w-2)gN}$.
In order to conclude that \eqref{E:classical embeds in overconvergent} is an isomorphism, we  need to show that the right hand side above has the same dimension as 
\[
\big[S_{(\underline k, w)}(K^p\Iw_p , \overline \QQ_p)^{\prod_\gothp U_\gothp\textrm{-slope}=\sum_{\gothp} \val_p(\lambda_p)}\big]=\big[\bigoplus_{\pi\in \scrA^{(\kb,w)}}
(\pi^{\infty,p})^{K^p}\otimes \big(\pi_p^{\Iw_p}\big)^{\prod_\gothp U_\gothp\textrm{-slope}=\sum_{\gothp} \val_p(\lambda_p)}\big].
\] 
It suffices to show that we have 
\begin{equation}\label{E:dim-comparison}
\dim \big(\pi_p^{K_p}\otimes (\tilde \rho_{\pi, l}^{\emptyset})^{\Phi \textrm{-slope} = Nwg-2N\sum_\gothp \val_p(\lambda_\gothp)}\big)=\dim \big(\pi_p^{\Iw_p}\big)^{\prod_\gothp U_\gothp\textrm{-slope}=\sum_{\gothp} \val_p(\lambda_p)}
\end{equation}
for every $\pi\in \scrA^{(\kb,w)}$.
 Note that if $\pi$ is Steinberg or supercuspidal at some places $\gothp\in \Sigma_p$, then the both sides above are equal to $0$ due to the slope condition \eqref{E:weak slope bound}. Assume therefore that $\pi$ is hyperspecial at $p$. Then $\pi_p^{K_p}$ is one-dimensional.  
  Let $\alpha_{\gothp}$ and $\beta_{\gothp}$ be the eigenvalues of $\Frob_{\gothp}$ on $\rho_{\pi,l}$ with $\val_p(\alpha_{\gothp})\leq \val_{p}(\beta_{\gothp})$ for any $\gothp\in \Sigma_p$.  
     Then $\alpha_{\gothp}$ and $\beta_{\gothp}$ also coincide with eigenvalues of $U_{\gothp}$ on $\pi_{\gothp}^{\Iw_{\gothp}}$ by Eichler-Shimura congruence relation.
     Hence, the  $\prod_{\gothp}U_{\gothp}$-slopes on $\pi_p^{\Iw_p}$ take values in $\gothS(U_p):=\{\sum_{\gothp\in \Sigma_p}\val_{p}(\gamma_{\gothp}): \gamma_{\gothp}\in \{\alpha_{\gothp},\beta_{\gothp}\},\gothp\in \Sigma_p\}$, 
     and the slopes of $\Phi$ take values in the set 
      $$
      \gothS(\Phi):=\Big\{
    2N\sum_{\gothp\in \Sigma_p}\frac{1}{d_{\gothp}}\big[\val_p(\alpha_{\gothp})i_{\gothp}+(d_{\gothp}-i_{\gothp})\val_{\gothp}(\beta_{\gothp})\big] - Ng(w-2)
      : i_{\gothp}\in \Z\cap [0, d_{\gothp}]\Big\}.
      $$
   Then if $\sum_{\gothp}\val_p(\lambda_{\gothp})$ is not in the set
    \[
   \widetilde \gothS(U_p):=\gothS(U_{p})\cup
    \{\sum_{\gothp}\val_p(\alpha_{\gothp})+\sum_{\gothp}\frac{i_{\gothp}}{d_\gothp}(\val_p(\beta_{\gothp})-\val_p(\alpha_{\gothp})): i_{\gothp}\in \Z\cap[0,d_{\gothp}], \forall \gothp\in \Sigma_p
   \}
    \] 
    then both sides of \eqref{E:dim-comparison} are equal to $0$. 
  Assume therefore $\sum_{\gothp}\val_p(\lambda_{\gothp})\in\tilde \gothS(U_p)$ and \eqref{E:weak slope bound} is satisfied. Note  that $\val_p(\alpha_{\gothp})<\val_p(\beta_{\gothp})$ for all $\gothp\in \Sigma_p$, since $\val_p(\alpha_{\gothp})+\val_p(\beta_{\gothp})=(w-1)d_{\gothp}$.  We claim that $\val_p(\lambda_{\gothp})=\sum_{\gothp}\val_p(\alpha_{\gothp})$, which is the minimal element of $\widetilde \gothS(U_{p})$. It would then follow that  both sides of \eqref{E:dim-comparison} have dimension $1$, and the proof will be finished. 
To prove the claim, it suffices to show that 
     \[
\sum_{\gothp\in \Sigma_{p}}\val_p(\alpha_{\gothp})+\min_{\gothp \in \Sigma_p}\frac{1}{d_{\gothp}} \big(\val_{p}(\beta_{\gothp})-\val_p(\alpha_{\gothp}) \big)>\sum_{\tau\in\Sigma_{\infty}}\frac{w-k_{\tau}}{2}+\frac{1}{g}\min_{\tau \in \Sigma_\infty}(k_{\tau}-1)-g,
     \] 
where  the right hand side is greater than $\sum_{\gothp}\val_p(\lambda_{\gothp})$ by assumption. Let $\gothp_0\in \Sigma_p$ be the $p$-adic place where the minimal of the left hand side is achieved. Since $\val_p(\alpha_{\gothp})\geq \sum_{\tau\in \Sigma_{\infty/\gothp}}\frac{w-k_{\tau}}{2}$ by admissibility condition and $\val_p(\alpha_{\gothp_0})+\val_{p}(\beta_{\gothp_0})=d_{\gothp_0}(w-1)$, we are easily reduced to showing that
          \[
\sum_{\tau\in \Sigma_{\infty/\gothp_{0}}}\frac{k_{\tau}-1}{2}-\big(\frac{1}{2}-\frac{1}{d_{\gothp_0}}\big) \big[\val_p({\beta_{\gothp_0}})-\val_p(\alpha_{\gothp_0})\big]>\frac{1}{g}\min_{\tau\in \Sigma_{\infty}}(k_{\tau}-1)-g,
     \]  
which is  trivially true if $d_{\gothp_0}\leq 2$, and follows easily from  $\val_p(\beta_{\gothp_0})-\val_{p}(\alpha_{\gothp_0})\leq \sum_{\tau\in \Sigma_{\infty/\gothp_0}}(k_{\tau}-1)$ if $d_{\gothp_0}>2$.
     \end{proof}

\subsection{Overconvergent Eigenforms of level $K_1(\gothN)$}
Let $\gothN$ be an integral ideal of $\cO_F$ prime to $p$. We put 
\[
K_1(\gothN)=
\bigg\{
\begin{pmatrix}a &b\\c&d\end{pmatrix}\in \GL_2(\widehat{\cO}_F)|a\equiv 1, c\equiv 0 \mod \gothN
\bigg\},
\]
and let $K_1(\gothN)^p$ be the prime-to-$p$ part. Since $K_1(\gothN)^p$ does not satisfy  Hypothesis~\ref{H:fine-moduli}, the theory in Section 3 does not apply directly. 
By Lemma~\ref{L:open-subgroup}, we choose an open compact subgroup $K^p\subseteq K_{1}(\gothN)^p$ that  satisfies Hypothesis~\ref{H:fine-moduli}. 
Consider the space $S^{\dagger}_{(\kb,w)}(K, L_{\wp})$ of overconvergent cusp forms of level $K=K^pK_p$ for a sufficiently large finite extension $L_{\wp}/\Q_p$. We have a natural action of  $\Gamma:=K_1(\gothN)^p/K^p$ on $S^{\dagger}_{(\kb,w)}(K,L_{\wp})$. We define the space of overconvergent cusp forms of level $K_1(\gothN)$ to be the invariant subspace 
\[
S^{\dagger}_{(\kb,w)}(K_1(\gothN), L_{\wp}):=S^{\dagger}_{(\kb,w)}(K,L_{\wp})^{\Gamma}.
\]
It is easy to see that the definition does not depend on the choice of $K^p$.
 The Hecke algebra $\scrH(K_1(\gothN)^p, L_{\wp})$ and the operators  $U_{\gothp}, S_{\gothp}$ for $ \gothp\in \Sigma_p$  act naturally on $S^{\dagger}_{(\kb,w)}(K_1(\gothN), L_{\wp})$. 
 We say $f\in S^{\dagger}_{(\kb,w)}(K_1(\gothN), L_{\wp})$ is a \emph{(normalized) overconvergent eigenform} if $f$ is a common eigenvector for all the Hecke operators in $\scrH'(K_1(\gothN)^p, L_{\wp})[U_{\gothp}, S_{\gothp}, S_{\gothp}^{-1}: \gothp\in \Sigma_p]$, and the first Fourier coefficient (the coefficient indexed by $1\in \cO_F$) is $1$. Here,  $\scrH'(K_1(\gothN)^p,L_{\wp})$ is the subalgebra of $\scrH(K_1(\gothN)^p,L_{\wp})$ generated by the usual Hecke operators $T_v$ for $v\nmid p\gothN$, $U_{v}$ for $v|\gothN$, together with  $S_v$ and $S_v^{-1}$ for all $v\nmid p$.

 Let $\widetilde{S}_{(\kb,w)}(K_1(\gothN)^p\Iw_p,L_{\wp})$ denote the subspace of classical Hilbert modular forms which vanish at the unramified cusps of the Hilbert modular variety of level $K_1(\gothN)^p\Iw_p$.
There are  natural Hecke equivariant injections  
 \[
S_{(\kb,w)}(K_1(\gothN)^p\Iw_p, L_{\wp})\hra \widetilde S_{(\kb,w)}(K_1(\gothN)^p\Iw_p, L_{\wp})\hra S^{\dagger}_{(\kb,w)}(K_{1}(\gothN), L_{\wp}),
 \]
where $S_{(\kb,w)}(K_1(\gothN)^p\Iw_p, L_{\wp})$ is the space of the classical Hilbert cusp forms.
%, and  $\widetilde S_{(\kb,w)}(K_1(\gothN)^p\Iw_p, L_{\wp})$ denotes the space of classical Hilbert modular forms of level $K_1(\gothN)$ (or $\Gamma_1(\gothN)$ in the classical language) that vanish at all unramified cusps of the Hilbert modular variety of level $K_1(\gothN)^p\Iw_p$.
% It is clear that $\widetilde S_{(\kb,w)}(K_1(\gothN)^p\Iw_p, L_{\wp})$ contains the subspace $ S^{(\kb,w)}(K_1(\gothN)^p\Iw_p, L_{\wp})$  of classical cusp forms. 
We will say a form $f\in S^{\dagger}_{(\kb,w)}(K_1(\gothN), L_{\wp})$ is a \emph{classical Hilbert modular form} if it lies in $\widetilde S^{(\kb,w)}(K_1(\gothN)^p\Iw_p, L_{\wp})$, and is a \emph{classical Hilbert cusp form} if it is in $S_{(\kb,w)}(K_1(\gothN)^p\Iw_p, L_{\wp})$.

 Note that  if  $(\kb,w)$ is not of parallel weight,  then  $\widetilde S_{(\kb,w)}(K_1(\gothN)^p\Iw_p,L_{\wp})$ coincides with $S_{(\kb,w)}(K_1(\gothN)^p\Iw_p, L_{\wp})$; but if $(\kb,w)$ is of parallel weight $k$, then $\widetilde S_{(\kb,w)}(K_1(\gothN)^p\Iw_p,L_{\wp})$ will contain 
  as well some Eisenstein series of level $K_1(\gothN)^p\Iw_{p}$ of $U_{\gothp}$-slope $d_{\gothp}(k-1)$ for all $\gothp\in\Sigma_p$. 
 Indeed, let $\chi$ be an algebraic Hecke character of $F$ with archimedean component given by $N_{F/\Q}^{k-2}$ and of conductor $\gothc$ dividing $\gothN$.  Write $\chi=\epsilon|\cdot|^{k-2}$ with $\epsilon$ a finite Hecke character. 
 Then there exists an Eisenstein series $E_{\chi}$ of weight $k$ such that $E_{\chi}$ is a common eigenvector of  $T_{\gothq}$ with eigenvalue $1+\epsilon(\gothq^{-1})N_{F/\Q}(\gothq)^{k-1}$ for all prime ideals $\gothq\nmid \gothc$. We take the $p$-stabilization $E_{\chi}'$ of $E_{\chi}$ such that $E_{\chi}'$ has level $K_1(\gothN)^p\Iw_p$ and it is a common eigenvector of  $U_{\gothp}$ with eigenvalue $\epsilon(\gothp^{-1}) N_{F/\Q}(\gothp)^{w-1}$ for all $\gothp\in \Sigma_p$. Then $E_{\chi}'$ vanishes at all unramified cusps at $p$ of the Hilbert modular variety of level $K_1(\gothN)^p\Iw_p$, and hence is contained in $\widetilde S_{(\kb,w)}(K_1(\gothN)^p\Iw_p,L_{\wp})$.

Recall that Kisin and Lai constructed in \cite{KL} various eigencurves $\calC_{(\kb,w)}(\bar\rho)$ that parametrize (normalized) overconvergent eigenforms with different weights.
 The points on the Kisin-Lai eigencurves that correspond to classical Hilbert eigenforms are called classical points.

\begin{theorem}
\label{C:zariski density}
On the Kisin-Lai eigencurves for overconvergent cusp  forms, classical points are Zariski dense.
\end{theorem}
\begin{proof}
This follows immediately from Proposition~\ref{P:weak classicality} by the same arguments as \cite[Theorem 2.20]{tian}.
\end{proof}
\begin{remark}
This Theorem also follows from the main results of \cite{ps}, where Pilloni and Stroh proved the classicality of overconvergent Hilbert modular forms using the method of analytic continuation.
\end{remark}

The following combines the work of many people.

\begin{prop}
\label{P:Galois representation for an overconvergent Hilbert modular form}
Let $f\in S^{\dagger}_{(\kb,w)}(K_1(\gothN), L_{\wp})$ be an overconvergent eigenform.
Then there exists  a $p$-adic  Galois representation $\rho_f$ of $G_F$ attached  to $f$ such that the following properties are satisfied:
\begin{itemize}
\item For every finite place $\gothl \nmid p\gothN$,  if $\lambda_\gothl$ denotes the eigenvalue of the Hecke operator $T_\gothl$ on $f$, then $\rho_f$ is unramified at $\gothl$ and the characteristic polynomial of $\rho_f(\Frob_\gothl)$ is $T^2 - \lambda_\gothl T + \epsilon(\gothl)N_{F/\Q}(\gothl)^{w-1}$, where $\epsilon(\gothl)$ is a root of unity and is equal to the eigenvalue of $S_{\gothl}$ on $f$ divided by $N_{F/\Q}(\gothl)^{w-2}$.

\item  For a place $\gothp \in \Sigma_p$, if $\lambda_\gothp$ is the eigenvalue of the $U_\gothp$-operator, then the local Galois representation $\rho_f|_{\Gal_{F_{\gothp}}}$ is Hodge-Tate  with  Hodge-Tate weights $\frac{w-k_\tau}{2}, \frac{w+k_\tau-2}{2}$ in the $\tau$-direction for each $\tau \in \Sigma_{\infty/\gothp}$, 
$\DD_\cris(\rho_f|_{\Gal_{F_{\gothp}}})^{\Frob_{\gothp} = \lambda_{\gothp}}$ is nonzero and its image in $\DD_{\dR, \tau}(\rho_f|_{\Gal_{F_{\gothp}}})$ lies in $\Fil^{(w-k_{\tau})/2}\DD_{\dR, \tau}(\rho_f|_{\Gal _F})$.

\item If $f$ is classical, then $\rho_f$ is semistable (including crystalline) at all places $\gothp \in \Sigma_p$.
\end{itemize}
\end{prop}
\begin{proof}
When $f$ is classical, the construction of $\rho_f$ is due to Carayol \cite{carayol-hilbert}, Taylor \cite{taylor}, and Blasius-Rogawski \cite{blasius-rogawski}. The verification of the properties for $\rho_f$ was done by Carayol \cite{carayol-hilbert} for places outside $p$  and    by  Saito \cite{saito} (plus a special case handled independently by T. Liu \cite{liutong} and Skinner \cite{skinner}) at places above $p$). 
For a general $f$, we consider a Kisin-Lai eigencurve $\calC$ that passes through $f$. 
Then the continuity of the Hecke eigenvalues define a pseudo-representations over the reduced subscheme of $\calC$.
Specializing this pseudo-representation to the point corresponding to $f$ provides $f$ with a Galois representation of $\Gal_F$.

The existence of crystalline periods can be proved using the recent work of Kedlaya, Pottharst, and the second author \cite{kedlaya-pottharst-xiao}, or independently R. Liu \cite{liu} on global triangulation.  Both works generalize the prior work of Kisin \cite{kisin}.
\end{proof}

%We consider the spherical Hecke algebra $\scrH^{\sph}=\bigotimes_{v} \Qpb[K_v\backslash \GL_2(F_{v})/K_v]$, where $v$ runs through the set of finite places of $F$ prime to $p$ such that  the open compact subgroup $K_v$ is hyperspecial.

%Let $\scrH^{\sph}$ be the sub-algebra of $\scrH(K_1(\gothN)^p,L_{\wp})$ generated by the operators $T_v=K_v\bigg(\begin{smallmatrix}\varpi_v^{-1} & 0\\0&1\end{smallmatrix}\bigg)K_v$, $S_{v}=\bigg(\begin{smallmatrix} \varpi_{v}^{-1} &0\\ 0&\varpi_{v}^{-1}\end{smallmatrix}\bigg)K_v$ and $S_{v}^{-1}$ for $v\nmid p\gothN$, where is an id\`ele  which is a uniformizer at $v$ and equals to $1$ at other places. Then $\scrH^{\sph}$ is commutative, and we call it spherical part of $\scrH(K_1(\gothN)^p,L_{\wp})$.

\begin{cor}
\label{C:classicality checked by tame eigenvalues}
Let $f\in S^{\dagger}_{(\kb,w)}(K_1(\gothN), L_{\wp})$ be an overconvergent eigenform. Assume that there exists a classical eigenform $\tilde f\in \widetilde{S}_{(\kb,w)}(K_1(\gothN)^p\Iw_p, L_{\wp})$ such that $f$ and $\tilde f$ have the same eigensystem for $\scrH(K_1(\gothN)^p,L_\wp)$. Then $f$ lies in $\widetilde S_{(\kb,w)}(K_1(\gothN)^p\Iw_p, L_{\wp})$.
\end{cor}
\begin{proof}
Let $\pi_{\tilde f}$ be the automorphic representation generated by $\tilde f$. Then $\pi_{\tilde f}$ has conductor $\gothc$ dividing $p\gothN$. We denote by $\Delta_{\tilde f}(\gothN)$ the set of  $K_1(\gothN)^p\Iw_{p}$-eigenforms contained in $\pi_{\tilde f}$,
 i.e. the set of the various $\gothq$-stabilizations of the newform in $\pi_{\tilde f}$ with  $\gothq$  dividing $p\gothN/\gothc$.
Since $f$ and $\tilde f$ have the same tame   Hecke eigensystem,  the (semisimple) $p$-adic Galois representation  $\rho_f$ is isomorphic to $\rho_{\tilde f}$ (or more canonically $\rho_{\pi_{\tilde f}}$) by Chebotarev density. 
In particular, $\DD_{\cris}(\rho_{f}|_{\Gal_{F_{\gothp}}})\simeq \DD_{\cris}(\rho_{\tilde f}|_{\Gal_{F_{\gothp}}})$ for every $\gothp\in \Sigma_p$. 
If $\lambda_{\gothp}(f)$ denotes the eigenvalue of $U_{\gothp}$ on $f$, then $\lambda_{\gothp}(f)$ appears as an eigenvalue of $\Frob_{\gothp}$ on $\DD_{\cris}(\rho_{\tilde f}|_{\Gal_{F_{\gothp}}})$ by Proposition~\ref{P:Galois representation for an overconvergent Hilbert modular form}.
  Then there exists $\tilde f'\in \Delta_{\tilde f}$ such that $\lambda_{\gothp}(f)=\lambda_{\gothp}(\tilde f')$ for all primes $\gothp\in \Sigma_p$. We conclude by the $q$-expansion principle that $f=\tilde f'$. 
\end{proof}

\begin{theorem}[Strong classicality]
\label{T:strong classicality}
Let $f$ be a cuspidal overconvergent Hilbert eigenform of multiweight $(\underline{k}, w)$ of  level $K_1(\gothN)=K_1(\gothN)^pK_p$ with $K_p=\GL_2(\cO_F\otimes\Z_p)$. 
Let
$$
d^{g-1}: \scrC^{g-1}=\bigoplus_{\tau\in\Sigma_{\infty}}S^{\dagger}_{(s_{\Sigma_{\infty}-\{\tau\}}\cdot\kb,w)}(K_1(\gothN),L_{\wp})\xra{\sum_{\tau\in \Sigma_{\infty}}
\Theta_{\tau, k_\tau-1}} \scrC^{g}=S^{\dagger}_{(\kb,w)}(K_1(\gothN),L_{\wp})
$$
denote the $(g-1)$st differential morphism of the complex $\scrC^\bullet$ \eqref{Equ:complex}. 
\begin{enumerate}
\item[(1)] If $f$ is not in the image  of $d^{g-1}$, then $f$ lies in $\widetilde S_{(\kb,w)}(K_1(\gothN)^p\Iw_p, L_{\wp})$.

 \item[(2)] For each $\gothp \in \Sigma_p$, let $\lambda_\gothp$ denote the eigenvalue of $f$ for the operator $U_\gothp$.
If 
\begin{equation}\label{E:strong-slope-bound}
\val_p(\lambda_\gothp) < \sum_{\tau \in \Sigma_{\infty/\gothp}} \frac{w-k_\tau}{2}+ \min_{\tau \in \Sigma_{\infty/\gothp}}\{k_\tau - 1\}
\end{equation}
 for each $\gothp \in \Sigma_\gothp$, then $f$ is lies in $S_{(\kb,w)}(K_1(\gothN)^p \Iw_p,L_{\wp})$.
 \end{enumerate}
\end{theorem}
\begin{proof}
(1) By Corollary~\ref{C:classicality checked by tame eigenvalues}, it suffices to prove the tame Hecke eigenvalues of $f$ coincide with those of a classical cuspidal eigenform. 
By  Theorem~\ref{Theorem:overconvergent}, $f$ gives rise to a non-zero cohomology class in $H^{g}_{\rig}(X^{\tor,\ord}, \D;\F^{(\kb,w)})$.  It follows then from Theorem~\ref{T:rigid coh of ordinary}(1) that 
 the tame Hecke eigenvalues of $f$ come from an automorphic representation $\pi$ of $\GL_{2,F}$ whose central character is an algebraic Hecke character with archimedean part $N_{F/\Q}^{w-2}$. Such a $\pi$  might be cuspidal, one-dimensional or Eisenstein.
 We have to exclude the case of one-dimensional representation, and then (1) would follow from Corollary~\ref{C:classicality checked by tame eigenvalues}.
 Assume in contrary that $f$ has the same tame Hecke spectrum as a one-dimensional automorphic representation:
 $$
 \GL_2(\AAA_F)\xra{\det}\AAA^{\times}_F\xra{\chi}\C^{\times}.
 $$ 
 Then $\chi$ is an algebraic Hecke character whose restriction to $(F\otimes \R)^{\times, \circ}$ is $N_{F/\Q}^{w/2-1}$. Via the fixed isomorphism $\iota_p:\C\simeq \Qpb$,  we have a well-defined $p$-adic character on $\AAA_F^{\times}$
  $$
  \chi_{p}: x\mapsto (\chi(x)\cdot N_{F/\Q}(x_{\infty})^{1-\frac w2}) N_{F/\Q}(x_{p})^{\frac w2-1}\in \Qpb^{\times},
  $$
  where  $x_{\infty}$ and $x_p$ are respectively the archimedean and the $p$-adic component of $x$. Note that $\chi_{p}$ is trivial on $(F\otimes \R)^{\times, \circ}\cdot F^{\times}$. By class field theory, it defines a $p$-adic character of $\Gal_F$. We put $\rho_{\pi,p}=\chi_p^{-1}$.
   Then $\rho_{\pi, p}$ is a one-dimensional Galois representation of $\Gal_F$ such that if $\gothl\nmid p\gothN$ is a place of $F$, $\Tr(\rho_{\pi, p}(\Frob_{\gothl}))$ is equal to the eigenvalue of the Hecke operator $T_{\gothl}$ on $f$.
   However, Proposition~\ref{P:Galois representation for an overconvergent Hilbert modular form} implies that  there is a two-dimensional $p$-adic representation $\rho_f$ of $\Gal_F$ satisfying the same property. Hence, we have $\Tr(\rho_{\pi,p}(\Frob_{\gothl}))=\Tr(\rho_f(\Frob_{\gothl}))$ for all unramified primes $\gothl$. By Chebotarev density, this implies that the semi-simplification of $\rho_f$ is equal to $\rho_{\pi,p}$, which is absurd. This finishes the proof of the first part.

(2) If $f$ satisfies the condition \eqref{E:strong-slope-bound}, then Corollary~\ref{Prop:slopes-ocv} implies that $f$ is not  in the image of $d^{g-1}$. Hence, $f$ is a classical (cuspidal) Hilbert eigenform by the first part of the Theorem. Such an $f$ can not be Eisenstein either, because an Eisenstein series of  (parallel) weight $(\kb,w)$ appearing in  $H^{g}_{\rig}(X^{\tor,\ord},\D;\F^{(\kb,w)})$ must have $(\prod_{\gothp\in \Sigma_p}U_{\gothp})$-slope $(w-1)g$, contradicting with the inequality \eqref{E:strong-slope-bound}. Thus, $f$ must be a classical cuspidal Hilbert modular form.
\end{proof}

%\begin{remark}
%The proof above of the strong classicality relies on  the $q$-expansion principle.
%With some care, one avoid  using it if one assumes the slopes of the classical spectrum does not satisfy certain (exceptional) equality.
%However, we feel that the correct way to avoid $q$-expansion principle is to consider the partial Frobenii as we will discuss in the next section.
%This will be particularly helpful when considering overconvergent automorphic forms over quaternionic Shimura varieties, where the $q$-expansion principle  is not available.
%Moreover, when extending current result to general PEL type Shimura varieties, the analogous of Corollary~\ref{C:classicality checked by tame eigenvalues} does not follow easily.  One may have to refer back to the computation in the next section.
%\end{remark}

\section{Computation of the Rigid Cohomology II}
\label{Section:classicality-II}

We keep the notation of the previous section. 
If we assume that Conjecture~\ref{Conj: partial frobenius} holds (e.g. when $p$ is inert in $F$ by Proposition~\ref{T:total Frobenius action}(3)), we can strengthen the weak cohomology computation Theorem~\ref{T:rigid coh of ordinary} to a stronger version by including the action of $U_\gothp^2$'s.

\begin{theorem}[Strong cohomology computation]
\label{T:rigid coh of ordinary strong}
Assume Conjecture~\ref{Conj: partial frobenius}.
We have the following isomorphism of modules in the Grothendieck group of finite-dimensional $\scrH(K^p, L_{\wp})[S_\gothp,S_\gothp^{-1}, U_\gothp^2: \gothp \in \Sigma_p]$-modules.
\begin{equation}
\label{E:rigid coh of ordinary strong}
\big[H^*_\rig(X^{\tor,\ord}, \D; \scrF^{(\underline{k}, w)}) \big] =
(-1)^g\cdot  \big[ S_{(\underline{k},w)}(K^p\Iw_p,L_{\wp}) \big].
\end{equation}
\end{theorem}

Before giving the proof of this theorem,
we deduce a corollary which is slightly stronger than   Theorem~\ref{T:strong classicality}
on the classicality of overconvergent cusp  forms.

\begin{cor}
\label{C:stronger classicality}
Assume Conjecture~\ref{Conj: partial frobenius}.
Let $(\underline{k}, w)$ be a  multiweight.
For every $\gothp\in \Sigma_p$, we put $s_\gothp = \sum_{\tau \in \Sigma_{\infty/\gothp}} \frac{w-k_\tau}{2}+ \min_{\tau \in \Sigma_{\infty/\gothp}}\{k_\tau - 1\}$.
Then the natural injection
\[
S_{(\underline{k},w)}(K^p\Iw_p,L_{\wp})^ {U_\gothp \textrm{-slope} < s_\gothp} \hookrightarrow
S^\dagger_{(\underline{k}
,w)}(K,L_{\wp})^ {U_\gothp \textrm{-slope} < s_\gothp}
\]
is an isomorphism, where the superscript means taking the part of $U_{\gothp}$-slope strictly less than $s_{\gothp}$ for every $\gothp\in\Sigma_p$.
In particular, if $f$ is an overconvergent cusp  form of multiweight $(\underline k,w)$ and level $K^pK_p$ with $K_p=\GL_2(\calO_F \otimes \ZZ_p)$, which is a generalized eigenvector of    $U_\gothp$  with slope strictly less than $s_\gothp$ for every $\gothp \in \Sigma_p$, then $f$ is classical.
\end{cor}
\begin{proof}
It suffices to compare the dimensions.
First of all, Theorem~\ref{Theorem:overconvergent} implies that
\begin{equation}
\label{E:BGG=rigid}
(\scrC^\bullet)^{U_\gothp \textrm{-slope} < s_\gothp} \cong \big(\mathrm{R}\Gamma_\rig(Y^{\tor, \ord}, \D; \scrF^{(\underline{k}, w)})\big)^ {U_\gothp \textrm{-slope} < s_\gothp}.
\end{equation}
By the slope inequality (Corollary~\ref{Prop:slopes-ocv}), the left hand side of \eqref{E:BGG=rigid} has only one term $(S^\dagger_{(\underline{k}
,w)})^{U_\gothp \textrm{-slope} < s_\gothp}$ in degree $g$.
So the right hand side of \eqref{E:BGG=rigid} is also concentrated in degree $g$.
By Theorem~\ref{T:rigid coh of ordinary strong} above, the right hand side of \eqref{E:BGG=rigid} in the Grothendieck group of finite-dimensional $ L_{\wp}[U_\gothp^2; \gothp \in \Sigma_p]$-modules is equal to $(-1)^g\cdot \big[ S_{(\underline{k},w)}(K^p\Iw_p)^ {U_\gothp \textrm{-slope} < s_\gothp} \big]$.
It follows that  $[(S^\dagger_{(\underline{k}
,w)})^{U_\gothp \textrm{-slope} < s_\gothp}]=[S_{(\underline{k},w)}(K^p\Iw_p)^ {U_\gothp \textrm{-slope} < s_\gothp}]$, and in particular
\[
\dim
(S^\dagger_{(\underline{k}
,w)})^ {U_\gothp \textrm{-slope} < s_\gothp} =\dim
S_{(\underline{k},w)}(K^p\Iw_p)^ {U_\gothp \textrm{-slope} < s_\gothp}.
\]
Therefore, the natural inclusion $S_{(\kb,w)}(K^p\Iw_p)^ {U_\gothp \textrm{-slope} < s_\gothp}\hra S^{\dagger}_{(\kb,w)}(K^pK_p)^ {U_\gothp \textrm{-slope} < s_\gothp}$ must be an isomorphism.
\end{proof}

\begin{remark}
Note that Corollary~\ref{C:stronger classicality} ensures classicality without assuming $f$ to be an eigenform for the tame Hecke operators. Moreover, its proof does not involve any argument from Section~\ref{Section:classicality-I}, especially with no reference to eigencurve and strong multiplicity one result.
\end{remark}

The proof of Theorem~\ref{T:rigid coh of ordinary strong} will occupy rest of the paper. We start by breaking up the theorem into a local computation at each place $\gothp \in \Sigma_p$.

\subsection{Reduction of the proof of Theorem~\ref{T:rigid coh of ordinary strong}}
\label{S:first reduction for the main theorem}
Recall that we have introduced the \v Cech symbols $e_{\tau}$ and $e_{\ttT}$ in Subsections \ref{S:cech-symbol} and \ref{P:spectral-sequence} to encode the action of  $\Fr_{\gothp}$ and $\Phi_{\gothp^2}$.
We need  more \v Cech symbols for the second relative cohomology of the $\PP^1$-bundle and the $\Phi_{\gothp^2}$-action on them.
For each $\tau \in \Sigma_\infty$,
 we introduce a \emph{\v Cech symbol} $f_\tau$ of degree $2$. 
  This means that $f_\tau \wedge f_{\tau'} = f_{\tau'} \wedge f_{\tau}$ and $f_\tau \wedge e_{\tau'} = e_{\tau'} \wedge f_\tau$ for all $\tau, \tau' \in \Sigma_\infty$.
For a subset $\ttT \subseteq \Sigma_\infty$, we put $f_\ttT = \bigwedge_{\tau \in \ttT} f_\tau$ (for a chosen order of $\ttT$).
The action of $\Phi_{\gothp^2}$ on the formal \v Cech symbols is defined to be 
\[
\Phi_{\gothp^2}(e_\tau) = p^2 e_{\sigma^{-2}\tau} \quad \textrm{and} \quad 
\Phi_{\gothp^2}(f_\tau) = p^2 f_{\sigma^{-2}\tau},
\]
for $\tau\in \Sigma_{\infty/\gothp}$,
 and is trivial on  $e_{\tau}$ and $f_{\tau}$ for $\tau\notin \Sigma_{\infty/\gothp}$.
For a subset $\ttT\subset \Sigma_{\infty}$, let $\ttS(\ttT)$ be the subset defined in Subsection~\ref{S:description-GO-strata}, and put $I_\ttT = \ttS(\ttT)_{\infty} - \ttT$ and $I_{\ttT,\gothp}=I_{\ttT}\cap \Sigma_{\infty/\gothp}$.
We fix isomorphisms $\Qpb\simeq \C\simeq \Qlb$ as usual.
Assuming Conjecture~\ref{Conj: partial frobenius}, 
we can conduct a computation finer than  Theorem~\ref{T:rigid coh of ordinary}
in the Grothendieck group of finite-dimensional $\scrH(K^p,\Qpb)[\Phi_{\gothp^2}; \gothp\in \Sigma_p]$-modules: 
\begin{small}
\begin{align*}
&\qquad \qquad
(-1)^g\big[H^\star_\rig(Y^{\tor,\ord}, \D; \scrF^{(\underline{k}, w)}) \otimes_{L_{\wp}}\Qpb\big] 
\\&\ 
\xlongequal{\eqref{E:spectral sequence for ordinary cohomology}}
\sum_{i=0}^g
(-1)^{g-i}
\big[ \bigoplus_{\#\ttT=i}
H^{\star}_{c,\rig}(Y_\ttT/L_\wp, \scrD^{(\underline k, w)})\otimes_{L_{\wp}}\Qpb e_\ttT \big]
\\&
\xlongequal{\textrm{Prop }\ref{P:comparison-rig-et}}
\sum_{i=0}^g
(-1)^{g-i}
\Big[ \bigoplus_{\#\ttT=i}
H^{\star}_{c,\et}(Y_{\ttT,\overline\FF_p}, \scrL_l^{(\underline k, w)})e_\ttT \Big]
\\&
\xlongequal{\textrm{Prop }\ref{C:cohomology of GO-strata}}
\sum_{i=0}^g
(-1)^{g-i}\sum_{j=0}^{i}(-1)^{g-i-j}
\Big[ \bigoplus_{\substack{\#\ttT=i,\\
\textrm{s.t.}\#I_{\ttT}=j}}
e_\ttT
\bigoplus_{\pi \in \scrA_{(\kb,w)}}
 (\pi_{\ttS(\ttT)}^\infty)^{K_{\ttS(\ttT)}}
\otimes \tilde \rho_{\pi,l}^{\ttS(\ttT)} \otimes
\bigotimes_{\tau \in I_{\ttT}}
(\overline \QQ_l \oplus \overline \QQ_l f_\tau)
\Big]
\\&\qquad \qquad
+ \sum_{i=0}^g
(-1)^{g-i}\delta_{\kb,2}
\Big[ \bigoplus_{\#\ttT=i}e_\ttT \bigoplus_{\pi\in \scrB_{w}}(\pi_{\ttS(\ttT)}^{\infty})^{K_{\ttS(\ttT)}}
\otimes \big(
\bigotimes_{\tau \in \Sigma_\infty-\ttT} (\overline \QQ_l \oplus \overline \QQ_l f_\tau) - \delta_{\ttT, \emptyset} \overline \QQ_l \big)
\Big]
\\
&\ =
\sum_{\pi \in \scrA_{(\underline k, w)}}\!\!\!\!
\big[
(\pi^{\infty,p})^{K^p}
\big]
\prod_{\gothp \in \Sigma_p}
\sum_{i_\gothp=0}^{d_\gothp}
\sum_{j_\gothp=0}^{i_\gothp}(-1)^{j_\gothp} 
\Big[\!\!\!\!
\bigoplus_{\substack{\#\ttT_{/\gothp}=i_\gothp,\\\textrm{s.t.}\#I_{\ttT,\gothp} = j_\gothp}}\!\!\!\!\!\!
e_{\ttT_{/\gothp}}  \tilde \rho_{\pi,\gothp,l}^{\ttS(\ttT)_{/\gothp}} \otimes
(\pi_{\ttS(\ttT),\gothp} )^{K_{\ttS(\ttT),\gothp}}
\otimes\!\! \bigotimes_{\tau \in I_{\ttT,\gothp}} (\overline \QQ_l \oplus \overline \QQ_l f_\tau)
\Big] 
\\
&\qquad +\delta_{\kb,2} (-1)^g
\sum_{\pi \in \scrB_w}
\big[ (\pi^{\infty})^{K}
\big] \Big( -[\overline \QQ_l] +
\prod_{\gothp\in\Sigma_p}
\sum_{i_\gothp=0}^{d_\gothp}
(-1)^{i_\gothp}
\big[
\bigoplus_{\#\ttT_{/\gothp}=i_\gothp}
e_{\ttT_{/\gothp}}\!\!\!\!\!\!
\bigotimes_{\tau \in \Sigma_{\infty/\gothp}-\ttT_{/\gothp}} \!\!\!(\overline \QQ_l \oplus \overline \QQ_l f_\tau)\; \big]
\Big),
\end{align*}
\end{small}where $\tilde \rho_{\pi,\gothp,l}^{\ttS(\ttT)_{/\gothp}}$ is isomorphic to $\bigotimes_{\Sigma_{\infty/\gothp} - \ttS(\ttT)_{/\gothp}}\textrm{-}\Ind_{\Gal_{F_\gothp}}^{\Gal_{\QQ_p}} \rho_{\pi, l}|_{\Gal_{F_\gothp}}$
but it carries an action of $\Phi_{\gothp^2}^{n_\gothp}$ , which is the same as the action of $\Frob_{p^{2n_\gothp}}$ on the tensorial induction representation multiplied with the number $\omega_\pi(\varpi_\gothp)^{n_\gothp(d_\gothp-\#\ttS(\ttT)_{/\gothp})/d_\gothp}$.
Here $n_\gothp$ denotes the minimal number such that $\sigma_\gothp^{n_\gothp}\ttT_{/\gothp} = \ttT_{/\gothp}$.

We will show that the long expression above is equal to
\[
\sum_{\pi \in \scrA_{(\underline k, w)}}
[(\pi^{\infty})^{K^p\Iw_p}]
=  \sum_{\pi \in \scrA_{(\underline k, w)}}
[(\pi^{\infty, p})^{K^p}]
\prod_{\gothp \in \Sigma_p}
\big[(\pi_\gothp)^{\Iw_\gothp} \big].
\]
For this, we need only to discuss separately for each $\pi \in \scrA_{(\underline k, w)}$ and $\scrB_w$.
We start with the latter where the computation is much easier.

\subsection{Contribution of the one-dimensional representations}
\label{S:contribution of one-dim repns}
We fix $\pi \in \scrB_w$.
By the discussion in the Subsection above, we need to show that the contribution of the $\pi$-component to $\big[ H^\star_\rig(Y^{\tor,\ord},\D;\scrF^{(\underline k, w)}) \big]$ is trivial.
For this, it is enough to show that, for each $\gothp \in \Sigma_p$, we have, in the Grothendieck group of finite-dimensional $\overline \QQ_l[\Phi_{\gothp}]$-modules,
\begin{equation}
\label{E:combinatorial equality in one-dim case}
\sum_{i_\gothp=0}^{d_\gothp}
(-1)^{i_\gothp}
\Big[
\bigoplus_{\#\ttT_{/\gothp}=i_\gothp}
e_{\ttT_{/\gothp}}
\bigotimes_{\tau \in \Sigma_{\infty/\gothp}-\ttT_{/\gothp}} (\overline \QQ_l \oplus \overline \QQ_l f_\tau) \Big]
\end{equation}
is equal to $[\overline \QQ_l]$, where $\Phi_\gothp$ acts on $\overline \QQ_l$ trivially and
\[
\Phi_\gothp(e_\tau) = pe_{\sigma_\gothp^{-1}\tau}, \quad \textrm{and}\quad
\Phi_\gothp(f_\tau) = pf_{\sigma_\gothp^{-1}\tau}.
\]
(We get back to our original statement by matching the action of $\Phi_{\gothp^2}$ with $\Phi_\gothp^2$.) 
Since the argument will be independent for each place; we will suppress the subscript $\gothp$ for the rest of this subsection.
We also label $e_\tau$ and $f_\tau$'s as $e_1, \dots, e_d$ and $f_1, \dots, f_d$ with the convention that the subscripts are taken modulo $d$ and $\Phi(e_i) = pe_{i-1}, \Phi(f_i) = pf_{i-1}$.

%{\color{blue}[Compare with 5.9]}
This will follow from some abstract nonsense of tensorial induction.
Recall the setup in Subsection~\ref{S:tensorial-induction}: $G$ a group, $H$ a subgroup of finite index, $\Sigma \subseteq G/H$ a finite subset, and $H'$ the stabilizer group of $\Sigma$ under the action of $G$ on $G/H$.
We fix representatives $s_1,\dots, s_r$ of $G/H$ and assume that $\Sigma = \{s_1 H, \dots, s_r H\}$ for some $r$.
Instead of starting with a representation of $H$, we start with a bounded complex $C^\bullet$ of $\ZZ[H]$-modules.
The \emph{tensorial induced complex}, denoted by $\otimes_\Sigma\textrm{-}\Ind_H^G C^\bullet$, is defined to be $\otimes_{i=1}^r C^\bullet_i$, where $C^\bullet_i$ is a copy of $C^\bullet$.  The action of $H'$ on this complex is given as follows: 
for each  $h'\in H'$, there exists a permutation $j$ of $r$ letters (depending on $h'$) such that for each $i\in \{1, \dots, r\}$,  $j(i) \in \{1,\dots, r\}$ is  the unique element  with  $h's_{j(i)} \in s_i H$. We define the action of $h'$ on $\otimes_\Sigma\textrm{-}\Ind_H^G C^\bullet$ by the linear combination of
\[
\xymatrix@R=0pt{
h' \colon
C_1^{a_1} \otimes \cdots \otimes C_r^{a_r} \ar[rr] &&
C_{1}^{a_{j(1)}} \otimes \cdots \otimes C_{r}^{a_{j(r)}}\\
h'(v_1\otimes \cdots \otimes v_r)
\ar[rr]&&
\mathrm{sgn}(j) \cdot 
(s_1^{-1} h' s_{j(1)})(v_{j(1)}) \otimes \cdots \otimes (s_r^{-1} h' s_{j(r)})(v_{j(r)}).
}
\]
\emph{Note that the sign function for $j$ is inserted to account for the sign convention in the definition of tensor the product of complexes.}
Such a construction does not depend on the choice of the coset representatives $s_i$'s and is functorial in $C^\bullet$. Moreover, it takes quasi-isomorphic complexes to quasi-isomorphic complexes (this is because when checking quasi-isomorphism one can forget about the group action.)

We will consider a particular case of the   tensorial induction:
$G=\ZZ$, $H=d\ZZ$, $\Sigma = G/H = \ZZ/d\ZZ$, and the complex is given by
$\overline \QQ_l \oplus \overline \QQ_l f \to \overline \QQ_l e$ in degrees $0$ and $1$, where the element $d\in H$ acts trivially on $\overline \QQ_l$ and acts by multiplication by $1/p$ on both $e$ and $f$, and the map is given by sending the first copy of $\overline \QQ_l$ to zero and sending $f$ to $e$.
It is clear that we have a quasi-isomorphism
\[
\overline \QQ_l \cong \otimes_\Sigma\textrm{-}\Ind_H^G \overline \QQ_l \xrightarrow{\textrm{quasi-isom}}
\otimes_\Sigma\textrm{-}\Ind_H^G\big( \overline \QQ_l \oplus \overline \QQ_l f \twoheadrightarrow \overline \QQ_l e\big).
\]
The upshot is that if we think of the action of $-1 \in G$ is the action of $\Phi$ from \eqref{E:combinatorial equality in one-dim case}, the expression \eqref{E:combinatorial equality in one-dim case} is exactly the image of the complex $\otimes_\Sigma\textrm{-}\Ind_H^G\big( \overline \QQ_l \oplus \overline \QQ_l f \twoheadrightarrow \overline \QQ_l e\big)$ in the Grothendieck group of finite-dimensional $\overline \QQ_l[\Phi]$-modules.
Then Theorem~\ref{T:rigid coh of ordinary strong} for one-dimensional $\pi$ would follow from this.

We now examine carefully the construction of the tensorial induction for $\overline \QQ_l \oplus \overline \QQ_l f \twoheadrightarrow \overline \QQ_l e$.
It is first of all isomorphic to $\otimes_{i=1}^d\big(\overline \QQ_l \oplus \overline \QQ_l f_i \twoheadrightarrow \overline \QQ_l e_i\big)$.  To properly account for the sign involved in tensorial induction, we need to declare that $e_i$ has degree $1$ as a \v Cech symbol.
The statement is now clear.

\subsection{Contribution of the cuspidal representations}

To prove the final statement in Subsection~\ref{S:first reduction for the main theorem} for a representation $\pi \in \scrA_{(\underline k,w)}$, we need to prove that, for each $\gothp \in \Sigma_p$, in the Grothendieck group of finite-dimensional $\overline \QQ_l[\Phi_{\gothp^2}]$-modules, we have an equality
\begin{equation}
\label{E:equality in the cuspidal case}
\big[(\pi_\gothp)^{\Iw_\gothp}\big] = 
\sum_{i_\gothp=0}^{d_\gothp}
\sum_{j_\gothp=0}^{i_\gothp}(-1)^{j_\gothp} 
\Big[
\bigoplus_{\substack{\#\ttT_{/\gothp}=i_\gothp,\\\textrm{s.t.}\#I_{\ttT,\gothp} = j_\gothp}}\!\!\!\!
e_{\ttT_{/\gothp}}  \tilde \rho_{\pi,\gothp,l}^{\ttS(\ttT)_{/\gothp}} \otimes
(\pi_{\ttS(\ttT),\gothp} )^{K_{\ttS(\ttT),\gothp}}
\otimes \bigotimes_{\tau \in I_{\ttT,\gothp}} (\overline \QQ_l \oplus \overline \QQ_l f_\tau) \Big],
\end{equation}
where $\Phi_{\gothp^2}$ acts on $(\pi_\gothp)^{\Iw_\gothp}$ by $N_{F/\Q}(\gothp)^2 S_\gothp / U_\gothp^2$.

When $\pi_\gothp$ is ramified, the left hand side \eqref{E:equality in the cuspidal case} is nonzero if and only if $\pi_\gothp$ is Steinberg, in which case it is one-dimensional with trivial $\Phi_{\gothp^2}$-action.
All terms on the right hand side of \eqref{E:equality in the cuspidal case} is zero except when $\ttT_{/\gothp} = \Sigma_{\infty/\gothp}$ and hence $I_{\ttT,\gothp} =\ttS(\ttT)_{\infty/\gothp}-\ttT_{/\gothp}=\emptyset$, in which case the contribution is one-dimensional with trivial $\Phi_{\gothp^2}$ if $\pi_\gothp$ is Steinberg and zero otherwise.  Both sides agree.

It is then left to prove \eqref{E:equality in the cuspidal case} when $\pi_\gothp$ is unramified.
Let $\alpha$ and $\beta$ be two eigenvalues of $\Frob_\gothp$-action on $\rho_{\pi, l}$ and let $\{v,w\}$ be a (generalized) eigenbasis  corresponding to the two eigenvalues respectively. Then the $S_\gothp$-eigenvalue on $\pi_\gothp$ is $\omega_\pi(\varpi_\gothp^{-1}) =\alpha \beta /p^{d_{\gothp}} $.
We take a square root $\lambda$ of $\alpha \beta / p^{d_\gothp}$ and put $\alpha_0 = \alpha / \lambda$ and $\beta_0 =\beta/ \lambda$ so that $\alpha_0  \beta_0 = p^{d_\gothp}$.
Then $\Phi_{\gothp^2}$ acts on $(\pi_\gothp)^{\Iw_p}$
with eigenvalues $\alpha^2_0$ and $\beta^2_0$.
We need to match this with the right hand side of \eqref{E:equality in the cuspidal case}.

We write $d=d_{\gothp}$, $i=i_{\gothp}$, $j=j_{\gothp}$,  and $\Phi=\Phi_{\gothp}$ to simplify the notation. 
We  label $\Sigma_{\infty/\gothp}$ by $\{\tau_1, \dots, \tau_d\}$ so that $\sigma(\tau_i) = \tau_{i+1}$ and $\tau_i = \tau_{i \pmod d}$.
We now try to rewrite the right hand of \eqref{E:equality in the cuspidal case} so that it is easier to work with.
We write $\tilde \rho_{\pi, \gothp, l}^{\ttS(\ttT)}$ as the vector space
$
\otimes_{\tau \in \Sigma_{\infty/\gothp}-\ttS(\ttT)_{\infty/\gothp}} (\Qlb v_{\tau} \oplus \Qlb w_{\tau})
$,
with the convention that the operator $\Phi$ acts on symbols by
\[
\Phi (v_{\tau_i}) = 
\left\{
\begin{array}{ll}
v_{\tau_{i-1}} &\textrm{if } i \neq 1,\\
\alpha_0 v_{\tau_n} &\textrm{if }i =1,
\end{array}
\right.
\quad \textrm{and}\quad
\Phi (w_{\tau_i}) = 
\left\{
\begin{array}{ll}
w_{\tau_{i-1}} &\textrm{if } i \neq 1,\\
\beta_0 w_{\tau_n} &\textrm{if }i =1.
\end{array}
\right.
\]
It is straightforward to check that the action of $\Phi_{\gothp^2}^d$ on $\tilde \rho_{\pi, \gothp, l}^{\ttS(\ttT)}$ is exactly the action of $\Phi^{2d}$ on $
\otimes_{\tau \in \Sigma_{\infty/\gothp}-\ttS(\ttT)_{\infty/\gothp}} (\Qlb v_{\tau} \oplus \Qlb w_{\tau})
$ given as above.
We also view these $v_\tau$'s and $w_\tau$'s as \v Cech symbols of degree $0$; namely, they commute with other \v Cech symbols.

%From now on, we suppress the subscript $\gothp$ from the notation in \eqref{E:equality in the cuspidal case}, and write $\Phi=\Phi_{\gothp}$ to simplify the presentation.
 We can rewrite the right hand side of \eqref{E:equality in the cuspidal case} as 
 \begin{small}
\begin{equation}
\label{E:expression of a bunch of tensors}
\sum_{i=0}^{d}
\sum_{j=0}^{i}(-1)^{j} 
\Big[
\bigoplus_{\substack{\#\ttT_{/\gothp}=i,\\ \textrm{s.t.}\#I_{\ttT,\gothp} = j}}
e_{\ttT}  
\bigotimes_{\tau \in \Sigma_{\infty/\gothp}-\ttS(\ttT)_{\infty/\gothp}} (\Qlb v_{\tau} \oplus \Qlb w_{\tau})
 \otimes
(\pi_{\ttS(\ttT),\gothp} )^{K_{\ttS(\ttT),\gothp}}
\otimes \bigotimes_{\tau \in I_{\ttT,\gothp}} (\overline \QQ_l \oplus \overline \QQ_l f_\tau) \Big],
\end{equation}
\end{small}
where $\Phi$ sends $e_\tau$ to $pe_{\sigma^{-1} \tau}$ and $f_\tau$ to $pf_{\sigma^{-1}\tau}$, and it acts trivially on $(\pi_{\ttS(\ttT),\gothp} )^{K_{\ttS(\ttT),\gothp}}$.
We need to show that \eqref{E:expression of a bunch of tensors}, in the Grothendieck group of finite-dimensional $\Qlb[\Phi]$-modules, is equal to the two-dimensional $\Qlb$-vector space where $\Phi$ acts with eigenvalues $\alpha_0$ and $\beta_0$.

We quickly point out that each term appearing in the expression above is of the form $\overline \QQ_l b_{\tau_1} \wedge \cdots \wedge b_{\tau_d}$, where each $b_{\tau_i}$ is exactly one of $e_{\tau_i}$, $f_{\tau_i}$, $v_{\tau_i}$, $w_{\tau_i}$ or empty.

\subsection{Cyclic words}

We now introduce some combinatorial way to describe terms and their contributions in \eqref{E:expression of a bunch of tensors}.
For each such expression above, we associated a \emph{cyclic word} (of length $d$), that is a 
 word $(a_1\cdots a_d)$ composed of letters of the following kinds:
\begin{itemize}
\item single letters $\alpha$ and $\beta$,
\item or short combinations: $\alpha'\beta'$, $\overline{\alpha \beta}$, and $\overline{\alpha'\beta'}$,
\end{itemize}
with the understanding that the last letter is considered adjacent to the first one, and the convention that $a_r = a_{r \bmod d}$.
The short combinations are viewed as two letters which always come together.
For example, $(\bar \beta \alpha\alpha \beta \bar \alpha)$ is a cyclic word, but $(\bar \alpha\alpha \beta \overline{\alpha \beta})$ is not.
We may \emph{rotate} each cyclic word by changing it from $(a_1\cdots a_d)$ to $(a_2 \cdots a_da_1)$.
The \emph{period} of a cyclic word $w$, denoted by $\per(w)$, is the minimal $r \in \NN$ such that $r$ times rotation of $w$   gives $w$ back.
In this case, $(a_1\cdots a_r)$ can be also viewed as a cyclic word (of length $r$).
We always have $\per(w)|d$.
Two words are called \emph{equivalent} if one may be turned into another using rotations.
For a cyclic word $w$, we use $[w]$ to denote its equivalence class.

To each term $\Qlb b_{\tau_1} \wedge \cdots \wedge b_{\tau_d}$ of \eqref{E:expression of a bunch of tensors}, we associate a cyclic word $(a_1\cdots a_d)$ as follows:
\begin{itemize}
\item
if $b_{\tau_i} = v_{\tau_i}$, we put $a_i = \alpha$;
\item
if $b_{\tau_i} = w_{\tau_i}$,
we put $a_i = \beta$;
\item
if $b_{\tau_i}= \emptyset$, then $b_{\tau_{i+1}}=e_{\tau_{i+1}}$ (i.e. $\tau_{i+1}\in \ttT$) and 
we put $a_ia_{i+1} = \overline{\alpha\beta}$;
\item
if $b_{\tau_i}= f_{\tau_i}$, then $b_{\tau_{i+1}}=e_{\tau_{i+1}}$ (i.e. $\tau_{i+1}\in \ttT$) and 
we put $a_ia_{i+1} = \overline{\alpha'\beta'}$;
\item
by the description of GO-strata (Subsection~\ref{S:description-GO-strata} and Theorem~\ref{T:GO-Hilbert}), the only unassigned $a_i$'s (for which we must have $b_{\tau_i} = e_{\tau_i}$) can be partitioned into disjoint union of pairs $a_j a_{j+1}$, to which we assign $\alpha'\beta'$.  (When $d$ is even and $\ttT_{/\gothp} = \Sigma_{\infty/\gothp}$, we have exactly two such partitions; we assign two cyclic words in this case.  In contrast, when $d$ is odd and $\ttT_{/\gothp}=\Sigma_{\infty/\gothp}$, the term in \eqref{E:expression of a bunch of tensors} is trivial because $(\pi_{\ttS(\ttT),\gothp} )^{K_{\ttS(\ttT),\gothp}}$-term is zero; this agrees with the fact that there is no cyclic words just consisting of short expressions $\alpha'\beta'$ since $d$ is odd.)
\end{itemize}
It is somewhat tedious but straightforward to check that this establishes a one-to-one correspondence between terms in \eqref{E:expression of a bunch of tensors} with all cyclic words of length $d$ discussed above.

We now discuss the contribution of the corresponding terms in \eqref{E:expression of a bunch of tensors} to the Grothendieck group of finite-dimensional $\Qlb[\Phi]$-modules.
For a cyclic word $w$ with period $r = \per(w)$, the contribution of the terms of \eqref{E:expression of a bunch of tensors} corresponding to all elements of $[w]$ is given by an $r$-dimensional representation
\begin{equation}
\label{E:Rw}
R_{[w]} = (-1)^a[\Qlb[\Phi] / (\Phi^r - \lambda)],
\end{equation}
where $a$ is the number of  short combinations
$\overline{\alpha\beta}$ and $\overline{\alpha'\beta'}$ in the cyclic word $(a_1\cdots a_d)$, and 
the number $\lambda$ is the product of
\begin{itemize}
\item
$\alpha_0$ for each $\alpha$ in the cyclic word $(a_1\cdots a_r)$,
\item
$\beta_0$ for each $\beta$ in the cyclic word $(a_1\cdots a_r)$,
\item
$\alpha_0\beta_0 = p^d$ for each $\overline{\alpha\beta}$  in the cyclic word $(a_1\cdots a_r)$,
\item
$p^{2d}$ for each ${\alpha'\beta'}$ and each $\overline{\alpha'\beta'}$  in the cyclic word $(a_1\cdots a_r)$, and
\item
a sign, which is nontrivial if and only if $d/r$ is even and there are odd number of pairs of $\overline {\alpha\beta}$ and $\overline{\alpha'\beta'}$ in $(a_1 \cdots a_r)$.
\end{itemize}
This does not depend on the choice of the representative $w$ in the equivalent class $[w]$.
For example, the representation associated to the equivalence class of $(\alpha \alpha'\beta' \overline{\alpha \beta}\alpha \alpha' \beta'\overline{\alpha \beta})$ is
$\Qlb[\Phi] / (\Phi^5 + p^{30}\alpha_0)$.

We now need to prove that the total contribution of all cyclic words to \eqref{E:expression of a bunch of tensors} is simply
$R_{[\alpha\cdots \alpha]}+R_{[\beta\cdots \beta]}$, which agrees with the contribution from $[(\pi_\gothp)^{\Iw_\gothp}]$.
In other words, we need to show that the contribution of all cyclic words except $(\alpha\cdots \alpha)$ and $(\beta\cdots \beta)$ cancel with each other.
For this, we need to properly group cyclic words together. We introduce some new terminology.

\begin{itemize}
\item
For a cyclic word $w$, its \emph{primitive form}, denoted by $\prim(w)$ is the cyclic word obtained by replacing all $\overline{\alpha\beta}$ by $\alpha\beta$ and all $\overline{\alpha'\beta'}$ by $\alpha' \beta'$.
Equivalent cyclic words have equivalent primitive forms.
We note that a cyclic word always has (nonstrictly) a bigger period than its primitive form, i.e. $\per(w)\geq \per(\prim(w))$.
The upshot of this terminology is that adding overline to either $\alpha\beta$ or $\alpha'\beta'$ will not change the absolute value on $\lambda$ in \eqref{E:Rw} but it will change the sign of 
$R_{[w]}$.

\item
We think of the difference between $\alpha\beta$ and $\overline{\alpha\beta}$ as being ``the conjugate of each other".
The same applies to $\alpha'\beta'$ and $\overline{\alpha'\beta'}$.
Hence we introduce the convention that $\overline{\overline{\alpha\beta}} = \alpha \beta$ and $\overline{\overline{\alpha\beta}} = \alpha \beta$.
The key observation is that, taking the conjugation of a short combination $\alpha\beta$ or $\alpha'\beta'$ will not change the absolute value of $\lambda$ in \eqref{E:Rw} but it will change the sign of 
$R_{[w]}$.
This allows us to cancel the contribution to \eqref{E:expression of a bunch of tensors}.
\end{itemize}

\vspace{5pt}
\underline{\textbf{Claim:}}
We group all cyclic words into packages with the same primitive forms up to equivalence.
For any equivalent class of primitive forms $[w_0]$ with period $\per(w_0) = r \neq 1$, all cyclic words $w$ with $[\prim(w)] = [w_0]$ have zero total contribution to the sum \eqref{E:expression of a bunch of tensors}.
\vspace{5pt}

Since the only cyclic words with period $1$ are $(\alpha\cdots \alpha)$ and $(\beta \cdots \beta)$, this claim would exactly prove our Theorem~\ref{T:rigid coh of ordinary strong} in the cuspidal case.

Since the Claim can be easily checked when $d = 1, 2$,  We assume $d \geq 3$ from now on.
Before proving the claim, we first indicate some simple cases to give the reader some feeling of the argument.

When $r=d$, the claim can be easily deduced as follows.
Let $w_0$ be as in the claim.
In this case, every cyclic word $w$ such that $\prim(w) = w_0$ will also have period equal to $d$.
Moreover, $w_0$ must have at least one adjacent $\alpha\beta$ or a short combination $\alpha'\beta'$.
We fix one such, say at the $i$th and the $(i+1)$st places.
Among all cyclic words whose primitive form is $w_0$, we may pair those which are identical at all places except at the $i$th and the $(i+1)$st, where they are conjugate of each other.
Their equivalent classes contribute the same representation to the sum \eqref{E:expression of a bunch of tensors}, but with different signs (since the signs are determined by the number of pairs of $\overline {\alpha\beta}$ and $\overline{\alpha'\beta'}$).
So the total sum in the Grothendieck group is zero.

A variant argument of this also works if $d/r$ is odd, as follows.
Let $w_0$ be as in the claim.
In this case, if a cyclic word $w$ satisfies $\prim(w) = w_0$, then we only know $r|\per(w)$. But this will not concern us.
We fix an adjacent $\alpha\beta$ or $\alpha'\beta'$ in $w_0$, say at $i$th and $(i+1)$st places.
Without loss of generality, we assume that $i \in \{1, \dots, r\}$.
Then we have adjacent $\alpha\beta$ or $\alpha'\beta'$ at the $(sr+i)$th and $(sr+i+1)$st places of $w_0$ for any $s \in \ZZ$.
For a cyclic word $w = (a_1 \cdots a_r)$ whose primitive form is $w_0$, we define its dual $w^\vee$ to be the following cyclic word
\[
w^\vee = (a_1\cdots \overline{a_ia_{i+1}} \cdots \overline{a_{i+r}a_{i+r+1}} \cdots \cdots \overline{a_{i+d-r} a_{i+d-r+1}} \cdots a_d).
\]
In other words, we take the conjugation of $w$ at the $(sr+i)$th and $(sr+i+1)$st places for all $s\in\ZZ$.
Note that the period of $w$ is still the same as $w^\vee$.
Hence their cyclic equivalent classes have the same contribution to the right hand side of \eqref{E:expression of a bunch of tensors}, but with signs differed by $(-1)^{d/r} = -1$.
So the total contribution is trivial again.

Clearly, a direct generalization of this argument would not work if $d/r$ is even.
We look at an example first.
Let $w_0 = (\alpha\beta\alpha\beta \alpha\beta\alpha\beta)$.
Then the list of cyclic words $w$ with $[\prim(w)] = [w_0]$ and the contribution to \eqref{E:expression of a bunch of tensors} of 
their equivalence classes is given as follows:
\begin{itemize}
\item[(i)] equivalence class of $(\alpha\beta\alpha\beta \alpha\beta\alpha\beta)$ contributes $\big[\Qlb[\Phi]/(\Phi^2 - q) \big]$,
\item[(ii)]  equivalence class of $(\overline{\alpha \beta}\alpha\beta \alpha\beta\alpha\beta)$ contributes $-\big[\Qlb[\Phi]/(\Phi^8 - q^4)\big]$,
\item[(iii)]  equivalence class of $(\overline{\alpha \beta}\overline{\alpha \beta} \alpha\beta\alpha\beta)$ contributes $\big[\Qlb[\Phi]/(\Phi^8 - q^4)\big]$,
\item[(iv)]  equivalence class of $(\overline{\alpha \beta}\alpha\beta \overline{\alpha \beta}\alpha\beta)$ contributes $\big[\Qlb[\Phi]/(\Phi^4 + q^2)\big]$,
\item[(v)]  equivalence class of $(\overline{\alpha \beta}\overline{\alpha \beta}\overline{\alpha \beta}\alpha\beta)$ contributes $-\big[\Qlb[\Phi]/(\Phi^8 - q^4)\big]$, and
\item[(vi)]  equivalence class of $(\overline{\alpha \beta}\overline{\alpha \beta}\overline{\alpha \beta}\overline{\alpha \beta})$ contributes $\big[\Qlb[\Phi]/(\Phi^2 + q)\big]$,
\end{itemize}
where $q = p^8$.
Note that, in (iii) and (v), the sign on power of $p$ is changed according to the last rule in \eqref{E:Rw}.
One sees that the factorization $\Phi^8-q^4 = (\Phi^4+q^2)(\Phi^2+q)(\Phi^2-q)$ is used to prove that the total contribution to the sum \eqref{E:expression of a bunch of tensors} is zero.

We now handle the most general case of the claim. 
We fix an adjacent $\alpha \beta$ or a short combination $\alpha'\beta'$ in $w_0$, at $i$th and $(i+1)$st places.
Without loss of generality, we assume that $i \in \{1, \dots, r\}$.
Assume that $d/r = 2^t s$ for $t \in \NN$ and $s$ odd. 
We fix a positive divisor $s''$ of $s$, and write $s'=s/s''$. 
Let $\CW_{s''}(w_0)$ denote the subset of the cyclic words $w$ of length $d$  whose primitive form is  $w_0$ and its period is of  the form  $\per(w)=2^{t''}s''r$ for some integer $0\leq t''\leq t$. 
For each $j \in \{0, \dots, 2^t-1\}$, we define  an operator $\ttr_{j}$ on $\CW_{s''}(w_0)$ as follows: 
for $w=(a_1\cdots a_d)\in \CW_{s''}(w_0)$,  $\ttr_j(w)=(b_1\cdots b_d)$ is given by  $b_{(m2^t+j)s''r+i}b_{(m2^t+j)s''r+i+1}=\overline{a_{(m2^t+j)s''r+i}a_{(m2^t+j)s''r+i+1}}$ for $0\leq m\leq s'-1$, and  $b_n=a_n$ for any other $n$'s. 
It is easy to see that $\ttr_{j}(w)\in \CW_{s''}(w_0)$, and two elements of $\CW_{s''}(w_0)$ are equivalent if and only if there images under $\ttr_j$ are equivalent.
%We have 
%\[
%\ttr_j(w) = (a_1 \cdots \overline{a_{jr+i} a_{jr+i+1}}\cdots \cdots \overline{a_{(2^t+j)r+i} a_{(2^t+j)r+i+1}}
%\cdots\cdots \overline{a_{((s-1)2^t+j)r+i}a_{((s-1)2^t+j+i+1)}}
%\cdots a_d).
%\]
Note that  $\ttr_j$ has order $2$.
Let  $\gothG_{2^t}\simeq (\Z/2\Z)^{2^t}$ denote the group generated by $\ttr_j$ for $0\leq j\leq 2^t-1$. 
 We now group the cyclic words in $\CW_{s''}(w_0)$ into $\gothG_{2^t}$-orbits, and as well as their equivalence classes. 
 We have the following

\vspace{5pt}
\underline{\textbf{Subclaim:}}
Let $\calW\subseteq \CW_{s''}(w_0)$ be a $\gothG_{2^t}$-orbit, and let $[\calW]$ be the associated set of equivalence classes.
Then the contribution of $[\calW]$ to the sum \eqref{E:expression of a bunch of tensors}  is zero.
\vspace{5pt}

It is clear the claim would follow this Subclaim.
From now on, we fix such a  $\calW$.
Among all periods of cyclic words in $\calW$, there is a minimal one, which we denote by $\tilde r=2^{t''} s''r $. Put $t' = t-t''$ so that  $d = (2^{t'} s') \tilde r$.
For any $w \in \calW$, $\per(w) / \tilde r$ is a power of $2$.
We fix an cyclic word $w^\star \in \calW$ with period $\tilde r$.

The case when $t' =0$ is simple, which we handle first.
This is the case when $d / \tilde r$ is odd.
Hence the periods for all cyclic words $w \in \calW$ are in fact the same.
We consider the action of $\ttr_0$, for example.
The $\Qlb[\Phi]$-module associated to the equivalence class of $w$ is the same as that of $\ttr_0(w)$.
But their contributions in \eqref{E:expression of a bunch of tensors} differ by a sign since $\ttr_0$ takes conjugation at $s'$ places.
The claim is proved in this case.

From now on, we assume that $t' >0$.
We will have to do an explicit computation.
For $k = 0, \dots, t'$, we consider the following set of operators
\[
\mathrm{Op}_k = \big\{
\ttr_j \circ \ttr_{j+2^{t''+k}} \circ \cdots \circ \ttr_{j+ 2^t-2^{t''+k}} \;\big|\; j = 0, \dots, 2^{t''+k}-1\big\},
\]
and  let $\langle \mathrm{Op}_k\rangle\subseteq \gothG_{2^t}$ denote the subgroup generated by $\mathrm{Op}_k$.
We have $\langle\mathrm{Op}_{k-1}\rangle\subseteq \langle\mathrm{Op}_k\rangle$ and $\langle\mathrm{Op}_{t'}\rangle= \gothG_{2^t}$.
Let $\Ob_k$ denote the orbit of $w^\star$ under the action of $\langle\mathrm{Op}_k\rangle$. 
It is not hard to see that, for each $k$, $\Ob_k$ consist of exactly those $w \in \calW$ such that $\per(w) | 2^{k+t''}s''r$.
Put $\Ob_{0}^{\circ}=\Ob_{0}$ and $\mathrm{Ob}_k^\circ = \mathrm{Ob}_k - \mathrm{Ob}_{k-1}$ for $k \geq 1$, and let $[ \Ob_k^\circ]$ denote the associated set of equivalence classes.
Thus $\mathrm{Ob}_k^\circ$ consists of exactly those cyclic words $w \in \calW$ such that $\per(w) = 2^{t''+k}s'' r$, and we have $[\calW]=\bigcup_{k=0}^{t'}[\Ob_{k}^{\circ}]$.
It is clear that the cardinalities of these sets are
\[
\#\mathrm{Ob}_0 = 2^{2^{t''}},\quad
\#\mathrm{Ob}_k^\circ = 2^{2^{t''+k}} - 2^{2^{t''+k-1}} \textrm{ for } k = 1, \dots, t'.
\]

We first look at the contribution from $[\Ob_0]$ to the sum \eqref{E:expression of a bunch of tensors}.
The equivalent class $w^\star$ corresponds to a representation $\Qlb[\Phi]/(\Phi^{\tilde r} - \lambda)$ for some $\lambda \in \Qlb$. 
Then among all  the cyclic words in $\mathrm{Ob}_0$, half of them (i.e. those obtained by applying even number of operators in $\mathrm{Op}_0$ to $w^\star$) correspond to the same representation as  $w^\star $, while the other half correspond to $\Qlb[\Phi]/(\Phi^{\tilde r} + \lambda)$.
  The change of the sign is a result of the last rule in \eqref{E:Rw}.
   We note moreover that any two cyclic words  in  $\Ob_0$ are not equivalent unless they are equal. Hence, the total contribution of $[\Ob_0]$ to \eqref{E:expression of a bunch of tensors} is $2^{2^{t''}-1}$ copies of $\Qlb[\Phi]/(\Phi^{\tilde r} - \lambda)$ and $2^{2^{t''}-1}$ copies of $\Qlb[\Phi]/(\Phi^{\tilde r} + \lambda)$.

Similarly, among the elements of $\mathrm{Ob}_k^\circ$ for $k =1, \dots, t'-1$, $2^{2^{t''+k}-1}$ of them are obtained by applying odd number of operators in $\mathrm{Op}_k$ to $w^\star$, while $2^{2^{t''+k}-1}-2^{2^{t''+k-1}}$ of them are obtained by using even number of operators. 
By the rules in  \eqref{E:Rw}, the representation corresponding to the elements in the first case is $\Qlb[\Phi]/(\Phi^{2^k\tilde r}+\lambda^{2^k})$, and that for the second case is $\Qlb[\Phi]/(\Phi^{2^k\tilde r}-\lambda^{2^k})$. Since every $2^{k}$ elements in $\Ob_k^{\circ}$ give rise to one equivalence class in $[\Ob^{\circ}_k]$, we see that the multiplicities of the two representations above given by  $[\Ob^{\circ}_k]$ are $2^{2^{t''+k}-1}/2^{k}=2^{2^{t''+k}-k-1}$ and $2^{2^{t''+k}-k-1}-2^{2^{t''+k-1}-k}$  respectively.

For $k=t'$, the representation associated to a cyclic equivalent class of a cyclic word $w^\star$ from $\mathrm{Ob}_{t'}^\circ$ is always $\Qlb[\Phi]/(\Phi^{2^{t'}\tilde r}-\lambda^{2^{t'}})$ (note that the index of the period is an odd number $s''$ now).
But the contribution to the sum \eqref{E:expression of a bunch of tensors} is the same as the contribution of $w^\star$ if and only if this element is obtained from $w^\star$ by applying even number of operators from $\mathrm{Op}_{t'}$.

In summary, we list all the representations we see in the following table.
\begin{center}
\begin{tabular}{|c|c|c|c|}
\hline
$k$ & Representation & Multiplicity & Sign of the  Contribution\\
\hline
$0$ & $\Qlb[\Phi]/(\Phi^{\tilde r} - \lambda)$ & $2^{2^{t''}-1}$ & same as $w^\star$\\
\hline
$0$ & $\Qlb[\Phi]/(\Phi^{\tilde r} + \lambda)$ & $2^{2^{t''}-1}$ & same as $w^\star$\\
\hline
$1$ & $\Qlb[\Phi]/(\Phi^{2\tilde r} - \lambda^2)$ & $2^{2^{t''+1} - 1-1} -2^{2^{t''}-1}$ & same as $w^\star$\\
\hline
$1$ & $\Qlb[\Phi]/(\Phi^{2\tilde r} + \lambda^2)$ & $2^{2^{t''+1}-1-1}$ & same as $w^\star$\\
\hline
\multicolumn{4}{|c|}{$\cdots \quad \cdots $}\\
\hline
$t'$ & $\Qlb[\Phi]/(\Phi^{2^{t'}\tilde r} - \lambda^{2^{t'}})$ & $2^{2^{t}-t'-1} - 2^{2^{t-1} -t'- 1}$ & same as $w^\star$\\
\hline
$t'$ & $\Qlb[\Phi]/(\Phi^{2^{t'}\tilde r} - \lambda^{2^{t'}})$ & $2^{2^{t}-t'-1}$ & opposite to $w^\star$\\
\hline
\end{tabular}
\end{center}
Then using  the factorization
\[
\Phi^{2^{t'}\tilde r} - \lambda^{2^{t'}} = (\Phi^{2^{t'-1}\tilde r}+\lambda^{2^{r'-1}})\cdots
(\Phi^{\tilde r} + \lambda)(\Phi^{\tilde r} - \lambda),
\]
it follows immediately
that the total contribution  of all $[\Ob^{\circ}_k]$ for $0\leq k\leq t'$ to \eqref{E:expression of a bunch of tensors} is zero. This finishes the proof of the Subclaim, and hence also Theorem~\ref{T:rigid coh of ordinary strong}.

\begin{remark}
Our proof of Theorem~\ref{T:rigid coh of ordinary strong} in the cuspidal case is very combinatorial.
It would be great if one can give a more conceptual or geometric proof.

If one needs only the action of a high power of $\Phi_{\gothp^2}$, the proof can be significantly simplified.
In fact, this suffices for the application to proving classicality result as stated in Corollary~\ref{C:stronger classicality}.
Nonetheless, we feel that Theorem~\ref{T:rigid coh of ordinary strong} has its own importance, and 
it deserves a trying for the most optimal statement.
\end{remark}

%We have the following weak vanishing theorem of the cohomology of ordinary locus.

%\begin{theorem}[Weak vanishing]
%\label{T:weak vanishing}
%Let $(\underline{k},w)$ be a regular multiweight.
%Let $r_i$ be the number of places above $p$ whose inertia degree is $i$.
%Then we have
%\[
%H^j_\rig(Y^{\tor,\ord}, \D; \scrF^{(\underline{k}, w)})_\cusp = 0 \textrm{ for } j < r_1 + 2r_2 + \sum_{s \geq 3} \lceil r/2\rceil r_s.
%\]
%\end{theorem}
%\begin{proof}
%To be added.
%\end{proof}

%\begin{remark}
%The theorem is a formal consequence of the vanishing of cuspidal cohomology, except that we can improve the coefficients on $r_2$ from $1$ to $2$.
%This seemingly marginal improvement gives us the edge to prove the strong version of classicality (that matches the prediction of Fontaine-Mazur conjecture) under the strong multiplicity one hypothesis for all compact quaternionic Shimura varieties.
%Unfortunately, without the strong multiplicity one result, we can only extend the strong version of classicality under the assumption that all inertia degrees are less than or equal to $2$.
%\end{remark}

\end{document}